\documentclass[pdflatex]{sn-jnl}% Math and Physical Sciences Numbered Reference Style
\usepackage{graphicx}%
\usepackage{multirow}%
\usepackage{amsmath,amssymb,amsfonts}%
\usepackage{amsthm}%
\usepackage{mathrsfs}%
\usepackage[title]{appendix}%
\usepackage{xcolor}%
\usepackage{textcomp}%
\usepackage{manyfoot}%
\usepackage{booktabs}%
\usepackage{algorithm}%
\usepackage{algorithmicx}%
\usepackage{algpseudocode}%
\usepackage{listings}%
\usepackage{bm}%
\usepackage[T1]{fontenc}  % 欧文フォントの割当を最新の規格に変更
\usepackage{lmodern}      % Latin Modern フォントを適用
\numberwithin{equation}{section}

\usepackage{mathtools}
\theoremstyle{thmstyleone}%
\newtheorem{theorem}{Theorem}% meant for sectionwise numbers
\newtheorem{proposition}{Proposition}% 
\newtheorem{lemma}{Lemma}[section]% 

\theoremstyle{thmstyletwo}%
\newtheorem{remark}{Remark}[section]%

\theoremstyle{thmstylethree}%
\newtheorem{definition}{Definition}[section]%

\begin{document}
	
\title[Article Title]{Twisted Gradient Flow Approach to Vegetation-Rainfall-Bushfire Interactions Model}
	
	
\author*[1]{\fnm{Tahir} \sur{Boudjeriou}}\email{t.boudjeriou@univ-boumerdes.dz}

\author[2]{\fnm{Akio} \sur{Ito}}\email{ito.akio.2015@gmail.com}
\equalcont{These authors contributed equally to this work.}
	
%	\author[1,2]{\fnm{Third} \sur{Author}}\email{iiiauthor@gmail.com}
%	\equalcont{These authors contributed equally to this work.}
	
\affil*[1]{\orgdiv{Institute of Electrical and Electronic Engineering}, \orgname{University of Boumerdes}, \orgaddress{\city{Boumerdes}, \postcode{35000}, 
\country{Algeria}}}
	
\affil[2]{\orgdiv{Indepedent Researcher}, \orgaddress{\street{11678-1, Node}, \city{Sosashi}, \postcode{289-3181}, \state{Chiba}, \country{Japan}}}
	
%	\affil[3]{\orgdiv{Department}, \orgname{Organization}, \orgaddress{\street{Street}, \city{City}, \postcode{610101}, \state{State}, \country{Country}}}
	
	
\abstract{We consider a vegetation-rainfall-bushfire interaction model consisting of two partial differential equations describing the evolutions of bushfire intensity 
and water availability, together with an ordinary differential equation governing the vegetation density.
The main feature of the model is that the ODE contains a nonlocal coefficient multiplying the vegetation density, which makes the mathematical analysis highly nontrivial.
In this paper, we regard this nonlocal coefficient as the generator of a deformation of the metric structure of the underlying real Hilbert space.
This interpretation enables us to formulate the ODE as a twisted gradient flow on state-dependent Hilbert spaces and to establish the existence 
of strong solutions to the associated initial-boundary value problem.}
	
	
%35Q92 PDEs in connection with biology, chemistry and other natural sciences
%35A01 Existence problems for PDEs: global existence, local existence, non-existence
%35A15 Variational methods applied to PDEs
%47J35 Nonlinear evolution equations
%35D35 Strong solutions to PDEs
\maketitle
%%%%%%%%%%%%%%%%%%%%%%%%%%%%%%%%%%%%%%%%%%%%%%%%%%%%%%%%%%%%%%%%%%%%%%%%%%%%%%%%%%%%%%%%%%%%%%%%%%%%%%%%%%%%%%%%%%%%%%%%%%%%%%%%%%%%%%%%
%%%%%%%%%%%%%%%%%%%%%%%%%%%%%%%%%%%%%%%%%%%%%%%%%%%%%%%%%%%%%%%%%%%%%%%%%%%%%%%%%%%%%%%%%%%%%%%%%%%%%%%%%%%%%%%%%%%%%%%%%%%%%%%%%%%%%%%%
%%%%%%%%%%%%%%%%%%%%%%%%%%%%%%%%%%%%%%%%%%%%%%%%%%%%%%%%%%%%%%%%%%%%%%%%%%%%%%%%%%%%%%%%%%%%%%%%%%%%%%%%%%%%%%%%%%%%%%%%%%%%%%%%%%%%%%%%
%%%%%%%%%%%%%%%%%%%%%%%%%%%%%%%%%%%%%%%%%%%%%%%%%%%%%%%%%%%%%%%%%%%%%%%%%%%%%%%%%%%%%%%%%%%%%%%%%%%%%%%%%%%%%%%%%%%%%%%%%%%%%%%%%%%%%%%%
%%%%%%%%%%%%%%%%%%%%%%%%%%%%%%%%%%%%%%%%%%%%%%%%%%%%%%%%%%%%%%%%%%%%%%%%%%%%%%%%%%%%%%%%%%%%%%%%%%%%%%%%%%%%%%%%%%%%%%%%%%%%%%%%%%%%%%%%
\section{Introduction}\label{Section-1}
%%%%%%%%%%%%%%%%%%%%%%%%%%%%%%%%%%%%%%%%%%%%%%%%%%%%%%%%%%%%%%%%%%%%%%%%%%%%%%%%%%%%%%%%%%%%%%%%%%%%%%%%%%%%%%%%%%%%%%%%%%%%%%%%%%%%%%%%
%%%%%%%%%%%%%%%%%%%%%%%%%%%%%%%%%%%%%%%%%%%%%%%%%%%%%%%%%%%%%%%%%%%%%%%%%%%%%%%%%%%%%%%%%%%%%%%%%%%%%%%%%%%%%%%%%%%%%%%%%%%%%%%%%%%%%%%%
%%%%%%%%%%%%%%%%%%%%%%%%%%%%%%%%%%%%%%%%%%%%%%%%%%%%%%%%%%%%%%%%%%%%%%%%%%%%%%%%%%%%%%%%%%%%%%%%%%%%%%%%%%%%%%%%%%%%%%%%%%%%%%%%%%%%%%%%
%%%%%%%%%%%%%%%%%%%%%%%%%%%%%%%%%%%%%%%%%%%%%%%%%%%%%%%%%%%%%%%%%%%%%%%%%%%%%%%%%%%%%%%%%%%%%%%%%%%%%%%%%%%%%%%%%%%%%%%%%%%%%%%%%%%%%%%%
%%%%%%%%%%%%%%%%%%%%%%%%%%%%%%%%%%%%%%%%%%%%%%%%%%%%%%%%%%%%%%%%%%%%%%%%%%%%%%%%%%%%%%%%%%%%%%%%%%%%%%%%%%%%%%%%%%%%%%%%%%%%%%%%%%%%%%%%
In this paper, we consider the following initial-boundary value problem, denoted by (P):=\{\eqref{Equation-1-1}--\eqref{Equation-1-6}\} throughout this paper, 
associated with the vegetation-rainfall-bushfire interaction (VRB for short) model introduced by S. Dipierro and E. Valdinoci in \cite{DV-2026} as 
one of the mathematical models describing ecohydrological-fire dynamics:
\begin{gather}
\label{Equation-1-1}
u_t-D_1\Delta u=u(c_1 v-c_2 w),\quad \text{a.e. in} \quad Q:=\Omega \times (0,\infty),\\
\label{Equation-1-2}
v_t+\left(\int_\Omega K(x,y)v(y,t)dy\right) v=v(c_3 w-c_4 u),\quad \text{a.e. in} \quad Q,\\
\label{Equation-1-3}
w_t-D_2 \Delta w=-c_5w+c_6-c_7 vw,\quad \text{a.e. in} \quad Q,\\
\label{Equation-1-4}
\nabla u \cdot \bm{n}=0,\quad \text{a.e. on} \quad \Sigma:=\Gamma \times (0,\infty),\\
\label{Equation-1-5}
\nabla w \cdot \bm{n}=0,\quad \text{a.e. on} \quad \Sigma,\\
\label{Equation-1-6}
(u(0),v(0),w(0))=(u_0,v_0,w_0),\quad \text{a.e. in} \quad \Omega,
\end{gather}
where $\Omega$ is a bounded domain in $\mathbb{R}^N~(N=1,2,3)$ with a smooth boundary $\Gamma:=\partial \Omega$\,;
$\bm{n}$ is an outer unit normal vector on $\Gamma$\,;
$D_i > 0~(i=1,2)$ and $c_i > 0~(i=1,2,\ldots,7)$ are positive constants.\\
\indent
We next explain the modeling background of the VRB system.
The VRB model was introduced to describe an ecosystem involving water reservoirs, flammable vegetation, and wildfires.
The unknown functions $u$, $v$ and $w$ describe the bushfire intensity, the vegetation density and water availability, respectively.\\
\indent
Firstly, the evolution of $u$ is determined by the equation \eqref{Equation-1-1} with boundary condition \eqref{Equation-1-4}.
The bushfire intensity $u$ diffuses uniformly with diffusion rate $D_1>0$, increases proportionally to the product $uv$ with proportionality constant $c_1>0$, 
and decreases proportionally to the product $uw$ with proportionality constant $c_2>0$.\\
\indent
Secondly, the evolution of $v$ is governed by \eqref{Equation-1-2}. Since \eqref{Equation-1-2} is an ordinary differential equation, the vegetation density $v$ 
does not diffuse in the spatial domain $\Omega$. It increases proportionally to the product $vw$ with proportionality constant $c_3>0$, and decreases proportionally 
to the product $uv$ with proportionality constant $c_4>0$.
The equation \eqref{Equation-1-2} also contains the nonlocal term
\begin{equation*}
\int_\Omega K(x,y) v(y)dy.
\end{equation*}
This term models the spatial competition among plants for resources such as water and nutrients.
The competition suppresses vegetation growth, and its intensity at the point $x$ depends on the distribution of vegetation over the entire domain $\Omega$.
Accordingly, the nonlocal quantity above acts as a spatially dependent coefficient in the self-limiting term
\begin{equation*}
\left(\int_\Omega K(x,y) v(y)dy\right)v.
\end{equation*}
\indent
Thirdly, the evolution of $w$ is governed by \eqref{Equation-1-3} with boundary condition \eqref{Equation-1-5}.
The water availability $w$ diffuses uniformly throughout the domain $\Omega$ with diffusion rate $D_2>0$, undergoes evaporation at the rate $c_5 w$, and is also absorbed 
by vegetation, which is modeled by the term $-c_7 vw$.
Moreover, water is replenished by rainfall, whose supply is modeled by the positive constant $c_6$.\\
\indent
Throughout this paper, we seek nonnegative solutions, since the unknown functions represent physical quantities.
In analyzing the VRB system mathematically, the main difficulty arises from the nonlocal term in \eqref{Equation-1-2} since the vegetation does not diffuse, that is, 
the smoothness of the vegetation density in the domain $\Omega$ is lost.
In order to overcome this difficulty, we incorporate the effect of the nonlocal term as a deformation of a usual metric $\|\cdot\|_{L^2(\Omega)}$ on the 
``\,{\sf flat}\,'' Hilbert space $L^2(\Omega)$.
More precisely, we deform $\|\cdot\|_{L^2(\Omega)}$ into the following one by the nonlocal coefficient:
\begin{equation*}
\int_\Omega \left(\int_\Omega K(x,y)v(y)dy\right) |\eta (x)|^2dx,\quad \forall \eta \in L^2(\Omega),
\end{equation*}
whose fundamental properties are established in Lemmas \ref{Lemma-2-1}--\ref{Lemma-2-3}.
This observation allows us to reformulate the nonlocal interaction as a deformation of the underlying Hilbert-space metric.
As a consequence, the VRB system can be analyzed within the framework of ``\,{\sf twisted gradient flows}\,'', which is treated as evolution inclusions 
on real Hilbert spaces with quasi-variational inner products in \cite{ito-2019,ito-2022} and explicitly named in \cite{ito-2026-1,ito-2027,ito-2026-2}.
To implement this reformulation, throughout this paper we assume that the following conditions are satisfied.
\begin{enumerate}\leftskip12pt
\item[(A1)] A nonnegative function $K \in C^1(\Omega \times \Omega)$ satisfies the following conditions:
\begin{enumerate}
\item[(a)] There exists a constant $K_1>0$ such that
\begin{equation*}
K_1 \le \inf_{x\,\in\,\Omega} \int_\Omega K(x,y)dy.
\end{equation*}
\item[(b)] There exists a constant $K_2>0$ such that 
\begin{equation*}
\sup_{x\,\in\,\Omega} \left(\int_\Omega |K(x,y)|^2dy\right)^{\frac{1}{2}} \le K_2.
\end{equation*}
\item[(c)] There exists a constant $K_3>0$ such that 
\begin{equation*}
\sup_{x\,\in\,\Omega} \left( \int_\Omega |\nabla_x K(x,y)|^2dy \right)^{\frac{1}{2}} \le K_3. 
\end{equation*}
Throughout this paper, we simply write $\nabla$ for $\nabla_x$.\\[-0.2cm]
\end{enumerate}
\noindent
As one of the typical examples of $K$, in \cite{LBCL-2009} it is modeled by the Gaussian function
\begin{equation*}
K(x,y) =\hat{\sigma} \exp(-\sigma |x-y|^2) \quad \text{for fixed constants}~\hat{\sigma}>0,\,\sigma >0.
\end{equation*}
\item[(A2)] An initial data $(u_0,v_0,w_0) \in (H^1(\Omega))^3$ satisfies
\begin{gather*}
u_0 \ge 0, \quad v_0 \ge 0,\quad w_0 \ge 0, \quad \text{a.e. in} \quad \Omega.
\end{gather*}
\end{enumerate}
We are now in a position to state the main result of this paper.
To the best of our knowledge, no global existence result of strong solutions has been established for the VRB system introduced in \cite{DV-2026}.
At first, we give the definition of strong global solution to the initial-boundary value problem (P) on $[0,\infty)$.
%%%%%
%%%%%
%%%%%
\begin{definition}\label{Definition-1-1}
A quadruple $(u,v,w)$ is called a strong global solution to (P) if and only if for every finite time $T>0$ the following properties are satisfied, that is, $(u,v,w)$ is a
strong solution to (P) on $[0,T]$:
\begin{enumerate}\leftskip10pt
\item[(s1)] $u \in W^{1,2}(0,T\,;L^2(\Omega)) \cap L^\infty (0,T\,;H^1(\Omega)) \cap L^2(0,T\,;H^2(\Omega))$.\\[-0.3cm]
\item[(s2)] $v \in W^{1,\infty}(0,T\,;L^2(\Omega)) \cap L^\infty (0,T\,;H^1(\Omega))$.\\[-0.3cm]
\item[(s3)] $w \in W^{1,2}(0,T\,;L^2(\Omega)) \cap L^\infty (0,T\,;H^1(\Omega)) \cap L^2(0,T\,;H^2(\Omega))$.\\[-0.3cm]
\item[(s4)] The equation (\ref{Equation-1-1}) with (\ref{Equation-1-4}) is satisfied in the following variational sense:
\begin{equation*}
\qquad u'(t)-D_1 \Delta_N u(t)=u(t) \{c_1 v(t)-c_2 w(t)\} \quad \text{in} \quad L^2(\Omega),\quad \text{a.e.}~t \in (0,T),
\end{equation*}
where $-\Delta_N : D(-\Delta_N) \rightarrow L^2(\Omega)$ is the Laplacian with homogeneous Neumann boundary condition whose domain is given by 
\begin{equation*}
\qquad D(-\Delta_N):=\left\{ \eta \in H^1(\Omega)\,;\,\Delta \eta \in L^2(\Omega)~~\text{with}~~\nabla \eta \cdot \bm{n}=0~~\text{for a.e. on}~~\Gamma \right\}.
\end{equation*}
\item[(s5)] The equation (\ref{Equation-1-2}) is satisfied in the following variational sense:
\begin{gather*}
\qquad v'(t)+\left( \int_\Omega K(x,y) v(y,t)dy\right) v(t)=v(t) \{c_3 w(t)-c_4 u(t)\}\\
\qquad \text{in} \quad L^2(\Omega),\quad \text{a.e.}~t \in (0,T),
\end{gather*}
\item[(s6)] The equation (\ref{Equation-1-3}) with (\ref{Equation-1-5}) is satisfied in the following variational sense:
\begin{equation*}
\qquad w'(t)-D_2 \Delta_N w(t)+c_5 w(t)=c_6 \cdot 1-c_7 v(t) w(t) \quad \text{in} \quad L^2(\Omega),\quad \text{a.e.}~t \in (0,T),
\end{equation*}
where we denote by $1$ the constant function equal to $1$ on $\Omega \times [0,T]$.\\[-0.3cm]
\item[(s7)] $(u(0),v(0),w(0))=(u_0,v_0,w_0) \quad \text{in} \quad (L^2(\Omega))^3$.
\end{enumerate}
Throughout this paper, the notation $f'$ denotes the time derivative of $f$ in $L^2(\Omega)$.
Whenever no confusion arises, we use the same notation $f_t=f'$ in the initial-boundary value problems throughout this paper.
\end{definition}
%%%%%
%%%%%
%%%%%
Now we give the main theorem in this paper.
%%%%%
%%%%%
%%%%%
\begin{theorem}\label{Theorem-1-1}
Assume that (A1) and (A2) hold. 
Then, the initial-boundary value problem (P) admits a nonnegative strong global solution $(u,v,w)$.
Here, global existence is understood in the sense of Definition \ref{Definition-1-1} for every finite $T>0$.
\end{theorem}
%%%%%
%%%%%
%%%%%
Since uniqueness is not guaranteed, we only construct a strong solution on any given finite time interval $[0,T]$ rather than a unique global trajectory on $[0,\infty)$.
In fact, Theorem \ref{Theorem-1-1} guarantees the existence of the non-blow-up strong global solutions to (P) for any physically meaningful time horizon.\\
\indent
As the key inequalities throughout this paper, we repeatedly use the continuous embedding $H^2(\Omega) \hookrightarrow L^\infty(\Omega)$: 
there exists a constant $K_4>0$ such that 
\begin{equation}\label{Equation-1-7}
\|\eta\|_{L^\infty (\Omega)} \le K_4 \|\eta\|_{H^2(\Omega)},\quad \forall \eta \in H^2(\Omega),
\end{equation}
and the Neumann elliptic regularity (see \cite{Barbu-1976}): there exists a constant $K_5>0$ such that
\begin{equation}\label{Equation-1-8}
\|\eta\|_{H^2(\Omega)} \le K_5 \left( \|\Delta_N \eta\|_{L^2(\Omega)}+\|\eta\|_{L^2(\Omega)} \right),\quad \forall \eta \in D(-\Delta_N).
\end{equation}
Besides, we use the Cauchy-Schwartz inequality and the Young inequality, repeatedly.
\begin{enumerate}\leftskip12pt
\item[(CS)] We have 
\begin{equation*}
|(\eta_1,\eta_2)_{L^2(\Omega)}| \le \|\eta_1\|_{L^2(\Omega)} \|\eta_2\|_{L^2(\Omega)},\quad \forall \eta_1,\eta_2 \in L^2(\Omega).
\end{equation*}
\item[(Y)] For any $p \in (1,\infty)$ and $\epsilon >0$ the following inequality holds:
\begin{equation*}
ab \le \frac{\epsilon}{p} a^p+\frac{1}{q \epsilon^{\frac{q}{p}}} b^q, \quad \forall a \ge 0,\,\forall b \ge 0,\quad \text{where} \quad \frac{1}{p}+\frac{1}{q}=1.
\end{equation*}
\end{enumerate}
Throughout this paper, we will use these well-known standard inequalities without further mention, and the following notation:\\
\begin{equation*}
\|\nabla \eta\|_{\bm{L}^2(\Omega)}:=\left( \int_\Omega |\nabla \eta (x)|^2dx\right)^{\frac{1}{2}},\quad \forall \eta \in H^1(\Omega).
\end{equation*}
\indent
At the end of this section, we briefly outline the strategy for proving Theorem \ref{Theorem-1-1} by describing the organization of this paper.
In Section \ref{Section-2}, we construct approximate problems for (P) that enable us not only to apply the general theory of ``\,{\sf twisted gradient flow}\,'' 
established in \cite{ito-2019,ito-2022} but also to justify the passage to the limit.
In fact, in Subsection \ref{Subsection-2-1}, we construct twisted spaces incorporating the nonlocal term into the underlying real Hilbert space $L^2(\Omega)$ by
introducing an auxiliary variable, called ``\,{\sf a twisted metric generator}\,'' in Subsection \ref{Subsection-2-3}, which serves to decouple the effect of the nonlocal term 
from the dynamics of $v$.
In Subsections \ref{Subsection-2-2}--\ref{Subsection-2-4}, we consider subsystems arising from the approximate problems and establish uniform estimates
for solutions to the subsystems which play important roles in the following sections.
In Section \ref{Section-3}, we show the existence and uniqueness of strong global-in-time solutions to the initial-boundary value problems for the approximate systems
with some uniform bounds of them.
Finally, in Section \ref{Section-4}, we establish the existence of strong solutions to (P) on any finite time interval by passing to the limit.
%%%%%%%%%%%%%%%%%%%%%%%%%%%%%%%%%%%%%%%%%%%%%%%%%%%%%%%%%%%%%%%%%%%%%%%%%%%%%%%%%%%%%%%%%%%%%%%%%%%%%%%%%%%%%%%%%%%%%%%%%%%%%%%%%
%%%%%%%%%%%%%%%%%%%%%%%%%%%%%%%%%%%%%%%%%%%%%%%%%%%%%%%%%%%%%%%%%%%%%%%%%%%%%%%%%%%%%%%%%%%%%%%%%%%%%%%%%%%%%%%%%%%%%%%%%%%%%%%%%
%%%%%%%%%%%%%%%%%%%%%%%%%%%%%%%%%%%%%%%%%%%%%%%%%%%%%%%%%%%%%%%%%%%%%%%%%%%%%%%%%%%%%%%%%%%%%%%%%%%%%%%%%%%%%%%%%%%%%%%%%%%%%%%%%
%%%%%%%%%%%%%%%%%%%%%%%%%%%%%%%%%%%%%%%%%%%%%%%%%%%%%%%%%%%%%%%%%%%%%%%%%%%%%%%%%%%%%%%%%%%%%%%%%%%%%%%%%%%%%%%%%%%%%%%%%%%%%%%%%
%%%%%%%%%%%%%%%%%%%%%%%%%%%%%%%%%%%%%%%%%%%%%%%%%%%%%%%%%%%%%%%%%%%%%%%%%%%%%%%%%%%%%%%%%%%%%%%%%%%%%%%%%%%%%%%%%%%%%%%%%%%%%%%%%
\section{Approximate problems and their subsystems}\label{Section-2}
%%%%%%%%%%%%%%%%%%%%%%%%%%%%%%%%%%%%%%%%%%%%%%%%%%%%%%%%%%%%%%%%%%%%%%%%%%%%%%%%%%%%%%%%%%%%%%%%%%%%%%%%%%%%%%%%%%%%%%%%%%%%%%%%%
%%%%%%%%%%%%%%%%%%%%%%%%%%%%%%%%%%%%%%%%%%%%%%%%%%%%%%%%%%%%%%%%%%%%%%%%%%%%%%%%%%%%%%%%%%%%%%%%%%%%%%%%%%%%%%%%%%%%%%%%%%%%%%%%%
%%%%%%%%%%%%%%%%%%%%%%%%%%%%%%%%%%%%%%%%%%%%%%%%%%%%%%%%%%%%%%%%%%%%%%%%%%%%%%%%%%%%%%%%%%%%%%%%%%%%%%%%%%%%%%%%%%%%%%%%%%%%%%%%%
%%%%%%%%%%%%%%%%%%%%%%%%%%%%%%%%%%%%%%%%%%%%%%%%%%%%%%%%%%%%%%%%%%%%%%%%%%%%%%%%%%%%%%%%%%%%%%%%%%%%%%%%%%%%%%%%%%%%%%%%%%%%%%%%%
%%%%%%%%%%%%%%%%%%%%%%%%%%%%%%%%%%%%%%%%%%%%%%%%%%%%%%%%%%%%%%%%%%%%%%%%%%%%%%%%%%%%%%%%%%%%%%%%%%%%%%%%%%%%%%%%%%%%%%%%%%%%%%%%%
In order to set approximate systems for the VRB model, for each $\delta \in (1,\infty)$ we prepare truncation functions 
$\alpha_{\text{{\tiny $\delta$}}} \in C(\mathbb{R})$ defined by 
\begin{equation*}
\alpha_{\text{{\tiny $\delta$}}} (r):=\min \left\{ \max \left\{ \delta^{-1},r\right\},\delta \right\},
\end{equation*}
which satisfies the following properties:
\begin{align}
\label{Equation-2-1}
\bullet~&\text{nondegeneracy and boundedness:} \quad \delta^{-1} \le \alpha_{\text{{\tiny $\delta$}}}(r) \le \delta,\quad \forall r \in \mathbb{R},\\
\label{Equation-2-2}
\bullet~&\text{contraction property:} \quad |\alpha_{\text{{\tiny $\delta$}}}(r_1)-\alpha_{\text{{\tiny $\delta$}}}(r_2)| \le |r_1-r_2|,\quad \forall r_1,r_2 \in \mathbb{R},
\end{align}
and $\beta_{\text{{\tiny $\delta$}}} \in C^1(\mathbb{R})$ defined by 
\begin{equation*}
\beta_{\text{{\tiny $\delta$}}} (r):=\left\{
\begin{array}{ll}
\delta \quad&\text{if} \quad r \in \left(\delta+\delta^{-1},\infty\right),\\[0.2cm]
-\dfrac{\delta}{4}\left(r-\delta-\delta^{-1}\right)^2+\delta \quad&\text{if} \quad r \in \left(\delta-\delta^{-1},\delta+\delta^{-1}\right],\\[0.4cm]
r \quad&\text{if} \quad r \in \left[-\delta+\delta^{-1},\delta-\delta^{-1}\right],\\[0.2cm]
\dfrac{\delta}{4}\left(r+\delta+\delta^{-1}\right)^2-\delta \quad&\text{if} \quad r \in \left[-\delta-\delta^{-1},-\delta+\delta^{-1}\right),\\[0.4cm]
-\delta \quad&\text{if} \quad r \in \left(-\infty,-\delta-\delta^{-1}\right),
\end{array}
\right.
\end{equation*}
which satisfies the following properties:
\begin{align}
\label{Equation-2-3}
\bullet~&\text{boundedness:} \quad -\min\{\delta,|r|\} \le \beta_{\text{{\tiny $\delta$}}}(r) \le \min\{\delta,|r|\}, \quad \forall r \in \mathbb{R},\\
\label{Equation-2-4}
\bullet~&\text{contraction property:} \quad |\beta_{\text{{\tiny $\delta$}}}(r_1)-\beta_{\text{{\tiny $\delta$}}}(r_2)| \le |r_1-r_2|,\quad \forall r_1,r_2 \in \mathbb{R},\\
\label{Equation-2-5}
\bullet~&\text{nonnegativity and boundedness of the derivative:} \quad 0 \le \beta_{\text{{\tiny $\delta$}}}'(r) \le 1,\quad \forall r \in \mathbb{R}.
\end{align}
Using the truncation $\alpha_{\text{{\tiny $\delta$}}}$, for any $\eta \in L^2(\Omega)$ we introduce the approximation 
$\mathcal{K}_{\text{{\tiny $\delta$}}}\eta$ of the nonlocal term in \eqref{Equation-1-2} by 
\begin{equation}\label{Equation-2-6}
(\mathcal{K}_{\text{{\tiny $\delta$}}}\eta)(x):=\int_\Omega K(x,y) \alpha_{\text{{\tiny $\delta$}}}(\eta (y))dy,\quad \text{a.e.}~x \in \Omega.
\end{equation}
As stated in Subsection \ref{Subsection-2-1} below, the truncation $\alpha_{\text{{\tiny $\delta$}}}$ preserves positivity and nondegeneracy of the coefficient, 
whereas $\beta_{\text{{\tiny $\delta$}}}$ controls the reaction terms uniformly in $\delta$.\\
\indent
In Sections \ref{Section-2} and \ref{Section-3}, for each $\delta \in (1,\infty)$, $\varepsilon \in (0,1)$ and $\nu \in (0,1)$ we consider the following approximate 
initial-boundary value problem $\text{(AP)}{}_{\text{{\tiny $\delta,\!\varepsilon,\!\nu$}}}$=\{\eqref{Equation-2-7}--\eqref{Equation-2-12}\}:
\begin{gather}
\label{Equation-2-7}
u'-D_1 \Delta u=\beta_{\text{{\tiny $\delta$}}}(u) (c_1 v-c_2 w),\quad \text{a.e. in} \quad Q,\\
\label{Equation-2-8}
v'+\mathcal{K}_{\text{{\tiny $\delta$}}}z\left(-\nu \Delta v+v\right)=v\left\{ c_3 \beta_{\text{{\tiny $\delta$}}} (w)-c_4 \beta_{\text{{\tiny $\delta$}}} (u) \right\},
\quad \text{a.e. in} \quad Q,\\
\label{Equation-2-9}
w'-D_2 \Delta w+c_5 w=c_6 -c_7 v \beta_{\text{{\tiny $\delta$}}} (w),\quad \text{a.e. in}\quad Q,\\
\label{Equation-2-10}
z'-\varepsilon \Delta z=v' \quad \text{a.e.} \quad Q,\\
\label{Equation-2-11}
\nabla u \cdot \bm{n}=\nabla v \cdot \bm{n}=\nabla z \cdot \bm{n}=\nabla w \cdot \bm{n}=0,\quad \text{a.e. on} \quad \Sigma,\\
\label{Equation-2-12}
(u(0),v(0),w(0),z(0))=(u_0,v_0,w_0,v_0),\quad \text{a.e. in} \quad \Omega.
\end{gather}
In the approximate problems $\text{(AP)}{}_{\text{{\tiny $\delta,\!\varepsilon,\!\nu$}}}$ for (P), we emphasize that the twisted metric generator $z$, serving as 
an auxiliary variable, is introduced so that the state-dependent coefficient $\mathcal{K}_{\text{{\tiny $\delta$}}}z$ in \eqref{Equation-2-8}, which plays the role of
an approximation of the nonlocal term, is generated by an evolution equation \eqref{Equation-2-10} independent of the variational structure of $v$.\\
%%%%%%%%%%%%%%%%%%%%%%%%%%%%%%%%%%%%%%%%%%%%%%%%%%%%%%%%%%%%%%%%%%%%%%%%%%%%%%%%%%%%%%%%%%%%%%%%%%%%%%%%%%%%%%%%%%%%%%%%%%%%%%%%%
%%%%%%%%%%%%%%%%%%%%%%%%%%%%%%%%%%%%%%%%%%%%%%%%%%%%%%%%%%%%%%%%%%%%%%%%%%%%%%%%%%%%%%%%%%%%%%%%%%%%%%%%%%%%%%%%%%%%%%%%%%%%%%%%%
%%%%%%%%%%%%%%%%%%%%%%%%%%%%%%%%%%%%%%%%%%%%%%%%%%%%%%%%%%%%%%%%%%%%%%%%%%%%%%%%%%%%%%%%%%%%%%%%%%%%%%%%%%%%%%%%%%%%%%%%%%%%%%%%%
\subsection{Twisted Hilbert spaces}\label{Subsection-2-1}
%%%%%%%%%%%%%%%%%%%%%%%%%%%%%%%%%%%%%%%%%%%%%%%%%%%%%%%%%%%%%%%%%%%%%%%%%%%%%%%%%%%%%%%%%%%%%%%%%%%%%%%%%%%%%%%%%%%%%%%%%%%%%%%%%
%%%%%%%%%%%%%%%%%%%%%%%%%%%%%%%%%%%%%%%%%%%%%%%%%%%%%%%%%%%%%%%%%%%%%%%%%%%%%%%%%%%%%%%%%%%%%%%%%%%%%%%%%%%%%%%%%%%%%%%%%%%%%%%%%
%%%%%%%%%%%%%%%%%%%%%%%%%%%%%%%%%%%%%%%%%%%%%%%%%%%%%%%%%%%%%%%%%%%%%%%%%%%%%%%%%%%%%%%%%%%%%%%%%%%%%%%%%%%%%%%%%%%%%%%%%%%%%%%%%
The key idea of this paper is to absorb the state-dependent coefficient into a family of nonlocal state-dependent Hilbert metrics, thereby converting the variable
coefficient into a state-dependent geometry.
The resulting family of ``\,{\sf twisted}\,'' Hilbert spaces satisfies the assumptions of the quasi-variational evolution framework developed in \cite{ito-2019}.
This construction enables us to establish an energy inequality despite the presence of the state-dependent coefficient $\mathcal{K}_{\delta}z$.
The crucial ingredient is the metric deformation inequality stated in Lemma \ref{Lemma-2-3}.\\
\indent
In order to deform the ``\,{\sf flat}\,'' space $L^2(\Omega)$ with the usual inner product $(\cdot,\cdot)_{L^2(\Omega)}$, for each approximate parameter 
$\delta \in (1,\infty)$ and each function $\bar{z} \in L^2(\Omega)$ we define an operator $M_{\text{{\tiny $\delta,\!\bar{z}$}}}$ on $L^2(\Omega)$ by 
\begin{equation}\label{Equation-2-13}
\forall v \in L^2(\Omega),\quad \left(M_{\text{{\tiny $\delta,\!\bar{z}$}}}v\right)(x):=
\left\{\left(\mathcal{K}_{\text{{\tiny $\delta$}}}\bar{z}\right)(x)\right\}^{-1}v(x),
\quad \forall x \in \Omega,
\end{equation}
where $\mathcal{K}_{\text{{\tiny $\delta$}}} \bar{z}$ is given by \eqref{Equation-2-6}.
Then, we show Lemmas \ref{Lemma-2-1}--\ref{Lemma-2-3}, which are originally obtained in \cite[Lemma 5.1]{ito-2026-2} and proposed as one of the typical examples
deforming the metric of $L^2(\Omega)$ by nonlocal term.
%%%%%
%%%%%
%%%%%
\begin{lemma}\label{Lemma-2-1}
For each $\delta \in (1,\infty)$ and $\bar{z} \in L^2(\Omega)$ the operator $M_{\text{{\tiny $\delta,\!\bar{z}$}}}$ is well-defined as an operator from 
$L^2(\Omega)$ into itself, which is linear, bounded and bijective.
\end{lemma}
%%%%%
%%%%%
%%%%%
\begin{proof}[Proof.]
Throughout this proof, we fix a parameter $\delta \in (1,\infty)$ and a function $\bar{z} \in L^2(\Omega)$.
Since from (A1) we have the following inequality for a.e. $x \in \Omega$:
\begin{align}
\label{Equation-2-14}
\frac{K_1}{\delta} &\le \frac{1}{\delta} \int_\Omega K(x,y)dy \le \int_\Omega K(x,y) \alpha_{\text{{\tiny $\delta$}}}(\bar{z}(y))dy\\
\nonumber
&=(\mathcal{K}_{\text{{\tiny $\delta$}}}\bar{z})(x) \le \left(\int_\Omega |K(x,y)|^2dy\right)^{\frac{1}{2}} 
\left(\int_\Omega |\alpha_{\text{{\tiny $\delta$}}}(\bar{z}(y))|^2dy\right)^{\frac{1}{2}} \le \delta K_2 |\Omega|^{\frac{1}{2}},
\end{align}
hence,
\begin{equation}\label{Equation-2-15}
\frac{1}{\delta K_2 |\Omega|^{\frac{1}{2}}} \le \left\{ (\mathcal{K}_{\text{{\tiny $\delta$}}}\bar{z})(x) \right\}^{-1} \le \frac{\delta}{K_1},
\end{equation}
it follows that the following inequality holds for all $v \in L^2(\Omega)$:
\begin{equation}\label{Equation-2-16}
\frac{1}{|\Omega|} \left(\frac{1}{\delta K_2} \right)^2 \|v\|_{L^2(\Omega)}^2 \le \int_\Omega \left|(M_{\text{{\tiny $\delta,\!\bar{z}$}}}v)(x)\right|^2dx 
\le \left(\frac{\delta}{K_1}\right)^2 \|v\|_{L^2(\Omega)}^2,
\end{equation}
which implies $M_{\text{{\tiny $\delta,\!\bar{z}$}}}v \in L^2(\Omega)$, hence, that the operator $M_{\text{{\tiny $\delta,\!\bar{z}$}}}$ is well-defined on $L^2(\Omega)$.
Since from \eqref{Equation-2-13} the operator $M_{\text{{\tiny $\delta,\!\bar{z}$}}}$ is linear on $L^2(\Omega)$, we see from \eqref{Equation-2-16} 
that it is bounded and injective on $L^2(\Omega)$.\\
\indent
To complete the proof of this lemma, we show that $M_{\text{{\tiny $\delta,\!\bar{z}$}}}$ is surjective on $L^2(\Omega)$.
Indeed, for any $v \in L^2(\Omega)$ we set 
\begin{equation}\label{Equation-2-17}
\bar{v}(x):=(\mathcal{K}_{\text{{\tiny $\delta$}}}\bar{z})(x) v(x),\quad \forall x \in \Omega.
\end{equation}
Then, from \eqref{Equation-2-14} we deduce 
\begin{equation*}
\int_\Omega |\bar{v}(x)|^2dx \le |\Omega| \left(\delta K_2\right)^2 \|v\|_{L^2(\Omega)}^2,
\end{equation*}
which implies $\bar{v} \in L^2(\Omega)$ and $M_{\text{{\tiny $\delta,\!\bar{z}$}}}\bar{v}=v$ in $L^2(\Omega)$ by using \eqref{Equation-2-13} and \eqref{Equation-2-17}.
\end{proof}
%%%%%
%%%%%
%%%%%
By Lemma \ref{Lemma-2-1}, we define a bilinear form $(\cdot,\cdot)_{\text{{\tiny $\delta,\!\bar{z}$}}}:L^2(\Omega) \times L^2(\Omega) \rightarrow \mathbb{R}$ by 
\begin{align}
\label{Equation-2-18}
(v_1,v_2)_{\text{{\tiny $\delta,\!\bar{z}$}}}:\!&=(M_{\text{{\tiny $\delta,\!\bar{z}$}}}v_1,v_2)_{L^2(\Omega)}\\
\nonumber
&=\int_\Omega \{(\mathcal{K}_{\text{{\tiny $\delta$}}}\bar{z})(x)\}^{-1} v_1(x) v_2 (x)dx,\quad \forall v_1,v_2 \in L^2(\Omega),
\end{align}
namely, the state-dependent coefficient $\mathcal{K}_{\text{{\tiny $\delta$}}}\bar{z}$ is absorbed into a twisted $L^2(\Omega)$-metric as is seen from
Lemma \ref{Lemma-2-3} below.
Before giving Lemma \ref{Lemma-2-3}, we show Lemma \ref{Lemma-2-2}, which gives some properties of 
$\{(\cdot,\cdot)_{\text{{\tiny $\delta,\!\bar{z}$}}}\,;\,\bar{z} \in L^2(\Omega)\}$ for all $\delta \in (1,\infty)$.
Estimate (b) quantifies how the metric changes when the state variable $\bar{z}$ varies.
This property is the key ingredient for the energy inequality of Lemma \ref{Lemma-2-3}. 
%%%%%
%%%%%
%%%%%
\begin{lemma}\label{Lemma-2-2}
For every $\delta \in (1,\infty)$ and $\bar{z} \in L^2(\Omega)$ the bilinear form  
$(\cdot,\cdot)_{\text{{\tiny $\delta,\!\bar{z}$}}}:L^2(\Omega) \times L^2(\Omega) \rightarrow \mathbb{R}$ defined by (\ref{Equation-2-18}) is a quasi-variational inner
product on $L^2(\Omega)$.\\
\indent
Moreover, for each $\delta \in (1,\infty)$ the class of inner products 
$\{(\cdot,\cdot)_{\text{{\tiny $\delta,\!\bar{z}$}}}\,;\,\bar{z} \in L^2(\Omega)\}$ satisfies the following properties:
\begin{enumerate}\leftskip6pt
\item[(1)] There exist constants $C_1>0$ and $C_2>0$ such that
\begin{equation*}
\delta^{-\frac{1}{2}} C_1 \|v\|_{L^2(\Omega)} \le \|v\|_{\text{{\tiny $\delta,\!\bar{z}$}}} \le \delta^{\frac{1}{2}} C_2 \|v\|_{L^2(\Omega)}, 
\quad \forall v \in L^2(\Omega),~\forall \bar{z} \in L^2(\Omega).
\end{equation*}
\item[(2)] There exists a constant $C_3>0$ such that 
\begin{equation*}
\left| \|v\|_{\text{{\tiny $\delta,\!\bar{z}_1$}}}^2-\|v\|_{\text{{\tiny $\delta,\!\bar{z}_2$}}}^2 \right| 
\le \delta^3 C_3 \|\bar{z}_1-\bar{z}_2\|_{L^2(\Omega)} \|v\|_{\text{{\tiny $\delta,\!\bar{z}_1$}}}^2, 
\quad \forall \bar{z}_1,\bar{z}_2 \in L^2(\Omega),\,\forall v \in L^2(\Omega).
\end{equation*}
\end{enumerate}
where for every $\delta \in (1,\infty)$ and $\bar{z} \in L^2(\Omega)$ we set 
$\|\,\cdot\,\|_{\text{{\tiny $\delta,\!\bar{z}$}}}:=\displaystyle{\sqrt{(\,\cdot\,,\,\cdot\,)_{\text{{\tiny $\delta,\!\bar{z}$}}}}}$.
\end{lemma}
%%%%%
%%%%%
%%%%%
\begin{proof}[Proof.]
From \eqref{Equation-2-13} and \eqref{Equation-2-18} we obtain the symmetry of $(\cdot,\cdot)_{\text{{\tiny $\delta,\!\bar{z}$}}}$:
\begin{equation*}
(v_1,v_2)_{\text{{\tiny $\delta,\!\bar{z}$}}}=(M_{\text{{\tiny $\delta,\!\bar{z}$}}}v_1,v_2)_{L^2(\Omega)}
=(M_{\text{{\tiny $\delta,\!\bar{z}$}}}v_2,v_1)_{L^2(\Omega)}=(v_2,v_1)_{\text{{\tiny $\delta,\!\bar{z}$}}}, \quad \forall v_1,v_2 \in L^2(\Omega).
\end{equation*}
Hence, we see from the nonnegativity of $K$ (cf. (A1)) and Lemma \ref{Lemma-2-1} that for any $\bar{z} \in L^2(\Omega)$ the bilinear form 
$(\cdot,\cdot)_{\text{{\tiny $\delta,\!\bar{z}$}}}$ defines an inner product on $L^2(\Omega)$.
Following \cite{ito-2019}, this is referred to as a ``\,{\sf quasi-variational inner product}\,'' due to its dependence on the state variable $\bar{z}$, 
which will eventually be replaced by the unknown function $z$ in (P).\\
\indent
Next, we show (1).
From \eqref{Equation-2-15} we get
\begin{equation}\label{Equation-2-19}
\frac{1}{\delta K_2 |\Omega|^{\frac{1}{2}}} \cdot \|v\|_{L^2(\Omega)}^2 \le \int_\Omega \left\{ (\mathcal{K}_{\text{{\tiny $\delta$}}}\bar{z})(x)\right\}^{-1} |v(x)|^2dx
\le \frac{\delta}{K_1} \cdot \|v\|_{L^2(\Omega)}^2.
\end{equation}
By defining the constants $C_1>0$ and $C_2>0$ as
\begin{equation*}
C_1:=\left(\frac{1}{K_2 |\Omega|^{\frac{1}{2}}}\right)^{\frac{1}{2}},\quad C_2:=\left( \frac{1}{K_1} \right)^{\frac{1}{2}},
\end{equation*}
we deduce from \eqref{Equation-2-19} that the required inequality holds.\\
\indent
In the rest of this proof, we show (2).
Since we have 
\begin{align}
\label{Equation-2-20}
&\,\left| \left\{(\mathcal{K}_{\text{{\tiny $\delta$}}}\bar{z}_1)(x)\right\}^{-1}-\left\{(\mathcal{K}_{\text{{\tiny $\delta$}}}\bar{z}_2)(x)\right\}^{-1}\right|\\
\nonumber
\le&\, \{(\mathcal{K}_{\text{{\tiny $\delta$}}}\bar{z}_1)(x)\}^{-1} \cdot \{(\mathcal{K}_{\text{{\tiny $\delta$}}}\bar{z}_2)(x)\}^{-1}
\int_\Omega K(x,y) \left| \alpha_{\text{{\tiny $\delta$}}} (\bar{z}_1(y))-\alpha_{\text{{\tiny $\delta$}}} (\bar{z}_2(y)) \right|dy\\
\nonumber
\le&\,\left(\dfrac{\delta}{K_1}\right)^2 \left( \int_\Omega |K(x,y)|^2dy\right)^{\frac{1}{2}} \|\bar{z}_1-\bar{z}_2\|_{L^2(\Omega)}.
\end{align}
From (1) of this lemma and \eqref{Equation-2-20} we obtain the following inequality:
\begin{align*}
\left| \|v\|_{\text{{\tiny $\delta,\!\bar{z}_1$}}}^2-\|v\|_{\text{{\tiny $\delta,\!\bar{z}_2$}}}^2\right| 
&\le K_2 \left(\frac{\delta}{K_1}\right)^2 \|\bar{z}_1-\bar{z}_2\|_{L^2(\Omega)} \|v\|_{L^2(\Omega)}^2\\
&\le \frac{\delta^3 K_2}{(C_1 K_1)^2} \cdot \left\|\bar{z}_1-\bar{z}_2\right\|_{L^2(\Omega)} \|v\|_{\text{{\tiny $\delta,\!\bar{z}_1$}}}^2, \quad \forall v \in L^2(\Omega).
\end{align*}
By defining the constant $C_3 > 0$ as
\begin{equation*}
C_3:=\frac{K_2}{(C_1 K_1)^2},
\end{equation*}
we get the required inequality in (2) of this lemma.
\end{proof}
%%%%%
%%%%%
%%%%%
By Lemma \ref{Lemma-2-2}, for each approximation parameter $\delta \in (1,\infty)$ we construct a family of
``\,{\sf twisted spaces}\,'' $\{L^2(\delta,\bar{z})\,;\,\bar{z} \in L^2(\Omega)\}$, where $L^2(\delta,\bar{z})$ denotes the real Hilbert space $L^2(\Omega)$ with the 
quasi-variational inner product $(\cdot,\cdot)_{\text{{\tiny $\delta,\!\bar{z}$}}}$ for any $\bar{z} \in L^2(\Omega)$ in this paper, 
while $L^2(\Omega)$ with the usual inner product is called ``\,{\sf flat}\,''.
From \eqref{Equation-2-17} in the proof of Lemma \ref{Lemma-2-1}, the inverse $M_{\text{{\tiny $\delta,\!\bar{z}$}}}^{-1}$ of 
$M_{\text{{\tiny $\delta,\!\bar{z}$}}}$ is given by
\begin{equation*}
\forall \tilde{v} \in L^2(\Omega), \quad \left(M_{\text{{\tiny $\delta,\!\bar{z}$}}}^{-1} \tilde{v} \right)(x)
=\left( \mathcal{K}_{\text{{\tiny $\delta$}}} \bar{z} \right)(x) \tilde{v}(x), \quad \forall x \in \Omega.
\end{equation*}
%
%%%%%
%%%%%
%%%%%
\begin{remark}
We see from (1) of Lemma \ref{Lemma-2-2} that one equivalence constant between $\|\cdot\|_{\text{{\tiny $\delta,\!\bar{z}$}}}$ and $\|\cdot\|_{L^2(\Omega)}$
degenerates as $\delta \uparrow \infty$ in the sense of 
\begin{equation*}
\lim_{\delta\,\uparrow\,\infty} \frac{C_1}{\sqrt{\delta}}=0
\end{equation*}
and the other blows up as $\delta \uparrow \infty$ in the sense of 
\begin{equation*}
\lim_{\delta\,\uparrow\,\infty} \sqrt{\delta}C_2=\infty.
\end{equation*}
These facts mean that the uniform equivalence between quasi-variational inner products $\{(\cdot,\cdot)_{\text{{\tiny $\delta,\!\bar{z}$}}}\,;\,\bar{z} \in L^2(\Omega)\}$ 
and the usual one $(\cdot,\cdot)$ collapses if $\delta \uparrow \infty$ occurs.\\
\indent
Moreover, we see from (b) of Lemma \ref{Lemma-2-2} that the Lipschitz estimate for the deformation of the metric is not uniform with respect to $\delta$
in the sense of 
\begin{equation*}
\lim_{\delta\,\uparrow\,\infty} \delta^3 C_3=\infty.
\end{equation*}
This fact means that the deformation rate between the metrics of twisted spaces $L^2(\delta,\bar{z}_1)$ and $L^2(\delta,\bar{z}_2)$ is uncontrollable 
if $\delta \uparrow \infty$ occurs as is seen from Lemma \ref{Lemma-2-3} below.  
\end{remark}
%%%%%
%%%%%
%%%%%
Moreover, we have Lemma \ref{Lemma-2-3} which was originally obtained in \cite{ito-2019}, and omit its proof in this paper.
%%%%%
%%%%%
%%%%%
\begin{lemma}\label{Lemma-2-3}
(\cite[Lemma 2.10]{ito-2019}) For each $v \in W^{1,2}(0,T\,;L^2(\Omega))$ and $\bar{z} \in W^{1,1}(0,T\,;L^2(\Omega))$ the function 
$t \mapsto \|v(t)\|_{\text{{\tiny $\delta,\!\bar{z}(t)$}}}^2$ is absolutely continuous from $[0,T]$ into $\mathbb{R}$, and we have the following energy inequality for the 
state-dependent norm associated with the quasi-variational inner products for a.e. $t \in (0,T)$:
\begin{equation*}
\left| \frac{d}{dt} \|v(t)\|_{\text{{\tiny $\delta,\!\bar{z}(t)$}}}^2-2(v'(t),v(t))_{\text{{\tiny $\delta,\!\bar{z}(t)$}}}\right| \le \delta^3 C_3
\|\bar{z}'(t)\|_{L^2(\Omega)} \|v(t)\|_{\text{{\tiny $\delta,\!\bar{z}(t)$}}}^2,
\end{equation*}
where the constant $C_3>0$ is the same one as in (2) of Lemma \ref{Lemma-2-2}. 
\end{lemma}
%%%%%
%%%%%
%%%%%
%%%%%%%%%%%%%%%%%%%%%%%%%%%%%%%%%%%%%%%%%%%%%%%%%%%%%%%%%%%%%%%%%%%%%%%%%%%%%%%%%%%%%%%%%%%%%%%%%%%%%%%%%%%%%%%%%%%%%%%%%%%%%%%%%%%%%%%%
%%%%%%%%%%%%%%%%%%%%%%%%%%%%%%%%%%%%%%%%%%%%%%%%%%%%%%%%%%%%%%%%%%%%%%%%%%%%%%%%%%%%%%%%%%%%%%%%%%%%%%%%%%%%%%%%%%%%%%%%%%%%%%%%%%%%%%%%
%%%%%%%%%%%%%%%%%%%%%%%%%%%%%%%%%%%%%%%%%%%%%%%%%%%%%%%%%%%%%%%%%%%%%%%%%%%%%%%%%%%%%%%%%%%%%%%%%%%%%%%%%%%%%%%%%%%%%%%%%%%%%%%%%%%%%%%%
\subsection{Twisted gradient flow on state-dependent spaces}\label{Subsection-2-2}
%%%%%%%%%%%%%%%%%%%%%%%%%%%%%%%%%%%%%%%%%%%%%%%%%%%%%%%%%%%%%%%%%%%%%%%%%%%%%%%%%%%%%%%%%%%%%%%%%%%%%%%%%%%%%%%%%%%%%%%%%%%%%%%%%%%%%%%%
%%%%%%%%%%%%%%%%%%%%%%%%%%%%%%%%%%%%%%%%%%%%%%%%%%%%%%%%%%%%%%%%%%%%%%%%%%%%%%%%%%%%%%%%%%%%%%%%%%%%%%%%%%%%%%%%%%%%%%%%%%%%%%%%%%%%%%%%
%%%%%%%%%%%%%%%%%%%%%%%%%%%%%%%%%%%%%%%%%%%%%%%%%%%%%%%%%%%%%%%%%%%%%%%%%%%%%%%%%%%%%%%%%%%%%%%%%%%%%%%%%%%%%%%%%%%%%%%%%%%%%%%%%%%%%%%%
For each $\nu \in (0,1)$ we define a proper l.s.c. convex function $\Phi_\nu : L^2(\Omega) \rightarrow \mathbb{R} \cup \{\infty\}$ by
\begin{equation}\label{Equation-2-21}
\Phi_\nu (\bar{v}):=\left\{
\begin{array}{ll}
\displaystyle{\frac{\nu}{2} \int_\Omega |\nabla \bar{v}(x)|^2dx+\frac{1}{2}\int_\Omega |\bar{v}(x)|^2dx},\quad&\text{if} \quad \bar{v} \in H^1(\Omega),\\[0.3cm]
\infty,\quad&\text{if} \quad \bar{v} \in L^2(\Omega) \setminus H^1(\Omega).
\end{array}
\right.
\end{equation}
By (1) of Lemma \ref{Lemma-2-2}, for each $\bar{z} \in L^2(\Omega)$ we are able to regard $\Phi_\nu$ as a proper l.s.c. convex function 
on the twisted Hilbert space $L^2(\delta,\bar{z})$.
Although $\Phi_\nu$ itself is independent of the state variable $z$, the induced subdifferential operator depends on $\bar{z}$ through 
the twisted Hilbert structure $L^2(\delta,\bar{z})$.
The following lemma identifies the subdifferential of $\Phi_\nu$ on the twisted Hilbert space $L^2(\delta,\bar{z})$.
The subsequent analysis relies on the abstract evolution framework developed in \cite{ito-2019,ito-2022}.
Especially, Lemmas \ref{Lemma-2-2} and \ref{Lemma-2-4} verify that the present state-dependent Hilbert structure satisfies the assumptions of the 
general theory established in \cite{ito-2019,ito-2022}.
%%%%%
%%%%%
%%%%%
\begin{lemma}\label{Lemma-2-4}
For each $\bar{z} \in L^2(\Omega)$ the following properties hold:
\begin{gather*}
D\left(\partial_{\text{{\tiny $\delta,\!\bar{z}$}}}\,\Phi_\nu\right)=D(-\Delta_N),\\
\partial_{\text{{\tiny $\delta,\!\bar{z}$}}}\,\Phi_\nu (\bar{v})=\left\{ M_{\text{{\tiny $\delta,\!\bar{z}$}}}^{-1}\left(-\nu \Delta_N \bar{v}+\bar{v}\right)\right\},
\quad \forall \bar{v} \in D\left(\partial_{\text{{\tiny $\delta,\!\bar{z}$}}}\,\Phi_\nu\right),
\end{gather*}
where the operator $-\Delta_N$ is the same as in (s4) in Theorem \ref{Theorem-1-1}.
\end{lemma}
%%%%%
%%%%%
%%%%%
\begin{proof}[Proof.]
For simplicity, for any $\bar{v} \in D(-\Delta_N)$ we set 
\begin{equation*}
\bar{v}^*:=M_{\text{{\tiny $\delta,\!\bar{z}$}}}^{-1} (-\nu \Delta_N \bar{v}+\bar{v}) \quad \text{in} \quad L^2(\Omega)=L^2(\delta,\bar{z}).
\end{equation*}
First, from \eqref{Equation-2-18} we have the following inequality for all $\tilde{v} \in H^1(\Omega)$:
\begin{align*}
\left(\bar{v}^*,\tilde{v}-\bar{v}\right)_{\text{{\tiny $\delta,\!\bar{z}$}}} &=\left( M_{\text{{\tiny $\delta,\!\bar{z}$}}}\bar{v}^*,\tilde{v}-\bar{v}\right)_{L^2(\Omega)}\\
&=\int_\Omega \{-\nu \Delta_N \bar{v}(x)+\bar{v}(x)\}(\tilde{v}-\bar{v})(x)dx\\
&=\nu \int_\Omega \nabla \bar{v}(x) \cdot \nabla (\tilde{v}-\bar{v})(x)dx+\int_\Omega \bar{v}(x)(\tilde{v}-\bar{v})(x)dx \le \Phi_\nu (\tilde{v})-\Phi_\nu (\bar{v}),
\end{align*}
which implies $\bar{v}^* \in \partial_{\text{{\tiny $\delta,\!\bar{z}$}}} \Phi_\nu (\bar{v})$, that is,
\begin{equation}\label{Equation-2-22}
\bar{v} \in D(\partial_{\text{{\tiny $\delta,\!\bar{z}$}}} \Phi_\nu),\quad 
\left\{M_{\text{{\tiny $\delta,\!\bar{z}$}}}^{-1} \left(-\nu \Delta_N \bar{v}+\bar{v}\right)\right\} \subset \partial_{\text{{\tiny $\delta,\!\bar{z}$}}}\,\Phi_\nu (\bar{v}).
\end{equation}
Since $\bar{v}$ is arbitrary, we obtain
\begin{equation}\label{Equation-2-23}
D(-\Delta_N) \subset D(\partial_{\text{{\tiny $\delta,\!\bar{z}$}}} \Phi_\nu).
\end{equation}
\indent
Conversely, we let $\bar{v}$ be any element of $D(\partial_{\text{{\tiny $\delta,\!\bar{z}$}}}\,\Phi_\nu)$ and consider any 
$\tilde{v}^* \in \partial_{\text{{\tiny $\delta,\!\bar{z}$}}}\,\Phi_\nu (\bar{v})$.
From the definition of $\partial_{\text{{\tiny $\delta,\!\bar{z}$}}}\,\Phi_\nu$ we have $\bar{v} \in H^1(\Omega)$ 
and the following inequality for all $\tilde{v} \in H^1(\Omega)$:
\begin{equation}\label{Equation-2-24}
(\tilde{v}^*,\tilde{v}-\bar{v})_{\text{{\tiny $\delta,\!\bar{z}$}}} \le \Phi_\nu (\tilde{v})-\Phi_\nu (\bar{v}).
\end{equation}
Substituting $\tilde{v}=\bar{v}+\lambda \hat{v} \in H^1(\Omega)$ for all $\hat{v} \in H^1(\Omega)$ and $\lambda \in \mathbb{R}$ into \eqref{Equation-2-24}, 
we see that the following inequalities hold: for $\lambda >0$
\begin{align}
\label{Equation-2-25}
(\tilde{v}^*,\hat{v})_{\text{{\tiny $\delta,\!\bar{z}$}}} &\le \nu \int_\Omega \nabla \bar{v}(x) \cdot \nabla \hat{v}(x)dx
+\frac{\lambda \nu}{2}\int_\Omega |\nabla \hat{v}(x)|^2dx\\
\nonumber
&\hspace*{2cm}+\int_\Omega \bar{v}(x) \hat{v}(x)dx+\frac{\lambda}{2}\int_\Omega |\hat{v}(x)|^2dx
\end{align}
and for $\lambda <0$
\begin{align}
\label{Equation-2-26}
(\tilde{v}^*,\hat{v})_{\text{{\tiny $\delta,\!\bar{z}$}}} &\ge \nu \int_\Omega \nabla \bar{v}(x) \cdot \nabla \hat{v}(x)dx
+\frac{\lambda \nu}{2}\int_\Omega |\nabla \hat{v}(x)|^2dx\\
\nonumber
&\hspace*{2cm}+\int_\Omega \bar{v}(x) \hat{v}(x)dx+\frac{\lambda}{2}\int_\Omega |\hat{v}(x)|^2dx
\end{align}
Taking the limits $\lambda \downarrow 0$ in \eqref{Equation-2-25} and $\lambda \uparrow 0$ in \eqref{Equation-2-26}, it yields
\begin{equation}\label{Equation-2-27}
(\tilde{v}^*,\hat{v})_{\text{{\tiny $\delta,\!\bar{z}$}}}=\nu \int_\Omega \nabla \bar{v}(x) \cdot \nabla \hat{v}(x)dx+\int_\Omega \bar{v}(x) \hat{v}(x)dx,
\quad \forall \hat{v} \in H^1(\Omega).
\end{equation}
Since the function 
\begin{equation*}
\hat{v} \in L^2(\Omega) \longmapsto (\tilde{v}^*,\hat{v})_{\text{{\tiny $\delta,\!\bar{z}$}}}-\int_\Omega \bar{v}(x) \hat{v}(x)dx
\end{equation*}
is continuous and linear on $L^2(\Omega)$, we see that there exists a function $\eta \in L^2(\Omega)$ such that 
\begin{equation*}
(\tilde{v}^*,\hat{v})_{\text{{\tiny $\delta,\!\bar{z}$}}}-\int_\Omega \bar{v}(x) \hat{v}(x)dx=(\eta,\hat{v})_{L^2(\Omega)},\quad \forall \hat{v} \in L^2(\Omega).
\end{equation*}
so, the following equality holds:
\begin{equation*}
(\eta,\hat{v})_{L^2(\Omega)}=\nu \int_\Omega \nabla \bar{v}(x) \cdot \nabla \hat{v}(x)dx,\quad \forall \hat{v} \in H^1(\Omega),
\end{equation*}
By the variational characterization of the Neumann Laplacian (which is equivalent to (s4) in Theorem \ref{Theorem-1-1}), it follows that 
\begin{equation*}
-\Delta_N \bar{v}=\frac{\eta}{\nu} \in L^2(\Omega),\quad \text{hence}, \quad \bar{v} \in D(-\Delta_N).
\end{equation*}
Therefore, by the Green formula for the Neumann Laplacian, the following equality holds in the weak sense:
\begin{gather}
\label{Equation-2-28}
\nu \int_\Omega \nabla \bar{v}(x) \cdot \nabla \hat{v}(x)dx+\int_\Omega \bar{v}(x) \hat{v}(x)dx=
\int_\Omega \left\{ -\nu \Delta_N \bar{v}(x)+\bar{v}(x)\right\} \hat{v}(x)dx,\\
\nonumber
\forall \hat{v} \in H^1(\Omega).
\end{gather}
Since $\bar{v} \in D(\partial_{\text{{\tiny $\delta,\!\bar{z}$}}}\,\Phi_\nu)$ is arbitrary, we obtain 
\begin{equation}\label{Equation-2-29}
D(\partial_{\text{{\tiny $\delta,\!\bar{z}$}}}\,\Phi_\nu) \subset D(-\Delta_N).
\end{equation}
From \eqref{Equation-2-27} and \eqref{Equation-2-28} we deduce the following equality for all $\hat{v} \in H^1(\Omega)$:
\begin{equation}\label{Equation-2-30}
(\tilde{v}^*,\hat{v})_{\text{{\tiny $\delta,\!\bar{z}$}}}=\int_\Omega \left\{ -\nu \Delta_N \bar{v}(x)+\bar{v}(x)\right\} \hat{v}(x)dx.
\end{equation}
Since $H^1(\Omega)$ is dense in $L^2(\Omega)$ and the quasi-variational norm $\|\cdot\|_{\text{{\tiny $\delta,\!\bar{z}$}}}$ is equivalent to the flat one 
$\|\cdot\|_{L^2(\Omega)}$ by (1) of Lemma \ref{Lemma-2-2}, we see that $H^1(\Omega)$ is also dense in $L^2(\delta,\bar{z})$.
Since both sides of \eqref{Equation-2-30} define continuous linear functionals on $H^1(\Omega)$, the equality extends to all $\hat v \in L^2(\delta,\bar{z})$,
which yields  
\begin{equation*}
(\tilde{v}^*,\hat{v})_{\text{{\tiny $\delta,\!\bar{z}$}}}=\left(-\nu \Delta_N \bar{v}+\bar{v},\hat{v}\right)_{L^2(\Omega)}
=\left( M_{\text{{\tiny $\delta,\!\bar{z}$}}}^{-1} \left(-\nu \Delta_N \bar{v}+\bar{v}\right),\hat{v}\right)_{\text{{\tiny $\delta,\!\bar{z}$}}},
\quad \forall \hat{v} \in L^2(\delta,\bar{z}).
\end{equation*}
Hence, we obtain
\begin{equation}\label{Equation-2-31}
\partial_{\text{{\tiny $\delta,\!\bar{z}$}}}\,\Phi_\nu (\bar{v}) \subset \left\{ M_{\text{{\tiny $\delta,\!\bar{z}$}}}^{-1} \left(-\nu \Delta_N \bar{v}+\bar{v}\right) \right\}.
\end{equation}
\indent
As a result of \eqref{Equation-2-22}, \eqref{Equation-2-23}, \eqref{Equation-2-29} and \eqref{Equation-2-31}, the proof of this lemma is complete.
\end{proof}
%%%%%
%%%%%
%%%%%
Now, for each $\tilde{z} \in W^{1,2}(0,T\,;L^2(\Omega))$ and $\tilde{g} \in L^2(0,T\,;L^2(\Omega))$ we consider the following initial-boundary 
value problem $\text{(AP1)}{}_{\text{{\tiny $\delta,\!\nu$}}}$:=\{\eqref{Equation-2-32}--\eqref{Equation-2-34}\} as the first auxiliary problem 
of $\text{(AP)}{}_{\text{{\tiny $\delta,\!\varepsilon,\!\nu$}}}$:
\begin{gather}
\label{Equation-2-32}
\left( \mathcal{K}_{\text{{\tiny $\delta$}}} \tilde{z}\right)^{-1} v'-\nu \Delta v+v=\left( \mathcal{K}_{\text{{\tiny $\delta$}}} \tilde{z}\right)^{-1} \tilde{g},
\quad \text{a.e. in} \quad Q_T,\\
\label{Equation-2-33}
\nabla v \cdot \bm{n}=0,\quad \text{a.e. on} \quad \Sigma_T,\\
\label{Equation-2-34}
v(0)=v_0, \quad \text{a.e. in} \quad \Omega.
\end{gather}
Lemmas \ref{Lemma-2-2} and \ref{Lemma-2-4} verify the assumptions of \cite[Theorem 3.1]{ito-2019}, including the uniform equivalence of norms 
and the characterization of the subdifferential.
Hence, applying \cite[Theorem 3.1]{ito-2019} together with Lemmas \ref{Lemma-2-2} and \ref{Lemma-2-4}, 
we obtain Proposition \ref{Proposition-2-1}, which guarantees the existence and uniqueness of solutions to $\text{(AP1)}{}_{\text{{\tiny $\delta,\!\nu$}}}$.
%%%%%
%%%%%
%%%%%
\begin{proposition}\label{Proposition-2-1}
For each pair $(\tilde{z},\tilde{g}) \in W^{1,2}(0,T;L^2(\Omega)) \times L^2(0,T;L^2(\Omega))$ we consider the following Cauchy problem 
$\text{(TGF)}{}_{\text{{\tiny $\delta,\!\nu$}}}$:=\{(\ref{Equation-2-35})--(\ref{Equation-2-37})\} of a twisted gradient flow with a forcing term on the time-dependent 
twisted spaces $\{L^2(\delta,\bar{z}(t))\,;\,0 \le t \le T\}$:
\begin{gather}
\label{Equation-2-35}
v'(t)+v^*(t)=\tilde{g}(t) \quad \text{in} \quad L^2(\delta,\tilde{z}(t)),\quad \text{a.e.}~t \in (0,T),\\
\label{Equation-2-36}
v^*(t) \in \partial_{\text{{\tiny $\delta,\!\tilde{z}(t)$}}} \Phi_\nu (v(t)),\quad \text{a.e.}~t \in (0,T),\\
\label{Equation-2-37}
v(0)=v_0 \quad \text{in} \quad L^2(\Omega),
\end{gather}
Then, the problem $\text{(TGF)}{}_{\text{{\tiny $\delta,\!\nu$}}}$ admits a unique solution 
\begin{gather}
\label{Equation-2-38}
v_{\text{{\tiny $\delta,\!\nu$}}}(v_0,\tilde{z},\tilde{g}) \in W^{1,2}(0,T\,;L^2(\Omega)) \cap L^\infty (0,T\,;H^1(\Omega))\\
\label{Equation-2-39}
\text{with} \quad v_{\text{{\tiny $\delta,\!\nu$}}}^*(v_0,\tilde{z},\tilde{g}) \in L^2(0,T\,;L^2(\Omega)).
\end{gather}
\indent
Moreover, the unique solution $v_{\text{{\tiny $\delta,\!\nu$}}}(v_0,\tilde{z},\tilde{g})$ is also the unique strong solution to 
$\text{(AP1)}{}_{\text{{\tiny $\delta,\!\nu$}}}$ on $[0,T]$, and has the regularity 
\begin{equation*}
v_{\text{{\tiny $\delta,\!\nu$}}}(v_0,\tilde z,\tilde g) \in L^2(0,T\,;H^2(\Omega)).
\end{equation*}
\end{proposition}
%%%%%
%%%%%
%%%%%
\begin{proof}[Proof.]
First, we apply \cite[Theorem 3.1]{ito-2019} together with Lemma \ref{Lemma-2-2} and the regularities 
\begin{equation*}
\tilde{z} \in W^{1,2}(0,T\,;L^2(\Omega)),\quad v_0 \in H^1(\Omega),\quad \tilde{g} \in L^2(0,T\,;L^2(\Omega)).
\end{equation*}
Then, it follows that $\text{(TGF)}{}_{\text{{\tiny $\delta,\!\nu$}}}$ has a unique solution $v_{\text{{\tiny $\delta,\!\nu$}}}(v_0,\tilde z,\tilde g)$
satisfying \eqref{Equation-2-38} and \eqref{Equation-2-39}.
\indent
In the rest of this proof, we set $\tilde{v}:=v_{\text{{\tiny $\delta,\!\nu$}}}(v_0,\tilde z,\tilde g)$.
By Lemma \ref{Lemma-2-4}, we deduce that for any $t \in [0,T]$ the subdifferential $\partial_{\text{{\tiny $\delta,\!\tilde{z}(t)$}}} \Phi_\nu$ is single-valued and 
\{\eqref{Equation-2-35},\,\eqref{Equation-2-36}\} can be rewritten into the following evolution equation for a.e. $t \in (0,T)$:
\begin{equation*}
\tilde{v}'(t)+M_{\text{{\tiny $\delta,\!\tilde{z}(t)$}}}^{-1} \left(-\nu \Delta_N \tilde{v}(t)+\tilde{v}(t) \right)=\tilde{g}(t) \quad \text{in} \quad L^2(\delta,\tilde{z}(t)),
\end{equation*}
hence, by Lemma \ref{Lemma-2-1}
\begin{equation}\label{Equation-2-40}
M_{\text{{\tiny $\delta,\!\tilde{z}(t)$}}}\tilde{v}'(t)-\nu \Delta_N \tilde{v}(t)+\tilde{v}(t)=M_{\text{{\tiny $\delta,\!\tilde{z}(t)$}}}\tilde{g}(t) \quad \text{in} \quad L^2(\Omega).
\end{equation}
Recalling the definition of $M_{\text{{\tiny $\delta,\!\tilde{z}(t)$}}}$ (cf. \eqref{Equation-2-13}), the equation \eqref{Equation-2-40} is equivalent to 
\eqref{Equation-2-32}, that is, \eqref{Equation-2-32} holds.
Besides, since by (a) of Lemma \ref{Lemma-2-2} the operator $M_{\text{{\tiny $\delta,\!\tilde{z}(t)$}}}$ is bounded in $L^2(\Omega)$ for all $t \in [0,T]$, 
from \eqref{Equation-2-40} we get 
\begin{equation*}
-\nu \Delta_N \tilde{v}=M_{\text{{\tiny $\delta,\!\tilde{z}$}}}\tilde{g}-M_{\text{{\tiny $\delta,\!\tilde{z}$}}}\tilde{v}'-\tilde{v} \in L^2(0,T\,;L^2(\Omega)).
\end{equation*}
Using the Neumann elliptic regularity \eqref{Equation-1-8}, we get the regularity $\tilde{v} \in L^2(0,T\,;H^2(\Omega))$, and complete the proof of this proposition.
\end{proof}
%%%%%
%%%%%
%%%%%
\indent
By Proposition \ref{Proposition-2-1}, we can define a single-valued solution operator 
\begin{equation*}
S_{\text{{\tiny $1,\!\delta,\!\nu$}}}:C([0,T]\,;L^2(\Omega)) \times L^2(0,T\,;L^2(\Omega)) \rightarrow C([0,T]\,;L^2(\Omega))
\end{equation*}
associated with the initial-boundary value problem $\text{(AP1)}{}_{\text{{\tiny $\delta,\!\nu$}}}$ by
\begin{gather*}
D(S_{\text{{\tiny $1,\!\delta,\!\nu$}}})=W^{1,2}(0,T\,;L^2(\Omega)) \times L^2(0,T\,;L^2(\Omega)),\\
R(S_{\text{{\tiny $1,\!\delta,\!\nu$}}}) \subset W^{1,2}(0,T\,;L^2(\Omega)) \cap L^\infty (0,T\,;H^1(\Omega)) \cap L^2(0,T\,;H^2(\Omega)),\\
S_{\text{{\tiny $1,\!\delta,\!\nu$}}} (\tilde{z},\tilde{g}):=v_{\text{{\tiny $\delta,\!\nu$}}}(v_0,\tilde{z},\tilde{g}).
\end{gather*}
We conclude this subsection by proving Lemmas \ref{Lemma-2-5}--\ref{Lemma-2-7}.
First, we show Lemma \ref{Lemma-2-5}.
In particular, this lemma guarantees that for each $\delta \in (1,\infty)$ the solution $v_{\text{{\tiny $\delta,\!\nu$}}}$ is bounded in $L^2(0,T\,;L^2(\Omega))$ uniformly 
with respect to $\nu \in (0,1)$, which comes from \eqref{Equation-2-21}.
%%%%%
%%%%%
%%%%%
\begin{lemma}\label{Lemma-2-5}
There exists a constant $C_4>0$, which depends on the following values:
\begin{equation*}
\delta,\quad T, \quad \|v_0\|_{L^2(\Omega)},\quad \int_0^T \|\tilde{z}'(t)\|_{L^2(\Omega)}dt, \quad \int_0^T \|\tilde{g}(t)\|_{L^2(\Omega)}^2dt,
\end{equation*}
such that the following boundedness estimate holds for all $\nu \in (0,1)$: 
\begin{equation*}
\max_{0 \le t \le T} \|(S_{\text{{\tiny $1,\!\delta,\!\nu$}}}(\tilde{z},\tilde{g}))(t)\|_{L^2(\Omega)}^2
+\int_0^T \Phi_\nu ((S_{\text{{\tiny $1,\!\delta,\!\nu$}}}(\tilde{z},\tilde{g}))(t))dt \le C_4.
\end{equation*}
\end{lemma}
%%%%%
%%%%%
%%%%%  
\begin{proof}[Proof.]
For simplicity we set $\tilde{v}:=S_{\text{{\tiny $1,\!\delta,\!\nu$}}}(\tilde{z},\tilde{g})$.
By Proposition \ref{Proposition-2-1}, we have
\begin{gather}
\label{Equation-2-41}
\tilde{v}'(t)+\tilde{v}^*(t)=\tilde{g}(t) \quad \text{in} \quad L^2(\delta,\tilde{z}(t)),\quad \text{a.e.}~t \in (0,T),\\
\label{Equation-2-42}
\tilde{v}^*(t) \in \partial_{\text{{\tiny $\delta,\!\tilde{z}(t)$}}} \Phi_\nu (\tilde{v}(t)),\quad \text{a.e.}~t \in (0,T).
\end{gather}
From Lemma \ref{Lemma-2-3} we deduce the following energy inequality for a.e. $t \in (0,T)$:
\begin{equation}\label{Equation-2-43}
\frac{d}{dt} \|\tilde{v}(t)\|_{\text{{\tiny $\delta,\!\tilde{z}(t)$}}}^2-2(\tilde{v}'(t),\tilde{v}(t))_{\text{{\tiny $\delta,\!\tilde{z}(t)$}}} 
\le \delta^3 C_3 \|\tilde{z}'(t)\|_{L^2(\Omega)} \|\tilde{v}(t)\|_{\text{{\tiny $\delta,\!\tilde{z}(t)$}}}^2.
\end{equation}
Substituting \eqref{Equation-2-41} into \eqref{Equation-2-43}, we derive the following inequality for a.e. $t \in (0,T)$:
\begin{align}
\label{Equation-2-44}
&\frac{d}{dt} \|\tilde{v}(t)\|_{\text{{\tiny $\delta,\!\tilde{z}(t)$}}}^2+2(\tilde{v}^*(t),\tilde{v}(t))_{\text{{\tiny $\delta,\!\tilde{z}(t)$}}}\\
\nonumber
&\hspace*{1cm}\le \delta^3 C_3 \|\tilde{z}'(t)\|_{L^2(\Omega)} \|\tilde{v}(t)\|_{\text{{\tiny $\delta,\!\tilde{z}(t)$}}}^2
+2(\tilde{g}(t),\tilde{v}(t))_{\text{{\tiny $\delta,\!\tilde{z}(t)$}}}.
\end{align}
Using $\Phi_\nu (0)=0$ and the definition of the subdifferential $\partial_{\text{{\tiny $\delta,\!\tilde{z}(t)$}}} \Phi_\nu$ with \eqref{Equation-2-42}, 
we obtain the following inequality for a.e. $t \in (0,T)$:
\begin{equation*}
-(\tilde{v}^*(t),\tilde{v}(t))_{\text{{\tiny $\delta,\!\tilde{z}(t)$}}}=(\tilde{v}^*(t),0-\tilde{v}(t))_{\text{{\tiny $\delta,\!\tilde{z}(t)$}}} \le 
\Phi_\nu (0)-\Phi_\nu (\tilde{v}(t))=-\Phi_\nu (\tilde{v}(t)),
\end{equation*}
so, we see from \eqref{Equation-2-44} that the following inequality holds for a.e. $t \in (0,T)$:  
\begin{align}
\label{Equation-2-45}
&\frac{d}{dt} \|\tilde{v}(t)\|_{\text{{\tiny $\delta,\!\tilde{z}(t)$}}}^2+2 \Phi_\nu (\tilde{v}(t))\\
\nonumber
&\hspace*{1cm}\le \left\{\delta^3 C_3 \|\tilde{z}'(t)\|_{L^2(\Omega)}+1\right\} \|\tilde{v}(t)\|_{\text{{\tiny $\delta,\!\tilde{z}(t)$}}}^2
+\|\tilde{g}(t)\|_{\text{{\tiny $\delta,\!\tilde{z}(t)$}}}^2.
\end{align}
Applying the Gronwall lemma to \eqref{Equation-2-45}, we derive the following inequality for all $t \in [0,T]$:
\begin{align}
\label{Equation-2-46}
&\,\|\tilde{v}(t)\|_{\text{{\tiny $\delta,\!\tilde{z}(t)$}}}^2+2\int_0^t \Phi_\nu (\tilde{v}(s)) 
\exp \left( \int_s^t \left\{\delta^3 C_3 \|\tilde{z}'(\tau)\|_{L^2(\Omega)}+1\right\}d\tau\right)ds\\
\nonumber
\le&\,\|v_0\|_{\text{{\tiny $\delta,\!v_0$}}}^2\exp \left( \int_0^t \left\{\delta^3 C_3 \|\tilde{z}'(s)\|_{L^2(\Omega)}+1\right\}ds\right)\\
\nonumber
&\,\hspace*{1cm}+\int_0^t \|\tilde{g}(s)\|_{\text{{\tiny $\delta,\!\tilde{z}(s)$}}}^2 
\exp \left( \int_s^t \left\{\delta^3 C_3 \|\tilde{z}'(\tau)\|_{L^2(\Omega)}+1\right\}d\tau\right)ds
\end{align}
Since $\Phi_\nu (\tilde{v}(t))$ is nonnegative for all $t \in [0,T]$, from (a) of Lemma \ref{Lemma-2-2} and \eqref{Equation-2-46} we have
\begin{align}
\label{Equation-2-47}
\bullet~&\left(\max_{0 \le t \le T} \|\tilde{v}(t)\|_{L^2(\Omega)}\right)^2 \le \left( \frac{\delta C_2}{C_1} \right)^2 \left\{\|v_0\|_{L^2(\Omega)}^2+
\int_0^T \|\tilde{g}(t)\|_{L^2(\Omega)}^2dt\right\}\\
\nonumber
&\hspace*{4cm}\times \exp \left( \int_0^T \left\{\delta^3 C_3 \|\tilde{z}'(t)\|_{L^2(\Omega)}+1\right\}dt\right),\\
\label{Equation-2-48}
\bullet~&\int_0^T \Phi_\nu (\tilde{v}(t))dt \le \frac{\delta (C_2)^2}{2} \left\{\|v_0\|_{L^2(\Omega)}^2+\int_0^T \|\tilde{g}(t)\|_{L^2(\Omega)}^2dt\right\}\\
\nonumber
&\hspace*{3cm}\times \exp \left( \int_0^T \left\{\delta^3 C_3 \|\tilde{z}'(t)\|_{L^2(\Omega)}+1\right\}dt\right).
\end{align}
Hence, from \eqref{Equation-2-47} and \eqref{Equation-2-48} the desired estimate in this lemma follows.
\end{proof}
%%%%%
%%%%%
%%%%%
Next, we show Lemma \ref{Lemma-2-6}.
This lemma implies that for each $\delta \in (1,\infty)$ the family of solutions $\{v_{\text{{\tiny $\delta,\!\nu$}}}\,;\,\nu \in (0,1)\}$ is bounded in 
$C([0,T]\,;L^2(\Omega)) \cap W^{1,2}(0,T\,;L^2(\Omega))$.
%%%%%
%%%%%
%%%%%
\begin{lemma}\label{Lemma-2-6}
There exists a constant $C_5>0$, which depends on the following values:
\begin{equation*}
\delta,\quad T, \quad \|v_0\|_{H^1(\Omega)},\quad \int_0^T \|\tilde{z}'(t)\|_{L^2(\Omega)}dt, 
\quad \int_0^T \|\tilde{g}(t)\|_{L^2(\Omega)}^2dt,
\end{equation*}
such that the following boundedness estimates hold for all $\nu \in (0,1)$:
\begin{align*}
&\int_0^T \|(S_{\text{{\tiny $1,\!\delta,\!\nu$}}}(\tilde{z},\tilde{g}))'(t)\|_{L^2(\Omega)}^2dt
+\sup_{0 \le t \le T} \Phi_\nu ((S_{\text{{\tiny $1,\!\delta,\!\nu$}}}(\tilde{z},\tilde{g}))(t))\\
&\hspace*{3cm}+\nu^2 \int_0^T \|(S_{\text{{\tiny $1,\!\delta,\!\nu$}}}(\tilde{z},\tilde{g}))(t)\|_{H^2(\Omega)}^2 dt \le C_5.
\end{align*}
\end{lemma}
%%%%%
%%%%%
%%%%%
\begin{proof}[Proof.]
We use the same notations as in the proof of Lemma \ref{Lemma-2-5}.
In order to prove this lemma, we use the following energy inequality, which was derived in \cite[Theorem 3.1]{ito-2019}:
for each $\lambda >0$ there exists a constant $C_{\text{{\tiny $2,\!6,\!1$}}}(\lambda) >0$ such that the following inequality holds for all $t \in [0,T]$:
\begin{gather}
\label{Equation-2-49}
\Phi_\nu (\tilde{v}(t))+\int_0^t (\tilde{v}'(s),\tilde{v}'(s)-\tilde{g}(s))_{\text{{\tiny $\delta,\!\tilde{z}(s)$}}}ds\\
\nonumber
\le \Phi_\nu (v_0)+\lambda \int_0^t \|\tilde{v}'(s)-\tilde{g}(s)\|_{\text{{\tiny $\delta,\!\tilde{z}(s)$}}}^2ds
+C_{\text{{\tiny $2,\!6,\!1$}}}(\lambda) \int_0^t \{\Phi_\nu (\tilde{v}(s))+1\} \|\tilde{z}'(s)\|_{L^2(\Omega)}ds.
\end{gather}
Although the functional $\Phi_\nu$ itself is independent of time, the associated Moreau-Yosida approximation depends on $t$ through the state-dependent 
inner product $(\cdot,\cdot)_{\text{{\tiny $\delta,\!\tilde{z}(t)$}}}$.
As a consequence, the time-dependence of the metric structure generates the additional perturbation term appearing in \eqref{Equation-2-49}.
The derivation of the energy inequality \eqref{Equation-2-49} follows from the abstract perturbation theory developed in \cite{ito-2022}.
More precisely, Lemmas \ref{Lemma-2-2} and \ref{Lemma-2-4} verify the structural assumptions required in the abstract framework, 
including the uniform equivalence of the state-dependent norms and the characterization of the associated subdifferential operators.
Therefore, we omit the proof of \eqref{Equation-2-49}.\\
\indent
Now, we take $\lambda =1/8$ and set $C_{\text{{\tiny $2,\!6,\!1$}}}:=C_{\text{{\tiny $2,\!6,\!1$}}}(1/8)$.
Using the Cauchy-Schwarz inequality and the Young inequality in \eqref{Equation-2-49}, we deduce the following inequality for a.e. $t \in (0,T)$:
\begin{gather}
\label{Equation-2-50}
\Phi_\nu (\tilde{v}(t))+\frac{1}{2} \int_0^t \|\tilde{v}'(s)\|_{\text{{\tiny $\delta,\!\tilde{z}(s)$}}}^2ds 
\le \Phi_\nu (v_0)+\frac{5}{4} \int_0^t \|\tilde{g}(s)\|_{\text{{\tiny $\delta,\!\tilde{z}(s)$}}}^2ds\\
\nonumber
+C_{\text{{\tiny $2,\!6,\!1$}}} \int_0^t \left\{\Phi_\nu (\tilde{v}(s))+1\right\}\|\tilde{z}'(s)\|_{L^2(\Omega)}ds.
\end{gather}
\indent
Now, in order to eliminate the perturbation term
\begin{equation*}
C_{\text{{\tiny $2,\!6,\!1$}}}\int_0^t \left\{ \Phi_\nu(\tilde v(s))+1\right\} \|\tilde z'(s)\|_{L^2(\Omega)}ds
\end{equation*}
we introduce a function $\xi : [0,T] \rightarrow \mathbb{R}$ by
\begin{gather}
\label{Equation-2-51}
\xi (t):=\int_0^t \left\{\Phi_\nu (\tilde{v}(s))+1 \right\}\|\tilde{z}'(s)\|_{L^2(\Omega)}ds \cdot \exp \left(-C_{\text{{\tiny $2,\!6,\!1$}}}\int_0^t \|\tilde{z}'(s)\|_{L^2(\Omega)}ds\right),\\
\nonumber
\forall t \in [0,T].
\end{gather}
By differentiating \eqref{Equation-2-51}, from \eqref{Equation-2-50} we obtain
\begin{align*}
\frac{d}{dt} \xi (t) &=\left[ \Phi_\nu (\tilde{v}(t))+1-C_{\text{{\tiny $2,\!6,\!1$}}}\int_0^t \left\{\Phi_\nu (\tilde{v}(s))+1\right\} \|\tilde{z}'(s)\|_{L^2(\Omega)}ds \right]\\
&\hspace*{2cm}\times \|\tilde{z}'(t)\|_{L^2(\Omega)} \exp \left(-C_{\text{{\tiny $2,\!6,\!1$}}}\int_0^t \|\tilde{z}'(s)\|_{L^2(\Omega)}ds\right)\\
&\le \left\{ \Phi_\nu (v_0)+1+\frac{5}{4} \int_0^t \|\tilde{g}(s)\|_{\text{{\tiny $\delta,\!\tilde{z}(s)$}}}^2ds \right\}\\
&\hspace*{2cm}\times \|\tilde{z}'(t)\|_{L^2(\Omega)} \exp \left(-C_{\text{{\tiny $2,\!6,\!1$}}}\int_0^t \|\tilde{z}'(s)\|_{L^2(\Omega)}ds\right),
\end{align*}
hence, because of $\xi (0)=0$
\begin{align}
\label{Equation-2-52}
\xi (t) &\le \int_0^t \left\{ \Phi_\nu (v_0)+1+\frac{5}{4} \int_0^s \|\tilde{g}(\tau)\|_{\text{{\tiny $\delta,\!\tilde{z}(\tau)$}}}^2d\tau \right\}\\
\nonumber
&\hspace*{2cm}\times \|\tilde{z}'(s)\|_{L^2(\Omega)} \exp \left(-C_{\text{{\tiny $2,\!6,\!1$}}}\int_0^s \|\tilde{z}'(\tau)\|_{L^2(\Omega)}d\tau\right).
\end{align}
From \eqref{Equation-2-51} and \eqref{Equation-2-52} we derive
\begin{gather}
\label{Equation-2-53}
\int_0^t \left\{\Phi_\nu (\tilde{v}(s))+1 \right\}\|\tilde{z}'(s)\|_{L^2(\Omega)}ds 
\le \int_0^t \left\{ \Phi_\nu (v_0)+1+\frac{5}{4} \int_0^s \|\tilde{g}(\tau)\|_{\text{{\tiny $\delta,\!\tilde{z}(\tau)$}}}^2d\tau \right\}\\
\nonumber
\times \|\tilde{z}'(s)\|_{L^2(\Omega)} \exp \left(C_{\text{{\tiny $2,\!6,\!1$}}} \int_s^t \|\tilde{z}'(\tau)\|_{L^2(\Omega)}d\tau\right)ds.
\end{gather}
Substituting \eqref{Equation-2-53} into \eqref{Equation-2-50}, we deduce the following inequality for all $t \in [0,T]$:
\begin{align*}
&\,\Phi_\nu (\tilde{v}(t))+\frac{1}{2} \int_0^t \|\tilde{v}'(s)\|_{\text{{\tiny $\delta,\!\tilde{z}(s)$}}}^2ds\\
\le&\,\Phi_\nu (v_0)+\frac{5}{4} \int_0^t \|\tilde{g}(s)\|_{\text{{\tiny $\delta,\!\tilde{z}(s)$}}}^2ds+C_{\text{{\tiny $2,\!6,\!1$}}} \int_0^t \left\{ \Phi_\nu (v_0)+1
+\frac{5}{4} \int_0^s \|\tilde{g}(\tau)\|_{\text{{\tiny $\delta,\!\tilde{z}(\tau)$}}}^2d\tau \right\}\\
&\,\hspace*{3cm}\times \|\tilde{z}'(s)\|_{L^2(\Omega)} \exp \left( C_{\text{{\tiny $2,\!6,\!1$}}} \int_s^t \|\tilde{z}'(\tau)\|_{L^2(\Omega)}d\tau\right)ds\\
\le&\,\left\{\frac{1}{2} \int_\Omega |\nabla v_0|^2dx+\frac{1}{2}\int_\Omega |v_0|^2dx+1+\frac{5 \delta (C_2)^2}{4} \int_0^T \|\tilde{g}(t)\|_{L^2(\Omega)}^2dt\right\}\\
&\,\hspace*{0.5cm} \times \left\{1+C_{\text{{\tiny $2,\!6,\!1$}}} \int_0^T \|\tilde{z}'(t)\|_{L^2(\Omega)}dt \cdot 
\exp \left(C_{\text{{\tiny $2,\!6,\!1$}}}\int_0^T \|\tilde{z}'(t)\|_{L^2(\Omega)}dt\right)\right\}
=:C_{\text{{\tiny $2,\!6,\!2$}}},
\end{align*}
which yields that the following estimates hold:
\begin{equation}\label{Equation-2-54}
\sup_{0 \le t \le T} \Phi_\nu (\tilde{v}(t))+\int_0^T \|\tilde{v}'(t)\|_{L^2(\Omega)}^2dt \le C_{\text{{\tiny $2,\!6,\!2$}}} \left\{1+\frac{2\delta}{(C_1)^2}\right\}.
\end{equation}
Going back to \eqref{Equation-2-40}, we have 
\begin{equation}\label{Equation-2-55}
\nu \Delta_N \tilde{v}(t)=\tilde{v}(t)+M_{\text{{\tiny $\delta,\!\tilde{z}(t)$}}}\tilde{v}'(t)-M_{\text{{\tiny $\delta,\!\tilde{z}(t)$}}}\tilde{g}(t) \quad
\text{in} \quad L^2(\Omega),\quad \text{a.e.}~t \in (0,T).
\end{equation}
From \eqref{Equation-1-8}, \eqref{Equation-2-16} and \eqref{Equation-2-55} we derive 
\begin{align*}
\nu \|\tilde{v}(t)\|_{H^2(\Omega)}&\le K_5 \left( \|\nu \Delta_N \tilde{v}(t)\|_{L^2(\Omega)}+\|\tilde{v}(t)\|_{L^2(\Omega)} \right)\\
&\le K_5 \left( \|M_{\text{{\tiny $\delta,\!\tilde{z}(t)$}}}\tilde{g}(t)\|_{L^2(\Omega)}+\|M_{\text{{\tiny $\delta,\!\tilde{z}(t)$}}}\tilde{v}'(t)\|_{L^2(\Omega)}
+2\|\tilde{v}(t)\|_{L^2(\Omega)} \right)\\
&\le K_5 \left(2+\frac{\delta}{K_1}\right) \left( \|\tilde{g}(t)\|_{L^2(\Omega)}+\|\tilde{v}'(t)\|_{L^2(\Omega)}+\|\tilde{v}(t)\|_{L^2(\Omega)} \right),
\end{align*}
which implies
\begin{align}
\label{Equation-2-56}
\nu^2 \int_0^T \|\tilde{v}(t)\|_{H^2(\Omega)}^2dt &\le 3 \left\{K_5 \left(1+\frac{\delta}{K_1}\right)\right\}^2 
\int_0^T \left( \|\tilde{g}(t)\|_{L^2(\Omega)}^2+\|\tilde{v}'(t)\|_{L^2(\Omega)}^2 \right)dt\\
\nonumber
&\hspace*{1cm}+3T \left\{K_5 \left(1+\frac{\delta}{K_1}\right)\right\}^2 \left( \max_{0 \le t \le T} \|\tilde{v}(t)\|_{L^2(\Omega)} \right)^2.
\end{align}
Combining \eqref{Equation-2-54} and \eqref{Equation-2-56} with Lemma \ref{Lemma-2-5}, we complete the proof of this lemma.
\end{proof}
%%%%%
%%%%%
%%%%%
\indent
At the end of this subsection, we show Lemma \ref{Lemma-2-7}, which guarantees the continuity property of the operator $S_{\text{{\tiny $1,\!\delta,\!\nu$}}}$,
which plays a key role to apply the Schauder fixed point theorem in Subsection \ref{Subsection-3-1}.
%%%%%
%%%%%
%%%%%
\begin{lemma}\label{Lemma-2-7}
We have the following convergence as $m \to \infty$:
\begin{equation}\label{Equation-2-57}
S_{\text{{\tiny $1,\!\delta,\!\nu$}}}(\tilde{z}_m,\tilde{g}_m) \longrightarrow S_{\text{{\tiny $1,\!\delta,\!\nu$}}}(\tilde{z},\tilde{g}) \quad \left\{
\begin{array}{l}
\text{in} \quad C([0,T]\,;L^2(\Omega)),\\[0.1cm]
\text{weakly in} \quad W^{1,2}(0,T\,;L^2(\Omega)),\\[0.1cm]
\text{weakly$^*$ in} \quad L^\infty (0,T\,;H^1(\Omega)),\\[0.1cm]
\text{weakly in} \quad L^2(0,T\,;H^2(\Omega)),
\end{array}
\right.
\end{equation}
whenever a sequence $\{(\tilde{z}_m,\tilde{g}_m)\}_{m \in \mathbb{N}} \subset D(S_{\text{{\tiny $1,\!\delta,\!\nu$}}})$ and a pair 
$(\tilde{z},\tilde{g}) \in D(S_{\text{{\tiny $1,\!\delta,\!\nu$}}})$ satisfy the following convergence as $m \to \infty$:
\begin{equation}
\label{Equation-2-58}
(\tilde{z}_m,\tilde{g}_m) \longrightarrow (\tilde{z},\tilde{g}) \quad \text{in} \quad C([0,T]\,;L^2(\Omega)) \times L^2(0,T\,;L^2(\Omega)),
\end{equation}
and the following boundedness:
\begin{equation}\label{Equation-2-59}
\sup_{m\,\in\,\mathbb{N}} \int_0^T \|z_m'(t)\|_{L^2(\Omega)} dt < \infty.
\end{equation}
\end{lemma}
%%%%%
%%%%%
%%%%%
\begin{proof}[Proof.]
In this proof we set $\tilde{v}:=S_{\text{{\tiny $1,\!\delta,\!\nu$}}}(\tilde{z},\tilde{g})$ and $\tilde{v}_m:=S_{\text{{\tiny $1,\!\delta,\!\nu$}}}(\tilde{z}_m,\tilde{g}_m)$ 
for all $m \in \mathbb{N}$.
Since from Lemma \ref{Lemma-2-6} with \eqref{Equation-2-59} the sequence $\{\tilde{v}_m\}_{m \in \mathbb N}$ is bounded in 
\begin{equation*}
W^{1,2}(0,T;L^2(\Omega)) \cap L^\infty (0,T;H^1(\Omega)) \cap L^2(0,T\,;H^2(\Omega)),
\end{equation*}
we have 
\begin{equation*}
\|\tilde{v}_m(t)-\tilde{v}_m(s)\|_{L^2(\Omega)} \le |t-s|^{\frac{1}{2}} \|\tilde{v}_m'\|_{L^2(0,T;L^2(\Omega))}, \quad \forall s,t \in [0,T],
\end{equation*}
hence, we see that the sequence $\{\tilde{v}_m\}_{m \in \mathbb N}$ is equicontinuous in $C([0,T];L^2(\Omega))$.
Since for each $t \in [0,T]$ the sequence $\{\tilde v_m(t)\}$ is bounded in $H^1(\Omega)$ and the embedding $H^1(\Omega) \subset L^2(\Omega)$ is compact,
the sequence $\{\tilde v_m(t)\}$ is relatively compact in $L^2(\Omega)$.
Combined with equicontinuity, the Ascoli-Arzel\`{a} theorem yields that there exist a subsequence, still denoted by $\{\tilde{v}_m\}_{m \in \mathbb N}$, and 
a function $\hat{v} \in C([0,T];L^2(\Omega))$ such that the following convergence holds as $m \to \infty$:
\begin{equation}\label{Equation-2-60}
\tilde{v}_m \longrightarrow \hat{v} \quad \left\{
\begin{array}{l}
\text{in} \quad C([0,T]\,;L^2(\Omega)),\\[0.1cm]
\text{weakly in} \quad W^{1,2}(0,T\,;L^2(\Omega)),\\[0.1cm]
\text{weakly$^*$ in} \quad L^\infty(0,T\,;H^1(\Omega)),\\[0.1cm]
\text{weakly in} \quad L^2(0,T\,;H^2(\Omega)).
\end{array}
\right.
\end{equation}
\indent
Now, we go back to Proposition \ref{Proposition-2-1}.
Using the definition of the subdifferential $\partial_{\text{{\tiny $\delta,\!\tilde{z}_m(t)$}}} \Phi_\nu$ on the time-dependent twisted space $L^2(\delta,\tilde{z}_m(t))$, 
we have 
\begin{gather*}
(\tilde{g}_m(t)-\tilde{v}_m'(t),\eta-\tilde{v}_m(t))_{\text{{\tiny $\delta,\!\tilde{z}_m(t)$}}} \le \Phi_\nu (\eta)-\Phi_\nu (\tilde{v}_m(t)),\\
\forall \eta \in L^2(\delta,\tilde{z}_m(t))=L^2(\Omega),\quad \text{a.e.}~t \in (0,T),
\end{gather*}
hence,
\begin{align}
\label{Equation-2-61}
&\int_0^T (\tilde{g}_m(t)-\tilde{v}_m'(t),\tilde{\xi}(t)-\tilde{v}_m(t))_{\text{{\tiny $\delta,\!\tilde{z}_m(t)$}}}dt\\
\nonumber
&\hspace*{1cm}\le \int_0^T \Phi_\nu (\tilde{\xi}(t))dt-\int_0^T \Phi_\nu (\tilde{v}_m(t))dt,\quad \forall \tilde{\xi} \in L^2(0,T\,;L^2(\Omega)).
\end{align}
Since the truncation operator $\alpha_{\text{{\tiny $\delta$}}}: \mathbb{R} \rightarrow \mathbb{R}$ is Lipschitz continuous (cf. \eqref{Equation-2-2}), the strong 
convergences \eqref{Equation-2-58} together with \eqref{Equation-2-60} and (2) of Lemma \ref{Lemma-2-2} guarantees the convergence of 
the quasi-variational inner products $(\cdot,\cdot)_{\text{{\tiny $\delta,\!\tilde{z}_m(t)$}}}$, whose detail proof is given in \cite[Lemma 2.2 and Lemma 3.4]{ito-2019}.
Passing to the limit in the left-hand side of \eqref{Equation-2-61}, and using the weak lower semicontinuity of the convex functional $\Phi_\nu$, we derive 
\begin{align}
\label{Equation-2-62}
&\int_0^T (\tilde{g}(t)-\hat{v}'(t),\tilde{\xi}(t)-\hat{v}(t))_{\text{{\tiny $\delta,\!\tilde{z}(t)$}}}dt\\
\nonumber
&\hspace*{1cm} \le \int_0^T \Phi_\nu (\tilde{\xi}(t))dt-\int_0^T \Phi_\nu (\hat{v}(t))dt,\quad \forall \tilde{\xi} \in L^2(0,T\,;L^2(\Omega)).
\end{align}
Applying \cite[Example 2.1.3]{Brezis-1973} together with \cite[Lemma 3.8]{ito-2019}, we infer from \eqref{Equation-2-62} that 
\begin{equation}\label{Equation-2-63}
\tilde{g}(t)-\hat{v}'(t) \in \partial_{\text{{\tiny $\delta,\!\tilde{z}(t)$}}} \Phi_\nu (\hat{v}(t)), \quad \text{a.e.}~t \in (0,T).
\end{equation}
By Lemma \ref{Lemma-2-4}, the inclusion \eqref{Equation-2-63} yields 
\begin{gather*}
\hat{v}'(t)+M_{\text{{\tiny $\delta,\!\tilde{z}(t)$}}}^{-1}\left(-\nu \Delta_N \hat{v}(t)+\hat{v}(t)\right)=\tilde{g}(t) \quad \text{in} \quad L^2(\delta,\tilde{z}(t)),
\quad \text{a.e.}~t \in (0,T).
\end{gather*}
Since $\tilde{v}_m(0)=v_0$ for all $m\in\mathbb N$ and $\tilde{v}_m \longrightarrow \hat{v}$ in $C([0,T];L^2(\Omega))$ as $m \to \infty$ (cf. \eqref{Equation-2-60}), 
we have $\hat{v}(0)=v_0$.
Hence, $\hat{v}$ is a solution to $\text{(AP1)}{}_{\text{{\tiny $\delta,\!\nu$}}}$ associated with the pair $(\tilde{z},\tilde{g})$.
By uniqueness of solutions to $\text{(AP1)}{}_{\text{{\tiny $\delta,\!\nu$}}}$, we obtain $\hat{v}=S_{\text{{\tiny $1,\!\delta,\!\nu$}}}(\tilde{z},\tilde{g})$.
Since every convergent subsequence of $\{\tilde{v}_m\}_{m \in \mathbb{N}}$ has the same limit, the whole sequence converges to 
$S_{\text{{\tiny $1,\!\delta,\!\nu$}}}(\tilde{z},\tilde{g})$ in $C([0,T];L^2(\Omega))$.
Besides, the weak and weak$^*$ convergences in \eqref{Equation-2-57} follow from the uniqueness of the limit together with \eqref{Equation-2-60}.
\end{proof}
%%%%%%%%%%%%%%%%%%%%%%%%%%%%%%%%%%%%%%%%%%%%%%%%%%%%%%%%%%%%%%%%%%%%%%%%%%%%%%%%%%%%%%%%%%%%%%%%%%%%%%%%%%%%%%%%%%%%%%%%%%%%%%%%%%%%%%%
%%%%%%%%%%%%%%%%%%%%%%%%%%%%%%%%%%%%%%%%%%%%%%%%%%%%%%%%%%%%%%%%%%%%%%%%%%%%%%%%%%%%%%%%%%%%%%%%%%%%%%%%%%%%%%%%%%%%%%%%%%%%%%%%%%%%%%%
%%%%%%%%%%%%%%%%%%%%%%%%%%%%%%%%%%%%%%%%%%%%%%%%%%%%%%%%%%%%%%%%%%%%%%%%%%%%%%%%%%%%%%%%%%%%%%%%%%%%%%%%%%%%%%%%%%%%%%%%%%%%%%%%%%%%%%%
\subsection{Evolution equation for the twisted metric generator}\label{Subsection-2-3}
%%%%%%%%%%%%%%%%%%%%%%%%%%%%%%%%%%%%%%%%%%%%%%%%%%%%%%%%%%%%%%%%%%%%%%%%%%%%%%%%%%%%%%%%%%%%%%%%%%%%%%%%%%%%%%%%%%%%%%%%%%%%%%%%%%%%%%%
%%%%%%%%%%%%%%%%%%%%%%%%%%%%%%%%%%%%%%%%%%%%%%%%%%%%%%%%%%%%%%%%%%%%%%%%%%%%%%%%%%%%%%%%%%%%%%%%%%%%%%%%%%%%%%%%%%%%%%%%%%%%%%%%%%%%%%%
%%%%%%%%%%%%%%%%%%%%%%%%%%%%%%%%%%%%%%%%%%%%%%%%%%%%%%%%%%%%%%%%%%%%%%%%%%%%%%%%%%%%%%%%%%%%%%%%%%%%%%%%%%%%%%%%%%%%%%%%%%%%%%%%%%%%%%%
We introduce ``\,{\sf a twisted metric generator}\,'' $\bar{z}$, whose evolution determines the state-dependent Hilbert structure through the nonlocal coefficient 
$\mathcal{K}_{\text{{\tiny $\delta$}}}\bar{z}$.\\
\indent
In the following argument, we set an admissible class $\mathcal{V}(v_0)$ for the function $v$ in the following way: a function $\tilde{v}$ belongs to $\mathcal{V}(v_0)$ 
if and only if $\tilde{v}(0)=v_0$ in $L^2(\Omega)$ and the regularity
\begin{equation*}
\tilde{v} \in W^{1,2}(0,T\,;L^2(\Omega)) \quad \text{with} \quad \sup_{0 \le t \le T} \Phi_\nu (\tilde{v}(t)) < \infty.
\end{equation*}
For each $z_0 \in H^1(\Omega)$ and $\tilde{v} \in \mathcal{V}(v_0)$ we consider the initial-boundary value problem $\text{(AP2)}{}_{\text{{\tiny $\varepsilon$}}}$:=
\{\eqref{Equation-2-64}--\eqref{Equation-2-66}\}, which generates the evolution of the metric variable $z$, as the second auxiliary problem 
of $\text{(AP)}{}_{\text{{\tiny $\delta,\!\varepsilon,\!\nu$}}}$:
\begin{gather}
\label{Equation-2-64}
z'-\varepsilon \Delta z=\tilde{v}_t,\quad \text{a.e.in}\quad Q_T,\\
\label{Equation-2-65}
\nabla z \cdot \bm{n}=0,\quad \text{a.e. on}\quad \Sigma_T,\\
\label{Equation-2-66}
z(0)=z_0,\quad \text{a.e. in} \quad \Omega.
\end{gather}
Especially, if the initial condition $z_0=v_0$ holds, the twisted metric generator $z$ initially coincides with the initial state of the twisted gradient flow.\\
\indent
First, we give the following lemma which provides the well-posedness and uniform estimates 
for the twisted metric generator system $\text{(AP2)}{}_{\text{{\tiny $\varepsilon$}}}$.
%%%%%
%%%%%
%%%%%
\begin{lemma}\label{Lemma-2-8}
For each $z_0 \in H^1(\Omega)$ and $\tilde{v} \in \mathcal{V}(v_0)$ the initial-boundary value problem $\text{(AP2)}{}_{\text{{\tiny $\varepsilon$}}}$ has a unique 
strong solution $z_{\text{{\tiny $\varepsilon$}}} (z_0,\tilde{v})$ on $[0,T]$ satisfying 
\begin{gather}
\label{Equation-2-67}
\hspace*{-0.8cm}z_{\text{{\tiny $\varepsilon$}}} (z_0,\tilde{v}) \in W^{1,2}(0,T\,;L^2(\Omega)) \cap L^\infty (0,T\,;H^1(\Omega)) \cap L^2(0,T\,;H^2(\Omega)),\\
\label{Equation-2-68}
(z_{\text{{\tiny $\varepsilon$}}}(z_0,\tilde{v}))'(t)-\varepsilon \Delta_N z_{\text{{\tiny $\varepsilon$}}}(z_0,\tilde{v}\,;t)=\tilde{v}'(t) \quad 
\text{in} \quad L^2(\Omega), \quad \text{a.e.}~t \in (0,T),\\
\label{Equation-2-69}
z_{\text{{\tiny $\varepsilon$}}}(z_0,\tilde{v}\,;0)=z_0 \quad \text{in} \quad L^2(\Omega).
\end{gather}
\indent
Moreover, there exists a constant $C_6>0$, which depends on the following values:
\begin{equation*}
T,\quad \|z_0\|_{H^1(\Omega)},\quad \int_0^T \|\tilde{v}'(t)\|_{L^2(\Omega)}^2dt,
\end{equation*}
such that the following estimate holds for all $\varepsilon \in (0,1)$:
\begin{align*}
&\left(\max_{0 \le t \le T} \|z_{\text{{\tiny $\varepsilon$}}} (z_0,\tilde{v}\,;t)\|_{L^2(\Omega)}\right)^2
+\int_0^T \|z_{\text{{\tiny $\varepsilon$}}}'(z_0,\tilde{v}\,;t)\|_{L^2(\Omega)}^2dt\\
&\hspace*{1cm}+\varepsilon \left(\sup_{0 \le t \le T} \|\nabla z_{\text{{\tiny $\varepsilon$}}} (z_0,\tilde{v}\,;t)\|_{\bm{L}^2(\Omega)} \right)^2
+\varepsilon^2 \int_0^T \|z_{\text{{\tiny $\varepsilon$}}} (z_0,\tilde{v}\,;t)\|_{H^2(\Omega)}^2dt \le C_6.
\end{align*}
Especially, as a result of this estimate we have the boundedness of the strong solutions $z_{\text{{\tiny $\varepsilon$}}} (z_0,\tilde{v})$ in 
$C([0,T]\,;L^2(\Omega)) \cap W^{1,2}(0,T\,;L^2(\Omega))$ uniformly with respect to $\varepsilon \in (0,1)$.
\end{lemma}
%%%%%
%%%%%
%%%%%
\begin{proof}[Proof.]
From the standard theory of linear parabolic PDEs, the problem $\text{(AP2)}_{\varepsilon}$ has a unique strong solution 
$\tilde{z}:=z_{\text{{\tiny $\varepsilon$}}} (z_0,\tilde{v})$ satisfying the required regularity \eqref{Equation-2-67}.
Hence, in the rest of this proof, we show its boundedness.\\
\indent
First, we take the inner product in both sides of \eqref{Equation-2-68} in $L^2(\Omega)$ by $\tilde{z}$.
Then, we obtain 
\begin{equation*}
\frac{1}{2} \frac{d}{dt} \|\tilde{z}(t)\|_{L^2(\Omega)}^2+\varepsilon \|\nabla \tilde{z}_\varepsilon (t)\|_{\bm{L}^2(\Omega)}^2=(\tilde{v}'(t),\tilde{z}(t))_{L^2(\Omega)},
\quad \text{a.e.}~t \in (0,T),
\end{equation*}
hence, the following inequality for a.e. $t \in (0,T)$:
\begin{equation}\label{Equation-2-70}
\frac{d}{dt} \|\tilde{z}(t)\|_{L^2(\Omega)}^2+2\epsilon \|\nabla \tilde{z}(t)\|_{\bm{L}^2(\Omega)}^2 \le \|\tilde{z}(t)\|_{L^2(\Omega)}^2+\|\tilde{v}'(t)\|_{L^2(\Omega)}^2.
\end{equation}
Applying the Gronwall inequality to \eqref{Equation-2-70}, we derive
\begin{align*}
&\|\tilde{z}(t)\|_{L^2(\Omega)}^2+2\epsilon \int_0^t e^{t-s}\|\nabla \tilde{z}(s)\|_{\bm{L}^2(\Omega)}^2ds\\
&\hspace*{1cm}\le e^t \|z_0\|_{L^2(\Omega)}^2+\int_0^t e^{t-s} \|\tilde{v}'(s)\|_{L^2(\Omega)}^2ds,\quad \forall t \in [0,T],
\end{align*}
which implies that the following estimates hold:
\begin{align}
\label{Equation-2-71}
\bullet~&\left(\max_{0 \le t \le T} \|\tilde{z}(t)\|_{L^2(\Omega)}\right)^2 \le e^T \left\{ \|z_0\|_{L^2(\Omega)}^2+\int_0^T \|\tilde{v}'(t)\|_{L^2(\Omega)}^2dt\right\},\\
\label{Equation-2-72}
\bullet~&\varepsilon \int_0^T \|\nabla \tilde{z}(t)\|_{\bm{L}^2(\Omega)}^2dt 
\le \frac{e^T}{2} \left\{ \|z_0\|_{L^2(\Omega)}^2+\int_0^T \|\tilde{v}'(t)\|_{L^2(\Omega)}^2dt\right\}.
\end{align}
\indent
Secondly, we take the inner product in both sides of \eqref{Equation-2-68} by $\tilde{z}'$.
Then, we obtain 
\begin{equation*}
\|\tilde{z}'(t)\|_{L^2(\Omega)}^2+\frac{\varepsilon}{2} \frac{d}{dt} \|\nabla \tilde{z}(t)\|_{\bm{L}^2(\Omega)}^2=(\tilde{v}'(t),\tilde{z}'(t))_{L^2(\Omega)},
\quad \text{a.e.}~t \in (0,T),
\end{equation*}
hence, the following inequality:
\begin{equation}\label{Equation-2-73}
\|\tilde{z}'(t)\|_{L^2(\Omega)}^2+\varepsilon \frac{d}{dt} \|\nabla \tilde{z}(t)\|_{\bm{L}^2(\Omega)}^2 \le \|\tilde{v}'(t)\|_{L^2(\Omega)}^2,\quad \text{a.e.}~t \in (0,T). 
\end{equation}
Integrating both sides of \eqref{Equation-2-73} on any interval $[0,t] \subset [0,T]$, we derive the following inequality for all $t \in [0,T]$:
\begin{equation*}
\int_0^t \|\tilde{z}'(s)\|_{L^2(\Omega)}^2ds+\varepsilon \|\nabla \tilde{z}(t)\|_{\bm{L}^2(\Omega)}^2 \le \|z_0\|_{H^1(\Omega)}^2+\int_0^t \|\tilde{v}'(s)\|_{L^2(\Omega)}^2ds,
\end{equation*}
which implies that the following estimates holds:
\begin{align}
\label{Equation-2-74}
\bullet~&\int_0^T \|\tilde{z}'(t)\|_{L^2(\Omega)}^2dt \le \|z_0\|_{H^1(\Omega)}^2+\int_0^T \|\tilde{v}'(t)\|_{L^2(\Omega)}^2,\\
\label{Equation-2-75}
\bullet~&\varepsilon \left( \sup_{0 \le t \le T} \|\nabla \tilde{z}(t)\|_{\bm{L}^2(\Omega)}^2 \right) \le \|z_0\|_{H^1(\Omega)}^2+\int_0^T \|\tilde{v}'(t)\|_{L^2(\Omega)}^2dt.
\end{align}
\indent
Thirdly, we use the Neumann elliptic regularity \eqref{Equation-1-8} again and repeat the similar argument as in the derivation of \eqref{Equation-2-56} in the proof of 
Lemma \ref{Lemma-2-6}.
Then, we have
\begin{align*}
\varepsilon \|\tilde{z}(t)\|_{H^2(\Omega)}&\le \varepsilon K_5 \left( \|\Delta_N \tilde{z}(t)\|_{L^2(\Omega)}+\|\tilde{z}(t)\|_{L^2(\Omega)} \right)\\
&\le K_5 \left( \|\tilde{v}'(t)\|_{L^2(\Omega)}+\|\tilde{z}'(t)\|_{L^2(\Omega)}+\|\tilde{z}(t)\|_{L^2(\Omega)}\right),
\end{align*}
which implies 
\begin{align}
\label{Equation-2-76}
\varepsilon^2 \int_0^T \|\tilde{z}(t)\|_{H^2(\Omega)}^2dt &\le 3(K_5)^2 \int_0^T \|\tilde{v}'(t)\|_{L^2(\Omega)}^2dt+3(K_5)^2 \int_0^T \|\tilde{z}'(t)\|_{L^2(\Omega)}^2dt\\
\nonumber
&\hspace*{2cm}+3T(K_5)^2 \left( \max_{0 \le t \le T} \|\tilde{z}(t)\|_{L^2(\Omega)}\right)^2.
\end{align}
\indent
Finally, combining all estimates \eqref{Equation-2-71}, \eqref{Equation-2-72} and \eqref{Equation-2-74}--\eqref{Equation-2-76}, we derive the required uniform 
estimates in this lemma.
\end{proof}
%%%%%
%%%%%
%%%%%
By Lemma \ref{Lemma-2-8}, we can define a single-valued solution operator
\begin{equation*}
S_{\text{{\tiny $2,\!\varepsilon,\!z_0$}}}:C([0,T]\,;L^2(\Omega)) \rightarrow C([0,T]\,;L^2(\Omega))
\end{equation*}
by the following way:
\begin{gather*}
D(S_{\text{{\tiny $2,\!\varepsilon,\!z_0$}}})=\mathcal{V}(v_0),\\
R(S_{\text{{\tiny $2,\!\varepsilon,\!z_0$}}}) \subset W^{1,2}(0,T\,;L^2(\Omega)) \cap L^\infty (0,T\,;H^1(\Omega)) \cap L^2(0,T\,;H^2(\Omega)),\\
S_{\text{{\tiny $2,\!\varepsilon,\!z_0$}}} (\tilde{v}):=z_{\varepsilon}(z_0,\tilde{v}).
\end{gather*}
At the end of this subsection, we show Lemma \ref{Lemma-2-9}, which guarantees the continuity property of the operator $S_{\text{{\tiny $2,\!\varepsilon,\!z_0$}}}$,
which also plays a key role to apply the Schauder fixed point theorem in Subsection \ref{Subsection-3-1}.
%%%%%
%%%%%
%%%%%
\begin{lemma}\label{Lemma-2-9}
We have the following convergence as $m \to \infty$:
\begin{equation}\label{Equation-2-77}
S_{\text{{\tiny $2,\!\varepsilon,\!z_0$}}}(\tilde{v}_m) \longrightarrow S_{\text{{\tiny $2,\!\varepsilon,\!z_0$}}}(\tilde{v}) \quad \left\{
\begin{array}{l}
\text{in} \quad C([0,T]\,;L^2(\Omega)),\\[0.1cm]
\text{weakly in} \quad W^{1,2}(0,T\,;L^2(\Omega)),
\end{array}
\right.
\end{equation}
whenever a sequence $\{\tilde{v}_m\}_{m \in \mathbb{N}} \subset \mathcal{V}(v_0)$ and a function $\tilde{v} \in \mathcal{V}(v_0)$ satisfy the following convergences 
as $m \to \infty$:
\begin{equation}\label{Equation-2-78}
\tilde{v}_m \longrightarrow \tilde{v} \quad \left\{
\begin{array}{l}
\text{in} \quad C([0,T]\,;L^2(\Omega)),\\[0.1cm]
\text{weakly in} \quad W^{1,2}(0,T\,;L^2(\Omega)).
\end{array}
\right.
\end{equation}
\end{lemma}
%%%%%
%%%%%
%%%%%
\begin{proof}
We set $\tilde{z}:=S_{\text{{\tiny $2,\!\varepsilon,\!z_0$}}}(\tilde{v})$ and $\tilde{z}_m:=S_{\text{{\tiny $2,\!\varepsilon,\!z_0$}}}(\tilde{v}_m)$ for all $m \in \mathbb{N}$.
Since from Lemma \ref{Lemma-2-8} the sequence $\{\tilde{z}_m\}_{m \in \mathbb N}$ is bounded in 
\begin{equation*}
W^{1,2}(0,T;L^2(\Omega)) \cap L^\infty (0,T;H^1(\Omega)) \cap L^2(0,T\,;H^2(\Omega)),
\end{equation*}
the Ascoli-Arzel\`{a} theorem yields that there exist a subsequence, still denoted by $\{\tilde{z}_m\}_{m \in \mathbb N}$, and a function 
$\hat{z} \in C([0,T];L^2(\Omega))$ such that the following convergences hold as $m \to \infty$:
\begin{equation}\label{Equation-2-79}
\tilde{z}_m \longrightarrow \hat{z} \quad \left\{
\begin{array}{l}
\text{in} \quad C([0,T]\,;L^2(\Omega)),\\[0.1cm]
\text{weakly in} \quad W^{1,2}(0,T\,;L^2(\Omega)),\\[0.1cm]
\text{weakly$^*$ in} \quad L^\infty(0,T\,;H^1(\Omega)),\\[0.1cm]
\text{weakly in} \quad L^2(0,T\,;H^2(\Omega)).
\end{array}
\right.
\end{equation}
Since $\{\tilde{z}_m\}_{m \in \mathbb{N}} \subset L^2(0,T\,;H^2(\Omega))$, for each $m \in \mathbb{N}$ the function $\tilde{z}_m$ satisfies the following weak
formulation:
\begin{align}
\label{Equation-2-80}
&\int_0^T (\tilde{z}_m'(t),\xi (t))_{L^2(\Omega)}dt-\varepsilon \int_0^T (\Delta_N \tilde{z}_m(t),\xi (t))_{L^2(\Omega)}dt\\
\nonumber
&\hspace*{2cm}=\int_0^T (\tilde{v}_m'(t),\xi (t))_{L^2(\Omega)}dt,\quad \forall \xi \in L^2(0,T\,;L^2(\Omega)).
\end{align}
Passing to the limit as $m \to \infty$ in \eqref{Equation-2-80}, by virtue of \eqref{Equation-2-78} and \eqref{Equation-2-79} we obtain 
\begin{align*}
&\int_0^T (\hat{z}'(t),\xi (t))_{L^2(\Omega)}dt-\varepsilon \int_0^T (\Delta_N \hat{z}(t),\xi (t))_{L^2(\Omega)}dt\\
&\hspace*{2cm}=\int_0^T (\tilde{v}'(t),\xi (t))_{L^2(\Omega)}dt,\quad \forall \xi \in L^2(0,T\,;L^2(\Omega)).
\end{align*}
Moreover, $\hat{z}(0)=z_0$ in $L^2(\Omega)$ holds by the convergence in $C([0,T]\,;L^2(\Omega))$.
Hence, $\hat{z}$ is a solution to $\text{(AP2)}_{\varepsilon}$ associated with $\tilde{v}$.
By the uniqueness of solutions to $\text{(AP2)}_{\varepsilon}$, we obtain $\hat{z}=S_{\text{{\tiny $2,\!\varepsilon,\!z_0$}}}(\tilde{v})$.
Since every convergent subsequence of $\{\tilde{z}_m\}_{m \in \mathbb{N}}$ has the same limit, the whole sequence converges to 
$S_{\text{{\tiny $2,\!\varepsilon,\!z_0$}}}(\tilde{v})$ in $C([0,T]\,;L^2(\Omega))$.
Besides, the weak convergence in \eqref{Equation-2-77} follows from the uniqueness of the limit together with \eqref{Equation-2-79}.
\end{proof}
%%%%%%%%%%%%%%%%%%%%%%%%%%%%%%%%%%%%%%%%%%%%%%%%%%%%%%%%%%%%%%%%%%%%%%%%%%%%%%%%%%%%%%%%%%%%%%%%%%%%%%%%%%%%%%%%%%%%%%%%%%%%%%%%%%%%%%%%
%%%%%%%%%%%%%%%%%%%%%%%%%%%%%%%%%%%%%%%%%%%%%%%%%%%%%%%%%%%%%%%%%%%%%%%%%%%%%%%%%%%%%%%%%%%%%%%%%%%%%%%%%%%%%%%%%%%%%%%%%%%%%%%%%%%%%%%%
%%%%%%%%%%%%%%%%%%%%%%%%%%%%%%%%%%%%%%%%%%%%%%%%%%%%%%%%%%%%%%%%%%%%%%%%%%%%%%%%%%%%%%%%%%%%%%%%%%%%%%%%%%%%%%%%%%%%%%%%%%%%%%%%%%%%%%%%
\subsection{Subsystem associated with the variables $u$ and $w$}\label{Subsection-2-4}
%%%%%%%%%%%%%%%%%%%%%%%%%%%%%%%%%%%%%%%%%%%%%%%%%%%%%%%%%%%%%%%%%%%%%%%%%%%%%%%%%%%%%%%%%%%%%%%%%%%%%%%%%%%%%%%%%%%%%%%%%%%%%%%%%%%%%%%%
%%%%%%%%%%%%%%%%%%%%%%%%%%%%%%%%%%%%%%%%%%%%%%%%%%%%%%%%%%%%%%%%%%%%%%%%%%%%%%%%%%%%%%%%%%%%%%%%%%%%%%%%%%%%%%%%%%%%%%%%%%%%%%%%%%%%%%%%
%%%%%%%%%%%%%%%%%%%%%%%%%%%%%%%%%%%%%%%%%%%%%%%%%%%%%%%%%%%%%%%%%%%%%%%%%%%%%%%%%%%%%%%%%%%%%%%%%%%%%%%%%%%%%%%%%%%%%%%%%%%%%%%%%%%%%%%%
In this subsection, for each $\tilde{v} \in \mathcal{V}(v_0)$ we consider the following initial-boundary value problem 
$\text{(AP3)}{}_{\text{{\tiny $\delta$}}}$:=\{\eqref{Equation-2-81}--\eqref{Equation-2-86}\} as the third auxiliary problem of 
$\text{(AP)}{}_{\text{{\tiny $\delta,\!\varepsilon,\!\nu$}}}$: 
\begin{gather}
\label{Equation-2-81}
u'-D_1\Delta u=\beta_{\text{{\tiny $\delta$}}}(u) (c_1 \tilde{v}-c_2 w), \quad \text{a.e. in} \quad Q_T,\\
\label{Equation-2-82}
w'-D_2 \Delta w+c_5 w=c_6 -c_7 \tilde{v} \beta_{\text{{\tiny $\delta$}}} (w),\quad \text{a.e. in} \quad Q_T,\\
\label{Equation-2-83}
\nabla u \cdot \nu=0,\quad \text{a.e. on} \quad \Sigma_T,\\
\label{Equation-2-84}
\nabla w \cdot \nu=0,\quad \text{a.e. on} \quad \Sigma_T,\\
\label{Equation-2-85}
u(0)=u_0,\quad \text{a.e. in} \quad \Omega,\\
\label{Equation-2-86}
w(0)=w_0,\quad \text{a.e. in} \quad \Omega.
\end{gather}
The purpose of introducing this subsystem is to construct a dynamical system for the remaining variables $(u,w)$ under a prescribed state $\tilde{v}$.
In the following argument, we establish the well-posedness of the auxiliary problem $\text{(AP3)}{}_{\text{{\tiny $\delta$}}}$.
%%%%%
%%%%%
%%%%%
\begin{lemma}\label{Lemma-2-10}
For each $\tilde{v} \in \mathcal{V}(v_0)$ the initial-boundary value problem $\text{(AP3)}{}_{\text{{\tiny $\delta$}}}$ has a unique strong solution 
$(u_{\text{{\tiny $\delta$}}} (u_0,w_0,\tilde{v}),w_{\text{{\tiny $\delta$}}} (w_0,\tilde{v}))$ on $[0,T]$ satisfying the following properties:
\begin{enumerate}\leftskip6pt
\item[(1)] $u_{\text{{\tiny $\delta$}}} (u_0,w_0,\tilde{v}) \in W^{1,2}(0,T\,;L^2(\Omega)) \cap L^\infty (0,T\,;H^1(\Omega)) \cap L^2(0,T\,;H^2(\Omega))$.\\[-0.2cm]
\item[(2)] $w_{\text{{\tiny $\delta$}}} (w_0,\tilde{v}) \in W^{1,2}(0,T\,;L^2(\Omega)) \cap L^\infty (0,T\,;H^1(\Omega)) \cap L^2(0,T\,;H^2(\Omega))$.\\[-0.2cm]
\item[(3)] \{(\ref{Equation-2-81}),\,(\ref{Equation-2-83})\} is satisfied in the following variational sense for a.e. $t \in (0,T)$:
\begin{align}
\label{Equation-2-87}
&u_{\text{{\tiny $\delta$}}}'(u_0,w_0,\tilde{v}\,;t)-D_1 \Delta_N u_{\text{{\tiny $\delta$}}} (u_0,w_0,\tilde{v}\,;t)\\
\nonumber
&\hspace*{1cm}=\beta_{\text{{\tiny $\delta$}}}(u_{\text{{\tiny $\delta$}}} (u_0,w_0,\tilde{v}\,;t)) 
\left\{c_1 \tilde{v}(t)-c_2 w_{\text{{\tiny $\delta$}}} (w_0,\tilde{v}\,;t)\right\} \quad \text{in} \quad L^2(\Omega).
\end{align}
\item[(4)] \{(\ref{Equation-2-82}),\,(\ref{Equation-2-84})\} is satisfied in the following variational sense for a.e. $t \in (0,T)$:
\begin{align}
\label{Equation-2-88}
&w_{\text{{\tiny $\delta$}}}'(w_0,\tilde{v}\,;t)-D_2 \Delta_N w_{\text{{\tiny $\delta$}}}(w_0,\tilde{v}\,;t)+c_5 w_{\text{{\tiny $\delta$}}} (w_0,\tilde{v}\,;t)\\
\nonumber
&\hspace*{1cm}=c_6 -c_7 \tilde{v}(t) \beta_{\text{{\tiny $\delta$}}} (w_{\text{{\tiny $\delta$}}} (w_0,\tilde{v}\,;t)) \quad \text{in} \quad L^2(\Omega).
\end{align}
\item[(5)] $u_{\text{{\tiny $\delta$}}} (u_0,w_0,\tilde{v}\,;0)=u_0 \quad \text{in} \quad L^2(\Omega)$.\\[-0.2cm]
\item[(6)] $w_{\text{{\tiny $\delta$}}} (w_0,\tilde{v}\,;0)=w_0 \quad \text{in} \quad L^2(\Omega)$.
\end{enumerate}
\end{lemma}
%%%%%
%%%%%
%%%%%
\begin{proof}
Since the truncation function $\beta_{\text{{\tiny $\delta$}}}$ is globally bounded and Lipschitz continuous on $\mathbb{R}$, the nonlinear term on the right-hand 
side of \eqref{Equation-2-82} is well-defined and belongs to $L^2(0,T\,;L^2(\Omega))$.
Therefore, by applying the Schauder fixed-point theorem combined with the standard maximal regularity theory for linear parabolic equations 
(cf. Ladyzhenskaya-Solonnikov-Ural'tseva \cite{LSU-1968}), the initial-boundary value problem \{\eqref{Equation-2-82},\eqref{Equation-2-84},\eqref{Equation-2-86}\}
admits a strong solution $w_{\text{{\tiny $\delta$}}}(w_0,\tilde{v})$ on $[0,T]$ satisfying (2), (4) and (6).
Since the proof of the uniqueness is based on a standard energy argument very similar to the one used in the proof of Lemma \ref{Lemma-2-13} below,
we omit the details here.\\
\indent
Once $w_{\text{{\tiny \(\delta \)}}} (w_0,\tilde{v})$ is determined, \eqref{Equation-2-81} becomes a semilinear parabolic equation for $u$.
Since the function $\beta_{\text{{\tiny \(\delta \)}}}$ is globally bounded, $\tilde{v} \in L^\infty (0,T\,;H^1(\Omega))$ and the obtained $w_{\text{{\tiny \(\delta \)}}}$ 
belongs to $L^\infty(0,T\,;H^1(\Omega)) \cap L^2(0,T\,;H^2(\Omega))$, the right-hand side of \eqref{Equation-2-81} also belongs to $L^2(0,T\,;L^2(\Omega))$.
Under the assumption $u_0 \in H^1(\Omega)$, using the standard theory of semilinear parabolic equations in \cite{LSU-1968} again along with the same uniqueness
argument, the initial-boundary value problem \{\eqref{Equation-2-81},\,\eqref{Equation-2-83},\,\eqref{Equation-2-85}\} admits a unique strong solution 
$u_{\text{{\tiny \(\delta \)}}} (u_0,w_0,\tilde{v})$ on $[0,T]$ satisfying (1), (3) and (5).
\end{proof}
%%%%%
%%%%%
%%%%%
\indent
By Lemma \ref{Lemma-2-10}, we can define a single-valued solution operator
\begin{equation*}
S_{\text{{\tiny $3,\!\delta$}}}:C([0,T]\,;L^2(\Omega)) \rightarrow \left(C([0,T]\,;L^2(\Omega))\right)^2
\end{equation*}
in the following way:
\begin{gather*}
D(S_{\text{{\tiny $3,\!\delta$}}})=\mathcal{V}(v_0),\\
R(S_{\text{{\tiny $3,\!\delta$}}}) \subset \left( W^{1,2}(0,T\,;L^2(\Omega)) \cap L^\infty (0,T\,;H^1(\Omega)) \cap L^2(0,T\,;H^2(\Omega)) \right)^2,\\
S_{\text{{\tiny $3,\!\delta$}}} (\tilde{v}):=(u_{\text{{\tiny $\delta$}}}(u_0,w_0,\tilde{v}),w_{\text{{\tiny $\delta$}}}(w_0,\tilde{v})),\quad \forall \tilde{v} \in \mathcal{V}(v_0).
\end{gather*}
\indent
In the next argument, we show Lemmas \ref{Lemma-2-11} and \ref{Lemma-2-12}, which give the uniform estimates of the solutions
$w_{\text{{\tiny $\delta$}}}(w_0,\tilde{v})$ and $u_{\text{{\tiny $\delta$}}}(u_0,w_0,\tilde{v})$ to $\text{(AP3)}{}_{\text{{\tiny $\delta$}}}$, respectively.
%%%%%
%%%%%
%%%%%
\begin{lemma}\label{Lemma-2-11}
There exists a constant $C_7>0$, which depends on the following values:
\begin{equation*}
\delta, \quad T, \quad \|w_0\|_{H^1(\Omega)}, \quad \int_0^T \|\tilde{v}(t)\|_{L^2(\Omega)}^2dt,
\end{equation*}
such that following boundedness holds:
\begin{equation*}
\int_0^T \|w_{\text{{\tiny $\delta$}}}' (w_0,\tilde{v}\,;t)\|_{L^2(\Omega)}^2
+\sup_{0 \le t \le T} \|w_{\text{{\tiny $\delta$}}}(w_0,\tilde{v}\,;t)\|_{H^1(\Omega)}^2
+\int_0^T \|w_{\text{{\tiny $\delta$}}} (w_0,\tilde{v}\,;t)\|_{H^2(\Omega)}^2dt \le C_7.
\end{equation*}
\end{lemma}
%%%%%
%%%%%
%%%%%
\begin{proof}[Proof.]
For simplicity, we set $\tilde{w}:=w_{\text{{\tiny $\delta$}}}(w_0,\tilde{v})$.
First of all, we take the inner product in both sides of \eqref{Equation-2-88} in $L^2(\Omega)$ by $\tilde{w}(t)$.
Then we obtain
\begin{align*}
&\,\frac{1}{2}\frac{d}{dt}\|\tilde{w}(t)\|_{L^2(\Omega)}^2+D_2 \|\nabla \tilde{w}(t)\|_{\bm{L}^2(\Omega)}^2+c_5 \|\tilde{w}(t)\|_{L^2(\Omega)}^2\\
&\,\hspace*{1cm}=c_6 \int_\Omega \tilde{w}(t)dx-c_7 \int_\Omega \tilde{v}(t) \beta_{\text{{\tiny $\delta$}}} (\tilde{w}(t)) \tilde{w}(t) dx,\quad \text{a.e.}~t \in (0,T),
\end{align*}
hence, the following inequality for a.e. $t \in (0,T)$:
\begin{gather}
\label{Equation-2-89}
\frac{d}{dt} \|\tilde{w}(t)\|_{L^2(\Omega)}^2+2D_2 \|\nabla \tilde{w}(t)\|_{\bm{L}^2(\Omega)}^2+c_5\|\tilde{w}(t)\|_{L^2(\Omega)}^2\\
\nonumber
\le \frac{2(\delta c_7)^2}{c_5} \|\tilde{v}(t)\|_{L^2(\Omega)}^2+\frac{2(c_6)^2|\Omega|}{c_5}.
\end{gather}
Applying the Gronwall lemma to \eqref{Equation-2-89}, we derive the following inequality for all $t \in [0,T]$:
\begin{align*}
&\|\tilde{w}(t)\|_{L^2(\Omega)}^2+2D_2 \int_0^t \|\nabla \tilde{w}(s)\|_{\bm{L}^2(\Omega)}^2 e^{-c_5(t-s)}ds\\
&\hspace*{0.5cm}\le \|w_0\|_{L^2(\Omega)}^2e^{-c_5t}+\int_0^t \left\{ \frac{2(\delta c_7)^2}{c_5} \|\tilde{v}(s)\|_{L^2(\Omega)}^2
+\frac{2(c_6)^2|\Omega|}{c_5}\right\} e^{-c_5(t-s)}ds,
\end{align*}
which implies 
\begin{equation}\label{Equation-2-90}
\max_{0 \le t \le T}\|\tilde{w}(t)\|_{L^2(\Omega)}^2 \le \|w_0\|_{L^2(\Omega)}^2
+\int_0^T \left\{ \frac{2(\delta c_7)^2}{c_5} \|\tilde{v}(t)\|_{L^2(\Omega)}^2+\frac{2(c_6)^2|\Omega|}{c_5}\right\}dt.
\end{equation}
\indent
Secondly, we take the inner product in both sides of \eqref{Equation-2-88} in $L^2(\Omega)$ by $w'(t)$.
Then we obtain the following equality for a.e. $t \in (0,T)$:
\begin{gather*}
\|\tilde{w}'(t)\|_{L^2(\Omega)}^2+\frac{d}{dt} \left(\frac{D_2}{2} \|\nabla \tilde{w}(t)\|_{\bm{L}^2(\Omega)}^2+\frac{c_5}{2} \|\tilde{w}(t)\|_{L^2(\Omega)}^2\right)\\
=c_6 \int_\Omega \tilde{w}'(t)dx-c_7 \int_\Omega \tilde{v}(t) \beta_{\text{{\tiny $\delta$}}} (\tilde{w}(t)) \tilde{w}'(t)dx.
\end{gather*}
From \eqref{Equation-2-3} we obtain 
\begin{align}
\label{Equation-2-91}
&\|\tilde{w}'(t)\|_{L^2(\Omega)}^2+\frac{d}{dt} \left(D_2 \|\nabla \tilde{w}(t)\|_{\bm{L}^2(\Omega)}^2+c_5 \|\tilde{w}(t)\|_{L^2(\Omega)}^2\right)\\
\nonumber
&\hspace*{1cm}\le 2\left\{(c_6)^2 |\Omega|+\left( \delta c_7\right)^2 \|\tilde{v}(t)\|_{L^2(\Omega)}^2\right\},\quad \text{a.e.}~t \in (0,T).
\end{align}
Integrating \eqref{Equation-2-91} on any interval $[0,t] \subset [0,T]$, we derive
\begin{gather*}
\int_0^t \|\tilde{w}'(s)\|_{L^2(\Omega)}^2ds+D_2\|\nabla \tilde{w}(t)\|_{\bm{L}^2(\Omega)}^2+c_5 \|\tilde{w}(t)\|_{L^2(\Omega)}^2\\
\le D_2\|\nabla w_0\|_{\bm{L}^2(\Omega)}^2+c_5\|w_0\|_{L^2(\Omega)}^2+2 \int_0^t \left\{(c_6)^2 |\Omega|
+\left( \delta c_7\right)^2 \|\tilde{v}(s)\|_{L^2(\Omega)}^2\right\}ds,
\end{gather*}
which implies
\begin{align}
\label{Equation-2-92}
\bullet~&\int_0^T \|\tilde{w}'(t)\|_{L^2(\Omega)}^2dt \le 2\max\{D_2,c_5\} \|w_0\|_{H^1(\Omega)}^2\\
\nonumber
&\,\hspace*{4cm}+2 \int_0^T \left\{(c_6)^2 |\Omega|+\left( \delta c_7\right)^2 \|\tilde{v}(t)\|_{L^2(\Omega)}^2\right\}dt,\\
\label{Equation-2-93}
\bullet~&\left(\sup_{0 \le t \le T} \|\tilde{w}(t)\|_{H^1(\Omega)}\right)^2 \le \frac{\max\{D_2,c_5\}}{\min\{D_2,c_5\}} \|w_0\|_{H^1(\Omega)}^2\\
\nonumber
&\,\hspace*{2cm}+\frac{2}{\min\{D_2,c_5\}} \int_0^T \left\{(c_6)^2 |\Omega|+\left( \delta c_7\right)^2 \|\tilde{v}(t)\|_{L^2(\Omega)}^2\right\}dt.
\end{align}
\indent
Thirdly, from \eqref{Equation-2-88}, \eqref{Equation-2-90} and \eqref{Equation-2-92} we derive the regularity $\Delta_N \tilde{w} \in L^2(0,T\,;L^2(\Omega))$.
Using the Neumann elliptic regularity \eqref{Equation-1-8} again, from \eqref{Equation-2-88} we obtain the following inequality for all $t \in (0,T)$: 
\begin{align}
\label{Equation-2-94}
&\,\int_0^T \|\tilde{w}(t)\|_{H^2(\Omega)}^2dt \le (K_5)^2 \int_0^T \left(\|\Delta_N \tilde{w}(t)\|_{L^2(\Omega)}+\|\tilde{w}(t)\|_{L^2(\Omega)} \right)^2dt\\
\nonumber
\le&\,\left(\frac{K_5}{D_2}\right)^2 \int_0^T \left\{ \|\tilde{w}'(t)\|_{L^2(\Omega)}+(D_2+c_5)\|\tilde{w}(t)\|_{L^2(\Omega)}
+\delta c_7 \|\tilde{v}(t)\|_{L^2(\Omega)}+c_6 |\Omega|^{\frac{1}{2}}\right\}^2dt\\
\nonumber
\le&\,\left(\frac{2K_5}{D_2}\right)^2 \int_0^T \|\tilde{w}'(t)\|_{L^2(\Omega)}^2dt+\left(\frac{2\delta c_7 K_5}{D_2} \right)^2\int_0^T \|\tilde{v}(t)\|_{L^2(\Omega)}^2
+T\left(\frac{2c_6 K_5 |\Omega|^{\frac{1}{2}}}{D_2}\right)^2\\
\nonumber
&\hspace*{2cm}+T\left\{2K_5 \left(1+\frac{c_5}{D_2}\right) \left( \sup_{0 \le t \le T} \|\tilde{w}(t)\|_{L^2(\Omega)}\right)\right\}^2.
\end{align}
\indent
Finally, combining all the estimates \eqref{Equation-2-90} and \eqref{Equation-2-92}--\eqref{Equation-2-94}, we derive the required uniform estimate.
\end{proof}
%%%%%
%%%%%
%%%%%
\begin{lemma}\label{Lemma-2-12}
There exists a constant $C_8>0$, which depends on the following values:
\begin{equation*}
\delta, \quad T, \quad \|u_0\|_{H^1(\Omega)}, \quad \|w_0\|_{H^1(\Omega)}, \quad \int_0^T \|\tilde{v}(t)\|_{L^2(\Omega)}^2dt,
\end{equation*}
such that following boundedness holds:
\begin{gather*}
\int_0^T \|(u_{\text{{\tiny $\delta$}}} (u_0,w_0,\tilde{v}))'(t)\|_{L^2(\Omega)}^2+\sup_{0 \le t \le T} \|u_{\text{{\tiny $\delta$}}} (u_0,w_0,\tilde{v}\,;t)\|_{H^1(\Omega)}^2\\
+\int_0^T \|u_{\text{{\tiny $\delta$}}} (u_0,w_0,\tilde{v}\,;t)\|_{H^2(\Omega)}^2dt \le C_8.
\end{gather*}
\end{lemma}
%%%%%
%%%%%
%%%%%
\begin{proof}[Proof.]
Throughout this proof, for simplicity we set $(\tilde{u},\tilde{w}):=(u_{\text{{\tiny $\delta$}}} (u_0,w_0,\tilde{v}),w_{\text{{\tiny $\delta$}}} (w_0,\tilde{v}))$.\\
\indent
Firstly, we take the inner product in both sides of \eqref{Equation-2-87} in $L^2(\Omega)$ by $\tilde{u}(t)$.
Then, we obtain the following equality for a.e. $t \in (0,T)$:
\begin{align*}
&\,\frac{1}{2}\frac{d}{dt}\|\tilde{u}(t)\|_{L^2(\Omega)}^2+D_1 \|\nabla \tilde{u}(t)\|_{\bm{L}^2(\Omega)}^2\\
&\,\hspace*{1cm}=c_1 \int_\Omega \tilde{v}(t) \beta_{\text{{\tiny $\delta$}}} (\tilde{u}(t)) \tilde{u}(t)dx
-c_2 \int_\Omega \tilde{w}(t)\beta_{\text{{\tiny $\delta$}}} (\tilde{u}(t)) \tilde{u}(t)dx,
\end{align*}
hence, the following inequality for a.e. $t \in (0,T)$:
\begin{align}
\label{Equation-2-95}
&\frac{d}{dt} \|\tilde{u}(t)\|_{L^2(\Omega)}^2+2D_1 \|\nabla \tilde{u}(t)\|_{\bm{L}^2(\Omega)}^2\\
\nonumber
&\hspace*{1cm}\le \|\tilde{u}(t)\|_{L^2(\Omega)}^2+2\delta^2 \left\{ (c_1)^2+(c_2)^2 \right\} \left( \|\tilde{v}(t)\|_{L^2(\Omega)}^2+\|\tilde{w}(t)\|_{L^2(\Omega)}^2 \right).
\end{align}
Applying the Gronwall lemma to \eqref{Equation-2-95} and using Lemma \ref{Lemma-2-11}, we obtain
\begin{align*}
&\,\|\tilde{u}(t)\|_{L^2(\Omega)}^2+2D_1 \int_0^t \|\nabla \tilde{u}(s)\|_{\bm{L}^2(\Omega)}^2 ds\\
\le&\,\|u_0\|_{L^2(\Omega)}^2e^t+2\delta^2 \left\{(c_1)^2+(c_2)^2\right\} \int_0^t \left( \|\tilde{v}(s)\|_{L^2(\Omega)}^2+\|\tilde{w}(s)\|_{L^2(\Omega)}^2 \right)e^{t-s}ds\\
\le&\,e^T \left[ \|u_0\|_{L^2(\Omega)}^2+2\delta^2 \left\{(c_1)^2+(c_2)^2\right\} 
\int_0^T \left( \|\tilde{v}(t)\|_{L^2(\Omega)}^2+C_7\right)dt\right], \quad \forall t \in [0,T],
\end{align*}
which implies 
\begin{align}
\label{Equation-2-96}
&\,\max_{0 \le t \le T}\|\tilde{u}(t)\|_{L^2(\Omega)}^2+2D_1 \int_0^T \|\nabla \tilde{u}(t)\|_{\bm{L}^2(\Omega)}^2dt\\
\nonumber
\le&\,2e^T \left[ \|u_0\|_{L^2(\Omega)}^2+2\delta^2 \left\{(c_1)^2+(c_2)^2\right\} \int_0^T \left( \|\tilde{v}(t)\|_{L^2(\Omega)}^2+C_7\right)dt\right].
\end{align}
\indent
Secondly, we take the inner product in both sides of \eqref{Equation-2-87} in $L^2(\Omega)$ by $\tilde{u}'(t)$.
Then, we obtain the following equality for a.e. $t \in (0,T)$: 
\begin{align*}
&\|\tilde{u}'(t)\|_{L^2(\Omega)}^2+\frac{d}{dt} \left(\frac{D_1}{2} \|\nabla \tilde{u}(t)\|_{\bm{L}^2(\Omega)}^2\right)\\
&\hspace*{1cm}=c_1 \int_\Omega \tilde{v}(t) \beta_{\text{{\tiny $\delta$}}}(\tilde{u}(t)) \tilde{u}'(t)dx
-c_2 \int_\Omega \tilde{w}(t) \beta_{\text{{\tiny $\delta$}}} (\tilde{u}(t)) \tilde{u}'(t)dx,
\end{align*}
hence, the following inequality for a.e. $t \in (0,T)$:
\begin{align}
\label{Equation-2-97}
&\|\tilde{u}'(t)\|_{L^2(\Omega)}^2+\frac{d}{dt} \left(D_1 \|\nabla \tilde{u}(t)\|_{\bm{L}^2(\Omega)}^2\right)\\
\nonumber
&\hspace*{1cm}\le 2\delta^2 \left\{ (c_1)^2+(c_2)^2 \right\} \left( \|\tilde{v}(t)\|_{L^2(\Omega)}^2+\|\tilde{w}(t)\|_{L^2(\Omega)}^2 \right).
\end{align}
Integrating \eqref{Equation-2-97} on any interval $[0,t] \subset [0,T]$ and using \eqref{Equation-2-96}, we obtain 
the following inequality for all $t \in [0,T]$:
\begin{align*}
&\int_0^t \|\tilde{u}'(s)\|_{L^2(\Omega)}^2ds+D_1 \|\nabla \tilde{u}(t)\|_{\bm{L}^2(\Omega)}^2\\
&\hspace*{1cm}\le D_1 \|\nabla u_0\|_{\bm{L}^2(\Omega)}^2+2 \delta^2 \left\{ (c_1)^2+(c_2)^2 \right\} 
\int_0^t \left( \|\tilde{v}(s)\|_{L^2(\Omega)}^2+\|\tilde{w}(s)\|_{L^2(\Omega)}^2 \right)ds,
\end{align*}
which implies
\begin{align}
\label{Equation-2-98}
&\int_0^T \|\tilde{u}'(t)\|_{L^2(\Omega)}^2dt+D_1 \left( \sup_{0 \le t \le T} \|\nabla \tilde{u}(t)\|_{\bm{L}^2(\Omega)} \right)^2\\
\nonumber
&\hspace*{1cm}\le 2\left[ D_1 \|u_0\|_{H^1(\Omega)}^2+2 \delta^2 \left\{ (c_1)^2+(c_2)^2 \right\} \int_0^T \left( \|\tilde{v}(t)\|_{L^2(\Omega)}^2+C_7\right)dt \right].
\end{align}
\indent
Thirdly, from \eqref{Equation-2-87}, \eqref{Equation-2-90} and \eqref{Equation-2-98} we have $\Delta_N \tilde{u} \in L^2(0,T\,;L^2(\Omega))$.
Using the Neumann elliptic regularity \eqref{Equation-1-8} again and repeating a similar argument of the derivation of \eqref{Equation-2-94},
we derive
\begin{align}
\label{Equation-2-99}
&\,\int_0^T \|\tilde{u}(t)\|_{H^2(\Omega)}^2dt \le (K_5)^2 \int_0^T \left( \|\Delta_N \tilde{u}(t)\|_{L^2(\Omega)}+\|\tilde{u}(t)\|_{L^2(\Omega)}\right)^2dt\\
\nonumber
\le&\,\left(\frac{K_5}{D_1}\right)^2 \int_0^T \left(\|\tilde{u}'(t)\|_{L^2(\Omega)}+\delta c_1 \|\tilde{v}(t)\|_{L^2(\Omega)}+\delta c_2 \|\tilde{w}(t)\|_{L^2(\Omega)}+
D_1\|\tilde{u}(t)\|_{L^2(\Omega)} \right)^2dt\\
\nonumber
\le&\,4\left(\frac{K_5}{D_1}\right)^2 \int_0^T \|\tilde{u}'(t)\|_{L^2(\Omega)}^2dt+4\left(\frac{\delta c_1 K_5}{D_1}\right)^2 \int_0^T \|\tilde{v}(t)\|_{L^2(\Omega)}^2dt\\
\nonumber
&\hspace*{1cm}+4T\left(\frac{\delta c_2 K_5}{D_1}\right)^2 \left( \max_{0\,\le\,t\,\le\,T} \|\tilde{w}(t)\|_{L^2(\Omega)}\right)^2+
4T(K_5)^2 \left( \max_{0 \le t \le T} \|\tilde{u}(t)\|_{L^2(\Omega)}\right)^2.
\end{align}
\indent
Finally, combining all the estimates \eqref{Equation-2-96}, \eqref{Equation-2-98} and \eqref{Equation-2-99} with Lemma \ref{Lemma-2-11}, 
we obtain the required estimate in this lemma.
\end{proof}
%%%%%
%%%%%
%%%%%
At the end of this subsection, we show Lemma \ref{Lemma-2-13}, which establishes the continuity of the solution operator $S_{\text{{\tiny $3,\!\delta$}}}$.
Just as with $S_{\text{{\tiny $1,\!\delta,\!\nu$}}}$ and $S_{\text{{\tiny $2,\!\varepsilon$}}}$, the operator $S_{\text{{\tiny $3,\!\delta$}}}$ plays a key role to
apply the Schauder fixed point theorem in Subsection \ref{Subsection-3-1}.
%%%%%
%%%%%
%%%%%
\begin{lemma}\label{Lemma-2-13}
We have the following convergence as $m \to \infty$:
\begin{equation}\label{Equation-2-100}
S_{\text{{\tiny $3,\!\delta$}}}(\tilde{v}_m) \longrightarrow S_{\text{{\tiny $3,\!\delta$}}}(\tilde{v}) \quad \text{in} \quad \left(C([0,T]\,;L^2(\Omega))\right)^2 
\quad \text{and} \quad \left(L^2(0,T\,;H^1(\Omega))\right)^2
\end{equation}
whenever a sequence $\{\tilde{v}_m\}_{m \in \mathbb{N}} \subset \mathcal{V}(v_0)$ and a function $\tilde{v} \in \mathcal{V}(v_0)$ satisfy 
\begin{equation}\label{Equation-2-101}
\tilde{v}_m \longrightarrow \tilde{v} \quad \text{in} \quad C([0,T]\,;L^2(\Omega)) \quad \text{as} \quad m \to \infty.
\end{equation}
\end{lemma}
%%%%%
%%%%%
%%%%%
\begin{proof}[Proof.]
For simplicity we set $(\tilde{U}_m,\tilde{W}_m,\tilde{V}_m):=(\tilde{u}_m-\tilde{u},\tilde{w}_m-\tilde{w},\tilde{v}_m-\tilde{v})$, where 
$(\tilde{u}_m,\tilde{w}_m):=S_{\text{{\tiny $3,\!\delta$}}}(\tilde{v}_m)$ for all $m \in \mathbb{N}$ and $(\tilde{u},\tilde{w}):=S_{\text{{\tiny $3,\!\delta$}}}(\tilde{v})$.
In this proof we use \eqref{Equation-2-3} and \eqref{Equation-2-4}, repeatedly.\\ 
\indent
Firstly, we consider the sequence $\{\tilde{W}_m\}_{m \in \mathbb{N}}$.
From \eqref{Equation-2-85} and \eqref{Equation-2-88} we have
\begin{gather}
\label{Equation-2-102}
\tilde{W}_m'(t)-D_2 \Delta_N \tilde{W}_m(t)+c_5 \tilde{W}_m(t)=-c_7 \tilde{v}_m(t) \beta_{\text{{\tiny $\delta$}}} (\tilde{w}_m(t))
+c_7 \tilde{v}(t) \beta_{\text{{\tiny $\delta$}}} (\tilde{w}(t))\\
\nonumber
\text{in} \quad L^2(\Omega),\quad \text{a.e.}~t \in (0,T),\\
\label{Equation-2-103}
\tilde{W}_m(0)=0 \quad \text{in} \quad L^2(\Omega).
\end{gather}
We take the inner product in both sides of \eqref{Equation-2-102} by $\tilde{W}_m(t)$.
From \eqref{Equation-2-4} and \eqref{Equation-2-3} we obtain the following inequality for a.e. $t \in (0,T)$:
\begin{gather}
\label{Equation-2-104}
\frac{1}{2} \frac{d}{dt} \|\tilde{W}_m(t)\|_{L^2(\Omega)}^2+\min\{D_2,\,c_5\} \|\tilde{W}_m(t)\|_{H^1(\Omega)}^2\\
\nonumber
\le \delta c_7 \|\tilde{V}_m(t)\|_{L^2(\Omega)} \|\tilde{W}_m(t)\|_{L^2(\Omega)}+c_7 \int_\Omega |\tilde{v}(t)| |\tilde{W}_m(t)|^2dx.
\end{gather}
Since the embedding $ H^1(\Omega)\hookrightarrow L^4(\Omega) $ is continuous, there exists a constant $K_6>0$ such that
\begin{equation}\label{Equation-2-105}
\|\xi\|_{L^4(\Omega)} \le K_6 \|\xi\|_{H^1(\Omega)},\quad \forall \xi \in H^1(\Omega).
\end{equation}
Using the generalized H\"{o}lder inequality, we see from \eqref{Equation-2-105} that the second term in the right-hand side of \eqref{Equation-2-104} 
is estimated as follows:
\begin{align}
\label{Equation-2-106}
&\,c_7\int_\Omega |\tilde{v}(t)| |\tilde{W}_m(t)|^2dx\\
\nonumber
\le&\,c_7\|\tilde{v}(t)\|_{L^4(\Omega)} \|\tilde{W}_m(t)\|_{L^4(\Omega)} \|\tilde{W}_m(t)\|_{L^2(\Omega)}\\
\nonumber
\le&\,c_7(K_6)^2 \|\tilde{v}(t)\|_{H^1(\Omega)} \|\tilde{W}_m(t)\|_{H^1(\Omega)} \|\tilde{W}_m(t)\|_{L^2(\Omega)}\\
\nonumber
\le&\,\frac{\min\{D_2,\,c_5\}}{2} \|\tilde{W}_m(t)\|_{H^1(\Omega)}^2
+\frac{(c_7)^2 (K_6)^4}{2\min\{D_2,\,c_5\}} \|\tilde{v}(t)\|_{H^1(\Omega)}^2 \|\tilde{W}_m(t)\|_{L^2(\Omega)}^2,
\end{align}
For the first term in the right-hand side of \eqref{Equation-2-104} we obtain 
\begin{equation}\label{Equation-2-107}
\delta c_7 \|\tilde{V}_m(t)\|_{L^2(\Omega)} \|\tilde{W}_m(t)\|_{L^2(\Omega)} \le \frac{\delta c_7}{2} \|\tilde{V}_m(t)\|_{L^2(\Omega)}^2
+\frac{\delta c_7}{2} \|\tilde{W}_m(t)\|_{L^2(\Omega)}^2.
\end{equation}
Substituting \eqref{Equation-2-106} and \eqref{Equation-2-107} into \eqref{Equation-2-104} and multiplying its result by $2$, we obtain the following inequality for a.e. $t \in (0,T)$:
\begin{gather}
\label{Equation-2-108}
\frac{d}{dt} \|\tilde{W}_m(t)\|_{L^2(\Omega)}^2+\min\{D_2,c_5\} \|\tilde{W}_m(t)\|_{H^1(\Omega)}^2\\
\nonumber
\le \left[ \frac{2 (c_7)^2 (K_6)^4}{\nu \min\{D_2,\,c_5\}} \Phi_\nu (\tilde{v}(t))+\delta c_7 \right] \| \tilde{W}_m(t)\|_{L^2(\Omega)}^2
+\delta c_7 \|\tilde{V}_m(t)\|_{L^2(\Omega)}^2.
\end{gather}
Applying the Gronwall lemma to \eqref{Equation-2-108} and using \eqref{Equation-2-103}, we obtain the following inequality for all $t \in [0,T]$:
\begin{gather*}
\|\tilde{W}_m(t)\|_{L^2(\Omega)}^2+\min\{D_2,c_5\} \int_0^t \|\tilde{W}_m(s)\|_{H^1(\Omega)}^2ds\\
\le \delta c_7 \int_0^t \|\tilde{V}_m(s)\|_{L^2(\Omega)}^2 \exp \left( \frac{(c_7)^2 (K_6)^4}{\min\{D_2,\,c_5\}} \int_s^t \Phi_\nu (\tilde{v}(\tau))d\tau
+\delta c_7 (t-s) \right)ds,
\end{gather*}
hence,
\begin{align}
\label{Equation-2-109}
&\,\max_{0 \le t \le T} \|\tilde{W}_m(t)\|_{L^2(\Omega)}^2+\int_0^T \|\tilde{W}_m(t)\|_{H^1(\Omega)}^2dt\\
\nonumber
\le&\,\delta c_7 T \left[ 1+\frac{1}{\min\{D_2,c_5\}} \right] \left( \max_{0\,\le\,t\,\le\,T} \|\tilde{V}(t)\|_{L^2(\Omega)} \right)^2\\
\nonumber
&\,\hspace*{1cm}\times \exp \left( \frac{(c_7)^2 (K_6)^4 T}{\min\{D_2,\,c_5\}} \sup_{0 \le t \le T} \Phi_\nu (\tilde{v}(t))+\delta c_7 T \right).
\end{align}
Combining \eqref{Equation-2-101} with \eqref{Equation-2-109}, we obtain the following convergence as $m \to \infty$:
\begin{equation}\label{Equation-2-110}
\tilde{w}_m \longrightarrow \tilde{w} \quad \text{in} \quad C([0,T]\,;L^2(\Omega)) \quad \text{and} \quad L^2(0,T\,;H^1(\Omega)).
\end{equation}
\indent
Secondly, we consider the sequence $\{\tilde{U}_m\}_{m \in \mathbb{N}}$.
From \eqref{Equation-2-85} and \eqref{Equation-2-87} we have 
\begin{gather}
\label{Equation-2-111}
\tilde{U}_m'(t)-D_1 \Delta_N \tilde{U}_m(t)=c_1 \beta_{\text{{\tiny $\delta$}}} (\tilde{u}_m(t)) \tilde{V}_m(t)+
c_1 \tilde{v}(t) \left\{ \beta_{\text{{\tiny $\delta$}}} (\tilde{u}_m(t))-\beta_{\text{{\tiny $\delta$}}} (\tilde{u}(t))\right\}\\
\nonumber
-c_2 \beta_{\text{{\tiny $\delta$}}} (\tilde{u}_m(t)) \tilde{W}_m(t)-c_2 \tilde{w}(t) \left\{ \beta_{\text{{\tiny $\delta$}}} (\tilde{u}_m(t))-
\beta_{\text{{\tiny $\delta$}}} (\tilde{u}(t))\right\}\\
\nonumber
\text{in} \quad L^2(\Omega),\quad \text{a.e.}~t \in (0,T),\\
\label{Equation-2-112}
\tilde{U}_m(0)=0 \quad \text{in} \quad L^2(\Omega).
\end{gather}
We take the inner product in both sides of \eqref{Equation-2-111} by $\tilde{U}_m(t)$.
From \eqref{Equation-2-3} and \eqref{Equation-2-4} we obtain
\begin{align}
\label{Equation-2-113}
\bullet~&c_1 \int_\Omega \beta_{\text{{\tiny $\delta$}}} (\tilde{u}_m(t)) |\tilde{V}_m(t)| |\tilde{U}_m(t)| dx \le \frac{\delta c_1}{2} \|\tilde{V}_m(t)\|_{L^2(\Omega)}^2
+\frac{\delta c_1}{2} \|\tilde{U}_m(t)\|_{L^2(\Omega)}^2,\\
\label{Equation-2-114}
\bullet~&c_2 \int_\Omega \beta_{\text{{\tiny $\delta$}}} (\tilde{u}_m(t)) |\tilde{W}_m(t)| |\tilde{U}_m(t)|dx \le \frac{\delta c_2}{2} \|\tilde{W}_m(t)\|_{L^2(\Omega)}^2
+\frac{\delta c_2}{2} \|\tilde{U}_m(t)\|_{L^2(\Omega)}^2.
\end{align}
Repeating the similar argument to the derivation of \eqref{Equation-2-106}, we obtain not only 
\begin{align}
\label{Equation-2-115}
&\,c_1 \int_\Omega \tilde{v}(t) \{ \beta_{\text{{\tiny $\delta$}}} (\tilde{u}_m(t))-\beta_{\text{{\tiny $\delta$}}} (\tilde{u}(t))\} \tilde{U}_m(t)dx\\
\nonumber
\le&\,c_1\|\tilde{v}(t)\|_{L^4(\Omega)} \|\tilde{U}_m(t)\|_{L^4(\Omega)} \|\tilde{U}_m(t)\|_{L^2(\Omega)}\\
\nonumber
\le&\,c_1 (K_6)^2 \|\tilde{v}(t)\|_{H^1(\Omega)} \|\tilde{U}_m(t)\|_{H^1(\Omega)} \|\tilde{U}_m(t)\|_{L^2(\Omega)}\\
\nonumber
\le&\,\frac{D_1}{4} \|\tilde{U}_m(t)\|_{H^1(\Omega)}^2+\frac{(c_1)^2(K_6)^4}{D_1} \|\tilde{v}(t)\|_{H^1(\Omega)}^2 \|\tilde{U}_m(t)\|_{L^2(\Omega)}^2\\
\nonumber
\le&\,\frac{D_1}{4} \|\nabla \tilde{U}_m(t)\|_{\bm{L}^2(\Omega)}^2+\left\{ \frac{2(c_1)^2(K_6)^4}{\nu D_1} \Phi_\nu (\tilde{v}(t))+\frac{D_1}{4} \right\}
\|\tilde{U}_m(t)\|_{L^2(\Omega)}^2,
\end{align}
but also
\begin{align}
\label{Equation-2-116}
&\,-c_2 \int_\Omega \tilde{w}(t) \{ \beta_{\text{{\tiny $\delta$}}} (\tilde{u}_m(t))-\beta_{\text{{\tiny $\delta$}}} (\tilde{u}(t))\} \tilde{U}_m(t)dx\\
\nonumber
\le&\,\frac{D_1}{4} \|\nabla \tilde{U}_m(t)\|_{\bm{L}^2(\Omega)}^2+\left\{ \frac{(c_2)^2 (K_6)^4}{D_1}\|\tilde{w}(t)\|_{H^1(\Omega)}^2+\frac{D_1}{4}\right\}
\|\tilde{U}_m(t)\|_{L^2(\Omega)}^2.
\end{align}
Using \eqref{Equation-2-113}--\eqref{Equation-2-116}, we obtain the following inequality for a.e. $t \in (0,T)$:
\begin{align}
\label{Equation-2-117}
&\,\frac{d}{dt} \|\tilde{U}_m(t)\|_{L^2(\Omega)}^2+D_1 \|\nabla \tilde{U}_m(t)\|_{\bm{L}^2(\Omega)}^2\\
\nonumber
\le&\,\left\{ \frac{4(c_1)^2(K_6)^4}{\nu D_1} \Phi_\nu (\tilde{v}(t))+\frac{2(c_2)^2 (K_6)^4}{D_1}\|\tilde{w}(t)\|_{H^1(\Omega)}^2+\delta (c_1+c_2)+D_1 \right\}
\|\tilde{U}_m(t)\|_{L^2(\Omega)}^2\\
\nonumber
&\,\hspace*{2cm}+\delta \left( c_1 \|\tilde{V}_m(t)\|_{L^2(\Omega)}^2+c_2 \|\tilde{W}_m(t)\|_{L^2(\Omega)}^2 \right).
\end{align}
Applying the Gronwall lemma to \eqref{Equation-2-117} and using \eqref{Equation-2-112}, we obtain the following inequality for all $t \in [0,T]$:
\begin{align*}
&\,\|\tilde{U}_m(t)\|_{L^2(\Omega)}^2+D_1 \int_0^t \|\nabla \tilde{U}_m(s)\|_{\bm{L}^2(\Omega)}^2ds\\ 
\le&\,\delta \int_0^t \left( c_1 \|\tilde{V}_m(s)\|_{L^2(\Omega)}^2+c_2 \|\tilde{W}_m(s)\|_{L^2(\Omega)}^2 \right)\\
&\hspace*{2cm} \times \exp \left( C_{\text{{\tiny $2,\!13,\!1$}}} \int_s^t \left\{ \Phi_\nu (\tilde{v}(\tau))+\|\tilde{w}(\tau)\|_{H^1(\Omega)}^2+1\right\}d\tau \right)ds,
\end{align*}
where we set 
\begin{equation*}
C_{\text{{\tiny $2,\!13,\!1$}}}:=\frac{4(c_1)^2(K_6)^4}{\nu D_1}+\frac{2(c_2)^2 (K_6)^4}{D_1}+\delta (c_1+c_2)+D_1.
\end{equation*}
hence, 
\begin{align}
\label{Equation-2-118}
&\,\max_{0 \le t \le T} \|\tilde{U}_m(t)\|_{L^2(\Omega)}^2+\int_0^T \|\nabla \tilde{U}_m(t)\|_{\bm{L}^2(\Omega)}^2dt\\
\nonumber
\le&\,\delta T \left(1+\frac{1}{D_1} \right) \exp \left( T C_{\text{{\tiny $2,\!13,\!1$}}} \left\{ \sup_{0 \le t \le T} \Phi_\nu (\tilde{v}(t))
+\left( \sup_{0 \le t \le T} \|\tilde{w}(t)\|_{H^1(\Omega)}\right)^2+1\right\} \right)\\
\nonumber
&\,\hspace*{1cm} \times \left\{ c_1\left( \max_{0 \le t \le T} \|\tilde{V}_m(t)\|_{L^2(\Omega)}\right)^2
+c_2\left(\max_{0 \le t \le T} \|\tilde{W}_m(t)\|_{L^2(\Omega)}\right)^2 \right\}.
\end{align}
Combining the convergences \eqref{Equation-2-101}, \eqref{Equation-2-110} and the estimate \eqref{Equation-2-118} with Lemma \ref{Lemma-2-11},
we obtain the following convergence as $m \to \infty$:
\begin{equation}\label{Equation-2-119}
\tilde{u}_m \longrightarrow \tilde{u} \quad \text{in} \quad C([0,T]\,;L^2(\Omega)) \quad \text{and} \quad L^2(0,T\,;H^1(\Omega)).
\end{equation}
\indent
Finally, from the convergences \eqref{Equation-2-110} and \eqref{Equation-2-119} we obtain \eqref{Equation-2-100}.
\end{proof}
%%%%%%%%%%%%%%%%%%%%%%%%%%%%%%%%%%%%%%%%%%%%%%%%%%%%%%%%%%%%%%%%%%%%%%%%%%%%%%%%%%%%%%%%%%%%%%%%%%%%%%%%%%%%%%%%%%%%%%%%%%%%%%%%%%%%%%%%
%%%%%%%%%%%%%%%%%%%%%%%%%%%%%%%%%%%%%%%%%%%%%%%%%%%%%%%%%%%%%%%%%%%%%%%%%%%%%%%%%%%%%%%%%%%%%%%%%%%%%%%%%%%%%%%%%%%%%%%%%%%%%%%%%%%%%%%%
%%%%%%%%%%%%%%%%%%%%%%%%%%%%%%%%%%%%%%%%%%%%%%%%%%%%%%%%%%%%%%%%%%%%%%%%%%%%%%%%%%%%%%%%%%%%%%%%%%%%%%%%%%%%%%%%%%%%%%%%%%%%%%%%%%%%%%%%
%%%%%%%%%%%%%%%%%%%%%%%%%%%%%%%%%%%%%%%%%%%%%%%%%%%%%%%%%%%%%%%%%%%%%%%%%%%%%%%%%%%%%%%%%%%%%%%%%%%%%%%%%%%%%%%%%%%%%%%%%%%%%%%%%%%%%%%%
%%%%%%%%%%%%%%%%%%%%%%%%%%%%%%%%%%%%%%%%%%%%%%%%%%%%%%%%%%%%%%%%%%%%%%%%%%%%%%%%%%%%%%%%%%%%%%%%%%%%%%%%%%%%%%%%%%%%%%%%%%%%%%%%%%%%%%%%
\section{Strong solutions to approximate problems}\label{Section-3}
%%%%%%%%%%%%%%%%%%%%%%%%%%%%%%%%%%%%%%%%%%%%%%%%%%%%%%%%%%%%%%%%%%%%%%%%%%%%%%%%%%%%%%%%%%%%%%%%%%%%%%%%%%%%%%%%%%%%%%%%%%%%%%%%%%%%%%%%
%%%%%%%%%%%%%%%%%%%%%%%%%%%%%%%%%%%%%%%%%%%%%%%%%%%%%%%%%%%%%%%%%%%%%%%%%%%%%%%%%%%%%%%%%%%%%%%%%%%%%%%%%%%%%%%%%%%%%%%%%%%%%%%%%%%%%%%%
%%%%%%%%%%%%%%%%%%%%%%%%%%%%%%%%%%%%%%%%%%%%%%%%%%%%%%%%%%%%%%%%%%%%%%%%%%%%%%%%%%%%%%%%%%%%%%%%%%%%%%%%%%%%%%%%%%%%%%%%%%%%%%%%%%%%%%%%
%%%%%%%%%%%%%%%%%%%%%%%%%%%%%%%%%%%%%%%%%%%%%%%%%%%%%%%%%%%%%%%%%%%%%%%%%%%%%%%%%%%%%%%%%%%%%%%%%%%%%%%%%%%%%%%%%%%%%%%%%%%%%%%%%%%%%%%%
%%%%%%%%%%%%%%%%%%%%%%%%%%%%%%%%%%%%%%%%%%%%%%%%%%%%%%%%%%%%%%%%%%%%%%%%%%%%%%%%%%%%%%%%%%%%%%%%%%%%%%%%%%%%%%%%%%%%%%%%%%%%%%%%%%%%%%%%
%%%%%%%%%%%%%%%%%%%%%%%%%%%%%%%%%%%%%%%%%%%%%%%%%%%%%%%%%%%%%%%%%%%%%%%%%%%%%%%%%%%%%%%%%%%%%%%%%%%%%%%%%%%%%%%%
%%%%%%%%%%%%%%%%%%%%%%%%%%%%%%%%%%%%%%%%%%%%%%%%%%%%%%%%%%%%%%%%%%%%%%%%%%%%%%%%%%%%%%%%%%%%%%%%%%%%%%%%%%%%%%%%
%%%%%%%%%%%%%%%%%%%%%%%%%%%%%%%%%%%%%%%%%%%%%%%%%%%%%%%%%%%%%%%%%%%%%%%%%%%%%%%%%%%%%%%%%%%%%%%%%%%%%%%%%%%%%%%%
\subsection{Existence of approximate strong local-in-time solutions}\label{Subsection-3-1}
%%%%%%%%%%%%%%%%%%%%%%%%%%%%%%%%%%%%%%%%%%%%%%%%%%%%%%%%%%%%%%%%%%%%%%%%%%%%%%%%%%%%%%%%%%%%%%%%%%%%%%%%%%%%%%%%
%%%%%%%%%%%%%%%%%%%%%%%%%%%%%%%%%%%%%%%%%%%%%%%%%%%%%%%%%%%%%%%%%%%%%%%%%%%%%%%%%%%%%%%%%%%%%%%%%%%%%%%%%%%%%%%%
%%%%%%%%%%%%%%%%%%%%%%%%%%%%%%%%%%%%%%%%%%%%%%%%%%%%%%%%%%%%%%%%%%%%%%%%%%%%%%%%%%%%%%%%%%%%%%%%%%%%%%%%%%%%%%%%
We fix any approximation parameters $\delta \in (1,\infty)$, $\varepsilon \in (0,1)$ and $\nu \in (0,1)$.
The main purpose of this subsection is to show Proposition \ref{Proposition-2}, which guarantees the existence of strong local-in-time solutions to the approximate 
problem $\text{(AP)}{}_{\text{{\tiny $\delta,\!\varepsilon,\!\nu$}}}$.
In the following argument, it is essential that any finite time $T>0$ is prescribed and fixed at first.
%%%%%
%%%%%
%%%%%
\begin{proposition}\label{Proposition-2}
Let any finite time $T>0$ be prescribed and fixed.
There exists a finite time $T_{\text{{\tiny $\delta$}}} \in (0,T]$, which is independent of the approximate parameters $\varepsilon \in (0,1)$ and $\nu \in (0,1)$, 
such that the approximate initial-boundary value problem $\text{(AP)}{}_{\text{{\tiny $\delta,\!\varepsilon,\!\nu$}}}$ admits a strong solution 
$(u,v,w,z):=(u_{\text{{\tiny $\delta,\!\varepsilon,\!\nu$}}},v_{\text{{\tiny $\delta,\!\varepsilon,\!\nu$}}},w_{\text{{\tiny $\delta,\!\varepsilon,\!\nu$}}},
z_{\text{{\tiny $\delta,\!\varepsilon,\!\nu$}}})$ on $[0,T_{\text{{\tiny $\delta$}}}]$ satisfying the following properties:
\begin{enumerate}\leftskip10pt
\item[(1)] $(u,v,z,w) \in (W^{1,2}(0,T_{\text{{\tiny $\delta$}}}\,;L^2(\Omega)) \cap 
L^\infty (0,T_{\text{{\tiny $\delta$}}}\,;H^1(\Omega)) \cap L^2(0,T_{\text{{\tiny $\delta$}}}\,;H^2(\Omega)))^4$.\\[-0.3cm]
\item[(2)] The system \{(\ref{Equation-2-7})--(\ref{Equation-2-10})\} with boundary conditions (\ref{Equation-2-11}) is satisfied in the following variational sense 
for a.e. $t \in (0,T_{\text{{\tiny $\delta$}}})$:
\begin{gather}
\label{Equation-3-1}
u'(t)-D_1 \Delta_N u(t)=\beta_{\text{{\tiny $\delta$}}}(u(t))\{c_1 v(t)-c_2 w(t)\} \quad \mbox{in} \quad L^2(\Omega),\\
\label{Equation-3-2}
M_{\text{{\tiny $\delta,\!z(t)$}}} v'(t)-\nu \Delta_N v(t)+v(t)=M_{\text{{\tiny $\delta,\!z(t)$}}}g_{\text{{\tiny $\delta$}}}(u,v,w\,;t) 
\quad \mbox{in} \quad L^2(\Omega),\\
\label{Equation-3-3}
w'(t)-D_2 \Delta_N w(t)+c_5 w(t)=c_6-c_7 v(t) \beta_{\text{{\tiny $\delta$}}} (w(t)) \quad \mbox{in} \quad L^2(\Omega),\\
\label{Equation-3-4}
z'(t)-\varepsilon \Delta_N z(t)=v'(t) \quad \mbox{in} \quad L^2(\Omega).
\end{gather}
where we define $g_{\text{{\tiny $\delta$}}}(u,v,w):[0,T] \rightarrow L^2(\Omega)$ by
\begin{equation*}
g_{\text{{\tiny $\delta$}}}(u,v,w\,;t):=v(t)\{c_3 \beta_{\text{{\tiny $\delta$}}} (w(t))-c_4 \beta_{\text{{\tiny $\delta$}}} (u(t))\},\quad \forall t \in [0,T].
\end{equation*}
\item[(3)] $(u(0),v(0),z(0),w(0))=(u_0,v_0,z_0,w_0) \quad \text{in} \quad (L^2(\Omega))^4$.
\end{enumerate}
\end{proposition}
%%%%%
%%%%%
%%%%%
While using the same notations as in Section \ref{Section-2}, for each $\tilde{v} \in \mathcal{V}(v_0)$ we set 
\begin{align}
\label{Equation-3-5}
&g_{\text{{\tiny $\delta$}}} (\tilde{v}):=\tilde{v}\{c_3 \beta_{\text{{\tiny $\delta$}}} (w_{\text{{\tiny $\delta$}}} (w_0,\tilde{v}))
-c_4 \beta_{\text{{\tiny $\delta$}}} (u_{\text{{\tiny $\delta$}}} (u_0,w_0,\tilde{v}))\},\\
\nonumber
&\hspace*{1cm}\text{where} \quad S_{\text{{\tiny $3,\!\delta$}}}(\tilde{v})=(u_{\text{{\tiny $\delta$}}} (u_0,w_0,\tilde{v}),w_{\text{{\tiny $\delta$}}} (w_0,\tilde{v})),
\quad \forall \tilde{v} \in \mathcal{V}(v_0),
\end{align}
and define an operator $S:C([0,T]\,;L^2(\Omega)) \rightarrow C([0,T]\,;L^2(\Omega))$ in the following way:
\begin{gather*}
D(S)=\mathcal{V}(v_0),\\
R(S) \subset W^{1,2}(0,T\,;L^2(\Omega)) \cap L^\infty (0,T\,;H^1(\Omega)) \cap L^2(0,T\,;H^2(\Omega)),\\
S(\tilde{v}):=S_{\text{{\tiny $1,\!\delta,\!\nu$}}} \left(S_{\text{{\tiny $2,\!\varepsilon,\!z_0$}}}(\tilde{v}),g_{\text{{\tiny $\delta$}}} (\tilde{v})\right).
\end{gather*}
Moreover, we fix a sufficiently large number $R>0$ and introduce a subset $\mathcal{V}_R(v_0,\nu) \subset \mathcal{V}(v_0)$ in the following manner:
$\tilde{v} \in \mathcal{V}_R(v_0,\nu)$ if and only if $\tilde{v} \in \mathcal{V}(v_0)$ and the following estimate is satisfied:
\begin{equation}\label{Equation-3-6}
\int_0^T \|\tilde{v}'(t)\|_{L^2(\Omega)}^2dt+\sup_{0 \le t \le T} \Phi_\nu (\tilde{v}(t)) \le R.
\end{equation}
Since from \eqref{Equation-2-21} we have 
\begin{equation}\label{Equation-3-7}
\frac{1}{2} \|\bar{v}\|_{L^2(\Omega)}^2 \le \Phi_\nu (\bar{v}),\quad \forall \bar{v} \in L^2(\Omega),~\forall \nu \in (0,1).
\end{equation}
\eqref{Equation-3-6} yields the following uniform boundedness:
\begin{equation}\label{Equation-3-8}
\sup_{z \in \mathcal{V}_R(v_0,\nu)} \left( \sup_{0 \le t \le T} \|\tilde{v}(t)\|_{L^2(\Omega)} \right)^2 \le 2R.
\end{equation}
First, we investigate the property of $\mathcal{V}_R(v_0,\nu)$.
%%%%%
%%%%%
%%%%%
\begin{lemma}\label{Lemma-3-1}
There exists a number $R_*>0$, which is independent of the approximation parameter $\nu \in (0,1]$, such that $\mathcal{V}_R(v_0,\nu)$ satisfies 
the following properties for all $R \ge R_*$:
\begin{enumerate}\leftskip6pt
\item[(1)] $\mathcal{V}_R(v_0,\nu)$ is nonempty and convex.\\[-0.3cm]
\item[(2)] $\mathcal{V}_R(v_0,\nu)$ is compact in $C([0,T]\,;L^2(\Omega))$.\\[-0.3cm]
\item[(3)] $\mathcal{V}_R(v_0,\nu)$ is weakly closed in $W^{1,2}(0,T\,;L^2(\Omega))$.\\[-0.3cm]
\item[(4)] $\mathcal{V}_R(v_0,\nu)$ is weakly$^*$ closed in $L^\infty (0,T\,;H^1(\Omega))$.
\end{enumerate}
\end{lemma}
%%%%%
%%%%%
%%%%%
\begin{proof}
We set a number $R_*>0$ by
\begin{equation*}
R_*:=\frac{1}{2}\int_\Omega |\nabla v_0|^2dx+\frac{1}{2} \int_\Omega |v_0|^2dx=\frac{1}{2} \|v_0\|_{H^1(\Omega)}^2,
\end{equation*}
and a function $\tilde{v}_0:[0,T] \rightarrow L^2(\Omega)$ defined by $\tilde{v}_0(t):=v_0$ for all $t \in [0,T]$.
Then, we have 
\begin{equation*}
\int_0^T \|\tilde{v}_0'(t)\|_{L^2(\Omega)}^2dt+\sup_{0 \le t \le T} \Phi_\nu (\tilde{v}_0(t)) \le R_*,
\end{equation*}
which implies $\mathcal{V}_R(v_0,\nu) \ne \emptyset$ for all $R \ge R_*$ and $\nu \in (0,1]$.
By the convexity of $\Phi_\nu$ and norms, together with their weak lower semicontinuity, the set $\mathcal{V}_R(v_0,\nu)$ is convex, weakly closed in
$W^{1,2}(0,T;L^2(\Omega))$ and weakly$^*$ closed in $L^\infty(0,T;H^1(\Omega))$.
Moreover, since 
\begin{equation*}
\mathcal{V}_R(v_0,\nu) \subset W^{1,2}(0,T\,;L^2(\Omega)) \cap L^\infty (0,T\,;H^1(\Omega)),
\end{equation*}
the compact embedding $H^1(\Omega) \hookrightarrow \hookrightarrow L^2(\Omega)$ and the Ascoli-Arzel\`{a} theorem, which is used in Section 2, imply that 
$\mathcal{V}_R(v_0,\nu)$ is compact in $C([0,T]\,;L^2(\Omega))$.
\end{proof}
%%%%%
%%%%%
%%%%%
\indent
In the following argument, we fix a sufficiently large number $R > R_*$, where the number $R_*$ is the same as in Lemma \ref{Lemma-3-1}.
%%%%%
%%%%%
%%%%%
\begin{lemma}\label{Lemma-3-2}
The following uniform estimate holds for all $t \in [0,T]$:
\begin{equation*}
\sup_{\nu\,\in\,(0,1)} \left( \sup_{\tilde{v}\,\in\,\mathcal{V}_R(v_0,\nu)} \int_0^t \|(S_{\text{{\tiny $2,\!\varepsilon,\!z_0$}}}(\tilde{v}))'(s)\|_{L^2(\Omega)}ds \right)
\le t^{\frac{1}{2}}\left( \|z_0\|_{H^1(\Omega)}+R\right).
\end{equation*}
\end{lemma}
%%%%%
%%%%%
%%%%%
\begin{proof}[Proof.]
In this proof we set $\tilde{z}:=S_{\text{{\tiny $2,\!\varepsilon,\!z_0$}}}(\tilde{v})$.
Since we have 
\begin{gather}
\label{Equation-3-9}
\tilde{z}'(t)-\varepsilon \Delta_N \tilde{z}(t)=\tilde{v}'(t) \quad \text{in} \quad L^2(\Omega),\quad \text{a.e.}~t \in (0,T),\\
\label{Equation-3-10}
\tilde{z}(0)=z_0 \quad \text{in} \quad L^2(\Omega),
\end{gather}
we take the inner product in both sides of \eqref{Equation-3-9} in $L^2(\Omega)$ by $\tilde{z}'(t)$ and obtain
\begin{equation*}
\|\tilde{z}'(t)\|_{L^2(\Omega)}^2+\frac{\varepsilon}{2} \frac{d}{dt} \|\nabla \tilde{z}(t)\|_{\bm{L}^2(\Omega)}^2
=\int_\Omega \tilde{v}'(t) \tilde{z}'(t)dx,\quad \text{a.e.}~t \in (0,T),
\end{equation*}
hence,
\begin{equation}\label{Equation-3-11}
\|\tilde{z}'(t)\|_{L^2(\Omega)}^2+\varepsilon \frac{d}{dt} \|\nabla \tilde{z}(t)\|_{\bm{L}^2(\Omega)}^2 \le \|\tilde{v}'(t)\|_{L^2(\Omega)}^2,\quad \text{a.e.}~t \in (0,T).
\end{equation}
Integrating \eqref{Equation-3-11} over $[0,T]$ and using \eqref{Equation-3-10}, we obtain
\begin{equation}\label{Equation-3-12}
\int_0^T \|\tilde{z}'(s)\|_{L^2(\Omega)}^2ds \le \varepsilon \|\nabla z_0\|_{\bm{L}^2(\Omega)}^2+\int_0^T \|\tilde{v}'(s)\|_{L^2(\Omega)}^2ds 
\le \|z_0\|_{H^1(\Omega)}^2+R^2. 
\end{equation}
Using the Cauchy-Schwarz inequality, the estimate \eqref{Equation-3-12} yields
\begin{equation*}
\int_0^t \|\tilde{z}'(s)\|_{L^2(\Omega)}ds \le t^{\frac{1}{2}} \left(\int_0^T \|\tilde{z}'(t)\|_{L^2(\Omega)}^2dt\right)^{\frac{1}{2}} 
\le t^{\frac{1}{2}} \left( \|z_0\|_{H^1(\Omega)}+R\right),
\end{equation*}
which implies that this lemma holds.
\end{proof}
%%%%%
%%%%%
%%%%%
In the next lemma, we investigate the properties of the external forcing term $g_{\text{{\tiny $\delta$}}} (\tilde{v})$ given by \eqref{Equation-3-5}.
%%%%%
%%%%%
%%%%%
\begin{lemma}\label{Lemma-3-3}
We have 
\begin{equation}\label{Equation-3-13}
g_{\text{{\tiny $\delta$}}} (\tilde{v}_m) \longrightarrow g_{\text{{\tiny $\delta$}}} (\tilde{v}) \quad \text{in} \quad L^2(0,T\,;L^2(\Omega)) \quad 
\text{as} \quad m \to \infty 
\end{equation}
whenever a sequence $\{\tilde{v}_m\}_{m\,\in\,\mathbb{N}} \subset \mathcal{V}_R(v_0,\nu)$ and a function $\tilde{v} \in \mathcal{V}(v_0)$ satisfy the following 
convergences as $m \to \infty$:
\begin{equation}\label{Equation-3-14}
\tilde{v}_m \longrightarrow \tilde{v} \quad \left\{
\begin{array}{l}
\text{in} \quad C([0,T]\,;L^2(\Omega)),\\[0.1cm]
\text{weakly in} \quad W^{1,2}(0,T\,;L^2(\Omega)),\\[0.1cm]
\text{weakly$^*$ in} \quad L^\infty (0,T\,;H^1(\Omega)).
\end{array}
\right.
\end{equation}
\indent
Moreover, there exists a constant $C_9>0$, which is independent of the approximate parameters $\delta \in (1,\infty)$ and $\nu \in (0,1)$, 
such that the following uniform estimate holds:
\begin{equation*}
\sup_{\nu\,\in\,(0,1)} \left\{ \sup_{\tilde{v}\,\in\,\mathcal{V}_R(v_0,\nu)} 
\left( \sup_{0 \le t \le T} \|(g_{\text{{\tiny $\delta$}}} (\tilde{v}))(t)\|_{L^2(\Omega)} \right) \right\} \le \delta C_9 R^{\frac{1}{2}}.
\end{equation*}
\end{lemma}
%%%%%
%%%%%
%%%%%
\begin{proof}[Proof.]
First of all, by Lemma \ref{Lemma-3-1} it follows $\tilde{v} \in \mathcal{V}_R(v_0,\nu)$.
Throughout this proof, for simplicity we set $(\tilde{u}_m,\tilde{w}_m):=S_{\text{{\tiny $3,\!\delta$}}}(\tilde{v}_m)$ for all $m \in \mathbb{N}$ and 
$(\tilde{u},\tilde{w}):=S_{\text{{\tiny $3,\!\delta$}}}(\tilde{v})$.
Using \eqref{Equation-2-3} and \eqref{Equation-2-4}, we obtain the following inequality for all $m \in \mathbb{N}$: 
\begin{align}
\label{Equation-3-15}
&\,\int_0^T \|(g_{\text{{\tiny $\delta$}}} (\tilde{v}_m))(t)-(g_{\text{{\tiny $\delta$}}} (\tilde{v}))(t)\|_{L^2(\Omega)}^2dt\\
\nonumber
\le&\,4(\delta c_3)^2 \int_0^T \|\tilde{v}_m(t)-\tilde{v}(t)\|_{L^2(\Omega)}^2dt+4(c_3)^2 \int_0^T \|\tilde{v}(t)\{\tilde{w}_m(t)-\tilde{w}(t)\}\|_{L^2(\Omega)}^2dt\\
\nonumber
&\,\hspace*{0.5cm}+4(\delta c_4)^2 \int_0^T \|\tilde{v}_m(t)-\tilde{v}(t)\|_{L^2(\Omega)}^2dt+4(c_4)^2 \int_0^T \|\tilde{v}(t)\{\tilde{u}_m(t)-\tilde{u}(t)\}\|_{L^2(\Omega)}^2dt.
\end{align}
Since from \eqref{Equation-2-105} we obtain 
\begin{align*}
\int_0^T \|\tilde{v}(t)\{\tilde{w}_m(t)-\tilde{w}(t)\}\|_{L^2(\Omega)}^2dt &\le \int_0^T \|\tilde{v}(t)\|_{L^4(\Omega)}^2 \|\tilde{w}_m(t)-\tilde{w}(t)\|_{L^4(\Omega)}^2dt\\
&\le \frac{2(K_6)^4}{\nu} \int_0^T \Phi_\nu (\tilde{v}(t)) \|\tilde{w}_m(t)-\tilde{w}(t)\|_{H^1(\Omega)}^2dt\\
&\le \frac{2(K_6)^4 R}{\nu} \int_0^T \|\tilde{w}_m(t)-\tilde{w}(t)\|_{H^1(\Omega)}^2dt 
\end{align*}
as well as 
\begin{equation*}
\int_0^T \|\tilde{v}(t)\{\tilde{u}_m(t)-\tilde{u}(t)\}\|_{L^2(\Omega)}^2dt \le \frac{2(K_6)^4 R}{\nu} \int_0^T \|\tilde{u}_m(t)-\tilde{u}(t)\|_{H^1(\Omega)}^2dt, 
\end{equation*}
we use Lemma \ref{Lemma-2-13} and obtain the following convergences as $m \to \infty$:
\begin{gather}
\label{Equation-3-16}
\tilde{v} (\tilde{w}_m-\tilde{w}) \longrightarrow 0 \quad \text{in} \quad L^2(0,T\,;L^2(\Omega)),\\
\label{Equation-3-17}
\tilde{v} (\tilde{u}_m-\tilde{u}) \longrightarrow 0 \quad \text{in} \quad L^2(0,T\,;L^2(\Omega)).
\end{gather}
Using \eqref{Equation-3-15} with all convergences \eqref{Equation-3-14}, \eqref{Equation-3-16} and \eqref{Equation-3-17}, 
it follows that the required convergence \eqref{Equation-3-13} in the first part of this lemma holds as $m \to \infty$.\\
\indent
In the rest of this proof, we show the uniform boundedness of $\{g_{\text{{\tiny $\delta$}}} (\tilde{v})\,;\,\tilde{v} \in \mathcal{V}_R(v_0,\nu)\}$.
Using \eqref{Equation-2-3} again and \eqref{Equation-3-8}, we obtain 
\begin{gather}
\label{Equation-3-18}
\|(g_{\text{{\tiny $\delta$}}} (\tilde{v}))(t)\|_{L^2(\Omega)}^2=\int_\Omega |\tilde{v}(t)|^2 |c_3 \beta_{\text{{\tiny $\delta$}}} (\tilde{w}(t))
-c_4 \beta_{\text{{\tiny $\delta$}}} (\tilde{u}(t))|^2 dx\\
\nonumber
\le \delta^2 (c_3+c_4)^2 \|\tilde{v}(t)\|_{L^2(\Omega)}^2 \le 2 \delta^2 (c_3+c_4)^2 R,
\end{gather}
which completes the proof by taking a constant $C_9:=2^{\frac{1}{2}} (c_3+c_4)$.
\end{proof}
%%%%%
%%%%%
%%%%%
The last lemma guarantees the uniform boundedness of $\{S(\tilde{v})\,;\,\tilde{v} \in \mathcal{V}(v_0)\}$, which enables us to proceed the Schauder fixed point argument.
%%%%%
%%%%%
%%%%%
\begin{lemma}\label{Lemma-3-4}
There exist constants 
\begin{gather*}
C_{10}:=C_{10}(\|v_0\|_{H^1(\Omega)})>0, \quad C_{11}:=C_{11}(\delta,R,\|z_0\|_{H^1(\Omega)})>0,\\
C_{12}:=C_{12}\left(\delta,R\right)>0,
\end{gather*}
such that the following uniform estimate holds for all $t \in [0,T]$:
\begin{equation*}
\sup_{\tilde{v}\,\in\,\mathcal{V}_R(v_0,\nu)} \left\{ \int_0^t \|(S(\tilde{v}))'(s)\|_{L^2(\Omega)}^2ds+\sup_{0 \le s \le t} \Phi_\nu ((S(\tilde{v}))(s)) \right\}
\le C_{10}+C_{11} t^{\frac{1}{2}}+C_{12} t.
\end{equation*}
\end{lemma}
%%%%%
%%%%%
%%%%%
\begin{proof}[Proof.]
In this proof we set $\hat{v}:=S(\tilde{v})$.
By Lemmas \ref{Lemma-2-6}, \ref{Lemma-3-2} and \ref{Lemma-3-3}, we have already obtained the following uniform estimate: there exists a constant 
$C_{\text{{\tiny $3,\!4,\!1$}}}>0$, which depends on the following values:
\begin{equation*}
\delta,\quad \|v_0\|_{H^1(\Omega)}, \quad \|z_0\|_{H^1(\Omega)}, \quad T, \quad R,
\end{equation*}
which is independent of $\varepsilon$, such that 
\begin{equation}\label{Equation-3-19}
\sup_{\tilde{v}\,\in\,\mathcal{V}_R(v_0,\nu)} \left\{ \sup_{0 \le t \le T} \Phi_\nu (\hat{v}(t))+\int_0^T \|\hat{v}'(t)\|_{L^2(\Omega)}^2dt\right\} \le C_{\text{{\tiny $3,\!4,\!1$}}}.
\end{equation}
Then, we go back to \eqref{Equation-2-50}.
Using (a) of Lemma \ref{Lemma-2-2} with \eqref{Equation-3-19}, we derive the following inequality for all $t \in [0,T]$:
\begin{align}
\label{Equation-3-20}
&\,\Phi_\nu (\hat{v}(t))+\frac{1}{2} \int_0^t \|\hat{v}'(s)\|_{\text{{\tiny $\delta,\!\tilde{z}(s)$}}}^2ds\\
\nonumber
\le&\,\Phi_\nu (v_0)+C_{\text{{\tiny $2,\!6,\!1$}}} \int_0^t \left\{\Phi_\nu (\hat{v}(s))+1\right\} \|\tilde{z}'(s)\|_{L^2(\Omega)}ds
+\frac{5}{4} \int_0^t \|(g_{\text{{\tiny $\delta$}}} (\tilde{v}))(s)\|_{\text{{\tiny $\delta,\!\tilde{z}(s)$}}}^2ds\\
\nonumber
\le&\,\frac{1}{2}\|v_0\|_{H^1(\Omega)}^2+C_{\text{{\tiny $2,\!6,\!1$}}} (C_{\text{{\tiny $3,\!4,\!1$}}}+1) \int_0^t \|\tilde{z}'(s)\|_{L^2(\Omega)}ds\\
\nonumber
&\,\hspace*{2cm}+\frac{5\delta (C_2)^2}{4}\left(\sup_{0 \le t \le T} \|(g_{\text{{\tiny $\delta$}}} (\tilde{v}))(t)\|_{L^2(\Omega)}\right)^2 t\\
\nonumber
\le&\,\frac{1}{2}\|v_0\|_{H^1(\Omega)}^2+C_{\text{{\tiny $2,\!6,\!1$}}}(C_{\text{{\tiny $3,\!4,\!1$}}}+1) \left( \|z_0\|_{H^1(\Omega)}+R\right) t^{\frac{1}{2}}
+\frac{5\delta^3 R(C_2 C_9)^2}{4}t.
\end{align}
Taking the constants $C_{10}>0$, $C_{11}>0$ and $C_{12}>0$ by
\begin{gather*}
C_{10}:=\frac{3}{2}\|v_0\|_{H^1(\Omega)}^2, \quad C_{11}:=3 C_{\text{{\tiny $2,\!6,\!1$}}} (C_{\text{{\tiny $3,\!4,\!1$}}}+1) \left( \|z_0\|_{H^1(\Omega)}+R\right),\\
C_{12}:=\frac{15\delta^3 R(C_2 C_9)^2}{4},
\end{gather*}
\eqref{Equation-3-20} yields the required uniform bounded estimate.
\end{proof}
%%%%%
%%%%%
%%%%%
Now we are ready for showing Proposition \ref{Proposition-2}.
%%%%%
%%%%%
%%%%%
\begin{proof}[Proof of Proposition \ref{Proposition-2}.]
We fix a number $R_0 \ge R_*$ satisfying $R_0 \ge C_{10}+1$, and choose a time $T_{\text{{\tiny $\delta$}}} \in (0,T]$, which is independent of $\nu \in (0,1)$ 
and $\varepsilon \in (0,1)$, such that 
\begin{equation}\label{Equation-3-21}
C_{11} T_{\text{{\tiny $\delta$}}}^{\frac{1}{2}}+C_{12} T_{\text{{\tiny $\delta$}}} \le 1.
\end{equation}
By Lemma \ref{Lemma-3-4} with \eqref{Equation-3-21} we have 
\begin{equation}\label{Equation-3-22}
S(\mathcal{V}_R(v_0,\nu)) \subset \mathcal{V}_R(v_0,\nu),\quad \forall \varepsilon \in (0,1).
\end{equation}
By Lemmas \ref{Lemma-2-7}, \ref{Lemma-2-6}, \ref{Lemma-2-13} and \ref{Lemma-3-3}, the operator 
$S:C([0,T_{\text{{\tiny $\delta$}}}]\,;L^2(\Omega)) \rightarrow C([0,T_{\text{{\tiny $\delta$}}}]\,;L^2(\Omega))$ is continuous satisfying \eqref{Equation-3-22}.
Using Lemma \ref{Lemma-3-1} and applying the Schauder fixed point theorem, we see that the operator $S$ has at least one fixed point $v$, that is, $S(v)=v$.
By the definition of the operators $S_{\text{{\tiny $1,\!\delta,\!\nu$}}}$, $S_{\text{{\tiny $2,\!\varepsilon,\!z_0$}}}$ and $S_{\text{{\tiny $3,\!\delta$}}}$, 
the fixed point $v$ of the operator $S$ generates the quadruple $(u,v,w,z)$ satisfying all required properties (s1)--(s3) in Theorem \ref{Proposition-2}.
Hence, $(u,v,w,z)$ is a strong solution to $\text{(AP)}{}_{\text{{\tiny $\delta,\!\varepsilon,\!\nu$}}}$ on $[0,T_{\text{{\tiny $\delta$}}}]$.
\end{proof}
%%%%%
%%%%%
%%%%%
%%%%%%%%%%%%%%%%%%%%%%%%%%%%%%%%%%%%%%%%%%%%%%%%%%%%%%%%%%%%%%%%%%%%%%%%%%%%%%%%%%%%%%%%%%%%%%%%%%%%%%%%%%%%%%%%
%%%%%%%%%%%%%%%%%%%%%%%%%%%%%%%%%%%%%%%%%%%%%%%%%%%%%%%%%%%%%%%%%%%%%%%%%%%%%%%%%%%%%%%%%%%%%%%%%%%%%%%%%%%%%%%%
%%%%%%%%%%%%%%%%%%%%%%%%%%%%%%%%%%%%%%%%%%%%%%%%%%%%%%%%%%%%%%%%%%%%%%%%%%%%%%%%%%%%%%%%%%%%%%%%%%%%%%%%%%%%%%%%
\subsection{Boundedness}\label{Subsection-3-2}
%%%%%%%%%%%%%%%%%%%%%%%%%%%%%%%%%%%%%%%%%%%%%%%%%%%%%%%%%%%%%%%%%%%%%%%%%%%%%%%%%%%%%%%%%%%%%%%%%%%%%%%%%%%%%%%%
%%%%%%%%%%%%%%%%%%%%%%%%%%%%%%%%%%%%%%%%%%%%%%%%%%%%%%%%%%%%%%%%%%%%%%%%%%%%%%%%%%%%%%%%%%%%%%%%%%%%%%%%%%%%%%%%
%%%%%%%%%%%%%%%%%%%%%%%%%%%%%%%%%%%%%%%%%%%%%%%%%%%%%%%%%%%%%%%%%%%%%%%%%%%%%%%%%%%%%%%%%%%%%%%%%%%%%%%%%%%%%%%%
In this subsection, we show some uniform boundedness of approximate strong local-in-time solutions to the approximate Cauchy problem 
$\text{(AP)}{}_{\text{{\tiny $\delta,\!\varepsilon,\!\nu$}}}$, which plays an important role not only to show the uniqueness of approximate strong local-in-time 
solutions in Subsection \ref{Subsection-3-3} but also to apply the continuation argument of the	approximate strong local-in-time solution for constructing approximate 
strong global-in-time solutions to $\text{(AP)}{}_{\text{{\tiny $\delta,\!\varepsilon,\!\nu$}}}$ in Section \ref{Subsection-3-4}.
In the following argument, we denote by a quadruple $(u_{\text{{\tiny $\delta,\!\varepsilon,\!\nu$}}},v_{\text{{\tiny $\delta,\!\varepsilon,\!\nu$}}},
w_{\text{{\tiny $\delta,\!\varepsilon,\!\nu$}}},z_{\text{{\tiny $\delta,\!\varepsilon,\!\nu$}}})$ a strong solution to the approximate initial-boundary value problem 
$\text{(AP)}{}_{\text{{\tiny $\delta,\!\varepsilon,\!\nu$}}}$ on $[0,T^*]$, where $T^*>0$ is a certain finite time depending on the approximate parameter 
$\delta \in (1,\infty)$ in general.\\
\indent
Firstly, we show Lemma \ref{Lemma-3-6}.
Especially, for each $\delta \in (1,\infty)$ Lemma \ref{Lemma-3-6} guarantees that the second component $v_{\text{{\tiny $\delta,\!\varepsilon,\!\nu$}}}$ 
of the strong solution to $\text{(AP)}{}_{\text{{\tiny $\delta,\!\varepsilon,\!\nu$}}}$ on $[0,T^*]$ is uniformly bounded in 
$C([0,T]\,;L^2(\Omega)) \cap W^{1,2}(0,T^*\,;L^2(\Omega))$ with 
respect to $\varepsilon \in (0,1)$ and $\nu \in (0,1)$, which comes from \eqref{Equation-3-7}.
%%%%%
%%%%%
%%%%%
\begin{lemma}\label{Lemma-3-6}
There exists a constant $C_{13}>0$, which depends on $\delta$, $T^*$ and $\|v_0\|_{H^1(\Omega)}$, such that the following uniform boundedness holds 
for all $\nu \in (0,1)$:
\begin{align*}
&\sup_{\varepsilon \in (0,1)} \left\{ \int_0^{T^*} \|v_{\text{{\tiny $\delta,\!\varepsilon,\!\nu$}}}'(t)\|_{L^2 (\Omega)}^2dt
+\sup_{0 \le t \le T^*} \Phi_\nu (v_{\text{{\tiny $\delta,\!\varepsilon,\!\nu$}}}(t)) \right.\\
&\hspace*{4cm}\left.+\nu^2 \int_0^{T^*} \|v_{\text{{\tiny $\delta,\!\varepsilon,\!\nu$}}}(t)\|_{H^2(\Omega)}^2dt \right\} \le C_{13}.
\end{align*}
In particular, the following uniform boundedness holds:
\begin{equation*}
\sup_{\varepsilon \in (0,1),\,\nu \in (0,1)} \left\{ \int_0^{T^*} \|v_{\text{{\tiny $\delta,\!\varepsilon,\!\nu$}}}'(t)\|_{L^2 (\Omega)}^2dt+
\frac{1}{2} \left( \max_{0 \le t \le T^*} \|v_{\text{{\tiny $\delta,\!\varepsilon,\!\nu$}}}(t)\|_{L^2(\Omega)}\right)^2 \right\} \le C_{13}.
\end{equation*}

\end{lemma}
%%%%%
%%%%%
%%%%%
\begin{proof}[Proof.]
For simplicity we set $(u,v,w,z):=(u_{\text{{\tiny $\delta,\!\varepsilon,\!\nu$}}},v_{\text{{\tiny $\delta,\!\varepsilon,\!\nu$}}},w_{\text{{\tiny $\delta,\!\varepsilon,\!\nu$}}},
z_{\text{{\tiny $\delta,\!\varepsilon,\!\nu$}}})$.
Taking the inner product in both sides of \eqref{Equation-3-2} in $L^2(\Omega)$ by $v'(t)$ and using \eqref{Equation-2-18}, 
we obtain
\begin{align}
\label{Equation-3-23}
&\int_\Omega \left(\mathcal{K}_{\text{{\tiny $\delta$}}}(z(t))\right)^{-1} |v'(t)|^2dx+\frac{d}{dt} \Phi_\nu (v(t))\\
\nonumber
&\hspace*{1cm}=\int_\Omega \left(\mathcal{K}_{\text{{\tiny $\delta$}}}(z(t))\right)^{-1} g_{\text{{\tiny $\delta$}}}(u,v,w\,;t) v'(t)dx,\quad \text{a.e.}~t \in (0,T^*).
\end{align}
In order to estimate the right-hand side of \eqref{Equation-3-23}, we use \eqref{Equation-2-15} and obtain the following inequality for a.e. $t \in (0,T^*)$:
\begin{gather*}
\frac{1}{\delta K_2 |\Omega|^{\frac{1}{2}}} \|v'(t)\|_{L^2(\Omega)}^2+\frac{d}{dt} \Phi_\nu (v(t)) \le \frac{\delta^2 (c_3+c_4)}{K_1} \int_\Omega |v(t)| |v'(t)|dx\\
\le \frac{1}{2\delta K_2 |\Omega|^{\frac{1}{2}}} \|v'(t)\|_{L^2(\Omega)}^2+\frac{\delta^5 K_2 |\Omega|^{\frac{1}{2}}}{2}
\left( \frac{c_3+c_4}{K_1}\right)^2 \|v(t)\|_{L^2(\Omega)}^2.
\end{gather*}
From \eqref{Equation-3-7} the above inequality yields the following inequality for a.e. $t \in (0,T^*)$:
\begin{equation}\label{Equation-3-24}
\frac{1}{2\delta K_2 |\Omega|^{\frac{1}{2}}} \|v'(t)\|_{L^2(\Omega)}^2+\frac{d}{dt} \Phi_\nu (v(t)) \le C_{\text{{\tiny $3,\!5,\!1$}}}(\delta) \Phi_\nu (v(t)),
\end{equation}
where we take the constant $C_{\text{{\tiny $3,\!5,\!1$}}}(\delta)>0$ by 
\begin{equation*}
C_{\text{{\tiny $3,\!5,\!1$}}}(\delta):=\delta^5 K_2 |\Omega|^{\frac{1}{2}} \left( \frac{c_3+c_4}{K_1}\right)^2.
\end{equation*}
Applying the Gronwall lemma to \eqref{Equation-3-24}, we obtain the following inequality for all $t \in [0,T^*]$:
\begin{equation*}
\frac{1}{2\delta K_2 |\Omega|^{\frac{1}{2}}} \int_0^t \|v'(s)\|_{L^2(\Omega)}^2ds+\Phi_\nu (v(t)) \le \Phi_\nu (v_0) \exp \bigl(C_{\text{{\tiny $3,\!5,\!1$}}}(\delta)T\bigr),
\end{equation*}
which implies 
\begin{equation}\label{Equation-3-25}
\int_0^T \|v'(t)\|_{L^2(\Omega)}^2dt+\sup_{0 \le t \le T^*} \Phi_\nu (v(t)) \le C_{\text{{\tiny $3,\!5,\!2$}}}(\delta)
\end{equation}
by taking the constant $C_{\text{{\tiny $3,\!5,\!2$}}}>0$ as follows:
\begin{equation*}
C_{\text{{\tiny $3,\!5,\!2$}}}(\delta):=\left(\frac{1}{2}+\delta K_2 |\Omega|^{\frac{1}{2}}\right) \|v_0\|_{H^1(\Omega)} ^2 \exp \bigl(C_{\text{{\tiny $3,\!5,\!1$}}}(\delta)T\bigr).
\end{equation*}
\indent
Next, we use the Neumann elliptic regularity (cf. \eqref{Equation-1-8}).
Then, we obtain
\begin{align}
\label{Equation-3-26}
&\,\nu^2 \int_0^{T^*} \|v(t)\|_{H^2(\Omega)}^2dt \le (K_5)^2 \int_0^{T^*} \left( \nu \|\Delta v(t)\|_{L^2(\Omega)}+\|v(t)\|_{L^2(\Omega)}\right)^2dt\\
\nonumber
\le&\,(K_5)^2 \int_0^{T^*} \left( \|M_{\text{{\tiny $\delta,\!z(t)$}}}v'(t)\|_{L^2 (\Omega)}+\|M_{\text{{\tiny $\delta,\!z(t)$}}} g_{\text{{\tiny $\delta$}}}(u,v,w\,;t)\|_{L^2(\Omega)}
+2\|v(t)\|_{L^2(\Omega)} \right)^2 dt\\
\nonumber
\le&\,3(K_5)^2 \int_0^{T^*} \|M_{\text{{\tiny $\delta,\!z(t)$}}}v'(t)\|_{L^2 (\Omega)}^2dt+12(K_5)^2 \int_0^{T^*} \|v(t)\|_{L^2(\Omega)}^2dt\\
\nonumber
&\,\hspace*{2cm}+3(K_5)^2 \int_0^{T^*} \|M_{\text{{\tiny $\delta,\!z(t)$}}} g_{\text{{\tiny $\delta$}}}(u,v,w\,;t)\|_{L^2(\Omega)}^2dt.
\end{align}
From \eqref{Equation-2-15} we obtain 
\begin{equation}\label{Equation-3-27}
\|M_{\text{{\tiny $\delta,\!z(t)$}}}v'(t)\|_{L^2 (\Omega)}^2=\int_\Omega |\mathcal{K}_{\text{{\tiny $\delta$}}}(z(t))|^{-2} |v'(t)|^2dx 
\le \left(\frac{\delta}{K_1}\right)^2 \|v'(t)\|_{L^2(\Omega)}^2,
\end{equation}
and by repeating the similar derivation to \eqref{Equation-3-18}
\begin{equation}\label{Equation-3-28}
\|M_{\text{{\tiny $\delta,\!z(t)$}}} g_{\text{{\tiny $\delta$}}}(u,v,w\,;t)\|_{L^2(\Omega)}^2 \le \delta^4 \left( \frac{c_3+c_4}{K_1} \right)^2 \|v(t)\|_{L^2(\Omega)}^2.
\end{equation}
Substituting \eqref{Equation-3-2} into \eqref{Equation-3-26}, from \eqref{Equation-3-7}, \eqref{Equation-3-27} and \eqref{Equation-3-28} we obtain 
\begin{equation}\label{Equation-3-29}
\nu^2 \int_0^{T^*} \|v(t)\|_{H^2(\Omega)}^2dt \le C_{\text{{\tiny $3,\!5,\!3$}}}(\delta)
\left( \int_0^{T^*} \|v'(t)\|_{L^2(\Omega)}^2dt+\sup_{0 \le t \le T^*} \Phi_\nu (v(t))\right),
\end{equation}
where the constant $C_{\text{{\tiny $3,\!5,\!3$}}}(\delta)>0$ is given by
\begin{equation*}
C_{\text{{\tiny $3,\!5,\!3$}}}(\delta):=3\left(\frac{\delta K_5}{K_1}\right)^2+6\delta^4 T^* \left\{ \frac{(c_3+c_4) K_5}{K_1} \right\}^2+24 T^* (K_5)^2.
\end{equation*}
\indent
Hence, \eqref{Equation-3-25} and \eqref{Equation-3-29} yield the required uniform boundedness in this lemma. 
\end{proof}
%%%%%
%%%%%
%%%%%
Secondly, we show Lemma \ref{Lemma-3-7} which implies that the third component $w_{\text{{\tiny $\delta,\!\varepsilon,\!\nu$}}}$ of the strong solution to 
$\text{(AP)}{}_{\text{{\tiny $\delta,\!\varepsilon,\!\nu$}}}$ on $[0,T^*]$ is uniformly bounded in 
\begin{equation*}
W^{1,2}(0,T^*;L^2(\Omega)) \cap L^\infty (0,T^*;H^1(\Omega)) \cap L^2(0,T^*;H^2(\Omega))
\end{equation*}
with respect to the approximate parameters $\varepsilon \in (0,1)$ and $\nu \in (0,1)$.
%%%%%
%%%%%
%%%%%
\begin{lemma}\label{Lemma-3-7}
There exists a constant $C_{14}>0$, which depends on $\delta$, $T^*$, $\|v_0\|_{H^1(\Omega)}$ and $\|w_0\|_{H^1(\Omega)}$, such that 
\begin{align*}
&\sup_{\varepsilon \in (0,1),\,\nu \in (0,1)} \left( \sup_{0 \le t \le T} \|w_{\text{{\tiny $\delta,\!\varepsilon,\!\nu$}}}(t)\|_{H^1(\Omega)}^2
+\int_0^T \|w_{\text{{\tiny $\delta,\!\varepsilon,\!\nu$}}}'(t)\|_{L^2 (\Omega)}^2dt \right.\\
&\hspace*{4cm}\left.+\int_0^T \|w_{\text{{\tiny $\delta,\!\varepsilon,\!\nu$}}}(t)\|_{H^2(\Omega)}^2 ds \right) \le C_{14}.
\end{align*}
\end{lemma}
%%%%%
%%%%%
%%%%%
\begin{proof}[Proof.]
For simplicity we set $(v,w):=(v_{\text{{\tiny $\delta,\!\varepsilon,\!\nu$}}},w_{\text{{\tiny $\delta,\!\varepsilon,\!\nu$}}})$, and repeat 
the similar argument to the proof of Lemma \ref{Lemma-2-11}.
We take the inner product in both sides of \eqref{Equation-3-3} in $L^2(\Omega)$ by $w'(t)$. 
From \eqref{Equation-2-3} we obtain the following inequality for a.e. $t \in (0,T^*)$:
\begin{gather*}
\|w'(t)\|_{L^2(\Omega)}^2+\frac{d}{dt} \left(\frac{D_2}{2} \|\nabla w(t)\|_{\bm{L}^2(\Omega)}^2+\frac{c_5}{2} \|w(t)\|_{L^2(\Omega)}^2\right)\\
\le c_6 \int_\Omega |w'(t)|dx+\delta c_7 \int_\Omega |v(t)| |w'(t)|dx.
\end{gather*}
Then, we see from Lemma \ref{Lemma-3-6} with \eqref{Equation-3-7} that the following inequality holds for a.e. $t \in (0,T^*)$:
\begin{gather}
\label{Equation-3-30}
\|w'(t)\|_{L^2(\Omega)}^2+\frac{d}{dt} \left(D_2 \|\nabla w(t)\|_{\bm{L}^2(\Omega)}^2+c_5 \|w(t)\|_{L^2(\Omega)}^2\right)\\
\nonumber
\le 2(c_6)^2|\Omega|+(\delta c_7)^2 \|v(t)\|_{L^2(\Omega)}^2 \le C_{\text{{\tiny $3,\!6,\!1$}}}(\delta), 
\end{gather}
where the constant $C_{\text{{\tiny $3,\!6,\!1$}}}(\delta)>0$ is given by
\begin{equation*}
C_{\text{{\tiny $3,\!6,\!1$}}}(\delta):=2(c_6)^2|\Omega|+4(\delta c_7)^2 C_{13}.
\end{equation*}
Integrating \eqref{Equation-3-30} on $[0,t] \subset [0,T^*]$, we obtain the following inequality for all $t \in [0,T^*]$:
\begin{equation*}
\int_0^t \|w'(s)\|_{L^2(\Omega)}^2ds+\min\{D_2,c_5\} \|w(t)\|_{H^1(\Omega)}^2 \le \max\{D_2,c_5\} \|w_0\|_{H^1(\Omega)}^2+T^*C_{\text{{\tiny $3,\!6,\!1$}}}(\delta),
\end{equation*}
hence,
\begin{equation}\label{Equation-3-31}
\int_0^{T^*} \|w'(t)\|_{L^2(\Omega)}^2dt+\left( \sup_{0 \le t \le T} \|w(t)\|_{H^1(\Omega)} \right)^2 \le C_{\text{{\tiny $3,\!6,\!2$}}}(\delta),
\end{equation}
where the constant $C_{\text{{\tiny $3,\!6,\!2$}}}(\delta)>0$ is given by
\begin{equation*}
C_{\text{{\tiny $3,\!6,\!2$}}}(\delta):=\left(1+\frac{1}{\min\{D_2,\,c_5\}}\right) \left[ \max\{D_2,\,c_5\} \|w_0\|_{H^1(\Omega)}+T^*C_{\text{{\tiny $3,\!6,\!1$}}}(\delta)\right].
\end{equation*}
Moreover, by repeating the similar argument to the derivation of \eqref{Equation-2-94}, the Neumann elliptic regularity (cf. \eqref{Equation-1-8}) with 
\eqref{Equation-3-7} yields the following inequality:
\begin{align}
\label{Equation-3-32}
\int_0^{T^*} \|w(t)\|_{H^2(\Omega)}^2dt &\le (K_5)^2 \int_0^{T^*} \left( \|\Delta_N w(t)\|_{L^2(\Omega)}+\|w(t)\|_{L^2(\Omega)}\right)^2dt\\
\nonumber
&\le C_{\text{{\tiny $3,\!6,\!3$}}}(\delta) \left\{ \int_0^{T^*} \|w'(t)\|_{L^2(\Omega)}^2dt+\left( \sup_{0 \le t \le T^*} \|v(t)\|_{L^2(\Omega)} \right)^2\right.\\
\nonumber
&\,\hspace*{3cm}\left.+\left(\sup_{0 \le t \le T^*} \|w(t)\|_{L^2(\Omega)}\right)^2+1 \right\},
\end{align}
where the constant $C_{\text{{\tiny $3,\!6,\!3$}}}(\delta)>0$ is given by
\begin{equation*}
C_{\text{{\tiny $3,\!6,\!3$}}}(\delta):=\left( \frac{2K_5}{D_2} \right)^2 \left\{ 2 (\delta c_7)^2 T^*+(D_2+c_5)^2T^*+(c_6)^2 |\Omega| T^*+1 \right\}.
\end{equation*}
\indent
Finally, combining the estimates \eqref{Equation-3-31} and \eqref{Equation-3-32} with Lemma \ref{Lemma-3-6}, we obtain the required uniform boundedness.
\end{proof}
%%%%%
%%%%%
%%%%%
Thirdly, we show Lemma \ref{Lemma-3-5}.
As with Lemma \ref{Lemma-3-4}, this lemma implies that the first component $u_{\text{{\tiny $\delta,\!\varepsilon,\!\nu$}}}$ of 
the strong solution to $\mbox{(AP)}{}_{\text{{\tiny $\delta,\!\varepsilon,\!\nu$}}}$ on $[0,T^*]$ is uniformly bounded in 
\begin{equation*}
W^{1,2}(0,T^*\,;\,L^2(\Omega)) \cap L^\infty (0,T^*\,;H^1(\Omega)) \cap L^2(0,T^*\,;H^2(\Omega))
\end{equation*}
with respect to the approximate parameters $\varepsilon \in (0,1)$ and $\nu \in (0,1)$.
%
%%%%%
%%%%%
%%%%%
\begin{lemma}\label{Lemma-3-5}
There exists a constant $C_{15}>0$, which depends on $\delta$, $T^*$, $\|u_0\|_{H^1(\Omega)}$, $\|v_0\|_{H^1(\Omega)}$ and $\|w_0\|_{H^1(\Omega)}$, such that
\begin{align*}
&\sup_{\nu \in (0,1),\,\varepsilon \in (0,1)} \left\{ \left( \sup_{0 \le t \le T^*} \|u_{\text{{\tiny $\delta,\!\varepsilon,\!\nu$}}}(t)\|_{H^1(\Omega)} \right)^2
+\int_0^{T^*} \|u_{\text{{\tiny $\delta,\!\varepsilon,\!\nu$}}}'(t)\|_{L^2 (\Omega)}^2dt\right.\\
&\hspace*{4cm}\left.+\int_0^{T^*} \|u_{\text{{\tiny $\delta,\!\varepsilon,\!\nu$}}}(t)\|_{H^2(\Omega)}^2dt \right\} \le C_{15}.
\end{align*}
\end{lemma}
%%%%%
%%%%%
%%%%%
\begin{proof}[Proof.]
For simplicity we set $(u,v,w):=(u_{\text{{\tiny $\delta,\!\varepsilon,\!\nu$}}},v_{\text{{\tiny $\delta,\!\varepsilon,\!\nu$}}},w_{\text{{\tiny $\delta,\!\varepsilon,\!\nu$}}})$,
and repeat the similar argument to the proof of Lemma \ref{Lemma-2-12}.\\
\indent
Firstly, we take the inner product in both sides of \eqref{Equation-3-1} in $L^2(\Omega)$ by $u(t)$, and repeat the similar argument to the derivation of 
\eqref{Equation-2-96}.
By Lemmas \ref{Lemma-3-6} and \ref{Lemma-3-7}, we obtain 
\begin{gather}
\label{Equation-3-33}
\left( \max_{0 \le t \le T^*} \|u(t)\|_{L^2(\Omega)}\right)^2 \le e^{T^*} \|u_0\|_{L^2(\Omega)}^2
+2(\delta c_1)^2 e^{T^*} \int_0^{T^*} \|v(t)\|_{L^2(\Omega)}^2dt\\
\nonumber
+2(\delta c_2)^2 e^{T^*} \int_0^{T^*} \|w(t)\|_{L^2(\Omega)}^2dt \le C_{\text{{\tiny $3,\!7,\!1$}}}(\delta),
\end{gather}
where the constant $C_{\text{{\tiny $3,\!7,\!1$}}}(\delta)>0$ is given by
\begin{equation*}
C_{\text{{\tiny $3,\!7,\!1$}}}(\delta):=e^{T^*} \left[ \|u_0\|_{L^2(\Omega)}^2+2 T^* \left\{ 2(\delta c_1)^2 C_{13}+(\delta c_2)^2 C_{14} \right\} \right].
\end{equation*}
\indent
Secondly, we take the inner product in both sides of \eqref{Equation-3-1} in $L^2(\Omega)$ by $u'(t)$, and repeat the similar argument of the derivation of 
\eqref{Equation-2-98}.
Using Lemmas \ref{Lemma-3-6} and \ref{Lemma-3-7} again, we obtain 
\begin{equation}\label{Equation-3-34}
\int_0^{T^*} \|u'(t)\|_{L^2(\Omega)}^2dt+\left( \sup_{0 \le t \le T^*} \|\nabla u(t)\|_{\bm{L}^2(\Omega)}\right)^2 \le C_{\text{{\tiny $3,\!7,\!2$}}}(\delta),
\end{equation}
where the constant $C_{\text{{\tiny $3,\!7,\!2$}}}(\delta)>0$ is given by
\begin{equation*}
C_{\text{{\tiny $3,\!7,\!2$}}}(\delta):=\left(1+\frac{1}{D_1}\right) \left[ \|\nabla u_0\|_{\bm{L}^2(\Omega)}^2
+2 T^* \left\{ 2(\delta c_1)^2 C_{13}+(\delta c_2)^2 C_{14} \right\} \right].
\end{equation*}
\indent
Thirdly, we use the Neumann elliptic regularity (cf. \eqref{Equation-1-8}) again, and repeat the similar argument to the derivation of \eqref{Equation-2-99}.
Then, we obtain
\begin{align}
\label{Equation-3-35}
&\,\int_0^{T^*} \|u(t)\|_{H^2(\Omega)}^2dt \le (K_5)^2 \int_0^{T^*} \left( \|\Delta u(t)\|_{L^2(\Omega)}+\|u(t)\|_{L^2(\Omega)}\right)^2dt\\
\nonumber
\le&\,C_{\text{{\tiny $3,\!7,\!3$}}}(\delta) \left\{ \int_0^{T^*} \|u'(t)\|_{L^2(\Omega)}^2dt+\left( \max_{0 \le t \le T^*} \|u(t)\|_{L^2(\Omega)} \right)^2 \right.\\
\nonumber
&\,\hspace*{3cm}\left.+\left( \max_{0 \le t \le T^*} \|v(t)\|_{L^2(\Omega)}\right)^2+\left( \max_{0 \le t \le T^*} \|w(t)\|_{L^2(\Omega)}\right)^2 \right\},
\end{align}
where the constant $C_{\text{{\tiny $3,\!7,\!3$}}}(\delta)>0$ is given by
\begin{equation*}
C_{\text{{\tiny $3,\!7,\!3$}}}(\delta):=\left( \frac{2 K_5}{D_1} \right)^2 \left[ 1+T^* \left\{ (D_1)^2+2 (\delta c_1)^2+(\delta c_2)^2\right\} \right].
\end{equation*}
\indent
Finally, combining the estimates \eqref{Equation-3-33}--\eqref{Equation-3-35} with Lemmas \ref{Lemma-3-6} and \ref{Lemma-3-7}, we obtain the required 
uniform boundedness.
\end{proof}
%%%%%
%%%%%
%%%%%
Finally, we show Lemma \ref{Lemma-3-8} which also gives the uniform boundedness of the fourth component $z_{\text{{\tiny $\delta,\!\varepsilon,\!\nu$}}}$ of 
the strong solution to $\mbox{(AP)}{}_{\text{{\tiny $\delta,\!\varepsilon,\!\nu$}}}$ on $[0,T^*]$.
Although for each $\delta \in (1,\infty)$ and $\varepsilon \in (0,1)$ the family $\{ z_{\text{{\tiny $\delta,\!\varepsilon,\!\nu$}}}\,;\,\nu \in (0,1)\}$ is bounded in
\begin{equation*}
W^{1,2}(0,T^*\,;\,L^2(\Omega)) \cap L^\infty (0,T^*\,;H^1(\Omega)) \cap L^2(0,T^*\,;H^2(\Omega)),
\end{equation*}
we only observe that for each $\delta \in (1,\infty)$ the family $\{ z_{\text{{\tiny $\delta,\!\varepsilon,\!\nu$}}}\,;\,\nu \in (0,1),\,\varepsilon \in (0,1)\}$ is bounded in 
$C([0,T^*]\,;L^2(\Omega)) \cap W^{1,2}(0,T^*\,;\,L^2(\Omega))$.
%
%%%%%
%%%%%
%%%%%
\begin{lemma}\label{Lemma-3-8}
There exists a constant $C_{16}>0$, which depends on $\delta$, $T^*$, $\|v_0\|_{H^1(\Omega)}$ and $\|z_0\|_{H^1(\Omega)}$, such that the following boundedness 
holds for all $\varepsilon \in (0,1)$:
\begin{align*}
&\sup_{\nu \in (0,1)} \left\{ \max_{0 \le t \le T^*} \|z_{\text{{\tiny $\delta,\!\varepsilon,\!\nu$}}}(t)\|_{L^2(\Omega)}^2
+\varepsilon \left( \sup_{0 \le t \le T^*} \|\nabla z_{\text{{\tiny $\delta,\!\varepsilon,\!\nu$}}}(t)\|_{\bm{L}^2(\Omega)}\right)^2\right.\\
&\hspace*{2cm}\left.+\varepsilon^2 \int_0^{T^*} \|z_{\text{{\tiny $\delta,\!\varepsilon,\!\nu$}}}(t)\|_{H^2(\Omega)}^2
+\int_0^{T^*} \|z_{\text{{\tiny $\delta,\!\varepsilon,\!\nu$}}}'(t)\|_{L^2 (\Omega)}^2dt \right\} \le C_{16}.
\end{align*}
\end{lemma}
%%%%%
%%%%%
%%%%%
\begin{proof}[Proof.]
For simplicity we set $(v,z):=(v_{\text{{\tiny $\delta,\!\varepsilon,\!\nu$}}},z_{\text{{\tiny $\delta,\!\varepsilon,\!\nu$}}})$, and repeat the similar argument in the proof 
of Lemma \ref{Lemma-2-8}.\\
\indent
Firstly, we take the inner product in both sides of \eqref{Equation-3-4} in $L^2(\Omega)$ by $z(t)$, and repeat the similar argument to the derivation of 
\eqref{Equation-2-71}.
Then, we obtain 
\begin{equation}\label{Equation-3-36}
\left(\max_{0 \le t \le T^*}\|z(t)\|_{L^2(\Omega)}\right)^2 \le e^{T^*} \left( \|z_0\|_{L^2(\Omega)}^2+\int_0^{T^*} \|v'(t)\|_{L^2(\Omega)}^2dt \right).
\end{equation}
Combining Lemma \ref{Lemma-3-6} with \eqref{Equation-3-36}, it follows
\begin{equation}\label{Equation-3-37}
\left(\max_{0 \le t \le T^*}\|z(t)\|_{L^2(\Omega)}\right)^2 \le C_{\text{{\tiny $3,\!8,\!1$}}}(\delta),
\end{equation}
where the constant $C_{\text{{\tiny $3,\!8,\!1$}}}(\delta)>0$ is given by 
\begin{equation*}
C_{\text{{\tiny $3,\!8,\!1$}}}(\delta):=e^{T^*} \left( \|z_0\|_{L^2(\Omega)}^2+C_{13} \right).
\end{equation*}
\indent
Secondly, we take the inner product in both sides of \eqref{Equation-3-4} in $L^2(\Omega)$ by $z'(t)$, and repeat the similar argument to the derivations of 
\eqref{Equation-2-74} and \eqref{Equation-2-75}.
Using Lemma \ref{Lemma-3-6} again, we obtain 
\begin{equation}\label{Equation-3-38}
\int_0^{T^*} \|z'(t)\|_{L^2(\Omega)}^2dt+\varepsilon \left(\sup_{0 \le t \le T^*} \|\nabla z(t)\|_{\bm{L}^2(\Omega)}\right)^2 \le 
2 \left( \|\nabla z_0\|_{\bm{L}^2(\Omega)}^2+C_{13} \right).
\end{equation}
\indent
Thirdly, we use the Neumann elliptic regularity (cf. \eqref{Equation-1-8}), and repeat the similar argument to the derivation of \eqref{Equation-2-76}.
By Lemma \ref{Lemma-3-6} again, we obtain 
\begin{align}
\label{Equation-3-39}
\varepsilon^2 \int_0^{T^*} \|z(t)\|_{H^2(\Omega)}^2dt &\le 3(K_5)^2 \int_0^{T^*} \|z'(t)\|_{L^2(\Omega)}^2dt\\
\nonumber
&\hspace*{1cm}+3(K_5)^2 \left\{ T^*\left( \max_{0 \le t \le T^*} \|z(t)\|_{L^2(\Omega)} \right)^2+C_{13} \right\}.
\end{align}
\indent
Finally, all the estimates \eqref{Equation-3-37}--\eqref{Equation-3-39} yields the required uniform boundedness.
\end{proof}
%%%%%
%%%%%
%%%%%
%%%%%%%%%%%%%%%%%%%%%%%%%%%%%%%%%%%%%%%%%%%%%%%%%%%%%%%%%%%%%%%%%%%%%%%%%%%%%%%%%%%%%%%%%%%%%%%%%%%%%%%%%%%%%%%%
%%%%%%%%%%%%%%%%%%%%%%%%%%%%%%%%%%%%%%%%%%%%%%%%%%%%%%%%%%%%%%%%%%%%%%%%%%%%%%%%%%%%%%%%%%%%%%%%%%%%%%%%%%%%%%%%
%%%%%%%%%%%%%%%%%%%%%%%%%%%%%%%%%%%%%%%%%%%%%%%%%%%%%%%%%%%%%%%%%%%%%%%%%%%%%%%%%%%%%%%%%%%%%%%%%%%%%%%%%%%%%%%%
\subsection{Uniqueness}\label{Subsection-3-3}
%%%%%%%%%%%%%%%%%%%%%%%%%%%%%%%%%%%%%%%%%%%%%%%%%%%%%%%%%%%%%%%%%%%%%%%%%%%%%%%%%%%%%%%%%%%%%%%%%%%%%%%%%%%%%%%%
%%%%%%%%%%%%%%%%%%%%%%%%%%%%%%%%%%%%%%%%%%%%%%%%%%%%%%%%%%%%%%%%%%%%%%%%%%%%%%%%%%%%%%%%%%%%%%%%%%%%%%%%%%%%%%%%
%%%%%%%%%%%%%%%%%%%%%%%%%%%%%%%%%%%%%%%%%%%%%%%%%%%%%%%%%%%%%%%%%%%%%%%%%%%%%%%%%%%%%%%%%%%%%%%%%%%%%%%%%%%%%%%%
The purpose of this section is to show Proposition \ref{Proposition-3}, which guarantees the uniqueness of local-in-time solutions to
$\text{(AP)}{}_{\text{{\tiny $\delta,\!\varepsilon,\!\nu$}}}$.
Just as the continuous embedding $H^2(\Omega) \hookrightarrow L^\infty (\Omega)$ (cf. \eqref{Equation-1-7}) played a crucial role in the proof of Lemma \ref{Lemma-2-13},
it also plays a crucial role in the proof of Proposition \ref{Proposition-3}.
Actually, the uniform boundedness in Lemmas \ref{Lemma-3-6}--\ref{Lemma-3-8} enables us to repeat the similar argument to the proof of Lemma \ref{Lemma-2-13}.
%%%%%
%%%%%
%%%%%
\begin{proposition}\label{Proposition-3}
Fix approximation parameters $\delta\in(1,\infty)$, $\nu\in(0,1)$, and $\varepsilon\in(0,1)$.
For any finite time $T>0$, the strong solution to the initial-boundary value problem $\text{(AP)}{}_{\text{{\tiny $\delta,\!\varepsilon,\!\nu$}}}$ on $[0,T]$,
if it exists, is unique.
\end{proposition}
%%%%%
%%%%%
%%%%%
\begin{proof}[Proof.]
Let $(u_1,v_1,w_1,z_1)$ and $(u_2,v_2,w_2,z_2)$ be two strong solutions to $\text{(AP)}{}_{\text{{\tiny $\delta,\!\varepsilon,\!\nu$}}}$ on $[0,T]$.
For simplicity, we set $(U,V,W,Z):=(u_2-u_1,v_2-v_1,w_2-w_1,z_2-z_1)$.\\
\indent
Firstly, we consider the function $U:[0,T] \rightarrow L^2(\Omega)$, which satisfies the following evolution equation: 
\begin{gather*}
U'(t)-D_1 \Delta_N U(t)=c_1 \beta_{\text{{\tiny $\delta$}}}(u_2(t))V(t)+c_1 v_1(t)\left\{ \beta_{\text{{\tiny $\delta$}}}(u_2(t))-\beta_{\text{{\tiny $\delta$}}}(u_1(t)) \right\}\\
-c_2 \beta_{\text{{\tiny $\delta$}}}(u_2(t)) W(t)-c_2 w_1(t)\left\{ \beta_{\text{{\tiny $\delta$}}}(u_2(t))-\beta_{\text{{\tiny $\delta$}}}(u_1(t)) \right\}\\
\text{in} \quad L^2(\Omega),\quad \text{a.e.}~t \in (0,T),
\end{gather*}
with the initial condition 
\begin{equation*}
U(0)=0 \quad \text{in} \quad L^2(\Omega).
\end{equation*}
We repeat the similar argument to the derivation of \eqref{Equation-2-117}, in which the triplets $(\tilde{u}_m,\tilde{v}_m,\tilde{w}_m)$ and 
$(\tilde{u},\tilde{v},\tilde{w})$ are replaced by $(u_2,v_2,w_2)$ and $(u_1,v_1,w_1)$, respectively.
Then, we obtain the following inequality for a.e. $t \in (0,T)$: 
\begin{gather}
\label{Equation-3-40}
\frac{d}{dt} \|U(t)\|_{L^2(\Omega)}^2+D_1\|\nabla U(t)\|_{\bm{L}^2(\Omega)}^2\\
\nonumber
\le \ell_1(t) \left( \|U(t)\|_{L^2(\Omega)}^2+\|V(t)\|_{L^2(\Omega)}^2+\|W(t)\|_{L^2(\Omega)}^2 \right),
\end{gather}
where the positive function $\ell_1:[0,T] \rightarrow \mathbb{R}$ is given by
\begin{gather*}
\ell_1(t):=\frac{4(c_1)^2(K_6)^4}{\nu D_1} \Phi_\nu (v_1(t))+\frac{2(c_2)^2 (K_6)^4}{D_1}\|w_1(t)\|_{H^1(\Omega)}^2+\delta (c_1+c_2)+D_1,\\
\forall t \in [0,T].
\end{gather*}
\indent
Secondly, we consider the function $V:[0,T] \rightarrow L^2(\Omega)$, which satisfies the following evolution equation:
\begin{gather}
\label{Equation-3-41}
M_{\text{{\tiny $\delta,\!z_2(t)$}}}v_2'(t)-M_{\text{{\tiny $\delta,\!z_1(t)$}}}v_1'(t)-\nu \Delta_N V(t)+V(t)=\sum_{j=1}^6 H_j(t)\\
\nonumber
\text{in} \quad L^2(\Omega),\quad\text{a.e.}~t \in (0,T),
\end{gather}
with the initial condition 
\begin{equation*}
V(0)=0 \quad \text{in} \quad L^2(\Omega),
\end{equation*}
where for all $t \in [0,T]$ the functions $H_j (t) \in L^2(\Omega)~(j=1,2,3,4,5,6)$ are given as follows:
\begin{equation*}
\left\{
\begin{array}{l}
H_1(t):=c_3 \left\{ M_{\text{{\tiny $\delta,\!z_2(t)$}}}-M_{\text{{\tiny $\delta,\!z_1(t)$}}} \right\} v_2(t) \beta_{\text{{\tiny $\delta$}}}(w_2(t)),\\[0.2cm]
H_2(t):=c_3 M_{\text{{\tiny $\delta,\!z_1(t)$}}} V(t) \beta_{\text{{\tiny $\delta$}}}(w_2(t)),\\[0.2cm]
H_3(t):=c_3M_{\text{{\tiny $\delta,\!z_1(t)$}}} v_1(t) \left\{ \beta_{\text{{\tiny $\delta$}}}(w_2(t))-\beta_{\text{{\tiny $\delta$}}}(w_1(t)) \right\},\\[0.2cm]
H_4(t):=-c_4 \left\{ M_{\text{{\tiny $\delta,\!z_2(t)$}}}-M_{\text{{\tiny $\delta,\!z_1(t)$}}} \right\} v_2(t) \beta_{\text{{\tiny $\delta$}}}(u_2(t)),\\[0.2cm]
H_5(t):=-c_4 M_{\text{{\tiny $\delta,\!z_1(t)$}}} V(t) \beta_{\text{{\tiny $\delta$}}}(u_2(t)),\\[0.2cm]
H_6(t):=-c_4M_{\text{{\tiny $\delta,\!z_1(t)$}}} v_1(t) \left\{ \beta_{\text{{\tiny $\delta$}}}(u_2(t))-\beta_{\text{{\tiny $\delta$}}}(u_1(t)) \right\}.
\end{array}
\right.
\end{equation*}
We take the inner product in both sides of \eqref{Equation-3-41} in $L^2(\Omega)$ by $V'(t)$.
As one of the benefits of introducing the twisted spaces $\{L^2(\nu,\bar{z})\,;\,\bar{z} \in L^2(\Omega)\}$, we obtain the following estimates by using 
the Cauchy-Schwarz inequality and the Yong inequality repeatedly:
\begin{enumerate}\leftskip10pt
\item[(i)] From \eqref{Equation-2-15} and \eqref{Equation-2-20} we obtain the following inequality for a.e. $t \in (0,T)$:
\begin{align*}
&\,(M_{\text{{\tiny $\delta,\!z_2(t)$}}}v_2'(t)-M_{\text{{\tiny $\delta,\!z_1(t)$}}}v_1'(t),V'(t))_{L^2(\Omega)}\\
=&\,\int_\Omega \left[ \left\{ \mathcal{K}_{\text{{\tiny $\delta$}}}z_2(t) \right\}^{-1}v_2'(t)
-\left\{ \mathcal{K}_{\text{{\tiny $\delta$}}}z_1(t) \right\}^{-1}v_1'(t)\right]V'(t)dx\\
\ge&\,\int_\Omega \{\mathcal{K}_{\text{{\tiny $\delta$}}}z_2(t)\}^{-1} |V'(t)|^2dx-\int_\Omega \left|\{\mathcal{K}_{\text{{\tiny $\delta$}}}z_2(t)\}^{-1}
-\{\mathcal{K}_{\text{{\tiny $\delta$}}}z_1(t)\}^{-1}\right||v_1'(t)| |V'(t)|dx\\
\ge&\,\frac{1}{\delta K_2 |\Omega|^{\frac{1}{2}}} \|V'(t)\|_{L^2(\Omega)}^2-K_2 \left(\frac{\delta}{K_1}\right)^2  \|v_1'(t)\|_{L^2(\Omega)} 
\|Z(t)\|_{L^2 (\Omega)} \|V'(t)\|_{L^2(\Omega)}\\
\ge&\,\left( \frac{1}{\delta K_2 |\Omega|^{\frac{1}{2}}}-\mu\right) \|V'(t)\|_{L^2(\Omega)}^2-\frac{\delta^4 (K_2)^2}{4\mu (K_1)^4}
\|v_1'(t)\|_{L^2(\Omega)}^2 \|Z(t)\|_{L^2 (\Omega)}^2.
\end{align*}
\item[(ii)] From \eqref{Equation-2-3} and \eqref{Equation-2-20} we obtain the following inequality for a.e. $t \in (0,T)$:
\begin{align*}
(H_1(t),V'(t))_{L^2(\Omega)}&
=c_3 \left( \left\{ M_{\text{{\tiny $\delta,\!z_2(t)$}}}-M_{\text{{\tiny $\delta,\!z_1(t)$}}} \right\} v_2(t) \beta_{\text{{\tiny $\delta$}}}(w_2(t)),V'(t)\right)_{L^2(\Omega)}\\
&\le \delta c_3 \int_\Omega \left| \{ \mathcal{K}_{\text{{\tiny $\delta$}}}z_2(t)\right\}^{-1}-\left\{ \mathcal{K}_{\text{{\tiny $\delta$}}}z_1(t) \}^{-1} \right| |v_2(t)| |V'(t)|dx\\
&\le \delta c_3 K_2 \left(\frac{\delta}{K_1}\right)^2 \|v_2(t)\|_{L^2(\Omega)} \|Z(t)\|_{L^2(\Omega)} \|V'(t)\|_{L^2(\Omega)}\\
&\le \mu \|V'(t)\|_{L^2(\Omega)}^2+\frac{\delta^6 (c_3 K_2)^2}{4\mu (K_1)^4} \|v_2(t)\|_{L^2(\Omega)}^2 \|Z(t)\|_{L^2(\Omega)}^2.
\end{align*}
\item[(iii)] From \eqref{Equation-2-3} and \eqref{Equation-2-15} we obtain the following inequality for a.e. $t \in (0,T)$:
\begin{align*}
(H_2(t),V'(t))_{L^2(\Omega)}&=c_3\int_\Omega \{\mathcal{K}_{\text{{\tiny $\delta$}}} z_1(t)\}^{-1} V(t) \beta_{\text{{\tiny $\delta$}}} (w_2(t)) V'(t)dx\\
&\le \frac{\delta^2 c_3}{K_1} \int_\Omega |V(t)| |V'(t)|dx \le \frac{\delta^2 c_3}{K_1} \|V(t)\|_{L^2(\Omega)} \|V'(t)\|_{L^2(\Omega)}\\
&\le \mu \|V'(t)\|_{L^2(\Omega)}^2+\frac{\delta^4 (c_3)^2}{4\mu (K_1)^2} \|V(t)\|_{L^2(\Omega)}^2.
\end{align*}
\item[(iv)] From \eqref{Equation-1-7}, \eqref{Equation-2-4} and \eqref{Equation-2-15} we obtain the following inequality for a.e. $t \in (0,T)$:
\begin{align*}
(H_3(t),V'(t))_{L^2(\Omega)}&=c_3\int_\Omega \{\mathcal{K}_{\text{{\tiny $\delta$}}} z_1(t)\}^{-1} v_1(t) 
\{ \beta_{\text{{\tiny $\delta$}}} (w_2(t))-\beta_{\text{{\tiny $\delta$}}} (w_1(t)) \} V'(t)dx\\
&\le \frac{\delta^2 c_3}{K_1} \int_\Omega |v_1(t)| |W(t)| |V'(t)|dx\\
&\le \frac{\delta^2 c_3}{K_1} \|v_1(t)\|_{L^\infty (\Omega)} \|W(t)\|_{L^2(\Omega)} \|V'(t)\|_{L^2(\Omega)}\\
&\le \mu \|V'(t)\|_{L^2(\Omega)}^2+\frac{\delta^4 (c_3 K_4)^2}{4\mu (K_1)^2} \|v_1(t)\|_{H^2(\Omega)}^2 \|W(t)\|_{L^2(\Omega)}^2.
\end{align*}
\item[(v)] We repeat the similar argument to the derivation of the estimate for $H_1$, and obtain the following inequality for a.e. $t \in (0,T)$:
\begin{align*}
(H_4(t),V'(t))_{L^2(\Omega)}&=c_4 \left( \left\{ M_{\text{{\tiny $\delta,\!z_2(t)$}}}-M_{\text{{\tiny $\delta,\!z_1(t)$}}} \right\} v_2(t) 
\beta_{\text{{\tiny $\delta$}}}(w_2(t)),V'(t)\right)_{L^2(\Omega)}\\
&\le \mu \|V'(t)\|_{L^2(\Omega)}^2+\frac{\delta^6 (c_4 K_2)^2}{4\mu (K_1)^4} \|v_2(t)\|_{L^2(\Omega)}^2 \|Z(t)\|_{L^2(\Omega)}^2.
\end{align*}
\item[(vi)] We repeat the similar argument to the derivation of the estimate for $H_2$, and obtain the following inequality for a.e. $t \in (0,T)$:
\begin{align*}
(H_5(t),V'(t))_{L^2(\Omega)}&=c_4\int_\Omega \{\mathcal{K}_{\text{{\tiny $\delta$}}} z_1(t)\}^{-1} V(t) \beta_{\text{{\tiny $\delta$}}} (u_2(t)) V'(t)dx\\
&\le \mu \|V'(t)\|_{L^2(\Omega)}^2+\frac{\delta^4 (c_4)^2}{4\mu (K_1)^2} \|V(t)\|_{L^2(\Omega)}^2.
\end{align*}
\item[(vii)] We repeat the similar argument to the derivation of the estimate for $H_3$, and obtain the following inequality for a.e. $t \in (0,T)$:
\begin{align*}
(H_6(t),V'(t))_{L^2(\Omega)}&=c_4\int_\Omega \{\mathcal{K}_{\text{{\tiny $\delta$}}} z_1(t)\}^{-1} v_1(t) 
\{ \beta_{\text{{\tiny $\delta$}}} (u_2(t))-\beta_{\text{{\tiny $\delta$}}} (u_1(t)) \} V'(t)dx\\
&\le \mu \|V'(t)\|_{L^2(\Omega)}^2+\frac{\delta^4 (c_4 K_4)^2}{4\mu (K_1)^2} \|v_1(t)\|_{H^2(\Omega)}^2 \|U(t)\|_{L^2(\Omega)}^2,
\end{align*}
\end{enumerate}
Adding all estimates (i)--(vii), we obtain the following inequality for a.e. $t \in (0,T)$:
\begin{gather}
\label{Equation-3-42}
\left(\frac{1}{\delta K_1 |\Omega|^{\frac{1}{2}}}-7\mu\right) \|V'(t)\|_{L^2(\Omega)}^2+\frac{d}{dt} \Phi_\nu (V(t))\\
\nonumber
\le \frac{\ell_2 (t)}{4\mu} \left( \|U(t)\|_{L^2(\Omega)}^2+\|V(t)\|_{L^2(\Omega)}^2+\|W(t)\|_{L^2(\Omega)}^2+\|Z(t)\|_{L^2(\Omega)}^2 \right),
\end{gather}
where the positive function $\ell: [0,T] \rightarrow \mathbb{R}$ is given by 
\begin{align*}
\ell_2 (t)&:=\frac{\delta^4 (K_2)^2}{(K_1)^4} \|v_1'(t)\|_{L^2(\Omega)}^2+\frac{\delta^6 (K_2)^2\{(c_3)^2+(c_4)^2\}}{(K_1)^4}  \|v_2(t)\|_{L^2(\Omega)}^2\\
&\hspace*{1cm}+\frac{\delta^4 \{(c_3)^2+(c_4)^2\}}{(K_1)^2}+\frac{\delta^4 (K_4)^2\{(c_3)^2+(c_4)^2\}}{(K_1)^2} \|v_1(t)\|_{H^2(\Omega)}^2,
\quad \forall t \in [0,T].
\end{align*}
\indent
Thirdly, we consider the function $W$, which satisfies the evolution equation:
\begin{gather*}
W'(t)-D_2 \Delta_N W(t)+c_5 W(t)\\
=-c_7 V(t) \beta_{\text{{\tiny $\delta$}}} (w_2(t))-c_7 v_1(t) \left\{ \beta_{\text{{\tiny $\delta$}}}(w_2(t))-\beta_{\text{{\tiny $\delta$}}}(w_1(t)) \right\}\\
\text{in} \quad L^2(\Omega),\quad \text{a.e.}~t \in (0,T),
\end{gather*}
with the initial condition 
\begin{equation*}
W(0)=0 \quad \text{in} \quad L^2(\Omega).
\end{equation*}
We repeat the similar argument to the derivation of \eqref{Equation-2-108} in Lemma \ref{Lemma-2-13}, in which the pairs $(\tilde{v}_m,\tilde{w}_m)$ and 
$(\tilde{v},\tilde{w})$ are replaced by $(v_2,w_2)$ and $(v_1,w_1)$, respectively.
Then, we obtain
\begin{align}
\label{Equation-3-43}
&\frac{d}{dt} \|W(t)\|_{L^2(\Omega)}^2+\min\{D_2,c_5\} \|W(t)\|_{H^1(\Omega)}^2\\
\nonumber
&\hspace*{1cm}\le \ell_3 (t) \left( \|W(t)\|_{L^2(\Omega)}^2+\delta c_7 \|V(t)\|_{L^2(\Omega)}^2 \right), \quad \text{a.e.}~t \in (0,T),
\end{align}
where the positive function $\ell_3:[0,T] \rightarrow \mathbb{R}$ is given by
\begin{equation*}
\ell_3 (t):=\frac{2 (c_7)^2 (K_6)^4}{\nu \min\{D_2,c_5\}} \Phi_\nu (v_1(t))+\delta c_7, \quad \forall t \in [0,T].
\end{equation*}
\indent
Fourthly, we consider the function $Z$.
Then, the function $Z$ is a strong solution to the following Cauchy problem on $[0,T]$:
\begin{gather}
\label{Equation-3-44}
Z'(t)-\varepsilon \Delta_N Z(t)=V'(t) \quad \text{in} \quad L^2(\Omega), \quad \text{a.e.}~t \in (0,T),\\
\nonumber
Z(0)=0 \quad \text{in} \quad L^2(\Omega).
\end{gather}
We take the inner product in both sides of \eqref{Equation-3-44} in $L^2(\Omega)$ by $Z(t)$.
For any $\mu >0$ we derive the following inequality for a.e. $t \in (0,T)$:
\begin{equation}\label{Equation-3-45}
\frac{d}{dt} \|Z(t)\|_{L^2(\Omega)}^2+2\varepsilon \|\nabla Z(t)\|_{\bm{L}^2(\Omega)}^2 \le \mu \|V'(t)\|_{L^2(\Omega)}^2+\frac{1}{\mu} \|Z(t)\|_{L^2(\Omega)}^2.
\end{equation}
\indent
Finally, we take out and fix a number $\mu:=\mu_* >0$ satisfying 
\begin{equation*}
\frac{1}{\delta K_1 |\Omega|^{\frac{1}{2}}}-8\mu_*>0.
\end{equation*}
From \eqref{Equation-3-40}, \eqref{Equation-3-42}, \eqref{Equation-4-43} and \eqref{Equation-3-45} we derive the following inequality for a.e. $t \in (0,T)$:
\begin{align}
\label{Equation-3-46}
&\,\frac{d}{dt} \left( \|U(t)\|_{L^2(\Omega)}^2+\Phi_\nu (V(t))+\|W(t)\|_{L^2(\Omega)}^2+\|Z(t)\|_{L^2(\Omega)}^2 \right)\\
\nonumber
&\,\hspace*{1cm}+C_{\text{{\tiny $3,\!1$}}} \left(\|\nabla U(t)\|_{\bm{L}^2(\Omega)}^2+\|V'(t)\|_{L^2(\Omega)}^2+\|W(t)\|_{H^1(\Omega)}^2
+\|\nabla Z(t)\|_{\bm{L}^2(\Omega)}^2\right)\\
\nonumber
\le&\,\ell (t) \left( \|U(t)\|_{L^2(\Omega)}^2+\Phi_\nu (V(t))+\|W(t)\|_{L^2(\Omega)}^2+\|Z(t)\|_{L^2(\Omega)}^2 \right),
\end{align}
where the constant $C_{\text{{\tiny $3,\!1$}}}>0$ and the positive function $\ell : [0,T] \rightarrow \mathbb{R}$ are given by
\begin{equation*}
C_{\text{{\tiny $3,\!1$}}}:=\min \left\{ D_1,\,\frac{1}{\delta K_1 |\Omega|^{\frac{1}{2}}}-8\mu^*,\,D_2,\,c_5,\,2\varepsilon \right\}, \quad 
\ell :=\ell_1+\frac{\ell_2}{4\mu^*}+\ell_3+\frac{1}{\mu^*}.
\end{equation*}
Since the regularities of $(u,v,,w,z)$ given in (1) of Proposition \ref{Proposition-2} yields $\ell \in L^1(0,T)$, 
we apply the Gronwall lemma to \eqref{Equation-3-46} and obtain the following equality by using the initial condition $(U(0),V(0),W(0),Z(0))=(0,0,0,0)$ in 
$(L^2(\Omega))^4$: 
\begin{equation*}
\|U(t)\|_{L^2(\Omega)}^2+\Phi_\nu (V(t))+\|Z(t)\|_{L^2(\Omega)}^2+\|W(t)\|_{L^2(\Omega)}^2=0,\quad \forall t \in [0,T],
\end{equation*}
which implies that strong solutions to $\text{(AP)}{}_{\text{{\tiny $\delta,\!\varepsilon,\!\nu$}}}$ on $[0,T]$ is unique.
\end{proof}
%%%%%
%%%%%
%%%%%
%%%%%%%%%%%%%%%%%%%%%%%%%%%%%%%%%%%%%%%%%%%%%%%%%%%%%%%%%%%%%%%%%%%%%%%%%%%%%%%%%%%%%%%%%%%%%%%%%%%%%%%%%%%%%%%%
%%%%%%%%%%%%%%%%%%%%%%%%%%%%%%%%%%%%%%%%%%%%%%%%%%%%%%%%%%%%%%%%%%%%%%%%%%%%%%%%%%%%%%%%%%%%%%%%%%%%%%%%%%%%%%%%
%%%%%%%%%%%%%%%%%%%%%%%%%%%%%%%%%%%%%%%%%%%%%%%%%%%%%%%%%%%%%%%%%%%%%%%%%%%%%%%%%%%%%%%%%%%%%%%%%%%%%%%%%%%%%%%%
\subsection{Approximate strong global-in-time solutions}\label{Subsection-3-4}
%%%%%%%%%%%%%%%%%%%%%%%%%%%%%%%%%%%%%%%%%%%%%%%%%%%%%%%%%%%%%%%%%%%%%%%%%%%%%%%%%%%%%%%%%%%%%%%%%%%%%%%%%%%%%%%%
%%%%%%%%%%%%%%%%%%%%%%%%%%%%%%%%%%%%%%%%%%%%%%%%%%%%%%%%%%%%%%%%%%%%%%%%%%%%%%%%%%%%%%%%%%%%%%%%%%%%%%%%%%%%%%%%
%%%%%%%%%%%%%%%%%%%%%%%%%%%%%%%%%%%%%%%%%%%%%%%%%%%%%%%%%%%%%%%%%%%%%%%%%%%%%%%%%%%%%%%%%%%%%%%%%%%%%%%%%%%%%%%%
In this subsection, we show Proposition \ref{Proposition-4}, for each $\delta \in (1,\infty)$, $\varepsilon \in (0,1)$ and $\nu \in (0,1)$ which guarantees 
the existence and uniqueness of global-in-time to the approximate initial-boundary value problem $\text{(AP)}{}_{\text{{\tiny $\delta,\!\varepsilon,\!\nu$}}}$.
%%%%%
%%%%%
%%%%%
\begin{proposition}\label{Proposition-4}
For each fixed approximation parameters $\delta \in (1,\infty)$, $\varepsilon \in (0,1)$ and $\nu \in (0,1)$ the approximate initial-boundary value problem 
$\text{(AP)}{}_{\text{{\tiny $\delta,\!\varepsilon,\!\nu$}}}$ admits a unique global-in-time solution 
$(u_{\text{{\tiny $\delta,\!\varepsilon,\!\nu$}}},v_{\text{{\tiny $\delta,\!\varepsilon,\!\nu$}}},w_{\text{{\tiny $\delta,\!\varepsilon,\!\nu$}}},
z_{\text{{\tiny $\delta,\!\varepsilon,\!\nu$}}})$ satisfying all properties (1)--(3) in Proposition \ref{Proposition-2} for all time $T>0$ in which the local existence time 
$T_{\text{{\tiny $\delta$}}}>0$ is replaced by any time $T>0$.
\end{proposition}
%%%%%
%%%%%
%%%%%
\begin{proof}[Proof.]
We see from the definition of global-in-time solutions to $\text{(AP)}{}_{\text{{\tiny $\delta,\!\varepsilon,\!\nu$}}}$ and Proposition \ref{Proposition-3} 
that for each fixed finite time $T>0$ the quadruple $(u_{\text{{\tiny $\delta,\!\varepsilon,\!\nu$}}},v_{\text{{\tiny $\delta,\!\varepsilon,\!\nu$}}},
w_{\text{{\tiny $\delta,\!\varepsilon,\!\nu$}}},z_{\text{{\tiny $\delta,\!\varepsilon,\!\nu$}}})$ is the unique solution to 
$\text{(AP)}{}_{\text{{\tiny $\delta,\!\varepsilon,\!\nu$}}}$ on $[0,T]$.
Since $T>0$ is arbitrary, it follows that the global-in-time solution to $\text{(AP)}{}_{\text{{\tiny $\delta,\!\varepsilon,\!\nu$}}}$ is unique.\\
\indent
In the following argument, we show the existence of strong global-in-time solutions to $\text{(AP)}{}_{\text{{\tiny $\delta,\!\varepsilon,\!\nu$}}}$ by contradiction.
We define the maximal existence time $T_{\text{{\tiny $\delta,\!\varepsilon,\!\nu$}}}^*$ by
\begin{equation}\label{Equation-3-47}
T_{\text{{\tiny $\delta,\!\varepsilon,\!\nu$}}}^*:=\sup \left\{ T>0\,;\,\text{$\text{(AP)}{}_{\text{{\tiny $\delta,\!\varepsilon,\!\nu$}}}$ has a solution on $[0,T]$} \right\}.
\end{equation}
Assume by contradiction that $T_{\text{{\tiny $\delta,\!\varepsilon,\!\nu$}}}^*<\infty$.
By Proposition \ref{Proposition-2}, we have $T_{\text{{\tiny $\delta,\!\varepsilon,\!\nu$}}}^*>0$.
Then, we take a strictly increasing sequence $\{T_m\}_{m \in \mathbb{N}}$ satisfying
\begin{equation}\label{Equation-3-48}
T_m \uparrow T_{\text{{\tiny $\delta,\!\varepsilon,\!\nu$}}}^* \quad \text{as} \quad m \to \infty,
\end{equation}
and let $(u_m,v_m,w_m,z_m)$ be the unique solution to $\text{(AP)}{}_{\text{{\tiny $\delta,\!\varepsilon,\!\nu$}}}$ on $[0,T_m]$ for each $m \in \mathbb{N}$.
By the uniqueness of strong local-in-time solutions, for all $m,n \in \mathbb{N}$ with $m < n$ we have 
\begin{gather}
\label{Equation-3-49}
(u_m(t),v_m(t),w_m(t),z_m(t))=(u_n(t),v_n(t),w_n(t),z_n(t))\\
\nonumber
\text{in} \quad (L^2(\Omega))^4,\quad \forall t \in [0,T_m].
\end{gather}
By the continuity of $(u_m,v_m,w_m,z_m)$ in $(C([0,T_m]\,;L^2(\Omega)))^4$, the terminal values 
\begin{equation*}
(u_m(T_m),v_m(T_m),w_m(T_m),z_m(T_m))~~\text{are well-defined as elements of}~~(L^2(\Omega))^4.
\end{equation*}
Moreover, by virtue of the uniform estimates obtained in Proposition \ref{Proposition-3} the sequence 
\begin{equation*}
\{(u_m(T_m),v_m(T_m),w_m(T_m),z_m(T_m))\}_{m \in \mathbb{N}}~~\text{is bounded in}~~(H^1(\Omega))^4.
\end{equation*}
Hence, there exists a subsequence of $\{(u_m(T_m),v_m(T_m),w_m(T_m),z_m(T_m))\}_{m \in \mathbb{N}}$, still denoted by the same notation, and a quadruple 
$(u_*,v_*,w_*,z_*) \in (H^1(\Omega))^4$ such that the following convergence holds as $m \to \infty$:
\begin{gather}\label{Equation-3-50}
(u_m(T_m),v_m(T_m),w_m(T_m),z_m(T_m)) \longrightarrow (u_*,v_*,w_*,z_*)\\
\nonumber
\text{in} \quad (L^2(\Omega))^4 \quad \text{and} \quad \text{weakly in} \quad (H^1(\Omega))^4.
\end{gather}
Using the quadruple $(u_*,v_*,w_*,z_*)$ as an initial datum, we consider the following Cauchy problem $\text{(AP)}{}_{\text{{\tiny $\delta,\!\varepsilon,\!\nu$}}}^*$:
\begin{gather}
\nonumber
u'(t)-D_1 \Delta_N u(t)=\beta_{\text{{\tiny $\delta$}}}(u(t)) \left\{ c_1 v(t)-c_2 w(t) \right\} 
\quad \text{in} \quad L^2(\Omega),\quad \text{a.e.}~t>0,\\
\nonumber
M_{\text{{\tiny$\delta,\!z(t)$}}}v'(t)-\nu \Delta_N v(t)+v(t)=M_{\text{{\tiny $\delta,\!z(t)$}}} v(t)
\left\{ c_3 \beta_{\text{{\tiny $\delta$}}} (w(t))-c_4 \beta_{\text{{\tiny $\delta$}}} (u(t)) \right\}\\
\nonumber
\text{in} \quad L^2(\Omega),\quad \text{a.e.}~t>0,\\
\nonumber
w'(t)-D_2 \Delta_N w(t)+c_5 w(t)=c_6 -c_7 v(t) \beta_{\text{{\tiny $\delta$}}} (w(t)) \quad \text{in} \quad L^2(\Omega), \quad \text{a.e.}~t>0,\\
\nonumber
z'(t)-\varepsilon \Delta_N z(t)=v'(t) \quad \text{in} \quad L^2(\Omega),\quad \text{a.e.}~t>0,\\
\label{Equation-3-51}
(u(0),v(0),z(0),w(0))=(u_*,v_*,z_*,w_*) \quad \text{in} \quad L^2(\Omega).
\end{gather}
Since all arguments as in Section \ref{Section-2} and Subsections \ref{Subsection-3-1}--\ref{Subsection-3-3} are available to 
$\text{(AP)}{}_{\text{{\tiny $\delta,\!\varepsilon,\!\nu$}}}^*$, we deduce that there exists a finite time $T_*>0$ such that 
$\text{(AP)}{}_{\text{{\tiny $\delta,\!\varepsilon,\!\nu$}}}^*$ admits a unique solution $(\hat{u},\hat{v},\hat{w},\hat{z})$ on $[0,T_*]$.
Using the sequence $\{(u_m,v_m,w_m,z_m)\}_{m \in \mathbb{N}}$ and $(\hat{u},\hat{v},\hat{w},\hat{z})$, we consider a quadruple $(\bar{u},\bar{v},\bar{w},\bar{z})$ 
as a function from $[0,T_{\text{{\tiny $\delta,\!\varepsilon,\!\nu$}}}^*+T_*]$ into $(L^2 (\Omega))^4$ defined by 
\begin{gather*}
(\bar{u}(t),\bar{v}(t),\bar{w}(t),\bar{z}(t)):=(u_m(t),v_m(t),w_m(t),z_m(t)),\\
\forall m \in \mathbb{N}~~\text{such that}~~t \in [0,T_m],\\
(\bar{u}(t),\bar{v}(t),\bar{w}(t),\bar{z}(t)):=(\hat{u}(t-T_{\text{{\tiny $\delta,\!\varepsilon,\!\nu$}}}^*),\hat{v}(t-T_{\text{{\tiny $\delta,\!\varepsilon,\!\nu$}}}^*),
\hat{w}(t-T_{\text{{\tiny $\delta,\!\varepsilon,\!\nu$}}}^*),\hat{z}(t-T_{\text{{\tiny $\delta,\!\varepsilon,\!\nu$}}}^*))\\
\text{if} \quad t \in [T_{\text{{\tiny $\delta,\!\varepsilon,\!\nu$}}}^*,T_{\text{{\tiny $\delta,\!\varepsilon,\!\nu$}}}^*+T_*].
\end{gather*}
This definition is independent of the choice of $m \in \mathbb{N}$ due to \eqref{Equation-3-49}, and we have 
\begin{equation*}
(\bar{u},\bar{v},\bar{w},\bar{z}) \in \left(C([0,T_{\text{{\tiny $\delta,\!\varepsilon,\!\nu$}}}^*+T_*]\,;L^2(\Omega))\right)^4
\end{equation*}
since from \eqref{Equation-3-48}--\eqref{Equation-3-50} and \eqref{Equation-3-51} we have
\begin{gather*}
(\bar{u}(t),\bar{u}(t),\bar{w}(t),\bar{z}(t)) \longrightarrow (u^*,v^*,w^*,z^*)=(\bar{u}(T_{\text{{\tiny $\delta,\!\varepsilon,\!\nu$}}}^*),
\bar{v}(T_{\text{{\tiny $\delta,\!\varepsilon,\!\nu$}}}^*),\bar{w}(T_{\text{{\tiny $\delta,\!\varepsilon,\!\nu$}}}^*),
\bar{z}(T_{\text{{\tiny $\delta,\!\varepsilon,\!\nu$}}}^*))\\
\text{in} \quad (L^2(\Omega))^4 \quad \text{as} \quad t \to T_{\text{{\tiny $\delta,\!\varepsilon,\!\nu$}}}^*.
\end{gather*}
Then, the quadruple $(\bar{u},\bar{v},\bar{z},\bar{w})$ is a solution to $\text{(AP)}{}_{\text{{\tiny $\delta,\!\varepsilon,\!\nu$}}}$ on 
$[0,T_{\text{{\tiny $\delta,\!\varepsilon,\!\nu$}}}^*+T_*]$.
This contradicts the definition of $T_{\text{{\tiny $\delta,\!\varepsilon,\!\nu$}}}^*$ given in \eqref{Equation-3-47}.
Hence, $T_{\text{{\tiny $\delta,\!\varepsilon,\!\nu$}}}^*=\infty$ must hold.
\end{proof}
%%%%%
%%%%%
%%%%%
%%%%%%%%%%%%%%%%%%%%%%%%%%%%%%%%%%%%%%%%%%%%%%%%%%%%%%%%%%%%%%%%%%%%%%%%%%%%%%%%%%%%%%%%%%%%%%%%%%%%%%%%%%%%%%%%
%%%%%%%%%%%%%%%%%%%%%%%%%%%%%%%%%%%%%%%%%%%%%%%%%%%%%%%%%%%%%%%%%%%%%%%%%%%%%%%%%%%%%%%%%%%%%%%%%%%%%%%%%%%%%%%%
%%%%%%%%%%%%%%%%%%%%%%%%%%%%%%%%%%%%%%%%%%%%%%%%%%%%%%%%%%%%%%%%%%%%%%%%%%%%%%%%%%%%%%%%%%%%%%%%%%%%%%%%%%%%%%%%
\subsection{Nonnegativity}\label{Subsection-3-5}
%%%%%%%%%%%%%%%%%%%%%%%%%%%%%%%%%%%%%%%%%%%%%%%%%%%%%%%%%%%%%%%%%%%%%%%%%%%%%%%%%%%%%%%%%%%%%%%%%%%%%%%%%%%%%%%%
%%%%%%%%%%%%%%%%%%%%%%%%%%%%%%%%%%%%%%%%%%%%%%%%%%%%%%%%%%%%%%%%%%%%%%%%%%%%%%%%%%%%%%%%%%%%%%%%%%%%%%%%%%%%%%%%
%%%%%%%%%%%%%%%%%%%%%%%%%%%%%%%%%%%%%%%%%%%%%%%%%%%%%%%%%%%%%%%%%%%%%%%%%%%%%%%%%%%%%%%%%%%%%%%%%%%%%%%%%%%%%%%%
In the rest part of this section, we show Proposition \ref{Proposition-5}, which states the nonnegativities of the components $u$, $v$ and $w$ of the 
global-in-time solution to $\text{(AP)}{}_{\text{{\tiny $\delta,\!\varepsilon,\!\nu$}}}$.
%%%%%
%%%%%
%%%%%
\begin{proposition}\label{Proposition-5}
For every approximation parameters $\delta \in (1,\infty)$, $\varepsilon \in (0,1)$ and $\nu \in (0,1)$ we let $(u_{\text{{\tiny $\delta,\!\varepsilon,\!\nu$}}},
v_{\text{{\tiny $\delta,\!\varepsilon,\!\nu$}}},w_{\text{{\tiny $\delta,\!\varepsilon,\!\nu$}}},z_{\text{{\tiny $\delta,\!\varepsilon,\!\nu$}}})$ be a strong 
global-in-time solution to the approximate initial-boundary value problem $\text{(AP)}{}_{\text{{\tiny $\delta,\!\varepsilon,\!\nu$}}}$.
Then, we have 
\begin{equation*}
u_{\text{{\tiny $\delta,\!\varepsilon,\!\nu$}}}(x,t) \ge 0,\quad v_{\text{{\tiny $\delta,\!\varepsilon,\!\nu$}}}(x,t) \ge 0,\quad 
w_{\text{{\tiny $\delta,\!\varepsilon,\!\nu$}}}(x,t) \ge 0,\quad \text{a.e.}~(x,t) \in Q.
\end{equation*}
\end{proposition}
%%%%%
%%%%%
%%%%%
\begin{proof}[Proof.]
We consider the following initial-boundary value problem $\text{(AAP)}{}_{\text{{\tiny $\delta,\!\varepsilon,\!\nu$}}}$=\{\eqref{Equation-3-52}--\eqref{Equation-3-57}\} 
as an auxiliary problem for $\text{(AP)}{}_{\text{{\tiny $\delta,\!\varepsilon,\!\nu$}}}$:
\begin{gather}
\label{Equation-3-52}
u'-D_1 \Delta u=\beta_{\text{{\tiny $\delta$}}} (u_{\text{{\tiny $+$}}}) \left\{c_1 v-c_2 w\right\},\quad \text{a.e. in} \quad Q,\\
\label{Equation-3-53}
v'=\mathcal{K}_{\text{{\tiny $\delta$}}}z \left( \nu \Delta v-v\right)+v_{\text{{\tiny $+$}}}
\left\{c_3 \beta_{\text{{\tiny $\delta$}}}(w)-c_4 \beta_{\text{{\tiny $\delta$}}}(u)\right\},
\quad \text{a.e. in} \quad Q,\\
\label{Equation-3-54}
w'=D_2\Delta w-c_5 w+c_6-c_7 v \beta_{\text{{\tiny $\delta$}}}(w_{\text{{\tiny $+$}}}),\quad \text{a.e. in} \quad Q,\\
\label{Equation-3-55}
z'=\varepsilon \Delta z+v',\quad \text{a.e. in} \quad Q,\\
\label{Equation-3-56}
\nabla u \cdot \bm{n}=\nabla v \cdot \bm{n}=\nabla w \cdot \bm{n}=\nabla z \cdot \bm{n}=0,\quad \text{a.e. on} \quad \Sigma,\\
\label{Equation-3-57}
u(0)=u_0,\quad v(0)=v_0,\quad w(0)=w_0,\quad z(0)=v_0,\quad \text{a.e. in} \quad \Omega,
\end{gather}
where we set $a_{\text{{\tiny $+$}}}:=\max\{a,0\}$ and $a_{\text{{\tiny $-$}}}:=\max\{-a,0\}$.
By repeating the construction used for the approximate initial-boundary value problem $\text{(AP)}{}_{\text{{\tiny $\delta,\!\varepsilon,\!\nu$}}}$, we readily obtain 
the existence of a strong global-in-time solution to $\text{(AAP)}{}_{\text{{\tiny $\delta,\!\varepsilon,\!\nu$}}}$.\\
\indent
Firstly, we multiply both sides of \eqref{Equation-3-52} by $-u_{\text{{\tiny $-$}}}$ and integrate the result over $\Omega$.
Combining the fact that $\beta_{\text{{\tiny $\delta$}}}(u_{\text{{\tiny $+$}}}) u_{\text{{\tiny $-$}}}=0$ a.e. in $Q$, we obtain 
\begin{equation}\label{Equation-3-58}
\frac{1}{2} \frac{d}{dt} \int_\Omega |u_{\text{{\tiny $-$}}}(x,t)|^2dx=-D_1 \int_\Omega |\nabla u_{\text{{\tiny $-$}}}(x,t)|^2 \le 0, \quad \text{a.e.}~t>0.
\end{equation}
Since the initial condition $u_0 \ge 0$ a.e. in $\Omega$, from \eqref{Equation-3-58} we obtain 
\begin{equation}\label{Equation-3-59}
\int_\Omega |u_{\text{{\tiny $-$}}}(x,t)|^2dx=0,\quad \text{hence},\quad u(x,t) \ge 0,\quad \text{a.e.}~x \in \Omega,~\forall t \ge 0.
\end{equation}
\indent
Secondly, we multiply both sides of \eqref{Equation-3-53} by $-v_{\text{{\tiny $-$}}}$ and integrate the result over $\Omega$.
Combining the fact that $\beta_{\text{{\tiny $\delta$}}}(v_{\text{{\tiny $+$}}}) v_{\text{{\tiny $-$}}}=0$ a.e. in $Q$, we obtain 
\begin{align*}
\frac{1}{2} \frac{d}{dt} \int_\Omega |v_{\text{{\tiny $-$}}}(x,t)|^2dx&=\nu\int_\Omega \left(\int_\Omega K(x,y) \alpha_{\text{{\tiny $\delta$}}}(z(y,t))dy\right)
v_{\text{{\tiny $-$}}}(x,t) \Delta v_{\text{{\tiny $-$}}}(x,t)dx\\
&\hspace*{1cm}-\int_\Omega \left(\int_\Omega K(x,y)\alpha_{\text{{\tiny $\delta$}}}(z(y,t))dy\right)|v_{\text{{\tiny $-$}}}(x,t)|^2dx\\
&=-\nu\int_\Omega \nabla \left( \left(\int_\Omega K(x,y) \alpha_{\text{{\tiny $\delta$}}}(z(y,t))dy\right)v_{\text{{\tiny $-$}}}(x,t)\right)
\cdot \nabla v_{\text{{\tiny $-$}}}(x,t)dx\\
&\hspace*{1cm}-\int_\Omega \left(\int_\Omega K(x,y)\alpha_{\text{{\tiny $\delta$}}}(z(y,t))dy\right)|v_{\text{{\tiny $-$}}}(x,t)|^2dx\\
&=-\nu\int_\Omega v_{\text{{\tiny $-$}}}(x,t)\left(\int_\Omega \alpha_{\text{{\tiny $\delta$}}}(z(y,t)) \nabla K(x,y) \cdot \nabla v_{\text{{\tiny $-$}}}(x,t)dy\right) dx\\
&\hspace*{1cm}-\nu\int_\Omega \left(\int_\Omega K(x,y)\alpha_{\text{{\tiny $\delta$}}}(z(y,t)dy\right) |\nabla v_{\text{{\tiny $-$}}}(x,t)|^2dx\\
&\hspace*{2cm}-\int_\Omega \left(\int_\Omega K(x,y)\alpha_{\text{{\tiny $\delta$}}}(z(y,t))dy\right)|v_{\text{{\tiny $-$}}}(x,t)|^2dx,
\end{align*}
which yields the following equality for a.e. $t>0$:
\begin{gather}
\label{Equation-3-60}
\frac{1}{2} \frac{d}{dt} \int_\Omega |v_{\text{{\tiny $-$}}}(x,t)|^2dx+\nu \int_\Omega (\mathcal{K}_{\text{{\tiny $\delta$}}}z)(x,t)|\nabla v_{\text{{\tiny $-$}}}(x,t)|^2dx+
\int_\Omega (\mathcal{K}_{\text{{\tiny $\delta$}}}z)(x,t)|v_{\text{{\tiny $-$}}}(x,t)|^2dx\\
\nonumber
=-\nu\int_\Omega v_{\text{{\tiny $-$}}}(x,t)\left(\int_\Omega \alpha_{\text{{\tiny $\delta$}}}(z(y,t)) \nabla K(x,y) \cdot \nabla v_{\text{{\tiny $-$}}}(x,t)dy\right) dx.
\end{gather}
From \eqref{Equation-2-14} we have 
\begin{equation}\label{Equation-3-61}
\nu \int_\Omega (\mathcal{K}_{\text{{\tiny $\delta$}}}z)(x,t) |\nabla v_{\text{{\tiny $-$}}}(x,t)|^2dx \ge 
\frac{\nu K_1}{\delta} \int_\Omega |\nabla v_{\text{{\tiny $-$}}}(x,t)|^2dx.
\end{equation}
Moreover, from (c) of (A1) we estimate the right-hand side of \eqref{Equation-3-60} as follows by using the Cauchy-Schwartz inequality and the Young inequality: 
\begin{align}
\label{Equation-3-62}
&\,-\nu\int_\Omega v_{\text{{\tiny $-$}}}(x,t)\left(\int_\Omega \alpha_{\text{{\tiny $\delta$}}}(z(y,t)) \nabla K(x,y) \cdot \nabla v_{\text{{\tiny $-$}}}(x,t)dy\right) dx\\
\nonumber
\le&\,\nu \delta \int_\Omega |v_{\text{{\tiny $-$}}}(x,t)|\left(\int_\Omega |\nabla K(x,y)| |\nabla v_{\text{{\tiny $-$}}}(x,t)|dy\right) dx\\
\nonumber
\le&\,\nu \delta K_3 |\Omega|^{\frac{1}{2}} \int_\Omega |v_{\text{{\tiny $-$}}}(x,t)| |\nabla v_{\text{{\tiny $-$}}}(x,t)|dx\\
\nonumber
\le&\,\nu \delta K_3 |\Omega|^{\frac{1}{2}} \left(\int_\Omega |v_{\text{{\tiny $-$}}}(x,t)|^2dx\right)^{\frac{1}{2}} 
\left(\int_\Omega |\nabla v_{\text{{\tiny $-$}}}(x,t)|^2dx\right)^{\frac{1}{2}}\\
\nonumber
\le&\,\frac{\nu K_1}{2\delta} \int_\Omega |\nabla v_{\text{{\tiny $-$}}}(x,t)|^2dx+\frac{\nu \delta^3(K_3)^2 |\Omega|}{2K_1} \int_\Omega |v_{\text{{\tiny $-$}}}(x,t)|^2dx.
\end{align}
Substituting \eqref{Equation-3-61} and \eqref{Equation-3-62} into \eqref{Equation-3-60} and using the nonnegativity of $\mathcal{K}_{\text{{\tiny $\delta$}}}z$, 
we obtain the following inequality for a.e. $t>0$:
\begin{equation}\label{Equation-3-63}
\frac{d}{dt} \int_\Omega |v_{\text{{\tiny $-$}}}(x,t)|^2dx+\frac{\nu K_1}{\delta} \int_\Omega |\nabla v_{\text{{\tiny $-$}}}(x,t)|^2dx 
\le \frac{\delta^3(K_3)^2 |\Omega|}{K_1} \int_\Omega |v_{\text{{\tiny $-$}}}(x,t)|^2dx.
\end{equation}
Applying the Gronwall lemma to \eqref{Equation-3-63} and using the condition $v_0 \ge 0$ a.e. in $\Omega$, we obtain 
\begin{equation}\label{Equation-3-64}
\int_\Omega |v_{\text{{\tiny $-$}}}(x,t)|^2dx=0,\quad \text{hence},\quad v(x,t) \ge 0,\quad \text{a.e.}~x \in \Omega,~\forall t \ge 0.
\end{equation}
\indent
Thirdly, we multiply both sides of \eqref{Equation-3-54} by $-w_{\text{{\tiny $-$}}}$ and integrate the result over $\Omega$.
Combining the fact that $\beta_{\text{{\tiny $\delta$}}}(w_{\text{{\tiny $+$}}}) w_{\text{{\tiny $-$}}}=0$ a.e. in $Q$, we obtain 
\begin{align}
\label{Equation-3-65}
\frac{1}{2}\frac{d}{dt} \int_\Omega |w_{\text{{\tiny $-$}}}(x,t)|^2dx&=-D_2\int_\Omega |\nabla w_{\text{{\tiny $-$}}}(x,t)|^2dx
-c_5\int_\Omega |w_{\text{{\tiny $-$}}}(x,t)|^2dx\\
\nonumber
&\hspace*{2cm}-c_6\int_\Omega w_{\text{{\tiny $-$}}}(x,t)dx \le 0,\quad \text{a.e.}~t>0.
\end{align}
Since the condition $w_0 \ge 0$ a.e. in $\Omega$, from \eqref{Equation-3-65} we obtain 
\begin{equation}\label{Equation-3-66}
\int_\Omega |w_{\text{{\tiny $-$}}}(x,t)|^2dx=0,\quad \text{hence},\quad w(x,t) \ge 0,\quad \text{a.e.}~x \in \Omega,~\forall t \ge 0.
\end{equation}
\indent
Finally, from \eqref{Equation-3-59}, \eqref{Equation-3-64} and \eqref{Equation-3-66} we obtain 
\begin{equation*}
u_{\text{{\tiny $+$}}}=u,\quad v_{\text{{\tiny $+$}}}=v,\quad w_{\text{{\tiny $+$}}}=w,\quad \text{a.e. in} \quad Q.
\end{equation*}
Hence, we see that the solution $(u,v,w,z)$ to $\text{(AAP)}{}_{\text{{\tiny $\delta,\!\varepsilon,\!\nu$}}}$ is also a strong global-in-time solution to 
$\text{(AP)}{}_{\text{{\tiny $\delta,\!\varepsilon,\!\nu$}}}$, and complete the proof of this theorem.
\end{proof}
%%%%%
%%%%%
%%%%% 
%%%%%%%%%%%%%%%%%%%%%%%%%%%%%%%%%%%%%%%%%%%%%%%%%%%%%%%%%%%%%%%%%%%%%%%%%%%%%%%%%%%%%%%%%%%%%%%%%%%%%%%%%%%%%%%%%%%%%%%%%%%%%%%%%%%%%%%%
%%%%%%%%%%%%%%%%%%%%%%%%%%%%%%%%%%%%%%%%%%%%%%%%%%%%%%%%%%%%%%%%%%%%%%%%%%%%%%%%%%%%%%%%%%%%%%%%%%%%%%%%%%%%%%%%%%%%%%%%%%%%%%%%%%%%%%%%
%%%%%%%%%%%%%%%%%%%%%%%%%%%%%%%%%%%%%%%%%%%%%%%%%%%%%%%%%%%%%%%%%%%%%%%%%%%%%%%%%%%%%%%%%%%%%%%%%%%%%%%%%%%%%%%%%%%%%%%%%%%%%%%%%%%%%%%%
%%%%%%%%%%%%%%%%%%%%%%%%%%%%%%%%%%%%%%%%%%%%%%%%%%%%%%%%%%%%%%%%%%%%%%%%%%%%%%%%%%%%%%%%%%%%%%%%%%%%%%%%%%%%%%%%%%%%%%%%%%%%%%%%%%%%%%%%
%%%%%%%%%%%%%%%%%%%%%%%%%%%%%%%%%%%%%%%%%%%%%%%%%%%%%%%%%%%%%%%%%%%%%%%%%%%%%%%%%%%%%%%%%%%%%%%%%%%%%%%%%%%%%%%%%%%%%%%%%%%%%%%%%%%%%%%%	
\section{Limit procedures}\label{Section-4}
%%%%%%%%%%%%%%%%%%%%%%%%%%%%%%%%%%%%%%%%%%%%%%%%%%%%%%%%%%%%%%%%%%%%%%%%%%%%%%%%%%%%%%%%%%%%%%%%%%%%%%%%%%%%%%%%%%%%%%%%%%%%%%%%%%%%%%%%
%%%%%%%%%%%%%%%%%%%%%%%%%%%%%%%%%%%%%%%%%%%%%%%%%%%%%%%%%%%%%%%%%%%%%%%%%%%%%%%%%%%%%%%%%%%%%%%%%%%%%%%%%%%%%%%%%%%%%%%%%%%%%%%%%%%%%%%%
%%%%%%%%%%%%%%%%%%%%%%%%%%%%%%%%%%%%%%%%%%%%%%%%%%%%%%%%%%%%%%%%%%%%%%%%%%%%%%%%%%%%%%%%%%%%%%%%%%%%%%%%%%%%%%%%%%%%%%%%%%%%%%%%%%%%%%%%
%%%%%%%%%%%%%%%%%%%%%%%%%%%%%%%%%%%%%%%%%%%%%%%%%%%%%%%%%%%%%%%%%%%%%%%%%%%%%%%%%%%%%%%%%%%%%%%%%%%%%%%%%%%%%%%%%%%%%%%%%%%%%%%%%%%%%%%%
%%%%%%%%%%%%%%%%%%%%%%%%%%%%%%%%%%%%%%%%%%%%%%%%%%%%%%%%%%%%%%%%%%%%%%%%%%%%%%%%%%%%%%%%%%%%%%%%%%%%%%%%%%%%%%%%%%%%%%%%%%%%%%%%%%%%%%%%	
Unless otherwise specified, throughout this section we fix a finite time $T>0$.
%%%%%%%%%%%%%%%%%%%%%%%%%%%%%%%%%%%%%%%%%%%%%%%%%%%%%%%%%%%%%%%%%%%%%%%%%%%%%%%%%%%%%%%%%%%%%%%%%%%%%%%%%%%%%
%%%%%%%%%%%%%%%%%%%%%%%%%%%%%%%%%%%%%%%%%%%%%%%%%%%%%%%%%%%%%%%%%%%%%%%%%%%%%%%%%%%%%%%%%%%%%%%%%%%%%%%%%%%%%
%%%%%%%%%%%%%%%%%%%%%%%%%%%%%%%%%%%%%%%%%%%%%%%%%%%%%%%%%%%%%%%%%%%%%%%%%%%%%%%%%%%%%%%%%%%%%%%%%%%%%%%%%%%%%	
\subsection{Limit procedures for the parameter $\nu$}\label{Subsection-4-1}
%%%%%%%%%%%%%%%%%%%%%%%%%%%%%%%%%%%%%%%%%%%%%%%%%%%%%%%%%%%%%%%%%%%%%%%%%%%%%%%%%%%%%%%%%%%%%%%%%%%%%%%%%%%%%
%%%%%%%%%%%%%%%%%%%%%%%%%%%%%%%%%%%%%%%%%%%%%%%%%%%%%%%%%%%%%%%%%%%%%%%%%%%%%%%%%%%%%%%%%%%%%%%%%%%%%%%%%%%%%
%%%%%%%%%%%%%%%%%%%%%%%%%%%%%%%%%%%%%%%%%%%%%%%%%%%%%%%%%%%%%%%%%%%%%%%%%%%%%%%%%%%%%%%%%%%%%%%%%%%%%%%%%%%%%
In this subsection, for each $\delta \in (1,\infty)$ and $\varepsilon \in (0,1)$ we consider the following initial-boundary value problem 
$\mbox{(AP)}{}_{\text{{\tiny $\delta,\!\varepsilon$}}}$:=\{\eqref{Equation-4-1}--\eqref{Equation-4-6}\}, which is formally derived by taking the approximation 
parameter $\nu=0$ in $\mbox{(AP)}{}_{\text{{\tiny $\delta,\!\varepsilon,\!\nu$}}}$:
\begin{gather}
\label{Equation-4-1}
u'-D_1 \Delta u=\beta_{\text{{\tiny $\delta$}}}(u)(c_1 v-c_2 w) \quad \text{a.e. in}\quad Q_T:=\Omega \times (0,T),\\
\label{Equation-4-2}
v'+(\mathcal{K}_\delta z) v=v\{c_3 \beta_\delta (w)-c_4 \beta_\delta (u)\} \quad \text{a.e. in} \quad Q_T,\\
\label{Equation-4-3}
w'-D_2 \Delta w+c_5w=c_6-c_7 v \beta_\delta (w) \quad \text{a.e. in} \quad Q_T,\\
\label{Equation-4-4}
z'-\varepsilon \Delta z=v' \quad \text{a.e. in} \quad Q_T,\\
\label{Equation-4-5}
\nabla u \cdot \bm{n}=\nabla w \cdot \bm{n}=\nabla z \cdot \bm{n}=0 \quad \text{a.e. on} \quad \Sigma_T:=\Gamma \times (0,T),\\
\label{Equation-4-6}
(u(0),v(0),w(0),z(0))=(u_0,v_0,w_0,v_0) \quad \text{a.e. in} \quad \Omega.
\end{gather}
\indent
At first, for each $\delta \in (1,\infty)$ and $\varepsilon \in (0,1)$ we give the definition of strong solutions to the approximate initial-boundary value problem 
$\mbox{(AP)}{}_{\text{{\tiny $\delta,\!\varepsilon$}}}$ on $[0,T]$.
%%%%%
%%%%%
%%%%%
\begin{definition}\label{Definition-6-1}
A quadruple $(u,v,w,z):=(u_{\text{{\tiny $\delta,\!\varepsilon$}}},v_{\text{{\tiny $\delta,\!\varepsilon$}}},w_{\text{{\tiny $\delta,\!\varepsilon$}}},
z_{\text{{\tiny $\delta,\!\varepsilon$}}})$ is called a strong solution to the approximate initial-boundary value problem 
$\mbox{(AP)}{}_{\text{{\tiny $\delta,\!\varepsilon$}}}$ on $[0,T]$ if and only if the following properties are satisfied:
\begin{enumerate}\leftskip12pt
\item[(a1)] $(u,w,z) \in \left(W^{1,2}(0,T\,;L^2(\Omega)) \cap L^\infty (0,T\,;H^1(\Omega)) \cap L^2(0,T\,;H^2(\Omega))\right)^3$.\\[-0.3cm]
\item[(a2)] $v \in W^{1,2}(0,T\,;L^2(\Omega))$.\\[-0.3cm]
\item[(a3)] The system \{\eqref{Equation-4-1}--\eqref{Equation-4-4}\} with the boundary condition \eqref{Equation-4-5} is satisfied in the following variational 
sense for a.e. $t \in (0,T)$:
\begin{gather}
\label{Equation-4-7}
u'(t)-D_1 \Delta_N u(t)=\beta_{\text{{\tiny $\delta$}}} (u(t))\{c_1 v(t)-c_2 w(t)\} \quad \text{in} \quad L^2(\Omega),\\
\label{Equation-4-8}
v'(t)+(M_{\text{{\tiny $\delta,\!z(t)$}}})^{-1} v(t)=v(t)\{c_3 \beta_\delta (w(t))-c_4 \beta_\delta (u(t))\} \quad \text{in} \quad L^2(\Omega),\\
\label{Equation-4-9}
w'(t)-D_2 \Delta_N w(t)+c_5w(t)=c_6-c_7 v(t) \beta_\delta (w(t)) \quad \text{in} \quad L^2(\Omega),\\
\label{Equation-4-10}
z'(t)-\varepsilon \Delta_N z(t)=v'(t) \quad \text{in} \quad L^2(\Omega).
\end{gather}
\item[(a4)] $(u(0),v(0),w(0),z(0))=(u_0,v_0,w_0,v_0) \quad \text{in} \quad (L^2(\Omega))^4$.
\end{enumerate}
\end{definition}
The main purpose of this section is to show Proposition \ref{Proposition-6}, which guarantees the existence of strong solutions to 
$\mbox{(AP)}{}_{\text{{\tiny $\delta,\!\varepsilon$}}}$ on any prescribed bounded time interval $[0,T]$ by taking the limit procedure as $\nu \downarrow 0$. 
%%%%%
%%%%%
%%%%%
\begin{proposition}\label{Proposition-6}
For every $\delta \in (1,\infty)$ and $\varepsilon \in (0,1)$ the initial-boundary value problem $\mbox{(AP)}{}_{\text{{\tiny $\delta,\!\varepsilon$}}}$ 
admits a nonnegative strong solution $(u_{\text{{\tiny $\delta,\!\varepsilon$}}},v_{\text{{\tiny $\delta,\!\varepsilon$}}},w_{\text{{\tiny $\delta,\!\varepsilon$}}},
z_{\text{{\tiny $\delta,\!\varepsilon$}}})$ on $[0,T]$ satisfying 
\begin{equation}\label{Equation-4-11}
u_{\text{{\tiny $\delta,\!\varepsilon$}}}(x,t) \ge 0, \quad v_{\text{{\tiny $\delta,\!\varepsilon$}}}(x,t) \ge 0, \quad 
w_{\text{{\tiny $\delta,\!\varepsilon$}}}(x,t) \ge 0,\quad \text{a.e.}~(x,t) \in Q_T.
\end{equation}
\indent
Moreover, the following uniform estimates hold:
\begin{enumerate}\leftskip6pt
\item[(1)] There exists a constant $C_{17}>0$, which depends on $\delta$, $T$, $\|u_0\|_{H^1(\Omega)}$, $\|v_0\|_{H^1(\Omega)}$ and $\|w_0\|_{H^1(\Omega)}$,
such that
\begin{gather*}
\sup_{\varepsilon\,\in\,(0,1)} \left\{ \int_0^T \|u_{\text{{\tiny $\delta,\!\varepsilon$}}}'(t)\|_{L^2(\Omega)}^2dt+
\left( \sup_{0 \le t \le T} \|u_{\text{{\tiny $\delta,\!\varepsilon$}}}(t)\|_{H^1(\Omega)} \right)^2 \right.\\
\hspace*{2cm}\left. +\int_0^T \|u_{\text{{\tiny $\delta,\!\varepsilon$}}}(t)\|_{H^2(\Omega)}^2 dt \right\} \le C_{17}.
\end{gather*}
\item[(2)] There exists a constant $C_{18}>0$, which depends on $\delta$, $T$ and $\|v_0\|_{H^1(\Omega)}$, such that
\begin{gather*}
\sup_{\varepsilon\,\in\,(0,1)} \left\{ \int_0^T \|v_{\text{{\tiny $\delta,\!\varepsilon$}}}'(t)\|_{L^2(\Omega)}^2dt+
\left( \max_{0 \le t \le T} \|v_{\text{{\tiny $\delta,\!\varepsilon$}}}(t)\|_{L^2(\Omega)} \right)^2 \right\} \le C_{18}.
\end{gather*}
\item[(3)] There exists a constant $C_{19}>0$, which depends on $\delta$, $T$, $\|v_0\|_{H^1(\Omega)}$ and $\|w_0\|_{H^1(\Omega)}$, such that
\begin{gather*}
\sup_{\varepsilon\,\in\,(0,1)} \left\{ \int_0^T \|w_{\text{{\tiny $\delta,\!\varepsilon$}}}'(t)\|_{L^2(\Omega)}^2dt+
\left( \sup_{0 \le t \le T} \|w_{\text{{\tiny $\delta,\!\varepsilon$}}}(t)\|_{H^1(\Omega)} \right)^2 \right.\\
\hspace*{2cm}\left. +\int_0^T \|w_{\text{{\tiny $\delta,\!\varepsilon$}}}(t)\|_{H^2(\Omega)}^2 dt \right\} \le C_{19}.
\end{gather*}
\item[(4)] There exists a constant $C_{20}>0$, which depends on $\delta$, $T$ and $\|v_0\|_{H^1(\Omega)}$, such that the following boundedness holds 
for all $\varepsilon \in (0,1)$:
\begin{align*}
&\left( \max_{0 \le t \le T^*} \|z_{\text{{\tiny $\delta,\!\varepsilon$}}}(t)\|_{L^2(\Omega)} \right)^2
+\varepsilon \left( \sup_{0 \le t \le T^*} \|\nabla z_{\text{{\tiny $\delta,\!\varepsilon$}}}(t)\|_{\bm{L}^2(\Omega)}\right)^2\\
&\hspace*{2cm}+\varepsilon^2 \int_0^{T^*} \|z_{\text{{\tiny $\delta,\!\varepsilon$}}}(t)\|_{H^2(\Omega)}^2
+\int_0^{T^*} \|z_{\text{{\tiny $\delta,\!\varepsilon$}}}'(t)\|_{L^2 (\Omega)}^2dt \le C_{20}.
\end{align*}
\end{enumerate}
\end{proposition}
%%%%%
%%%%%
%%%%%
Since the proof of Proposition \ref{Proposition-6} is so long, we divide it in the following parts.
%%%%%
%%%%%
%%%%%
\begin{proof}[Proof of (a1), (a2), (a4) and (\ref{Equation-4-11}) in Definition \ref{Definition-6-1}.] 
By Lemmas \ref{Lemma-3-6}--\ref{Lemma-3-8}, there exist a sequence $\{\delta_n\}_{n\,\in\,\mathbb{N}} \subset (0,1)$ and a quadruple 
$(u_{\text{{\tiny $\delta,\!\varepsilon$}}},v_{\text{{\tiny $\delta,\!\varepsilon$}}},w_{\text{{\tiny $\delta,\!\varepsilon$}}},z_{\text{{\tiny $\delta,\!\varepsilon$}}})$ 
such that 
\begin{gather*}
(u_{\text{{\tiny $\delta,\!\varepsilon$}}},w_{\text{{\tiny $\delta,\!\varepsilon$}}},z_{\text{{\tiny $\delta,\!\varepsilon$}}})
\in \left(W^{1,2}((0,T)\,;L^2(\Omega)) \cap L^\infty ((0,T)\,;H^1(\Omega)) \cap L^2((0,T)\,;H^2(\Omega))\right)^3,\\
v_{\text{{\tiny $\delta,\!\varepsilon$}}} \in W^{1,2}((0,T)\,;L^2(\Omega)),
\end{gather*}
and the following convergences hold as $n \to \infty$\,:
\begin{gather}
\label{Equation-4-12}
\nu_n \downarrow 0,\\
\label{Equation-4-13}
u_n:=u_{\text{{\tiny $\delta,\!\varepsilon,\!\nu_n$}}} \longrightarrow u_{\text{{\tiny $\delta,\!\varepsilon$}}} \quad \left\{
\begin{array}{l}
\text{a.e. in} \quad Q_T,\\[0.1cm]
\text{in} \quad C([0,T]\,;L^2(\Omega)),\\[0.1cm]
\text{weakly in} \quad W^{1,2}(0,T\,;L^2(\Omega)),\\[0.1cm]
\text{weakly$^*$ in} \quad L^\infty (0,T\,;H^1(\Omega)),\\[0.1cm]
\text{weakly in} \quad L^2 (0,T\,;H^2(\Omega)),
\end{array}
\right.\\
\label{Equation-4-14}
v_n:=v_{\text{{\tiny $\delta,\!\varepsilon,\!\nu_n$}}} \longrightarrow v_{\text{{\tiny $\delta,\!\varepsilon$}}} \quad \left\{
\begin{array}{l}
\text{weakly$^*$ in} \quad L^\infty (0,T\,;L^2(\Omega)),\\[0.1cm]
\text{weakly in} \quad W^{1,2}(0,T\,;L^2(\Omega)),
\end{array}
\right.\\
\label{Equation-4-15}
w_n:=w_{\text{{\tiny $\delta,\!\varepsilon,\!\nu_n$}}} \longrightarrow w_{\text{{\tiny $\delta,\!\varepsilon$}}} \quad \left\{
\begin{array}{l}
\text{a.e. in} \quad Q_T,\\[0.1cm]
\text{in} \quad C([0,T]\,;L^2(\Omega)),\\[0.1cm]
\text{weakly in} \quad W^{1,2}((0,T)\,;L^2(\Omega)),\\[0.1cm]
\text{weakly$^*$ in} \quad L^\infty (0,T\,;H^1(\Omega)),\\[0.1cm]
\text{weakly in} \quad L^2(0,T\,;H^2(\Omega)),
\end{array}
\right.\\
\label{Equation-4-16}
z_n:=z_{\text{{\tiny $\delta,\!\varepsilon,\!\nu_n$}}} \longrightarrow z_{\text{{\tiny $\delta,\!\varepsilon$}}} \quad \left\{
\begin{array}{l}
\text{a.e. in} \quad Q_T,\\[0.1cm]
\text{in} \quad C([0,T]\,;L^2(\Omega)),\\[0.1cm]
\text{weakly in} \quad W^{1,2}(0,T\,;L^2(\Omega)),\\[0.1cm]
\text{weakly$^*$ in} \quad L^\infty (0,T\,;H^1(\Omega)),\\[0.1cm]
\text{weakly in} \quad L^2 (0,T\,;H^2(\Omega)).
\end{array}
\right.
\end{gather}
By these convergences and Proposition \ref{Proposition-5}, the properties (a1), (a2) and \eqref{Equation-4-11} are obtained. 
Moreover, since $(u_n(0),v_n(0),w_n(0),z_n(0))=(u_0,v_0,w_0,v_0)$ in $(L^2(\Omega))^4$ for all $n \in \mathbb{N}$, we obtain (a4).
\end{proof}
%%%%%
%%%%%
%%%%%
Next, we show (a3) by using the properties of the truncation $\beta_{\text{{\tiny $\delta$}}}$ (cf. \eqref{Equation-2-3} and \eqref{Equation-2-4}), 
whose proof is divided into two parts below.
%%%%%
%%%%%
%%%%%
\begin{proof}[Proof of (\ref{Equation-4-7}), (\ref{Equation-4-9}) and (\ref{Equation-4-10}) of (a3) in Definition \ref{Definition-6-1}.]
Throughout this proof, for simplicity we set $(u,v,w,z):=(u_{\text{{\tiny $\delta,\!\varepsilon$}}},v_{\text{{\tiny $\delta,\!\varepsilon$}}},
w_{\text{{\tiny $\delta,\!\varepsilon$}}},z_{\text{{\tiny $\delta,\!\varepsilon$}}})$.
Using \eqref{Equation-2-4}, we see from \eqref{Equation-4-13} and \eqref{Equation-4-15} that 
the following convergences hold as $n \to \infty$:
\begin{align}
\label{Equation-4-17}
&\beta_{\text{{\tiny $\delta$}}}(u_n) \longrightarrow \beta_{\text{{\tiny $\delta$}}}(u) \quad \text{a.e. in} \quad Q_T,\\
\label{Equation-4-18}
&\beta_{\text{{\tiny $\delta$}}}(w_n) \longrightarrow \beta_{\text{{\tiny $\delta$}}}(w) \quad \text{a.e. in} \quad Q_T.
\end{align}
Moreover, from \eqref{Equation-2-3}, for any $\xi \in L^2(0,T\,;L^2(\Omega))$ we have 
\begin{gather}
\label{Equation-4-19}
\left\{ \beta_{\text{{\tiny $\delta$}}}(u_n)-\beta_{\text{{\tiny $\delta$}}}(u)\right\} \xi,
~\left\{ \beta_{\text{{\tiny $\delta$}}}(w_n)-\beta_{\text{{\tiny $\delta$}}}(w)\right\} \xi  \in L^2(0,T\,;L^2(\Omega)),\\
\label{Equation-4-20}
\max\left\{ \left|\left\{\beta_{\text{{\tiny $\delta$}}}(u_n)-\beta_{\text{{\tiny $\delta$}}}(u)\right\} \xi \right|,\,
\left|\left\{\beta_{\text{{\tiny $\delta$}}}(w_n)-\beta_{\text{{\tiny $\delta$}}}(w)\right\} \xi \right| \right\} \le 2\delta |\xi|,\quad \text{a.e. in}\quad Q_T.
\end{gather}
By \eqref{Equation-4-17}--\eqref{Equation-4-20}, the Lebesgue dominated convergence theorem yields the following convergences as $n \to \infty$ for all
$\xi \in L^2(0,T\,;L^2(\Omega))$:
\begin{gather}
\label{Equation-4-21}
\beta_{\text{{\tiny $\delta$}}}(u_n)\xi \longrightarrow \beta_{\text{{\tiny $\delta$}}}(u)\xi \quad \text{in} \quad L^2(0,T\,;L^2(\Omega)),\\
\label{Equation-4-22}
\beta_{\text{{\tiny $\delta$}}}(w_n)\xi \longrightarrow \beta_{\text{{\tiny $\delta$}}}(w)\xi \quad \text{in} \quad L^2(0,T\,;L^2(\Omega)).
\end{gather}
Since $\xi \in L^2(0,T\,;L^2(\Omega))$ is arbitrary, we see from \eqref{Equation-4-14}, \eqref{Equation-4-21} and \eqref{Equation-4-22} that 
the following convergences hold as $n \to \infty$:
\begin{gather}
\label{Equation-4-23}
w_n \beta_{\text{{\tiny $\delta$}}}(u_n) \longrightarrow w \beta_{\text{{\tiny $\delta$}}}(u) \quad \text{weakly in} \quad L^2(0,T\,;L^2(\Omega)),\\
\label{Equation-4-24}
v_n \beta_{\text{{\tiny $\delta$}}}(u_n) \longrightarrow v \beta_{\text{{\tiny $\delta$}}}(u) \quad \text{weakly in} \quad L^2(0,T\,;L^2(\Omega)),\\
\label{Equation-4-25}
v_n \beta_{\text{{\tiny $\delta$}}}(w_n) \longrightarrow v \beta_{\text{{\tiny $\delta$}}}(w) \quad \text{weakly in} \quad L^2(0,T\,;L^2(\Omega)).
\end{gather}
\indent
Since for each $n \in \mathbb{N}$ the quadruple $(u_n,v_n,w_n,z_n)$ is a strong global-in-time solutions to $\text{(AP)}{}_{\text{{\tiny $\delta,\!\varepsilon,\!\nu$}}}$,
we have the following equalities for all $\xi \in L^2(0,T\,;L^2(\Omega))$:
\begin{align*}
\bullet~~&\int_0^T \bigl(u_n'(t)-D_1 \Delta_N u_n(t)-\beta_{\text{{\tiny $\delta$}}} (u_n(t))\{c_1 v_n(t)-c_2 w_n(t)\},\xi (t)\bigr)_{L^2(\Omega)}dt=0,\\
\bullet~~&\int_0^T \bigl(w_n'(t)-D_2 \Delta_N w_n(t)+c_5 w_n(t)-c_6+c_7 v_n (t) \beta_{\text{{\tiny $\delta$}}} (w_n(t)),\xi (t)\bigr)_{L^2(\Omega)}dt=0,\\
\bullet~~&\int_0^T \bigl(z_n'(t)-\varepsilon \Delta_N z_n(t)-v_n'(t),\xi (t)\bigr)_{L^2(\Omega)}dt=0.
\end{align*}
Taking the limit $n \to \infty$ in the above equalities combining \eqref{Equation-4-13}--\eqref{Equation-4-16} and \eqref{Equation-4-23}--\eqref{Equation-4-25},
\begin{align*}
\bullet~~&\int_0^T \bigl(u'(t)-D_1 \Delta_N u(t)-\beta_{\text{{\tiny $\delta$}}} (u(t))\{c_1 v(t)-c_2 w(t)\},\xi (t)\bigr)_{L^2(\Omega)}dt=0,\\
\bullet~~&\int_0^T \bigl(w'(t)-D_2 \Delta_N w(t)+c_5 w(t)-c_6+c_7 v (t) \beta_{\text{{\tiny $\delta$}}} (w(t)),\xi (t)\bigr)_{L^2(\Omega)}dt=0,\\
\bullet~~&\int_0^T \bigl(z'(t)-\varepsilon \Delta_N z(t)-v'(t),\xi (t)\bigr)_{L^2(\Omega)}dt=0,
\end{align*}
which imply that evolutions \eqref{Equation-4-7}, \eqref{Equation-4-9} and \eqref{Equation-4-10} holds.
\end{proof}
%%%%%
%%%%%
%%%%%
We show \eqref{Equation-4-8} of (a3) in Definition \ref{Definition-6-1} whose proof will demonstrate the advantage of deforming the structure of 
the underlying flat space $L^2(\Omega)$ by the nonlocal mobility.  
%%%%%
%%%%%
%%%%%
\begin{proof}[Proof of (\ref{Equation-4-8}) of (a3) in Definition \ref{Definition-6-1}.]
Since from (b) of (A1) and \eqref{Equation-2-20} we have the following inequality for all $(x,t) \in Q_T$:
\begin{equation*}
\left| (\mathcal{K}_{\text{{\tiny $\delta$}}} z_n)^{-1}(x,t)-(\mathcal{K}_{\text{{\tiny $\delta$}}} z)^{-1}(x,t)\right| \le (K_2)^{\frac{1}{2}} 
\left( \frac{\delta}{K_1}\right)^2 \left( \max_{0\,\le\,t\,\le\,T} \|z_n(t)-z(t)\|_{L^2(\Omega)}\right),
\end{equation*}
which implies that the following convergence holds as $n \to \infty$:
\begin{equation}\label{Equation-4-26}
(\mathcal{K}_{\text{{\tiny $\delta$}}} z_n)^{-1} \longrightarrow (\mathcal{K}_{\text{{\tiny $\delta$}}} z)^{-1} \quad \text{a.e. in} \quad Q_T.
\end{equation}
From \eqref{Equation-4-17}, \eqref{Equation-4-18} and \eqref{Equation-4-26} the following convergences also hold as $n \to \infty$:
\begin{align}
\label{Equation-4-27}
&(\mathcal{K}_{\text{{\tiny $\delta$}}} z_n)^{-1} \beta_{\text{{\tiny $\delta$}}}(u_n) \longrightarrow 
(\mathcal{K}_{\text{{\tiny $\delta$}}} z)^{-1}\beta_{\text{{\tiny $\delta$}}}(u) \quad \text{a.e. in} \quad Q_T,\\
\label{Equation-4-28}
&(\mathcal{K}_{\text{{\tiny $\delta$}}} z_n)^{-1} \beta_{\text{{\tiny $\delta$}}}(w_n) \longrightarrow 
(\mathcal{K}_{\text{{\tiny $\delta$}}} z)^{-1}\beta_{\text{{\tiny $\delta$}}}(w) \quad \text{a.e. in} \quad Q_T.
\end{align}
Moreover, from \eqref{Equation-2-15} and \eqref{Equation-2-3}, for any $\xi \in L^2(0,T\,;L^2(\Omega))$ we have 
\begin{gather}
\label{Equation-4-29}
\left\{ (\mathcal{K}_{\text{{\tiny $\delta$}}} z_n)^{-1}-(\mathcal{K}_{\text{{\tiny $\delta$}}} z)^{-1}\right\} \xi \in L^2(0,T\,;L^2(\Omega)),\\
\label{Equation-4-30}
\left\{ (\mathcal{K}_{\text{{\tiny $\delta$}}} z_n)^{-1} \beta_{\text{{\tiny $\delta$}}}(u_n)
-(\mathcal{K}_{\text{{\tiny $\delta$}}} z)^{-1}\beta_{\text{{\tiny $\delta$}}}(u)\right\} \xi \in L^2(0,T\,;L^2(\Omega)),\\
\label{Equation-4-31}
~\left\{ (\mathcal{K}_{\text{{\tiny $\delta$}}} z_n)^{-1} \beta_{\text{{\tiny $\delta$}}}(w_n)
-(\mathcal{K}_{\text{{\tiny $\delta$}}} z)^{-1}\beta_{\text{{\tiny $\delta$}}}(w)\right\} \xi  \in L^2(0,T\,;L^2(\Omega)),\\
\label{Equation-4-32}
\left| \left\{ (\mathcal{K}_{\text{{\tiny $\delta$}}} z_n)^{-1} \beta_{\text{{\tiny $\delta$}}}(u_n)
-(\mathcal{K}_{\text{{\tiny $\delta$}}} z)^{-1}\beta_{\text{{\tiny $\delta$}}}(u)\right\} \xi \right| \le \frac{2\delta^2}{K_1} |\xi|,\quad \text{a.e. in}\quad Q_T,\\
\label{Equation-4-33}
\left| \left\{ (\mathcal{K}_{\text{{\tiny $\delta$}}} z_n)^{-1} \beta_{\text{{\tiny $\delta$}}}(w_n)
-(\mathcal{K}_{\text{{\tiny $\delta$}}} z)^{-1}\beta_{\text{{\tiny $\delta$}}}(w)\right\} \xi \right| \le \frac{2\delta^2}{K_1} |\xi|,\quad \text{a.e. in}\quad Q_T.
\end{gather}
By \eqref{Equation-4-27}--\eqref{Equation-4-33}, the Lebesgue dominated convergence theorem yields the following convergences as $n \to \infty$ for all
$\xi \in L^2(0,T\,;L^2(\Omega))$:
\begin{gather}
\label{Equation-4-34}
M_{\text{{\tiny $\delta,\!z_n$}}}\xi \longrightarrow M_{\text{{\tiny $\delta,\!z$}}} \xi \quad \text{in} \quad L^2(0,T\,;L^2(\Omega)),\\
\label{Equation-4-35}
(M_{\text{{\tiny $\delta,\!z_n$}}} \beta_{\text{{\tiny $\delta$}}}(u_n)) \xi \longrightarrow 
(M_{\text{{\tiny $\delta,\!z$}}} \beta_{\text{{\tiny $\delta$}}}(u)) \xi \quad \text{in} \quad L^2(0,T\,;L^2(\Omega)),\\
\label{Equation-4-36}
(M_{\text{{\tiny $\delta,\!z_n$}}} \beta_{\text{{\tiny $\delta$}}}(w_n)) \xi \longrightarrow 
(M_{\text{{\tiny $\delta,\!z$}}} \beta_{\text{{\tiny $\delta$}}}(w)) \xi \quad \text{in} \quad L^2(0,T\,;L^2(\Omega)).
\end{gather}
Moreover, by Lemma \ref{Lemma-3-6}, for any $\eta \in L^2(0,T\,;H^1(\Omega))$ we have 
\begin{align*}
\left|\int_0^T (-\nu_n \Delta_N v_n(t),\eta (t))_{L^2(\Omega)}dt\right| &\le \nu_n \int_0^T \|\nabla v_n(t)\|_{\bm{L}^2(\Omega)} \|\nabla \eta (t)\|_{\bm{L}^2(\Omega)} dt\\
&\le \nu_n^{\frac{1}{2}} \left( \int_0^T 2\Phi_{\nu_n} (v_n(t))dt\right)^{\frac{1}{2}} \left(\int_0^T \|\eta (t)\|_{H^1(\Omega)}^2 dt\right)^{\frac{1}{2}}\\
&\le \left(\nu_n T C_{14}\right)^{\frac{1}{2}} \left(\int_0^T \|\eta (t)\|_{H^1(\Omega)}^2 dt\right)^{\frac{1}{2}},
\end{align*}
which implies from \eqref{Equation-4-12} that the following convergence holds:
\begin{equation}\label{Equation-4-37}
\lim_{n \to \infty}\int_0^T (-\nu_n \Delta_N v_n(t),\eta (t))_{L^2(\Omega)}dt=0, \quad \forall \eta \in L^2(0,T\,;H^1(\Omega)).
\end{equation}
\indent
Since for each $n \in \mathbb{N}$ the quadruple $(u_n,v_n,w_n,z_n)$ is a strong global-in-time solutions to $\text{(AP)}{}_{\text{{\tiny $\delta,\!\varepsilon,\!\nu_n$}}}$,
we have the following equality for all $\xi \in L^2(0,T\,;L^2(\Omega))$:
\begin{align*}
&\int_0^T \bigl( M_{\text{{\tiny $\delta,\!z_n(t)$}}}v_n'(t)-\nu_n \Delta_N v_n(t)+v_n(t),\xi (t)\bigr)_{L^2(\Omega)}dt\\
&\hspace*{1cm}=\int_0^T \bigl( (M_{\text{{\tiny $\delta,\!z_n(t)$}}}v_n(t))
\{c_3\beta_{\text{{\tiny $\delta$}}} (w_n(t))-c_4\beta_{\text{{\tiny $\delta$}}}(u_n(t))\},\xi (t)\bigr)_{L^2(\Omega)}dt.
\end{align*}
Taking the limit $n \to \infty$ in the above equality and using \eqref{Equation-4-34}--\eqref{Equation-4-37}, we obtain the following equality for all 
$\eta \in L^2(0,T\,;H^1(\Omega))$:
\begin{align}
\label{Equation-4-38}
&\int_0^T \bigl( M_{\text{{\tiny $\delta,\!z(t)$}}}v'(t)+v(t),\eta (t)\bigr)_{L^2(\Omega)}dt\\
\nonumber
&\hspace*{1cm}=\int_0^T \bigl( (M_{\text{{\tiny $\delta,\!z(t)$}}}v(t))
\{c_3\beta_{\text{{\tiny $\delta$}}} (w(t))-c_4\beta_{\text{{\tiny $\delta$}}}(u(t))\},\eta (t)\bigr)_{L^2(\Omega)}dt.
\end{align}
Since $H^1(\Omega)$ is dense in $L^2(\Omega)$, the equality \eqref{Equation-4-38} holds for all $\xi \in L^2(0,T\,;L^2(\Omega))$. 
Hence, we see that evolutions \eqref{Equation-4-10} holds.
\end{proof}
%%%%%
%%%%%
%%%%%
Finally, we show the boundedness of $(u_{\text{{\tiny $\delta,\!\varepsilon$}}},v_{\text{{\tiny $\delta,\!\varepsilon$}}},w_{\text{{\tiny $\delta,\!\varepsilon$}}},
z_{\text{{\tiny $\delta,\!\varepsilon$}}})$.
%%%%%
%%%%%
%%%%%
\begin{proof}[Proof of (1)--(4).]
Applying Lemmas \ref{Lemma-3-5}--\ref{Lemma-3-8} to $u_n$, $v_n$, $w_n$ and $z_n$, respectively, and passing to the limit inferior, we obtain (1)–(4).
Hence, we complete the proof of Proposition \ref{Proposition-6}.
\end{proof}
%%%%%%%%%%%%%%%%%%%%%%%%%%%%%%%%%%%%%%%%%%%%%%%%%%%%%%%%%%%%%%%%%%%%%%%%%%%%%%%%%%%%%%%%%%%%%%%%%%%%%%%%%%%%%
%%%%%%%%%%%%%%%%%%%%%%%%%%%%%%%%%%%%%%%%%%%%%%%%%%%%%%%%%%%%%%%%%%%%%%%%%%%%%%%%%%%%%%%%%%%%%%%%%%%%%%%%%%%%%
%%%%%%%%%%%%%%%%%%%%%%%%%%%%%%%%%%%%%%%%%%%%%%%%%%%%%%%%%%%%%%%%%%%%%%%%%%%%%%%%%%%%%%%%%%%%%%%%%%%%%%%%%%%%%	
\subsection{Limit procedures for the parameter $\varepsilon$}\label{Subsection-4-2}
%%%%%%%%%%%%%%%%%%%%%%%%%%%%%%%%%%%%%%%%%%%%%%%%%%%%%%%%%%%%%%%%%%%%%%%%%%%%%%%%%%%%%%%%%%%%%%%%%%%%%%%%%%%%%
%%%%%%%%%%%%%%%%%%%%%%%%%%%%%%%%%%%%%%%%%%%%%%%%%%%%%%%%%%%%%%%%%%%%%%%%%%%%%%%%%%%%%%%%%%%%%%%%%%%%%%%%%%%%%
%%%%%%%%%%%%%%%%%%%%%%%%%%%%%%%%%%%%%%%%%%%%%%%%%%%%%%%%%%%%%%%%%%%%%%%%%%%%%%%%%%%%%%%%%%%%%%%%%%%%%%%%%%%%%	
In this subsection, for each $\delta \in (1,\infty)$ we consider the following initial-boundary value problem 
$\mbox{(AP)}{}_{\text{{\tiny $\delta$}}}$:=\{\eqref{Equation-4-39}--\eqref{Equation-4-43}\}, which is formally derived by taking the approximation 
parameter $\varepsilon =0$ in $\mbox{(AP)}_{\text{{\tiny $\delta,\!\varepsilon$}}}$:
\begin{gather}
\label{Equation-4-39}
u'-D_1 \Delta u=\beta_{\text{{\tiny $\delta$}}}(u)(c_1 v-c_2 w) \quad \text{a.e. in}\quad Q_T,\\
\label{Equation-4-40}
v'+(\mathcal{K}_\delta z) v=v\{c_3 \beta_\delta (w)-c_4 \beta_\delta (u)\} \quad \text{a.e. in} \quad Q_T,\\
\label{Equation-4-41}
w'-D_2 \Delta w+c_5w=c_6-c_7 v \beta_\delta (w) \quad \text{a.e. in} \quad Q_T,\\
\label{Equation-4-42}
\nabla u \cdot \bm{n}=\nabla w \cdot \bm{n}=0 \quad \text{a.e. on} \quad \Sigma_T,\\
\label{Equation-4-43}
(u(0),v(0),w(0))=(u_0,v_0,w_0) \quad \text{a.e. in} \quad \Omega.
\end{gather}
\indent
At first, for each $\delta \in (1,\infty)$ we give the definition of strong solutions to the approximate initial-boundary value problem 
$\mbox{(AP)}{}_{\text{{\tiny $\delta$}}}$ on $[0,T]$.
%%%%%
%%%%%
%%%%%
\begin{definition}\label{Definition-4-2}
A triplet $(u,v,w):=(u_{\text{{\tiny $\delta$}}},v_{\text{{\tiny $\delta$}}},w_{\text{{\tiny $\delta$}}})$ is called a strong solution to the approximate initial-boundary 
value problem $\mbox{(AP)}{}_{\text{{\tiny $\delta$}}}$ on $[0,T]$ if and only if the following properties are satisfied:
\begin{enumerate}\leftskip12pt
\item[(b1)] $(u,w) \in \left(W^{1,2}(0,T\,;L^2(\Omega)) \cap L^\infty (0,T\,;H^1(\Omega)) \cap L^2(0,T\,;H^2(\Omega))\right)^2$.\\[-0.3cm]
\item[(b2)] $v \in W^{1,2}(0,T\,;L^2(\Omega)) \cap L^\infty (0,T\,;H^1(\Omega))$.\\[-0.3cm]
\item[(b3)] The system \{\eqref{Equation-4-39}--\eqref{Equation-4-41}\} with the boundary conditions \eqref{Equation-4-42} is satisfied in the following 
variational sense for a.e. $t \in (0,T)$:
\begin{gather}
\label{Equation-4-44}
u'(t)-D_1 \Delta_N u(t)=\beta_{\text{{\tiny $\delta$}}} (u(t))\{c_1 v(t)-c_2 w(t)\} \quad \text{in} \quad L^2(\Omega),\\
\label{Equation-4-45}
v'(t)+(M_{\text{{\tiny $\delta,\!z(t)$}}})^{-1} v(t)=v(t)\{c_3 \beta_\delta (w(t))-c_4 \beta_\delta (u(t))\} \quad \text{in} \quad L^2(\Omega),\\
\label{Equation-4-46}
w'(t)-D_2 \Delta_N w(t)+c_5w(t)=c_6-c_7 v(t) \beta_\delta (w(t)) \quad \text{in} \quad L^2(\Omega).
\end{gather}
\item[(b4)] $(u(0),v(0),w(0))=(u_0,v_0,w_0)$ in $(L^2(\Omega))^3$.
\end{enumerate}
\end{definition}
%%%%%
%%%%%
%%%%%
The main purpose of this subsection is to show Proposition \ref{Proposition-7}, which guarantees the existence of nonnegative strong solutions to 
$\mbox{(AP)}{}_{\text{{\tiny $\delta$}}}$ on $[0,T]$.
%%%%%
%%%%%
%%%%%
\begin{proposition}\label{Proposition-7}
For every $\delta \in (1,\infty)$ the approximate initial-boundary value problem $\mbox{(AP)}{}_{\text{{\tiny $\delta$}}}$ admits at least 
one strong solution $(u_{\text{{\tiny $\delta$}}},v_{\text{{\tiny $\delta$}}},w_{\text{{\tiny $\delta$}}})$ on $[0,T]$ satisfying 
\begin{equation}\label{Equation-4-73}
u_{\text{{\tiny $\delta$}}}(x,t) \ge 0, \quad v_{\text{{\tiny $\delta$}}}(x,t) \ge 0, \quad w_{\text{{\tiny $\delta$}}}(x,t) \ge 0, \quad \text{a.e.}~(x,t) \in Q_T.
\end{equation}
\indent
Moreover, the following estimates hold:
\begin{enumerate}\leftskip6pt
\item[(1)] There exists a constant $C_{21}>0$, which depends on $\delta$, $T$, $\|u_0\|_{H^1(\Omega)}$, $\|v_0\|_{H^1(\Omega)}$ and $\|w_0\|_{H^1(\Omega)}$,
such that
\begin{equation*}
\int_0^T \|u_{\text{{\tiny $\delta$}}}'(t)\|_{L^2(\Omega)}^2dt+\left( \sup_{0 \le t \le T} \|u_{\text{{\tiny $\delta$}}}(t)\|_{H^1(\Omega)} \right)^2
+\int_0^T \|u_{\text{{\tiny $\delta$}}}(t)\|_{H^2(\Omega)}^2 dt \le C_{21}.
\end{equation*}
\item[(2)] There exists a constant $C_{22}>0$, which depends on $\delta$, $T$ and $\|v_0\|_{H^1(\Omega)}$, such that
\begin{equation*}
\int_0^T \|v_{\text{{\tiny $\delta$}}}'(t)\|_{L^2(\Omega)}^2dt+\left( \max_{0 \le t \le T} \|v_{\text{{\tiny $\delta$}}}(t)\|_{L^2(\Omega)} \right)^2 \le C_{22}.
\end{equation*}
\item[(3)] There exists a constant $C_{23}>0$, which depends on $\delta$, $T$, $\|v_0\|_{H^1(\Omega)}$ and $\|w_0\|_{H^1(\Omega)}$, such that
\begin{equation*}
\int_0^T \|w_{\text{{\tiny $\delta$}}}'(t)\|_{L^2(\Omega)}^2dt+\left( \sup_{0 \le t \le T} \|w_{\text{{\tiny $\delta$}}}(t)\|_{H^1(\Omega)} \right)^2
+\int_0^T \|w_{\text{{\tiny $\delta$}}}(t)\|_{H^2(\Omega)}^2 dt \le C_{23}.
\end{equation*}
\end{enumerate}
\end{proposition}
%%%%%
%%%%%
%%%%%
In order to show Proposition \ref{Proposition-7}, for each $\varepsilon \in (0,1)$ we consider the strong solution 
$(u_{\text{{\tiny $\delta,\!\varepsilon$}}},v_{\text{{\tiny $\delta,\!\varepsilon$}}},w_{\text{{\tiny $\delta,\!\varepsilon$}}},z_{\text{{\tiny $\delta,\!\varepsilon$}}})$
to the approximate initial-boundary value problem $\mbox{(AP)}{}_{\text{{\tiny $\delta,\!\varepsilon$}}}$, and take the limit procedure as 
$\varepsilon \downarrow 0$.
By comparing (a2) in Definition \ref{Definition-6-1} and (b2) in Definition \ref{Definition-4-2}, it follows that the regularity of $v_{\text{{\tiny $\delta$}}}$ is much 
stronger than that of $v_{\text{{\tiny $\delta,\!\varepsilon$}}}$.
That is, the first objective is to recover the regularity of $v_{\text{{\tiny $\delta$}}}$, which is lost by taking the limit $\nu_n \downarrow 0$ as $n \to \infty$ 
in the proof of Proposition \ref{Proposition-6} (cf. \eqref{Equation-4-12}).
Actually, without recovering the regularity of $v_{\text{{\tiny $\delta$}}}$, we cannot especially obtain the following convergence for the nonlocal term as 
$\varepsilon \downarrow 0$:
\begin{equation*}
\int_\Omega K(\cdot,y) \alpha_{\text{{\tiny $\delta$}}}(z_{\text{{\tiny $\delta,\!\varepsilon$}}}(y))dy \longrightarrow 
\int_\Omega K(\cdot,y) \alpha_{\text{{\tiny $\delta$}}}(v_{\text{{\tiny $\delta$}}}(y))dy \quad \text{in} \quad L^2(\Omega).
\end{equation*}
\indent
Firstly, we show Lemma \ref{Lemma-4-1} while considering the regularity of $v_{\text{{\tiny $\delta$}}}$.
%%%%%
%%%%%
%%%%%
\begin{lemma}\label{Lemma-4-1}
There exists a constant $C_{24}>0$, which depends on $\|v_0\|_{L^4(\Omega)}$, $\|w_0\|_{L^2(\Omega)}$ and $T$, such that the following uniform estimate 
is satisfied:
\begin{equation*}
\sup_{\delta \in (1,\infty),\,\varepsilon\,\in\,(0,1)} \left( \sup_{0 \le t \le T} \|v_{\text{{\tiny $\delta,\!\varepsilon$}}}(t)\|_{L^4(\Omega)} \right) \le C_{24}.
\end{equation*}
\end{lemma}
%%%%%
%%%%%
%%%%%
\begin{proof}[Proof.]
Throughout this proof, for simplicity we set $(u_{\text{{\tiny $\varepsilon$}}},v_{\text{{\tiny $\varepsilon$}}},w_{\text{{\tiny $\varepsilon$}}},z_{\text{{\tiny $\varepsilon$}}})
:=(u_{\text{{\tiny $\delta,\!\varepsilon$}}},v_{\text{{\tiny $\delta,\!\varepsilon$}}},w_{\text{{\tiny $\delta,\!\varepsilon$}}},z_{\text{{\tiny $\delta,\!\varepsilon$}}})$.
From \eqref{Equation-4-2} and \eqref{Equation-4-11} we obtain
\begin{equation}\label{Equation-4-48}
v_{\text{{\tiny $\varepsilon$}}}' \le c_3 v_{\text{{\tiny $\varepsilon$}}} w_{\text{{\tiny $\varepsilon$}}}, \quad \text{a.e. in} \quad Q_T, 
\end{equation}
by using the following inequality (cf. \eqref{Equation-2-3}):
\begin{equation*}
0 \le \beta_{\text{{\tiny $\delta$}}}(w_{\text{{\tiny $\varepsilon$}}}) \le w_{\text{{\tiny $\varepsilon$}}},\quad \text{a.e. in} \quad Q_T.
\end{equation*}
Applying the Gronwall lemma to \eqref{Equation-4-48}, we obtain 
\begin{equation}\label{Equation-4-49}
v_{\text{{\tiny $\varepsilon$}}} (x,t) \le v_0(x) \exp \left( c_3 \int_0^T w_{\text{{\tiny $\varepsilon$}}} (x,t)dt \right),\quad \text{a.e.}~x \in \Omega,~\forall t \in [0,T].
\end{equation}
Taking the fourth power on both sides of \eqref{Equation-4-49} and integrating the result over $\Omega$, we obtain 
\begin{equation}\label{Equation-4-50}
0 \le \int_\Omega |v_{\text{{\tiny $\varepsilon$}}} (t)|^4dx \le \int_\Omega |v_0|^4 
\exp \left(4c_3 \int_0^T w_{\text{{\tiny $\varepsilon$}}} (s)ds\right) dx,\quad \forall t \in [0,T]. 
\end{equation}
Since $v$ and $w$ are nonnegative, it follows from \eqref{Equation-4-3} that 
\[
w'_{\text{{\tiny $\varepsilon$}}} \le D_{2}\Delta w_{\text{{\tiny $\varepsilon$}}} + c_{6}, \quad \text{a.e. in} \quad Q_T,
\]
whereas by the comparison principle and the order-preserving property of the Neumann heat semigroup
$\left(e^{\tau D_{2}\Delta}\right)_{\tau\ge0}$ yields the following inequality holds for all $t \in [0,T]$:
\[
w_{\text{{\tiny $\varepsilon$}}}(\cdot,t)
\le e^{t D_{2}\Delta}w_0
+\int_0^t e^{(t-s)D_{2}\Delta}c_{6}\,ds, \quad \text{a.e. in} \quad \Omega.
\]
By the $L^2$--$L^\infty$ smoothing estimate for the Neumann heat semigroup, see \cite[Lemma 1.3]{Winkler}, there exists $C>0$,  which depends only on 
$\Omega$, such that the following inequality holds for all $t \in [0,T]$:
\begin{equation}\label{Equation-4-51}
w_{\text{{\tiny $\varepsilon$}}}(\cdot,t) \le C\left(1+t^{-\frac{N}{4}}\right)\|w_0\|_{L^2(\Omega)}+c_{6} t,
\quad \text{a.e. in} \quad \Omega.
\end{equation}
Inserting \eqref{Equation-4-51} into (\ref{Equation-4-50}) and using the fact that $n\leq 3$, we may integrate with respect to time $t$ to obtain
the following inequality for all $t \in [0,T]$:
\begin{align*}
\int_{\Omega}|v_{\text{{\tiny $\varepsilon$}}} (t)|^4dx &\le \int_\Omega |v_0|^4 
\exp \left(4c_3 \int_0^T w_{\text{{\tiny $\varepsilon$}}} (s)ds\right) dx\\
&\leq \exp\left(4c_{3}\int_{0}^{T}\left(C\left(1+t^{-\frac{N}{4}}\right)\|w_0\|_{L^2(\Omega)}
+c_{6} t\right)\right)\int_{\Omega}|v_{0}|^4\,dx \\
&= \exp\left(4c_{3}\left(C\left(T+\frac{4}{4-N}T^{\frac{4-N}{4}}\right)\|w_{0}\|_{L^2(\Omega)}+\frac{c_{6}T^2}{2}\right)\right)\int_{\Omega}|v_{0}|^4\,dx.
\end{align*}
Consequently, there exists $C_{24}>0$ such that 
\begin{equation*}
\sup_{0 \le t \le T} \int_\Omega |v_{\text{{\tiny $\varepsilon$}}} (x,t)|^4\,dx \le C_{24},
\end{equation*}
which implies that this lemma holds.
\end{proof}
%%%%%
%%%%%
%%%%%
\begin{lemma}\label{Lemma-4-2}
There exists a constant $C_{25}>0$, which depends on the following values:
\begin{equation*}
\delta, \quad T, \quad \|u_0\|_{H^1(\Omega)},\quad \|v_0\|_{H^1(\Omega)},\quad \|w_0\|_{H^1(\Omega)},
\end{equation*}
such that the following uniform estimate is satisfied:
\begin{equation*}
\sup_{\varepsilon \in (0,1)} \left\{ \int_0^T \|v_{\text{{\tiny $\delta,\!\varepsilon$}}}'(t)\|_{L^2(\Omega)}^2dt
+\left( \sup_{0 \le t \le T} \|v_{\text{{\tiny $\delta,\!\varepsilon$}}}(t)\|_{H^1(\Omega)} \right)^2 \right\} \le C_{25}.
\end{equation*}
\end{lemma}
%%%%%
%%%%%
%%%%%
\begin{proof}[Proof.]
We use the same notation in the proof of Lemma \ref{Lemma-4-1}.
Solving the ODE \eqref{Equation-4-2} with the initial condition $v(0)=v_0$ in \eqref{Equation-4-6}, we obtain
\begin{align}
\label{Equation-4-54}
v_{\varepsilon}(x,t)&=v_0(x) \exp\left(-\int_{0}^{t}\int_{\Omega}K(x,y)\alpha_{\text{{\tiny $\delta$}}}(z_{\varepsilon}(y,s))\,dyds
+c_{3}\int_{0}^{t}\beta_{\text{{\tiny $\delta$}}}(w_{\text{{\tiny $\varepsilon$}}}(x,s))\,ds\right.\\
\nonumber
&\hspace*{5cm}\left.-c_{4}\int_{0}^{t}\beta_{\text{{\tiny $\delta$}}}(u_{\varepsilon}(x,s))\,ds\right).
\end{align}
Differentiating \eqref{Equation-4-54} with respect to space and using (c) of (A1), we obtain
\begin{align}\label{Equation-4-55}
\nabla v_{\varepsilon}(x,t)&=v_{\varepsilon}(x,t) \left(\int_{0}^{t}\int_{\Omega} \alpha_{\text{{\tiny $\delta$}}}(z_{\varepsilon}(y,s)) \nabla K(x,y)\,dyds\right.\\
\nonumber
&\hspace*{1cm}\left.+c_{3}\int_{0}^{t} \beta_{\text{{\tiny $\delta$}}}'(w_{\text{{\tiny $\varepsilon$}}}(x,s)) \nabla w_{\text{{\tiny $\varepsilon$}}}(x,s)\,ds
-c_{4}\int_{0}^{t} \beta_{\text{{\tiny $\delta$}}}'(u_{\varepsilon}(x,s))\nabla u_{\varepsilon}(x,s)\,ds\right)\\
\nonumber
&\hspace*{0.5cm}+\exp\left(-\int_{0}^{t}\int_{\Omega}K(x,y)\alpha_{\text{{\tiny $\delta$}}}(z_{\varepsilon}(y,s))\,dyds
+c_{3}\int_{0}^{t}\beta_{\text{{\tiny $\delta$}}}(w_{\text{{\tiny $\varepsilon$}}}(x,s))\,ds\right.\\
\nonumber
&\hspace*{5cm}\left.-c_{4}\int_{0}^{t}\beta_{\text{{\tiny $\delta$}}}(u_{\varepsilon}(x,s))\,ds\right)\nabla v_{0}(x).
\end{align}
Multiplying \eqref{Equation-4-55} by $\nabla v_{\varepsilon}(x,t)$ and integrating its result over $\Omega$, from \eqref{Equation-2-1} and \eqref{Equation-2-5} we obtain 
the following inequality for all $t \in [0,T]$:
\begin{align}
\label{Equation-4-56}
\|\nabla v_{\varepsilon}(t)\|_{\bm{L}^2(\Omega)}^2&\le \int_{\Omega}|\nabla v_{0}||\nabla v_{\varepsilon}(t)|\,dx \cdot 
\exp\left(c_{3}\int_{0}^{t}\|w_{\text{{\tiny $\varepsilon$}}}(s)\|_{L^{\infty}(\Omega)}\,ds\right)\\
\nonumber
&\hspace*{1cm}+\int_{0}^{t} \int_{\Omega} |\nabla v_{\varepsilon}(t)| |v_{\varepsilon}(t)| \left(\int_\Omega |z_{\varepsilon}(y,s)| |\nabla K(x,y)| \,dy\right) dx ds\\
\nonumber
&\hspace*{1cm}+c_{3}\int_{0}^{t} \int_{\Omega} |\nabla w_{\text{{\tiny $\varepsilon$}}}(s)| |\nabla v_{\varepsilon}(t)| |v_{\varepsilon}(t)|\,dxds\\
\nonumber
&\hspace*{1cm}+c_{4}\int_{0}^{t} \int_{\Omega} |\nabla u_{\varepsilon}(s)| |\nabla v_{\varepsilon}(t)| |v_{\varepsilon}(t)|\,dxds.
\end{align}
\indent
We now estimate each term on the right-hand side of \eqref{Equation-4-56}.\\
\indent
Firstly, we consider the first term.
According to (3) in Proposition \ref{Proposition-6} and the continuous embedding $H^{2}(\Omega)\hookrightarrow L^{\infty}(\Omega)$ (cf. \eqref{Equation-1-7}), we have 
\begin{align}
\label{Equation-4-57}
&\,\int_{\Omega}|\nabla v_{0}||\nabla v_{\varepsilon}(t)|\,dx \cdot \exp\left(c_3 \int_{0}^{T}\|w_{\text{{\tiny $\varepsilon$}}} (t)\|_{L^{\infty}(\Omega)}\,dt\right)\\
\nonumber
\le&\,\|\nabla v_{0}\|_{\bm{L}^2(\Omega)}\|\nabla v_{\varepsilon}(t)\|_{\bm{L}^2(\Omega)}\exp \left( c_3 K_4 \int_0^T \|w(t)\|_{H^2(\Omega)}\,dt\right)\\
\nonumber
\le&\,\frac{1}{8}\|\nabla v_{\varepsilon}(t)\|_{\bm{L}^2(\Omega)}^2+2\|v_0\|_{H^1(\Omega)}^2 \exp \left( 2c_3 K_4 (T C_{19})^{\frac{1}{2}} \right),
\end{align}
where the constant $C_{19}>0$ is the same as in (3) in Proposition \ref{Proposition-6}.\\
\indent
For the second term, from (c) of (A1), (4) in Proposition \ref{Proposition-6} and Lemma \ref{Lemma-4-1} we obtain
\begin{align}
\label{Equation-4-58}
&\,\int_{0}^{t} \int_{\Omega} |v_{\varepsilon}(t)||\nabla v_{\varepsilon}(t)| \left( \int_\Omega |\nabla K(x,y)||z_{\varepsilon}(y,s)|\,dy\right) dxds\\
\nonumber
\le&\, \int_0^t \int_\Omega |v_{\varepsilon}(t)| |\nabla v_{\varepsilon}(t)| \left( \int_\Omega |\nabla K(x,y)|^2dx\right)^{\frac{1}{2}} 
\left( \int_\Omega |z_{\varepsilon}(s)|^2 dx \right)^{\frac{1}{2}}ds\\
\nonumber
\le&\, \|v_{\varepsilon}(t)\|_{L^2(\Omega)} \|\nabla v_{\varepsilon}(t)\|_{\bm{L}^2(\Omega)} \cdot 
\sup_{x\,\in\,\Omega} \left( \int_\Omega |\nabla K(x,y)|^2dx\right)^{\frac{1}{2}} \int_0^t \|z_{\varepsilon}(s)\|_{L^2(\Omega)}ds\\
\nonumber
\le&\,\frac{1}{8}\|\nabla v_{\varepsilon}(t)\|_{\bm{L}^2 (\Omega)}^{2}+2 C_{20} (TK_3 C_{24})^2 |\Omega|^{\frac{1}{2}}, \quad \forall t \in [0,T].
\end{align}
\indent
For the third term, by Lemma \ref{Lemma-4-1} and the continuous embedding $H^1(\Omega) \hookrightarrow L^4(\Omega)$ (cf. \eqref{Equation-2-105}), we obtain 
the following inequality for all $t \in [0,T]$;
\begin{align}
\label{Equation-4-59}
&\,c_{3}\int_{0}^{t}\int_\Omega |\nabla w_{\text{{\tiny $\varepsilon$}}}(x,s)| |\nabla v_{\varepsilon}(x,t)| |v_{\varepsilon}(x,t)|\,dxds\\
\nonumber
\le&\,c_{3}\|\nabla v_{\varepsilon}(t)\|_{\bm{L}^2(\Omega)} \|v_{\varepsilon}(t)\|_{L^{4}(\Omega)} \int_{0}^{t}\|\nabla w_{\text{{\tiny $\varepsilon$}}}(s)\|_{\bm{L}^4(\Omega)}\,ds\\
\nonumber
\le&\,\frac{1}{8}\|\nabla v_{\varepsilon}(t)\|^2_{\bm{L}^2(\Omega)}+2(c_3)^2 T \left( \sup_{0 \le t \le T} \|v_{\varepsilon}(t)\|_{L^4(\Omega)} \right)^{2}
\int_{0}^{T}\|\nabla w(t)\|_{\bm{L}^4(\Omega)}^{2}dt\\
\nonumber
\le&\,\frac{1}{8}\|\nabla v_{\varepsilon}(t)\|^2_{\bm{L}^2(\Omega)}+2(c_3 C_{24})^2 T \int_{0}^{T}\|\nabla w(t)\|_{\bm{L}^4(\Omega)}^{2}dt.
\end{align}
By the Gagliardo-Nirenberg inequality: there exist constants $K_7>0$ and $K_8>0$ such that 
\begin{equation*}
\|\nabla \eta \|_{\bm{L}^4(\Omega)}^{2} \le K_7 \|\Delta \eta \|_{L^{2}(\Omega)}^{\frac{N}{2}} 
\|\nabla \eta \|_{\bm{L}^2(\Omega)}^{\frac{4-N}{2}}+K_8 \|\nabla \eta\|_{\bm{L}^2(\Omega)}^{2}, \quad \forall \eta \in H^2(\Omega), 
\end{equation*}
it follows from (3) of Proposition \ref{Proposition-6} that the following inequality holds for all $t \in [0,T]$: 
\begin{align}
\label{Equation-4-60}
\|\nabla w_{\text{{\tiny $\varepsilon$}}}(t)\|_{\bm{L}^4(\Omega)}^{2} &\le K_7 \|\Delta w_{\text{{\tiny $\varepsilon$}}}(t)\|_{L^{2}(\Omega)}^{\frac{N}{2}} 
\|\nabla w_{\text{{\tiny $\varepsilon$}}}(t)\|_{\bm{L}^2(\Omega)}^{\frac{4-N}{2}}+K_8 \|\nabla w_{\text{{\tiny $\varepsilon$}}}(t)\|_{\bm{L}^2(\Omega)}^{2}\\
\nonumber
&\le \frac{K_7 N}{4} \|\Delta w_{\text{{\tiny $\varepsilon$}}}(t)\|_{L^2(\Omega)}^{2}+\left\{ \frac{K_7(4-N)}{4}+K_8\right\} \|\nabla w_{\text{{\tiny $\varepsilon$}}}(t)\|_{\bm{L}^2(\Omega)}^{2}\\
\nonumber
&\le \frac{K_7 N}{4} \|\Delta w_{\text{{\tiny $\varepsilon$}}}(t)\|_{L^2(\Omega)}^{2}+\left\{ \frac{K_7(4-N)}{4}+K_8\right\} C_{19}.
\end{align}
Using \eqref{Equation-4-60} and (3) of Proposition \ref{Proposition-6} again, we obtain 
\begin{gather}
\label{Equation-4-61}
\int_{0}^{T}\|\nabla w_{\text{{\tiny $\varepsilon$}}}(t)\|^2_{\bm{L}^{4}(\Omega)}dt \le C_{19} \left[ \frac{K_7 N}{4}+T\left\{ \frac{K_7(4-N)}{4}+K_8\right\} \right]. 
\end{gather}
Substituting \eqref{Equation-4-61} into \eqref{Equation-4-59} and using Lemma \ref{Lemma-4-1}, we see that there exists 
a constant $C_{\text{{\tiny $4,\!2,\!1$}}}>0$ such that the following inequality holds for all $t \in [0,T]$\,:
\begin{equation}\label{Equation-4-62}
c_{3}\int_{0}^{t}\int_\Omega |\nabla w_{\text{{\tiny $\varepsilon$}}}(x,s)| |\nabla v_{\varepsilon}(x,t)| |v_{\varepsilon}(x,t)|\,dxds 
\le \frac{1}{8} \|\nabla v_{\varepsilon} (t)\|^2_{\bm{L}^2(\Omega)}+C_{\text{{\tiny $4,\!2,\!1$}}},
\end{equation}
where the constant $C_{\text{{\tiny $4,\!2,\!1$}}}>0$ is given by 
\begin{equation*}
C_{\text{{\tiny $4,\!2,\!1$}}}:=2 (c_3 C_{24})^2 C_{19} T\left[ \frac{K_7 N}{4}+T\left\{ \frac{K_7(4-N)}{4}+K_8\right\} \right].
\end{equation*}
\indent
The fourth term is estimated analogously by repeating the similar argument to $w_{\text{{\tiny $\varepsilon$}}}$ for the function $u_{\varepsilon}$.
Then, we see that there exists a constant $C_{\text{{\tiny $4,\!2,\!2$}}}>0$ such that the following inequality holds for all $t \in [0,T]$:
\begin{equation}\label{Equation-4-63}
c_{4} \int_{0}^{t}\int_\Omega |\nabla u_{\varepsilon}(x,s)| |\nabla v_{\varepsilon}(x,t)| |v_{\varepsilon}(x,t)|\,dxds 
\le \frac{1}{8}\|\nabla v_{\varepsilon}(t)\|_{\bm{L}^{2}(\Omega)}^{2}+C_{\text{{\tiny $4,\!2,\!2$}}}.
\end{equation}
Substituting \eqref{Equation-4-57}, \eqref{Equation-4-58}, \eqref{Equation-4-62} and \eqref{Equation-4-63} into \eqref{Equation-4-56}, we obtain 
\begin{equation*}
\|\nabla v_{\varepsilon}(t)\|_{\bm{L}^2(\Omega)}^{2} \le \frac{1}{2}\|\nabla v_{\varepsilon}(t)\|_{\bm{L}^2(\Omega)}^{2}+C_{\text{{\tiny $4,\!2,\!3$}}},
\end{equation*}
that is,
\begin{equation}\label{Equation-4-64}
\sup_{0 \le t \le T} \|\nabla v_{\varepsilon}(t)\|_{\bm{L}^{2}(\Omega)} \le C_{\text{{\tiny $4,\!2,\!3$}}},
\end{equation}
where the constant $C_{\text{{\tiny $4,\!2,\!3$}}}>0$ is given by
\begin{equation*}
C_{\text{{\tiny $4,\!2,\!3$}}}:=2\|v_0\|_{H^1(\Omega)}^2 \exp \left( 2c_3 K_4 (T C_{19})^{\frac{1}{2}} \right)+2 C_{20} (TK_3 C_{24})^2 |\Omega|^{\frac{1}{2}}
+C_{\text{{\tiny $4,\!2,\!1$}}}+C_{\text{{\tiny $4,\!2,\!2$}}}.
\end{equation*}
\indent
Finally, we take the inner product in both sides of \eqref{Equation-4-8} in $L^2(\Omega)$ by $v_{\varepsilon}'(t)$.
Using the H\"{o}lder inequality, we obtain the following inequality for a.e. $t \in (0,T)$:
\begin{align*}
\|v_{\varepsilon}'(t)\|_{L^2(\Omega)}^2 &\le \int_\Omega \left( \int_\Omega K(x,y) \alpha_\delta (z_{\varepsilon}(y,t))dy \right) |v_{\varepsilon}(x,t)| |v_{\varepsilon}'(x,t)| dx\\
&\hspace*{1cm}+c_3 \int_\Omega |v_{\varepsilon}(t)| |w_{\text{{\tiny $\varepsilon$}}}(t)| |v_{\varepsilon}'(t)|dx
+c_4 \int_\Omega |v_{\varepsilon}(t)| |u_{\varepsilon}(t)| |v_{\varepsilon}'(t)|\,dx\\
&\le K_2 \|z_{\text{{\tiny $\varepsilon$}}} (t)\|_{L^2(\Omega)} \|v_{\text{{\tiny $\varepsilon$}}} (t)\|_{L^2(\Omega)} \|v_{\text{{\tiny $\varepsilon$}}}'(t)\|_{L^2(\Omega)}\\
&\hspace*{1cm}+\|v_{\varepsilon}(t)\|_{L^4(\Omega)} \|v_{\varepsilon}'(t)\|_{L^2(\Omega)} \left( c_3 \|w_{\text{{\tiny $\varepsilon$}}}(t)\|_{L^4(\Omega)}+
+c_4 \|u_{\varepsilon}(t)\|_{L^4(\Omega)} \right),
\end{align*}
which gives 
\begin{align}
\label{Equation-4-65}
\|v_{\varepsilon}'(t)\|_{L^2(\Omega)}^2 &\le 4K_2 \left( \sup_{0 \le t \le T} \|z_{\text{{\tiny $\varepsilon$}}} (t)\|_{L^2(\Omega)}\right)^2
\left( \sup_{0 \le t \le T} \|v_{\text{{\tiny $\varepsilon$}}} (t)\|_{L^2(\Omega)}\right)^2\\
\nonumber
&\hspace*{1cm}+4(c_3)^2 \left( \sup_{0 \le t \le T} \|v_{\text{{\tiny $\varepsilon$}}} (t)\|_{L^4(\Omega)}\right)^2 
\left( \sup_{0 \le t \le T} \|w_{\text{{\tiny $\varepsilon$}}} (t)\|_{L^4(\Omega)}\right)^2\\
\nonumber
&\hspace*{1cm}+4(c_4)^2 \left( \sup_{0 \le t \le T} \|v_{\text{{\tiny $\varepsilon$}}} (t)\|_{L^4(\Omega)}\right)^2 
\left( \sup_{0 \le t \le T} \|u_{\text{{\tiny $\varepsilon$}}} (t)\|_{L^4(\Omega)}\right)^2.
\end{align}
By the continuous embedding $H^1(\Omega) \hookrightarrow L^4(\Omega)$ (cf. \eqref{Equation-2-105}) again, from Proposition \ref{Proposition-6} 
and Lemma \ref{Lemma-4-1} with \eqref{Equation-4-65} we obtain
\begin{equation}\label{Equation-4-66}
\int_0^T \|v_{\text{{\tiny $\varepsilon$}}}'(t)\|_{L^2(\Omega)}^2 \le C_{\text{{\tiny $4,\!2,\!4$}}},
\end{equation}
where the constant $C_{\text{{\tiny $4,\!2,\!4$}}}>0$ is given by
\begin{equation*}
C_{\text{{\tiny $4,\!2,\!4$}}}:=4 (C_{24})^2 \left\{ K_2C_{20} |\Omega|+(c_3 K_6)^2 C_{19}+(c_4 K_6)^2 C_{17} \right\}.
\end{equation*}
The boundedness \eqref{Equation-4-64} and \eqref{Equation-4-66} yield the required uniform estimates in this lemma.
\end{proof}
%%%%%
%%%%%
%%%%%
In the next lemma, we consider the convergence of the sequences $\{v_{\text{{\tiny $\delta,\!\varepsilon_n$}}}\}_{n \in \mathbb{N}}$ and 
$\{z_{\text{{\tiny $\delta,\!\varepsilon_n$}}}\}_{n \in \mathbb{N}}$.
In essence, this lemma implies that there exists a sequence $\{\varepsilon_n\}_{n \in \mathbb{N}}$ with $\varepsilon \downarrow 0$ as $n \to \infty$ 
such that the corresponding subsequences $\{v_{\text{{\tiny $\delta,\!\varepsilon_n$}}}\}_{n \in \mathbb{N}}$ and 
$\{z_{\text{{\tiny $\delta,\!\varepsilon_n$}}}\}_{n \in \mathbb{N}}$ share the same limit.
As a result, the independent evolution of the metric generator $z_{\text{{\tiny $\delta,\!\varepsilon$}}}$ in the approximate initial-boundary value problem 
$\mbox{(AP)}{}_{\text{{\tiny $\delta,\!\varepsilon$}}}$ vanishes in $\mbox{(AP)}{}_{\text{{\tiny $\delta$}}}$, 
allowing us to identify the two limit functions in $\mbox{(AP)}{}_{\text{{\tiny $\delta$}}}$.
%%%%%
%%%%%
%%%%%
\begin{lemma}\label{Lemma-4-3}
There exist a sequence $\{\varepsilon_n\}_{n \in \mathbb{N}} \subset (0,1)$ and a nonnegative function
$v_{\text{{\tiny $\delta$}}} \in W^{1,2}(0,T\,;L^2(\Omega)) \cap L^\infty (0,T\,;H^1(\Omega))$ such that the following convergences hold as $n \to \infty$\,:
\begin{gather}
\label{Equation-4-67}
\varepsilon_n \downarrow 0,\\
\label{Equation-4-68}
z_{\text{{\tiny $\delta,\!\varepsilon_n$}}} \longrightarrow v_{\text{{\tiny $\delta$}}} \quad \text{in} \quad C([0,T]\,;L^2(\Omega)),\\
\label{Equation-4-69}
v_{\text{{\tiny $\delta,\!\varepsilon_n$}}} \longrightarrow v_{\text{{\tiny $\delta$}}} \quad \left\{
\begin{array}{l}
\text{in} \quad C([0,T]\,;L^2(\Omega)),\\[0.1cm]
\text{weakly in} \quad W^{1,2}(0,T\,;L^2(\Omega)),\\[0.1cm]
\text{weakly$^*$ in} \quad L^\infty (0,T\,;H^1(\Omega)).
\end{array}
\right.
\end{gather}
\end{lemma}
%%%%%
%%%%%
%%%%%
\begin{proof}[Proof.]
Throughout this proof, for simplicity we set 
$(v_{\text{{\tiny $\varepsilon$}}},z_{\text{{\tiny $\varepsilon$}}}):=(v_{\text{{\tiny $\delta,\!\varepsilon$}}},z_{\text{{\tiny $\delta,\!\varepsilon$}}})$.
From \eqref{Equation-4-10} in (a3) and (a4) of Definition \ref{Definition-4-2} the pair $(v_{\text{{\tiny $\varepsilon$}}},z_{\text{{\tiny $\varepsilon$}}})$ satisfies 
\eqref{Equation-4-70} and \eqref{Equation-4-71}:
\begin{gather}
\label{Equation-4-70}
(z_{\text{{\tiny $\varepsilon$}}}-v_{\text{{\tiny $\varepsilon$}}})'(t)-\varepsilon \Delta_N z_{\text{{\tiny $\varepsilon$}}}(t)=0 
\quad \text{in} \quad L^2(\Omega), \quad \text{a.e.}~t \in (0,T),\\
\label{Equation-4-71}
(z_{\text{{\tiny $\varepsilon$}}}-v_{\text{{\tiny $\varepsilon$}}})(0)=0 \quad \text{in} \quad L^2(\Omega).
\end{gather}
Taking the inner product in both sides of \eqref{Equation-4-70} in $L^2(\Omega)$ by $(z_{\text{{\tiny $\varepsilon$}}}-v_{\text{{\tiny $\varepsilon$}}})(t)$,
we obtain the following equality for a.e. $t \in (0,T)$:
\begin{equation}\label{Equation-4-72}
\frac{1}{2} \frac{d}{dt} \|(z_{\text{{\tiny $\varepsilon$}}}-v_{\text{{\tiny $\varepsilon$}}})(t)\|_{L^2(\Omega)}^2
-\varepsilon (\Delta_N z_{\text{{\tiny $\varepsilon$}}}(t),(z_{\text{{\tiny $\varepsilon$}}}-v_{\text{{\tiny $\varepsilon$}}})(t))_{L^2(\Omega)}=0.
\end{equation}
By Proposition \ref{Proposition-6} and Lemma \ref{Lemma-4-2}, from \eqref{Equation-4-71} we obtain the following equality:
\begin{gather}
\label{Equation-4-73}
-\varepsilon (\Delta_N z_{\text{{\tiny $\varepsilon$}}}(t),z_{\text{{\tiny $\varepsilon$}}}-v_{\text{{\tiny $\varepsilon$}}})(t))_{L^2(\Omega)}
=\varepsilon \int_\Omega \nabla z_{\text{{\tiny $\varepsilon$}}}(t) \cdot \nabla (z_{\text{{\tiny $\varepsilon$}}}-v_{\text{{\tiny $\varepsilon$}}})(t)dx\\
\nonumber
=\varepsilon \|\nabla (z_{\text{{\tiny $\varepsilon$}}}-v_{\text{{\tiny $\varepsilon$}}})(t)\|_{\bm{L}^2(\Omega)}^2
+\varepsilon \int_\Omega \nabla v_{\text{{\tiny $\varepsilon$}}}(t) \cdot \nabla (z_{\text{{\tiny $\varepsilon$}}}-v_{\text{{\tiny $\varepsilon$}}})(t)dx,
\quad \text{a.e.}~t \in (0,T).
\end{gather}
Substituting \eqref{Equation-4-73} into \eqref{Equation-4-72}, we obtain 
\begin{equation*}
\frac{1}{2}\frac{d}{dt} \|(z_{\text{{\tiny $\varepsilon$}}}-v_{\text{{\tiny $\varepsilon$}}})(t)\|_{L^2(\Omega)}^2
+\varepsilon \|\nabla (z_{\text{{\tiny $\varepsilon$}}}-v_{\text{{\tiny $\varepsilon$}}})(t)\|_{\bm{L}^2(\Omega)}^2
=-\varepsilon \int_\Omega \nabla v_{\text{{\tiny $\varepsilon$}}}(t) \cdot \nabla (z_{\text{{\tiny $\varepsilon$}}}-v_{\text{{\tiny $\varepsilon$}}})(t)dx,
\end{equation*}
hence, the following inequality for a.e. $t \in (0,T)$:
\begin{equation}\label{Equation-4-74}
\frac{d}{dt} \|(z_{\text{{\tiny $\varepsilon$}}}-v_{\text{{\tiny $\varepsilon$}}})(t)\|_{L^2(\Omega)}^2
+\varepsilon \|\nabla (z_{\text{{\tiny $\varepsilon$}}}-v_{\text{{\tiny $\varepsilon$}}})(t)\|_{\bm{L}^2(\Omega)}^2
\le \varepsilon \|\nabla v_{\text{{\tiny $\varepsilon$}}}(t)\|_{\bm{L}^2(\Omega)}^2,
\end{equation}
Integrating \eqref{Equation-4-74} on any interval $[0,t] \subset [0,T]$ and using \eqref{Equation-4-71}, we obtain 
\begin{equation}\label{Equation-4-75}
\max_{0\,\le\,t\,\le\,T} \|(z_{\text{{\tiny $\varepsilon$}}}-v_{\text{{\tiny $\varepsilon$}}})(t)\|_{L^2(\Omega)}^2 
\le \varepsilon \int_0^T \|\nabla v_{\text{{\tiny $\varepsilon$}}} (t)\|_{\bm{L}^2(\Omega)}^2
\end{equation}
By Lemma \ref{Lemma-4-2}, we see from \eqref{Equation-4-75} that the following estimate holds for all $\varepsilon \in (0,1)$:
\begin{equation*}
\max_{0\,\le\,t\,\le\,T} \|(z_{\text{{\tiny $\varepsilon$}}}-v_{\text{{\tiny $\varepsilon$}}})(t)\|_{L^2(\Omega)}^2 \le \varepsilon T C_{25},
\end{equation*}
which implies 
\begin{equation}\label{Equation-4-76}
z_{\text{{\tiny $\varepsilon$}}}-v_{\text{{\tiny $\varepsilon$}}} \longrightarrow 0 \quad \text{in} \quad C([0,T]\,;L^2(\Omega))
\quad \text{as} \quad \varepsilon \downarrow 0.
\end{equation}
\indent
On the other hand, we see from Lemma \ref{Lemma-4-2} that there exists a sequence $\{\varepsilon_n\}_{n \in \mathbb{N}} \subset (0,1)$ and a nonnegative function 
$v:=v_{\text{{\tiny $\delta$}}} \in W^{1,2}(0,T\,;L^2(\Omega)) \cap L^\infty(0,T\,;H^1(\Omega))$ such that the required convergences \eqref{Equation-4-67} 
and \eqref{Equation-4-69} holds as $n \to \infty$.\\
\indent
Since we have the following inequality for all $n \in \mathbb{N}$ and $t \in [0,T]$:
\begin{equation*}
\|(z_{\varepsilon_n}-v)(t)\|_{L^2(\Omega)} \le \|(z_{\varepsilon_n}-v_{\varepsilon_n})(t)\|_{L^2(\Omega)}+\|(v_{\varepsilon_n}-v)(t)\|_{L^2(\Omega)},
\end{equation*}
the convergences \eqref{Equation-4-67}, \eqref{Equation-4-69} and \eqref{Equation-4-76} yield the required convergence \eqref{Equation-4-68} as $n \to \infty$.
This concludes that this lemma holds.
\end{proof}
%%%%%
%%%%%
%%%%%
Now we are in a position to show Proposition \ref{Proposition-7}.
%%%%%
%%%%%
%%%%%
\begin{proof}[Proof of Proposition \ref{Proposition-7}.]
Let the sequence $\{\varepsilon_n\}_{n \in \mathbb{N}}$ and the function $v_{\text{{\tiny $\delta$}}}$ be the same ones as in Lemma \ref{Lemma-4-3}, and 
set $v_n:=v_{\text{{\tiny $\delta,\!\varepsilon_n$}}}$ for all $n \in \mathbb{N}$ for simplicity.
By Proposition \ref{Proposition-6}, there exist a subsequence of $\{\varepsilon_n\}_{n \in \mathbb{N}}$, which is denoted the same notation, and 
a triplet $(u_{\text{{\tiny $\delta$}}},w_{\text{{\tiny $\delta$}}},z_{\text{{\tiny $\delta$}}})$ such that the following convergences hold as $n \to \infty$:
\begin{align}
\label{Equation-4-77}
&u_n:=u_{\text{{\tiny $\delta,\!\varepsilon_n$}}} \longrightarrow u_{\text{{\tiny $\delta$}}} \quad \left\{
\begin{array}{l}
\text{a.e. in} \quad Q_T,\\[0.1cm]
\text{in} \quad C([0,T]\,;L^2(\Omega)),\\[0.1cm]
\text{weakly in} \quad W^{1,2}(0,T\,;L^2(\Omega)),\\[0.1cm]
\text{weakly$^*$ in} \quad L^\infty (0,T\,;H^1(\Omega)),\\[0.1cm]
\text{weakly in} \quad L^2(0,T\,;H^2(\Omega)),
\end{array}
\right.\\
\label{Equation-4-78}
&w_n:=w_{\text{{\tiny $\delta,\!\varepsilon_n$}}} \longrightarrow w_{\text{{\tiny $\delta$}}} \quad \left\{
\begin{array}{l}
\text{a.e. in} \quad Q_T,\\[0.1cm]
\text{in} \quad C([0,T]\,;L^2(\Omega)),\\[0.1cm]
\text{weakly in} \quad W^{1,2}(0,T\,;L^2(\Omega)),\\[0.1cm]
\text{weakly$^*$ in} \quad L^\infty(0,T\,;H^1(\Omega)),\\[0.1cm]
\text{weakly in} \quad L^2(0,T\,;H^2(\Omega)),
\end{array}
\right.\\
\label{Equation-4-79}
&z_n:=z_{\text{{\tiny $\delta,\!\varepsilon_n$}}} \longrightarrow z_{\text{{\tiny $\delta$}}} \quad \left\{
\begin{array}{l}
\text{weakly$^*$ in} \quad L^\infty (0,T\,;L^2(\Omega)),\\[0.1cm]
\text{weakly in} \quad W^{1,2}((0,T)\,;L^2(\Omega)).
\end{array}
\right.
\end{align}
Using Lemma \ref{Lemma-4-3} and repeating the similar argument to the proof of Proposition \ref{Proposition-6} due to 
\eqref{Equation-4-69}, \eqref{Equation-4-77} and \eqref{Equation-4-78} (cf. (cf. \eqref{Equation-4-14}), \eqref{Equation-4-13} and \eqref{Equation-4-15}, respectively),
it follows that the limit triplet $(u_{\text{{\tiny $\delta$}}},v_{\text{{\tiny $\delta$}}},w_{\text{{\tiny $\delta$}}})$ satisfies \eqref{Equation-4-44} and \eqref{Equation-4-46} 
with the initial condition $(u_{\text{{\tiny $\delta$}}}(0),v_{\text{{\tiny $\delta$}}}(0),w_{\text{{\tiny $\delta$}}}(0))=(u_0,v_0,w_0)$.\\
\indent
Next, we investigate the function $z_{\text{{\tiny $\delta$}}}$.
By Proposition \ref{Proposition-6}, from \eqref{Equation-4-10} we obtain the following inequality for all $\eta \in L^2(0,T\,;H^1(\Omega))$:
\begin{gather*}
\,\left| \int_0^T (z_n'(t)-v_n'(t),\eta (t))_{L^2(\Omega)}dt \right|
=\varepsilon_n \left| \int_0^T (\nabla z_n(t),\nabla \eta (t))_{\bm{L}^2(\Omega)}dt \right|\\
\le \varepsilon_n^{\frac{1}{2}} \left( \varepsilon_n \int_0^T \|\nabla z_n(t)\|_{\bm{L}^2(\Omega)}^2dt\right)^{\frac{1}{2}}
\left( \int_0^T \|\nabla \eta (t)\|_{\bm{L}^2(\Omega)}^2dt \right)^{\frac{1}{2}}\\
\le \left( \varepsilon_n T C_{20} \right)^{\frac{1}{2}} \left( \int_0^T \|\nabla \eta (t)\|_{\bm{L}^2(\Omega)}^2dt \right)^{\frac{1}{2}},
\end{gather*}
which implies 
\begin{equation}\label{Equation-4-80}
\lim_{n \to \infty} \int_0^T (z_n'(t)-v_n'(t),\eta (t))_{L^2(\Omega)}dt=0.
\end{equation}
Since the sequence $\{z_n'-v_n'\}_{n \in \mathbb{N}}$ is bounded in $L^2(0,T\,;L^2(\Omega))$ and 
$H^1(\Omega)$ is dense in $L^2(\Omega)$, it follows that \eqref{Equation-4-80} is valid for all $\eta \in L^2(0,T\,;L^2(\Omega))$, that is, the 
following convergence holds as $n \to \infty$:
\begin{equation}\label{Equation-4-81}
z_n'-v_n' \longrightarrow 0 \quad \text{weakly in} \quad L^2(0,T\,;L^2(\Omega)).
\end{equation}
From \eqref{Equation-4-69}, \eqref{Equation-4-79} and \eqref{Equation-4-81} we obtain $z'=v'$ in $L^2(0,T\,;L^2(\Omega))$, hence, by using the initial condition 
$v_{\text{{\tiny $\delta$}}}(0)=z_{\text{{\tiny $\delta$}}}(0)=v_0$
\begin{equation}\label{Equation-4-82}
z_{\text{{\tiny $\delta$}}}(t)=v_0+\int_0^t z_{\text{{\tiny $\delta$}}}'(s)ds=v_0+\int_0^t v_{\text{{\tiny $\delta$}}}'(s)ds=v_{\text{{\tiny $\delta$}}}(t)
\quad \text{in} \quad L^2(\Omega),\quad \forall t \in [0,T].
\end{equation}
The equation \eqref{Equation-4-82} implies that the metric generator $z_{\text{{\tiny $\delta$}}}$ coincides with the function $v_{\text{{\tiny $\delta$}}}$ in the 
limit system as $\varepsilon_n \downarrow 0$ and the evolution equation \eqref{Equation-4-10} vanishes in $\mbox{(AP)}{}_{\text{{\tiny $\delta$}}}$.\\
\indent
In the rest of this proof, we consider the function $v_{\text{{\tiny $\delta$}}}$.
At first, it follow from \eqref{Equation-4-45} that for each $n \in \mathbb{N}$ the following equality holds for all $\xi \in L^2(0,T\,;L^2(\Omega))$:
\begin{gather}\label{Equation-4-83}
\int_0^T (v_n'(t),\xi (t))_{L^2(\Omega)}dt+\int_0^T \int_\Omega \left(\int_\Omega K(x,y) 
\alpha_{\text{{\tiny $\delta$}}} (z_n(y,t))dy \right) v_n(t) \xi (t)dxdt\\
\nonumber
=\int_0^T \int_\Omega v_n(t) \left\{ c_3 \beta_{\text{{\tiny $\delta$}}}(w_n(t))-
c_4 \beta_{\text{{\tiny $\delta$}}}(u_n(t)) \right\} \xi (t)dxdt.
\end{gather}
From \eqref{Equation-4-69} we obtain 
\begin{equation}\label{Equation-4-84}
\lim_{n \to \infty} \int_0^T (v_n'(t),\xi (t))_{L^2(\Omega)}dt=\int_0^T (v_{\text{{\tiny $\delta$}}}'(t),\xi (t))_{L^2(\Omega)}dt.
\end{equation}
While using \eqref{Equation-4-69}, \eqref{Equation-4-77} and \eqref{Equation-4-78}, we repeat the similar argument to the proof of \eqref{Equation-4-7},
\eqref{Equation-4-9} and \eqref{Equation-4-10} of (a3) in Definition \ref{Definition-6-1}.
Then, we obtain 
\begin{align}
\label{Equation-4-85}
&\lim_{n \to \infty} \int_0^T \int_\Omega v_n(t) \left\{ c_3 \beta_{\text{{\tiny $\delta$}}}(w_n(t))-
c_4 \beta_{\text{{\tiny $\delta$}}}(u_n(t)) \right\} \xi (t)dxdt\\
\nonumber
&\hspace*{1cm}=\int_0^T \int_\Omega v_{\text{{\tiny $\delta$}}}(t) \left\{ c_3 \beta_{\text{{\tiny $\delta$}}}(w_{\text{{\tiny $\delta$}}}(t))-
c_4 \beta_{\text{{\tiny $\delta$}}}(u_{\text{{\tiny $\delta$}}}(t)) \right\} \xi (t)dxdt.
\end{align}
Since we obtain the following inequalities: from (c) of (A1) and \eqref{Equation-2-1}
\begin{align*}
&\,\left| \int_0^T \int_\Omega \left(\int_\Omega K(x,y)\alpha_{\text{{\tiny $\delta$}}} (z_n(y,t))dy \right) 
\left\{ v_n(t)-v_{\text{{\tiny $\delta$}}}(t) \right\} \xi (t)dxdt \right|\\
\le&\,\delta |\Omega|^{\frac{1}{2}} \sup_{x \in \Omega} \left( \int_\Omega |K(x,y)|^2dy \right)^{\frac{1}{2}} 
\int_0^T \int_\Omega |v_n(t)-v_{\text{{\tiny $\delta$}}}(t)| |\xi (t)|\,dxdt\\
\le&\,\delta (T |\Omega|)^{\frac{1}{2}} K_2 \left( \max_{0 \le t \le T} \|v_n(t)-v_{\text{{\tiny $\delta$}}}(t)\|_{L^2(\Omega)}\right)
\left( \int_0^T \|\xi (t)\|_{L^2(\Omega)}^2dt \right)^{\frac{1}{2}},
\end{align*}
and from \eqref{Equation-2-2} with (c) of (A3) again
\begin{align*}
&\,\left| \int_0^T \int_\Omega \left(\int_\Omega K(x,y)\left\{ \alpha_{\text{{\tiny $\delta$}}} (z_n(y,t))
-\alpha_{\text{{\tiny $\delta$}}} (v_{\text{{\tiny $\delta$}}}(y,t)) \right\}dy \right) v_{\text{{\tiny $\delta$}}}(t) \xi (t)dxdt \right|\\
\le&\,\int_0^T \int_\Omega \left( \int_\Omega |K(x,y)|^2dy \right)^{\frac{1}{2}} 
\|z_n(t)-z_{\text{{\tiny $\delta$}}}(t)\|_{L^2(\Omega)} |v_{\text{{\tiny $\delta$}}}(t)| |\xi (t)|\,dxdt\\
\le&\,\delta K_2 \left( \max_{0 \le t \le T} \|z_n(t)-v_{\text{{\tiny $\delta$}}}(t)\|_{L^2(\Omega)}\right)
\left( \int_0^T \|v_{\text{{\tiny $\delta$}}}(t)\|_{L^2(\Omega)}^2dt \right)^{\frac{1}{2}} \left( \int_0^T \|\xi (t)\|_{L^2(\Omega)}^2dt \right)^{\frac{1}{2}}.
\end{align*}
from \eqref{Equation-4-68} and \eqref{Equation-4-69} we obtain
\begin{align}\label{Equation-4-86}
&\lim_{n \to \infty} \int_0^T \int_\Omega \left(\int_\Omega K(x,y) 
\alpha_{\text{{\tiny $\delta$}}} (z_n(y,t))dy \right) v_n(t) \xi (t)dxdt\\
\nonumber
&\hspace*{1cm}=\int_0^T \int_\Omega \left(\int_\Omega K(x,y) 
\alpha_{\text{{\tiny $\delta$}}} (v_{\text{{\tiny $\delta$}}}(y,t))dy \right) v_{\text{{\tiny $\delta$}}}(t) \xi (t)dxdt.
\end{align}
Taking the limit $n \to \infty$ in \eqref{Equation-4-83} and using \eqref{Equation-4-84}--\eqref{Equation-4-86}, it follows that the following equality 
holds for all $\xi \in L^2(0,T\,;L^2(\Omega))$:
\begin{gather*}
\int_0^T (v_{\text{{\tiny $\delta$}}}'(t),\xi (t))_{L^2(\Omega)}dt+\int_0^T \int_\Omega \left(\int_\Omega K(x,y) 
\alpha_{\text{{\tiny $\delta$}}} (v_{\text{{\tiny $\delta$}}}(y,t))dy \right) v_{\text{{\tiny $\delta$}}}(t) \xi (t)dxdt\\
=\int_0^T \int_\Omega v_{\text{{\tiny $\delta$}}}(t) \left\{ c_3 \beta_{\text{{\tiny $\delta$}}}(w_{\text{{\tiny $\delta$}}}(t))-
c_4 \beta_{\text{{\tiny $\delta$}}}(u_{\text{{\tiny $\delta$}}}(t)) \right\} \xi (t)dxdt,
\end{gather*}
which implies that \eqref{Equation-4-45} holds, and complete the proof of Proposition \ref{Proposition-7}.
\end{proof}
%%%%%%%%%%%%%%%%%%%%%%%%%%%%%%%%%%%%%%%%%%%%%%%%%%%%%%%%%%%%%%%%%%%%%%%%%%%%%%%%%%%%%%%%%%%%%%%%%%%%%%%%%%%%%
%%%%%%%%%%%%%%%%%%%%%%%%%%%%%%%%%%%%%%%%%%%%%%%%%%%%%%%%%%%%%%%%%%%%%%%%%%%%%%%%%%%%%%%%%%%%%%%%%%%%%%%%%%%%%
%%%%%%%%%%%%%%%%%%%%%%%%%%%%%%%%%%%%%%%%%%%%%%%%%%%%%%%%%%%%%%%%%%%%%%%%%%%%%%%%%%%%%%%%%%%%%%%%%%%%%%%%%%%%%	
\subsection{Proof of Theorem \ref{Theorem-1-1}}\label{Subsection-4-3}
%%%%%%%%%%%%%%%%%%%%%%%%%%%%%%%%%%%%%%%%%%%%%%%%%%%%%%%%%%%%%%%%%%%%%%%%%%%%%%%%%%%%%%%%%%%%%%%%%%%%%%%%%%%%%
%%%%%%%%%%%%%%%%%%%%%%%%%%%%%%%%%%%%%%%%%%%%%%%%%%%%%%%%%%%%%%%%%%%%%%%%%%%%%%%%%%%%%%%%%%%%%%%%%%%%%%%%%%%%%
%%%%%%%%%%%%%%%%%%%%%%%%%%%%%%%%%%%%%%%%%%%%%%%%%%%%%%%%%%%%%%%%%%%%%%%%%%%%%%%%%%%%%%%%%%%%%%%%%%%%%%%%%%%%%
The main objective is to show Theorem \ref{Theorem-1-1}.
By Proposition \ref{Proposition-7}, for every finite time $T>0$ the approximate initial-boundary value problem $\text{(AP)}{}_{\text{{\tiny $\delta$}}}$ has a strong 
solution $(u_{\text{{\tiny $\delta$}}},v_{\text{{\tiny $\delta$}}},w_{\text{{\tiny $\delta$}}})$ on $[0,T]$.
As you see from Proposition \ref{Proposition-7}, we have not obtained any uniform estimates for the family 
$\{(u_{\text{{\tiny $\delta$}}},v_{\text{{\tiny $\delta$}}},w_{\text{{\tiny $\delta$}}})\,;\,\delta \in (1,\infty)\}$ of solutions to $\text{(AP)}{}_{\text{{\tiny $\delta$}}}$ 
on $[0,T]$, except, the following lemma, which is directly obtained by Lemma \ref{Lemma-4-1}.
%%%%%
%%%%%
%%%%%
\begin{lemma}\label{Lemma-4-4}
The following uniform estimate holds: 
\begin{equation*}
\sup_{\delta \in (1,\infty)} \left( \sup_{0 \le t \le T} \|v_{\text{{\tiny $\delta$}}}(t)\|_{L^4(\Omega)} \right) \le C_{24},
\end{equation*}
where the constant $C_{24}>0$ is the same as in Lemma \ref{Lemma-4-1}.
\end{lemma}
%%%%%
%%%%%
%%%%%
First of all, we show Lemma \ref{Lemma-4-5}.
%%%%%
%%%%%
%%%%%
\begin{lemma}\label{Lemma-4-5}
There exists a constant $C_{26}>0$, which depends on $T$, $\|w_0\|_{H^1(\Omega)}$ and $\|u_0\|_{H^1(\Omega)}$, such that the following uniform estimate holds:
\begin{equation*}
\sup_{\delta \in (1,\infty)} \left\{ \left(\sup_{0 \le t\le T} \|u_{\text{{\tiny $\delta$}}}(t)\|_{L^4(\Omega)}\right)+
\left(\sup_{0 \le t\le T} \|w_{\text{{\tiny $\delta$}}}(t)\|_{L^4(\Omega)}\right) \right\} \le C_{26}.
\end{equation*}
\end{lemma}
%%%%%
%%%%%
%%%%%
\begin{proof}[Proof.]
For simplicity we set $(u,v,w):=(u_{\text{{\tiny $\delta$}}},v_{\text{{\tiny $\delta$}}},w_{\text{{\tiny $\delta$}}})$.
Testing (\ref{Equation-4-41}) with $w^3$ and then integrating by parts over $\Omega$ yields the following inequality for a.e. $t \in (0,T)$:
\begin{equation*}
\frac{1}{4}\frac{d}{dt}\int_{\Omega} w^4\,dx +3D_{2}\int_{\Omega}w^2|\nabla w|^2\,dx+c_{5}\int_{\Omega} w^4\,dx
=c_{6}\int_{\Omega}w^3\,dx-c_{7}\int_{\Omega} v\beta_{\delta}(w)w^{3}\,dx,
\end{equation*}
while by dropping the nonnegative second  term on the left-hand side of previous identity and the last negative term on the right-hand side, we infer 
\begin{equation}\label{REV100}
\frac{1}{4}\frac{d}{dt}\int_{\Omega} w^4+c_{5}\int_{\Omega} w^4\,dx \leq c_{6}\int_{\Omega}w^3\,dx.
\end{equation}
Then, an application of Young's inequality ensures that
\begin{equation}\label{REV101} 
c_{6}\int_{\Omega}w^3\,dx\leq \frac{c_{5}}{2}\int_{\Omega} w^4\,dx+\frac{1}{4} \left( \frac{3}{2c_5} \right)^3 (c_6)^4 |\Omega|.
\end{equation}
Combining \eqref{REV101} with \eqref{REV100} gives 
\begin{equation*}
\frac{d}{dt}\int_{\Omega} w^4\,dx+2c_{5}\int_{\Omega} w^4\,dx\leq \left( \frac{3}{2c_5} \right)^3 (c_6)^4 |\Omega|,
\end{equation*}
whence, by an ODE comparison principle, we arrive at
\begin{equation}\label{Equation-4-87}
\|w(\cdot, t)\|^{4}_{L^4(\Omega)}\leq C_{\text{{\tiny $4,\!5,\!1$}}},\quad\mbox{for all}\; t\in [0,T],
\end{equation}
where the constant $C_{\text{{\tiny $4,\!5,\!1$}}}>0$ is given by
\begin{equation*}
C_{\text{{\tiny $4,\!5,\!1$}}}:=\sup\left\{\|w_{0}\|^{4}_{L^4(\Omega)}, \frac{1}{2c_{5}} \left( \frac{3}{2c_5} \right)^3 (c_6)^4 |\Omega \right\}.
\end{equation*}
\indent
Next, we consider the function $u$.
Since $u \in L^\infty (0,T\,;H^1(\Omega))$ and $H^1(\Omega) \hookrightarrow L^6(\Omega)$, we have $\{u(t)\}^3 \in L^\infty (0,T\,;L^2(\Omega))$,
which enables us to take the inner product in both sides of \eqref{Equation-4-39} in $L^2(\Omega)$ by $\{u(t)\}^3$.
From \eqref{Equation-2-3} we obtain the following inequality for a.e. $t \in (0,T)$:
\begin{equation}\label{Equation-4-88}
\frac{1}{4} \frac{d}{dt} \|u_{\text{{\tiny $\delta$}}}(t)\|_{L^4(\Omega)}^4-D_1\int_\Omega (\Delta u_{\text{{\tiny $\delta$}}}(t)) \{u_{\text{{\tiny $\delta$}}}(t)\}^3dx
\le c_1 \int_\Omega \{u_{\text{{\tiny $\delta$}}}(t)\}^4 v_{\text{{\tiny $\delta$}}}(t)dx.
\end{equation}
Since we have the following inequality:
\begin{align}
\label{Equation-4-89}
&\,-D_1\int_\Omega (\Delta_N u_{\text{{\tiny $\delta$}}}(t)) \{u_{\text{{\tiny $\delta$}}}(t)\}^3dx
=3D_1\int_\Omega \{u_{\text{{\tiny $\delta$}}}(t)\}^2 |\nabla (u_{\text{{\tiny $\delta$}}}(t)|^2dx\\
\nonumber
&\,\hspace*{2cm}=\frac{3D_1}{4}\int_\Omega |\nabla (u_{\text{{\tiny $\delta$}}}(t))^2|^2dx \le D_1 \|\Delta_N u(t)\|_{L^2(\Omega)} \|u(t)\|_{L^6(\Omega)}^3 < \infty
\end{align}
and by using the generalized H\"{o}lder inequality and Lemma \ref{Lemma-4-4}
\begin{align}\label{Equation-4-90}
\int_\Omega \{u_{\text{{\tiny $\delta$}}}(t)\}^4 v_{\text{{\tiny $\delta$}}}(t)dx &\le |\Omega|^{\frac{1}{4}} \|v_{\text{{\tiny $\delta$}}}(t)\|_{L^4(\Omega)} 
\left( \int_\Omega |u_{\text{{\tiny $\delta$}}}(t)|^8dx\right)^{\frac{1}{2}}\\
\nonumber
&\le C_{24} |\Omega|^{\frac{1}{4}} \left( \int_\Omega |u_{\text{{\tiny $\delta$}}}(t)|^8dx\right)^{\frac{1}{2}},
\end{align}
we substitute \eqref{Equation-4-89} and \eqref{Equation-4-90} into \eqref{Equation-4-88}.
Then, we obtain
\begin{equation}\label{Equation-4-91}
\frac{d}{dt} \|u(t)\|_{L^4(\Omega)}^4+3D_1 \|\nabla (u(t))^2\|_{\bm{L}^2(\Omega)}^2 \le 4c_{1} C_{24} |\Omega| \|u(t)\|_{L^8(\Omega)}^4,\quad \text{a.e.}~t \in (0,T).
\end{equation}
Subsequently, applying the following Gagliardo-Nirenberg inequality (cf. \eqref{Equation-4-60}): there exist constants $K_9>0$ and $K_{10}>0$ such that
\begin{equation*}
\|\eta \|^2_{L^4(\Omega)} \le K_9 \|\nabla \eta \|^{\frac{N}{2}}_{\bm{L}^2(\Omega)}
\|\eta \|^{\frac{4-N}{2}}_{L^2(\Omega)}+K_{10}\|\eta \|_{L^2(\Omega)}^{2}, \quad \forall \eta \in H^1(\Omega),
\end{equation*} 
we obtain
\begin{align}\label{Equation-4-92}
\|u(t)\|_{L^8(\Omega)}^4&=\|(u(t))^2\|_{L^4(\Omega)}^2\\
\nonumber
&\le K_9 \|\nabla (u(t))^2\|^{\frac{N}{2}}_{L^2(\Omega)}\|(u(t))^2\|^{\frac{4-N}{2}}_{L^2(\Omega)}+K_{10}\|(u(t))^2\|_{L^2(\Omega)}^{2}\\
\nonumber
&= K_9 \|\nabla (u(t))^2\|^{\frac{N}{2}}_{L^2(\Omega)}\|u(t)\|_{L^4(\Omega)}^{4-N}+K_{10}\|u(t)\|_{L^4(\Omega)}^{4},
\end{align}
as well as 
\begin{equation}\label{Equation-4-93}
\|\nabla (u(t))^2\|^{\frac{N}{2}}_{\bm{L}^2(\Omega)} \|u(t)\|^{4-N}_{L^4(\Omega)} 
\le \hat{\mu} \|\nabla (u(t))^2\|_{\bm{L}^2(\Omega)}^2+\frac{4-N}{4} \left(\frac{N}{4\hat{\mu}}\right)^{\frac{N}{4-N}} \|u(t)\|_{L^4(\Omega)}^4,
\end{equation}
where the number $\hat{\mu}>0$ is given by 
\begin{equation*}
\hat{\mu}:=\frac{D_1}{2c_1 K_9 C_{24} |\Omega|^{\frac{1}{4}}}.
\end{equation*}
Combining (\ref{Equation-4-92}) with (\ref{Equation-4-93}) and substituting its result into \eqref{Equation-4-91}, we obtain 
\begin{equation}\label{Equation-4-94}
\frac{d}{dt}\|u(t)\|_{L^4(\Omega)}^4+D_1 \|\nabla (u(t))^2\|_{\bm{L}^2(\Omega)}^2 \le C_{\text{{\tiny $4,\!5,\!2$}}} \|u(t)\|_{L^4(\Omega)}^4, \quad \text{a.e.}~t \in (0,T),
\end{equation}
where the constant $C_{\text{{\tiny $4,\!5,\!2$}}}>0$ is given by
\begin{equation*}
C_{\text{{\tiny $4,\!5,\!2$}}}:=4c_1 C_{24} |\Omega|^{\frac{1}{4}} \left\{ \frac{K_9(4-N)}{4} \left(\frac{N}{4\hat{\mu}}\right)^{\frac{N}{4-N}}+K_{10} \right\}.
\end{equation*}
Applying the Gronwall lemma to \eqref{Equation-4-94}, we obtain the following boundedness:
\begin{equation}\label{Equation-4-95}
\left( \sup_{0 \le t \le T} \|u(t)\|_{L^4(\Omega)} \right)^4 \le \|u_0\|_{L^2(\Omega)}^2 \exp \left( {\color{red}C_{\text{{\tiny $4,\!5,\!2$}}}}T \right).
\end{equation}
\indent
As a result of \eqref{Equation-4-87} and \eqref{Equation-4-95}, it follows that the required uniform estimate holds.
Hence, the proof is completed.
\end{proof}
%%%%%
%%%%%
%%%%%
By Lemmas \ref{Lemma-4-4} and \ref{Lemma-4-5}, we show the uniform estimates which enable us to proceed the limit procedure as $\delta \uparrow \infty$
in the following lemmas.
%%%%%
%%%%%
%%%%%
\begin{lemma}\label{Lemma-4-6}
There exists a constant $C_{27}>0$, which depends on the following values:
\begin{equation*}
T, \quad \|u_0\|_{H^1(\Omega)}, \quad \|v_0\|_{H^1(\Omega)}, \quad \|w_0\|_{H^1 (\Omega)},
\end{equation*}
such that the following uniform estimate holds:
\begin{align*}
&\sup_{\delta \in (1,\infty)} \left\{ \left( \sup_{0 \le t \le T} \|u_{\text{{\tiny $\delta$}}}(t)\|_{H^1(\Omega)}\right)^2+
\int_0^T \|u_{\text{{\tiny $\delta$}}}(t)\|_{H^2(\Omega)}^2dt\right.\\
&\hspace*{2.5cm}\left.+\left( \sup_{0 \le t \le T} \|w_{\text{{\tiny $\delta$}}}(t)\|_{H^1(\Omega)}\right)^2+
\int_0^T \|w_{\text{{\tiny $\delta$}}}(t)\|_{H^2(\Omega)}^2dt\right\} \le C_{27}.
\end{align*}
\end{lemma}
%%%%%
%%%%%
%%%%%
\begin{proof}[Proof.]
For simplicity, we also set $(u,v,w):=(u_{\text{{\tiny $\delta$}}},v_{\text{{\tiny $\delta$}}},w_{\text{{\tiny $\delta$}}})$.
Since $u \in L^2(0,T\,;H^2(\Omega))$, we take the inner product in both sides of \eqref{Equation-4-44} in $L^2(\Omega)$ by $-\Delta_N u(t)$.
Then, we obtain the following inequality for a.e. $t \in (0,T)$: 
\begin{align}
\label{Equation-4-96}
&\frac{1}{2}\frac{d}{dt} \|\nabla u(t)\|_{\bm{L}^2(\Omega)}^2+D_{1}\|\Delta_N u(t)\|_{L^2(\Omega)}^2\\
\nonumber
&\hspace*{1cm}=-c_1 \int_\Omega \beta_{\text{{\tiny $\delta$}}}(u(t)) v(t) \Delta_N u(t)dx+c_2 \int_\Omega \beta_{\text{{\tiny $\delta$}}}(u(t)) w(t) \Delta_N u(t)dx.
\end{align}
By Lemmas \ref{Lemma-4-4} and \ref{Lemma-4-5}, from \eqref{Equation-2-3} we obtain
\begin{align}
\label{Equation-4-97}
&\,-c_1 \int_\Omega \beta_{\text{{\tiny $\delta$}}}(u(t)) v(t) \Delta_N u(t)dx+c_{2}\int_\Omega \beta_{\text{{\tiny $\delta$}}}(u(t)) w(t) \Delta_N u(t)dx\\
\nonumber
\le&\,\|u(t)\|_{L^4 (\Omega)} \left( c_1 \|v(t)\|_{L^4 (\Omega)}+c_2 \|w(t)\|_{L^4 (\Omega)} \right)\|\Delta_N u(t)\|_{L^2 (\Omega)}\\
\nonumber
\le&\,\frac{D_1}{2} \|\Delta_N u(t)\|_{L^2 (\Omega)}^2+\frac{1}{2D_1} \left\{ C_{26} (c_1 C_{24}+c_2 C_{26}) \right\}^2.
\end{align}
Combining \eqref{Equation-4-97} with (\ref{Equation-4-96}) and integrating over any time interval $[0,t] \subset [0,T]$, 
we obtain the following inequality for all $t \in [0,T]$:
\begin{align*}
&\|\nabla u(t)\|_{\bm{L}^2(\Omega)}^2+D_1 \int_0^t \|\Delta_N u(s)\|_{L^2(\Omega)}^2 ds\\
&\hspace*{1cm}\le \|u_0\|_{H^1(\Omega)}^2+\frac{T}{D_1} \left\{ C_{26} (c_1 C_{24}+c_2 C_{26}) \right\}^2,
\end{align*}
which implies that the following estimate holds:
\begin{equation}\label{Equation-4-98}
\sup_{0 \le t \le T} \|\nabla u(t)\|_{\bm{L}^2(\Omega)}^2+\int_0^T \|\Delta_N u(t)\|_{L^2(\Omega)}^2 dt \le C_{\text{{\tiny $4,\!8,\!1$}}},
\end{equation}
where the constant $C_{\text{{\tiny $4,\!8,\!1$}}}>0$ is given by
\begin{equation*}
C_{\text{{\tiny $4,\!8,\!1$}}}:=\left(1+\frac{1}{D_1}\right) \left[ \|u_0\|_{H^1(\Omega)}^2+\frac{T}{D_1} \left\{ C_{26} (c_1 C_{24}+c_2 C_{26}) \right\}^2\right].
\end{equation*}
By Lemma \ref{Lemma-4-5} again, from \eqref{Equation-4-98} we obtain
\begin{equation}\label{Equation-4-99}
\left( \sup_{0 \le t \le T} \|u(t)\|_{H^1(\Omega)} \right)^2 \le |\Omega|^{\frac{1}{2}}(C_{26})^2+C_{\text{{\tiny $4,\!8,\!1$}}}.
\end{equation}
Moreover, we apply the Neumann elliptic regularity \eqref{Equation-1-8} and use Lemma \ref{Lemma-4-5} with \eqref{Equation-4-98} again.
Then, we obtain 
\begin{align}
\label{Equation-4-100}
\int_0^T \|u(t)\|_{H^2(\Omega)}^2 dt &\le 2(K_5)^2 \left\{ \int_0^T \|\Delta_N u(t)\|_{L^2(\Omega)}^2 dt+\int_0^T \|u(t)\|_{L^2(\Omega)}^2 dt \right\}\\
\nonumber
&\le 2(K_5)^2 \left\{ C_{\text{{\tiny $4,\!8,\!1$}}}+T |\Omega|^{\frac{1}{2}} (C_{26})^2 \right\}.
\end{align}
\indent
Secondly, we consider the function $w$.
Since $w \in L^2(0,T\,;H^2(\Omega))$, we take the inner product in both sides of \eqref{Equation-4-46} in $L^2(\Omega)$ by $-\Delta_N w(t)$.
Then, we obtain the following inequality for a.e. $t \in (0,T)$: 
\begin{align*}
&\,\frac{1}{2}\frac{d}{dt} \|\nabla w(t)\|_{\bm{L}^2(\Omega)}^2+D_{2}\|\Delta_N w(t)\|_{L^2(\Omega)}^2\\
=&\,c_5 \int_\Omega w(t) \Delta_N w(t)dx-c_6 \int_\Omega \Delta_N w(t)dx+c_7 \int_\Omega v(t) \beta_{\text{{\tiny $\delta$}}}(w(t)) \Delta_N w(t)dx\\
\le&\,c_5 \|w(t)\|_{L^2(\Omega)} \|\Delta_N w(t)\|_{L^2(\Omega)}+c_6 |\Omega|^{\frac{1}{2}} \|\Delta_N w(t)\|_{L^2(\Omega)}\\
&\,\hspace*{3cm}+c_7 \|v(t)\|_{L^4(\Omega)} \|w(t))\|_{L^4(\Omega)}\|\Delta_N w(t)\|_{L^2(\Omega)}.
\end{align*}
By Lemmas \ref{Lemma-4-4} and \ref{Lemma-4-5}, we repeat the similar arguments to the derivations of \eqref{Equation-4-99} and \eqref{Equation-4-100}.
Then, we see that there exists a constant $C_{\text{{\tiny $4,\!8,\!2$}}}>0$ such that 
\begin{equation}\label{Equation-4-101}
\left( \sup_{0 \le t \le T} \|w(t)\|_{H^1(\Omega)}\right)^2+\int_0^T \|w(t)\|_{H^2(\Omega)}^2dt \le C_{\text{{\tiny $4,\!8,\!2$}}}. 
\end{equation}
\indent
As a result of \eqref{Equation-4-99}--\eqref{Equation-4-101}, we obtain the required uniform estimate in this lemma.
The proof is completed.
\end{proof}
%%%%%
%%%%%
%%%%%
\begin{lemma}\label{Lemma-4-7}
There exists a constant $C_{28}>0$, which depends on the following values:
\begin{equation*}
T, \quad \|u_0\|_{H^1(\Omega)}, \quad \|v_0\|_{H^1(\Omega)}, \quad \|w_0\|_{H^1 (\Omega)},
\end{equation*}

such that the following uniform estimate holds:
\begin{align*}
&\sup_{\delta \in (1,\infty)} \left( \int_0^T \|u_{\text{{\tiny $\delta$}}}'(t)\|_{L^2(\Omega)}^2dt+
\int_0^T \|w_{\text{{\tiny $\delta$}}}'(t)\|_{L^2(\Omega)}^2dt\right) \le C_{28}.
\end{align*}
\end{lemma}
%%%%%
%%%%%
%%%%%
\begin{proof}[Proof.]
For simplicity we set $(u,v,w):=(u_{\text{{\tiny $\delta$}}},v_{\text{{\tiny $\delta$}}},w_{\text{{\tiny $\delta$}}})$.
By Lemmas \ref{Lemma-4-4} and \ref{Lemma-4-5} with \eqref{Equation-2-3} and \eqref{Equation-4-11}, we obtain the following estimates:
\begin{equation}\label{Equation-4-102}
\left\{
\begin{array}{l}
\|\beta_{\text{{\tiny $\delta$}}} (u(t)) v(t)\|_{L^2(\Omega)}^2 \le \|u(t)\|_{L^4(\Omega)}^2 \|v(t)\|_{L^4(\Omega)}^2 \le (C_{24} C_{26})^2,\\[0.2cm]
\|\beta_{\text{{\tiny $\delta$}}} (u(t)) w(t)\|_{L^2(\Omega)}^2 \le \|u(t)\|_{L^4(\Omega)}^2 \|w(t)\|_{L^4(\Omega)}^2 \le (C_{26})^4,\\[0.2cm]
\|v(t) \beta_{\text{{\tiny $\delta$}}} (w(t))\|_{L^2(\Omega)}^2 \le \|v(t)\|_{L^4(\Omega)}^2 \|w(t)\|_{L^4(\Omega)}^2 \le (C_{24} C_{26})^2.
\end{array}
\right.
\end{equation}
\indent
At first, from \eqref{Equation-4-44} we obtain 
\begin{gather}
\label{Equation-4-103}
\int_0^T \|u'(t)\|_{L^2(\Omega)}^2dt \le 3(D_1)^2 \int_0^T \|\Delta_N u(t)\|_{L^2(\Omega)}^2dt\\
\nonumber
+3(c_1)^2 \int_0^T \|u(t) v(t)\|_{L^2(\Omega)}^2dt+3(c_2)^2 \int_0^T \|u(t) w(t)\|_{L^2(\Omega)}^2dt.
\end{gather}
Combining Lemma \ref{Lemma-4-6} with \eqref{Equation-4-102} and \eqref{Equation-4-103}, we obtain the following estimate:
\begin{equation}\label{Equation-4-104}
\int_0^T \|u'(t)\|_{L^2(\Omega)}^2dt \le C_{\text{{\tiny $4,\!7,\!1$}}},
\end{equation}
where the constant $C_{\text{{\tiny $4,\!7,\!1$}}}>0$ is given by
\begin{equation*}
C_{\text{{\tiny $4,\!7,\!1$}}}:=3C_{27}(D_1)^2+3(c_1 C_{24} C_{26})^2+3(c_2)^2 (C_{26})^4.
\end{equation*}
\indent
Next, from \eqref{Equation-4-46} we obtain 
\begin{gather}
\label{Equation-4-105}
\int_0^T \|w'(t)\|_{L^2(\Omega)}^2dt \le (2D_1)^2 \int_0^T D_1\|\Delta_N w(t)\|_{L^2(\Omega)}^2dt\\
\nonumber
+(2c_5)^2 \int_0^T \|w(t)\|_{L^2(\Omega)}^2dt+(2c_7)^2 \int_0^T \|v(t) w(t)\|_{L^2(\Omega)}^2dt\\
\nonumber
+(2c_2)^2 \int_0^T \|u(t) w(t)\|_{L^2(\Omega)}^2dt+(2c_6)^2T|\Omega|.
\end{gather}
Combining Lemma \ref{Lemma-4-6} again with \eqref{Equation-4-102} and \eqref{Equation-4-105}, we obtain the following estimate:
\begin{equation}\label{Equation-4-106}
\int_0^T \|w'(t)\|_{L^2(\Omega)}^2dt \le C_{\text{{\tiny $4,\!7,\!2$}}},
\end{equation}
where the constant $C_{\text{{\tiny $4,\!7,\!2$}}}>0$ is given by
\begin{equation*}
C_{\text{{\tiny $4,\!7,\!2$}}}:=C_{27}(2D_2)^2+(2c_5 C_{26})^2 T |\Omega|^{\frac{1}{2}}+(2c_2 C_{24} C_{26})^2+(2c_6)^2 T |\Omega|.
\end{equation*}
\indent
As a result of \eqref{Equation-4-104} and \eqref{Equation-4-106}, we obtain the required uniform estimate in this lemma.
The proof is completed.
\end{proof}
%%%%%
%%%%%
%%%%%
\begin{lemma}\label{Lemma-4-8}
There exists a constant $C_{29}>0$, which depends on the following values:
\begin{equation*}
 T, \quad \|u_0\|_{H^1(\Omega)}, \quad \|v_0\|_{H^1(\Omega)}, \quad \|w_0\|_{H^1(\Omega)},
\end{equation*}
such that the following uniform estimate holds:
\begin{equation*}
\sup_{\delta \in (1,\infty)} \left\{ \|v_{\text{{\tiny $\delta$}}}'(t)\|_{L^\infty (0,T\,;L^2(\Omega))}+
\left( \sup_{0 \le t \le T} \|v_{\text{{\tiny $\delta$}}}(t)\|_{H^1(\Omega)} \right)^2 \right\} \le C_{29}.
\end{equation*}
\end{lemma}
%%%%%
%%%%%
%%%%%
\begin{proof}[Proof.]
We show this lemma by using the similar argument of Lemma \ref{Lemma-4-2}.
For simplicity we set $(u,v,w):=(u_{\text{{\tiny $\delta$}}},v_{\text{{\tiny $\delta$}}},w_{\text{{\tiny $\delta$}}})$, and repeat the similar argument to the proof of 
Lemma \ref{Lemma-4-2}, in which $z_{\text{{\tiny $\varepsilon$}}}$ is replaced by $v$.\\
We see from \eqref{Equation-4-45} and the initial condition $v(0)=v_0$ that the following equality holds for a.e. $(x,t) \in Q_T$, which is compared with 
\eqref{Equation-4-54}:
\begin{align}
\label{Equation-4-107}
v(x,t)&=v_0(x) \exp \left( -\int_0^t \int_{\Omega}K(x,y)\alpha_{\text{{\tiny $\delta$}}}(v(x,s))dyds\right.\\
\nonumber
&\hspace*{3cm}\left.+c_3 \int_0^t \beta_{\text{{\tiny $\delta$}}}(w(x,s))ds-c_4 \int_0^t \beta_{\text{{\tiny $\delta$}}}(u(x,s))ds \right).
\end{align}
\indent
Now, we differentiate (\ref{Equation-4-107}) with respect to space and use (c) of (A1).
Then, we obtain 
\begin{equation}\label{Equation-4-108}
\nabla v(x,t)=F_1(x,t)+F_2(x,t)+F_3(x,t)+F_4(x,t),
\end{equation}
where the functions $F_j: \Omega \times [0,T] \rightarrow \mathbb{R}^N~(j=1,2,3,4)$ are defined as follows:
\begin{align*}
F_1(x,t)&:=\exp\left(-\int_0^t \int_{\Omega}K(x,y)\alpha_{\text{{\tiny $\delta$}}}(v(y,s))\,dyds+c_3 \int_0^t \beta_{\text{{\tiny $\delta$}}}(w(x,s))\,ds\right.\\
&\hspace*{3cm}\left.-c_4 \int_0^t \beta_{\text{{\tiny $\delta$}}}(u(x,s))\,ds\right) \nabla v_0(x),\\
F_2(x,t)&:=v(x,t) \int_{0}^{t}\int_{\Omega} \nabla K(x,y)\alpha_{\text{{\tiny $\delta$}}}(v(x,s))\,dyds,\\
F_3(x,t)&:=c_3 v(x,t) \int_0^t \beta_{\text{{\tiny $\delta$}}}'(w(x,s)) \nabla w(x,s)\,ds,\\
F_4(x,t)&:=-c_4 v(x,t) \int_0^t \beta_{\text{{\tiny $\delta$}}}'(u(x,s)) \nabla u(x,s)\,ds.
\end{align*}
We multiply both sides of \eqref{Equation-4-108} by $\nabla v(t)$ and integrate the result over $\Omega$.
In the following argument, we use \eqref{Equation-2-1}, \eqref{Equation-2-3} and \eqref{Equation-2-5} repeatedly.\\
\indent
First, using Lemma \ref{Lemma-4-6} and the continuous embedding $H^{2}(\Omega)\hookrightarrow L^{\infty}(\Omega)$ (cf. \eqref{Equation-1-7}), we have 
the following inequality for all $t \in [0,T]$:
\begin{align}
\label{Equation-4-109}
\int_\Omega F_1(t) \cdot \nabla v(t)dx&\le \int_\Omega |\nabla v_0(x)| |\nabla v(x,t)| \exp \left( c_3 \int_0^t w(x,s)ds \right)dx\\
\nonumber
&\le \|\nabla v_0\|_{\bm{L}^2(\Omega)} \|\nabla v(t)\|_{\bm{L}^2(\Omega)} \exp \left( c_3 \int_0^T \|w(t)\|_{L^\infty (\Omega)}dt \right)\\
\nonumber
&\le \|\nabla v_0\|_{\bm{L}^2(\Omega)} \|\nabla v(t)\|_{\bm{L}^2(\Omega)} 
\exp \left( c_3 K_4 T^{\frac{1}{2}} \left( \int_0^T \|w(t)\|_{H^2(\Omega)}^2dt \right)^{\frac{1}{2}} \right)\\
\nonumber
&\le \frac{1}{8} \|\nabla v(t)\|_{\bm{L}^2(\Omega)}^2+2\|v_0\|_{H^1(\Omega)}^2 \exp \left( 2 c_3 K_4 (T C_{27})^{\frac{1}{2}} \right),
\end{align}
which is the same argumentation as the derivation of \eqref{Equation-4-57}.\\
\indent
Secondly, from (c) of (A3) and Lemma \ref{Lemma-4-4} we obtain 
\begin{align}
\label{Equation-4-110}
\int_\Omega F_2(t) \cdot \nabla v(t)dx&\le \int_0^t \int_\Omega |\nabla v(x,t)| |v(x,t)| \left( \int_\Omega |\nabla K(x,y)| |v(y,s)|\,dy \right) dxds\\
\nonumber
&\le \int_0^t \int_\Omega |\nabla v(x,t)| |v(x,t)| \left( \int_\Omega |\nabla K(x,y)|^2dy\right)^{\frac{1}{2}} \|v(s)\|_{L^2(\Omega)}dxds\\
\nonumber
&\le K_3 |\Omega| \|v(t)\|_{L^4(\Omega)} \|\nabla v(t)\|_{\bm{L}^2(\Omega)} \int_0^t \|v(s)\|_{L^4(\Omega)}ds\\
\nonumber
&\le \frac{1}{8} \|\nabla v(t)\|_{\bm{L}^2(\Omega)}^2+2(K_3 T |\Omega|)^2 (C_{24})^4,
\end{align}
which should be compared with the derivation of \eqref{Equation-4-58}.\\
\indent
Thirdly, for the rest terms we repeat the same arguments to the derivations of \eqref{Equation-4-62} and \eqref{Equation-4-63}.
Then, it follows that there exists a constant $C_{\text{{\tiny $4,\!8,\!1$}}}>0$ such that 
\begin{equation}\label{Equation-4-111}
\int_\Omega F_3(t) \cdot \nabla v(t)dx+\int_\Omega F_4(t) \cdot \nabla v(t)dx \le \frac{1}{4} \|\nabla v(t)\|_{\bm{L}^2(\Omega)}^2+C_{\text{{\tiny $4,\!8,\!1$}}}.
\end{equation}
\indent
Now, we combine all estimates \eqref{Equation-4-109}--\eqref{Equation-4-111}, and obtain 
\begin{equation*}
\|\nabla v(t)\|_{\bm{L}^2(\Omega)}^2=\sum_{j=1}^4 \int_\Omega F_j (t) \cdot \nabla v(t)dx 
\le \frac{1}{2} \|\nabla v(t)\|_{\bm{L}^2(\Omega)}^2+C_{\text{{\tiny $4,\!8,\!2$}}}, \quad \forall t \in [0,T],
\end{equation*}
where the constant $C_{\text{{\tiny $4,\!8,\!2$}}}>0$ is given by
\begin{equation*}
C_{\text{{\tiny $4,\!8,\!2$}}}:=2\|v_0\|_{H^1(\Omega)}^2 \exp \left( 2 c_3 K_4 (T C_{27})^{\frac{1}{2}} \right)+2(K_3 T |\Omega|)^2 (C_{24})^4+C_{\text{{\tiny $4,\!8,\!1$}}},
\end{equation*}
and thus
\begin{equation}\label{Equation-4-112}
\sup_{0 \le t \le T} \|\nabla v(t)\|_{\bm{L}^2(\Omega)}^2 \le C_{\text{{\tiny $4,\!8,\!2$}}}.
\end{equation}
\indent
In the rest of this proof, we estimate $v'$ in $L^2(0,T\,;L^2(\Omega))$.
We take the inner product in both sides of \eqref{Equation-4-45} in $L^2(\Omega)$ by $v'(t)$.
Since $(u,v,w) \in (L^\infty (0,T\,;L^4(\Omega)))^3$, we can apply the generalized H\"{o}lder inequality and obtain the following estimates:
\begin{align*}
&|(v(t)\beta_{\text{{\tiny $\delta$}}}(w(t)),v'(t))_{L^2(\Omega)}| \le \frac{1}{4c_3} \|v'(t)\|_{L^2(\Omega)}^2+c_3\|v(t)\|_{L^4(\Omega)}^2 \|w(t)\|_{L^4(\Omega)}^2,\\
&|(v(t)\beta_{\text{{\tiny $\delta$}}}(u(t)),v'(t))_{L^2(\Omega)}| \le \frac{1}{4c_4} \|v'(t)\|_{L^2(\Omega)}^2+c_4\|v(t)\|_{L^4(\Omega)}^2 \|u(t)\|_{L^4(\Omega)}^2,
\end{align*}
and from (c) of (A1)
\begin{align*}
\left((M_{\text{{\tiny $\delta,\!v(t)$}}})^{-1}v(t),v'(t)\right)_{L^2(\Omega)} &\le \int_\Omega \left(\int_\Omega K(x,y)v(y,t)dy\right) v(x,t) v'(x,t)dx\\
&\le \left\{ \sup_{x \in \Omega} \left(\int_\Omega |K(x,y)|^2dy \right)^{\frac{1}{2}}\right\} \|v(t)\|_{L^2(\Omega)}^2 \|v'(t)\|_{L^2(\Omega)}\\
&\le \frac{1}{4}\|v'(t)\|_{L^2(\Omega)}^2+|\Omega| (K_3)^2 \|v(t)\|_{L^4(\Omega)}^4.
\end{align*}
Combining all inequalities in the above with Lemmas \ref{Lemma-4-4} and \ref{Lemma-4-5}, we obtain the following inequality for a.e. $t \in (0,T)$:
\begin{align*}
\|v'(t)\|_{L^2(\Omega)}^2 &\le \frac{3}{4}\|v'(t)\|_{L^2(\Omega)}^2+(c_3)^2\|v(t)\|_{L^4(\Omega)}^2 \|w(t)\|_{L^4(\Omega)}^2\\
&\hspace*{1cm}+(c_4)^2\|v(t)\|_{L^4(\Omega)}^2 \|u(t)\|_{L^4(\Omega)}^2+|\Omega| (K_3)^2 \|v(t)\|_{L^4(\Omega)}^4\\
&\le \frac{3}{4}\|v'(t)\|_{L^2(\Omega)}^2+(c_3 C_{24} C_{26})^2+(c_4 C_{24} C_{26})^2+|\Omega| (K_3)^2 (C_{24})^4,
\end{align*}
which yields 
\begin{equation}\label{Equation-4-113}
\|v'(t)\|_{L^2(\Omega)}^2 \le 4 \left\{ (c_3 C_{24} C_{26})^2+(c_4 C_{24} C_{26})^2+|\Omega| (K_3)^2 (C_{24})^4 \right\}.
\end{equation}
\indent
Finally, from \eqref{Equation-4-112} and \eqref{Equation-4-113} the required uniform estimate is obtained, and the proof is completed.
\end{proof}
%%%%%
%%%%%
%%%%%
Now we are in a position to show Theorem \ref{Theorem-1-1}
%%%%%
%%%%%
%%%%%
\begin{proof}[Proof of the nonnegativity and (s1)--(s4) and (s7) in Definition \ref{Definition-1-1}.]
By Lemmas \ref{Lemma-4-4}--\ref{Lemma-4-8}, we see that there exists a sequence $\{\delta_n\}_{n \in \mathbb{N}}$ and a triplet $(u,v,w)$ such that 
the following convergences hold as $n \to \infty$:
\begin{gather}
\label{Equation-4-114}
\delta_n \downarrow 0,\\
\label{Equation-4-115}
u_n:=u_{\text{{\tiny $\delta_n$}}} \longrightarrow u \quad \left\{
\begin{array}{l}
\text{a.e. in} \quad Q_T,\\[0.1cm]
\text{in} \quad C([0,T]\,;L^2(\Omega)),\\[0.1cm]
\text{weakly in} \quad W^{1,2}(0,T\,;L^2(\Omega)),\\[0.1cm]
\text{weakly$^*$ in} \quad L^\infty (0,T\,;H^1(\Omega)),\\[0.1cm]
\text{weakly in} \quad L^2(0,T\,;H^2(\Omega)),
\end{array}
\right.\\
\label{Equation-4-116}
v_n:=v_{\text{{\tiny $\delta_n$}}} \longrightarrow v \quad \left\{
\begin{array}{l}
\text{a.e. in} \quad Q_T,\\[0.1cm]
\text{in} \quad C([0,T]\,;L^2(\Omega)),\\[0.1cm]
\text{weakly$^*$ in} \quad W^{1,\infty}(0,T\,;L^2(\Omega)),\\[0.1cm]
\text{weakly$^*$ in} \quad L^\infty (0,T\,;H^1(\Omega)),
\end{array}
\right.\\[0.1cm]
\label{Equation-4-117}
w_n:=w_{\text{{\tiny $\delta_n$}}} \longrightarrow w \quad \left\{
\begin{array}{l}
\text{a.e. in} \quad Q_T,\\[0.1cm]
\text{in} \quad C([0,T]\,;L^2(\Omega)),\\[0.1cm]
\text{weakly in} \quad W^{1,2}(0,T\,;L^2(\Omega)),\\[0.1cm]
\text{weakly$^*$ in} \quad L^\infty (0,T\,;H^1(\Omega)),\\[0.1cm]
\text{weakly in} \quad L^2(0,T\,;H^2(\Omega)),
\end{array}
\right.
\end{gather}
which implies that the limit triplet $(u,v,w)$ satisfies the regularity properties (s1)--(s3) and the initial conditions (s7) with the nonnegativities by 
combining with \eqref{Equation-4-73}: 
\begin{equation*}
u(x,t) \ge 0, \quad v(x,t) \ge 0, \quad w(x,t) \ge 0,\quad \text{a.e.}~(x,t) \in Q_T.
\end{equation*}
In order to show the property (s4), we use \eqref{Equation-2-3} and \eqref{Equation-2-4}.
For every $n \in \mathbb{N}$ we obtain that the following equality holds for all $\xi \in L^2(0,T\,;L^2(\Omega))$:
\begin{align}
\label{Equation-4-118}
&\int_0^T (u_n'(t),\xi (t))_{L^2(\Omega)}dt-D_1\int_0^T (\Delta_N u_n(t),\xi (t))_{L^2(\Omega)}dt\\
\nonumber
&\hspace*{0.5cm}=c_1 \int_0^T (\beta_{\text{{\tiny $\delta_n$}}}(u_n(t))v_n(t),\xi (t))_{L^2(\Omega)}dt
-c_2 \int_0^T (\beta_{\text{{\tiny $\delta_n$}}}(u_n(t))w_n(t),\xi (t))_{L^2(\Omega)}dt.
\end{align}
In this proof, we use the interpolation inequality:
\begin{equation}\label{Equation-4-119}
\|\eta\|_{L^4(\Omega)} \le \|\eta\|_{L^2(\Omega)}^{\frac{1}{4}} \|\eta\|_{L^6(\Omega)}^{\frac{3}{4}},\quad \forall \eta \in L^6(\Omega),
\end{equation}
and the continuous embedding $H^1(\Omega) \hookrightarrow L^6(\Omega)$: there exists a constant $K_{11}>0$ such that
\begin{equation}\label{Equation-4-120}
\|\eta\|_{L^6(\Omega)} \le K_{11} \|\eta\|_{H^1(\Omega)},\quad \forall \eta \in H^1(\Omega).
\end{equation}
Combining \eqref{Equation-4-115} with \eqref{Equation-4-119} and \eqref{Equation-4-120} and using Lemma \ref{Lemma-4-6}, we obtain 
\begin{align*}
&\,\sup_{0 \le t \le T} \|u_n(t)-u(t)\|_{L^4(\Omega)}\\
\le&\,(K_{11})^{\frac{3}{4}} \left( \max_{0 \le t \le T} \|u_n(t)-u(t)\|_{L^2(\Omega)} \right)^{\frac{1}{4}}
\left( \sup_{0 \le t \le T} \|u_n(t)-u(t)\|_{H^1(\Omega)} \right)^{\frac{3}{4}}\\
\le&\,(K_{11})^{\frac{3}{4}} \left\{ (C_{27})^{\frac{1}{2}}+\sup_{0 \le t \le T} \|u(t)\|_{H^1(\Omega)} \right\}^{\frac{3}{4}}
\left( \max_{0 \le t \le T} \|u_n(t)-u(t)\|_{L^2(\Omega)} \right)^{\frac{1}{4}},
\end{align*}
which implies 
\begin{equation}\label{Equation-4-121}
u_n \longrightarrow u \quad \text{in} \quad L^\infty (0,T\,;L^4(\Omega)) \quad \text{as} \quad n \to \infty.
\end{equation}
Similarly, we obtain the following convergences as $n \to \infty$:
\begin{equation}\label{Equation-4-122}
(v_n,w_n) \longrightarrow (v,w) \quad \text{in} \quad \left( L^\infty (0,T\,;L^4(\Omega)) \right)^2 \quad \text{as} \quad n \to \infty.
\end{equation}
\indent
Next, we use the generalized H\"{o}lder inequality, repeatedly.
Firstly, from Lemma \ref{Lemma-4-5} we obtain the following inequality for all $n \in \mathbb{N}$:
\begin{align*}
&\,\left| \int_0^T \int_\Omega \{v_n(t)-v(t)\} \beta_{\text{{\tiny $\delta_n$}}}(u_n(t)) \xi (t) dxdt \right|\\
\le&\,\int_0^T \|v_n(t)-v(t)\|_{L^4(\Omega)} \|u_n(t)\|_{L^4(\Omega)} \|\xi (t)\|_{L^2(\Omega)}dt\\
\le&\,T^{\frac{1}{2}} \left( \sup_{0 \le t \le T} \|u_n(t)\|_{L^4(\Omega)} \right) \left( \sup_{0 \le t \le T} \|v_n(t)-v(t)\|_{L^4(\Omega)} \right)
\left( \int_0^T \|\xi (t)\|_{L^2(\Omega)}^2dt \right)^{\frac{1}{2}}\\
\le&\,C_{26} T^{\frac{1}{2}} \left( \sup_{0 \le t \le T} \|v_n(t)-v(t)\|_{L^4(\Omega)} \right) \left( \int_0^T \|\xi (t)\|_{L^2(\Omega)}^2dt \right)^{\frac{1}{2}},
\end{align*}
which implies the following convergence by using \eqref{Equation-4-122}:
\begin{equation}\label{Equation-4-123}
\lim_{n \to \infty} \int_0^T \int_\Omega \{v_n(t)-v(t)\} \beta_{\text{{\tiny $\delta_n$}}}(u_n(t)) \xi (t) dxdt=0.
\end{equation}
Secondly, from Lemma \ref{Lemma-4-4} we obtain the following inequality for all $n \in \mathbb{N}$:
\begin{align*}
&\,\left| \int_0^T \int_\Omega v(t) \left\{ \beta_{\text{{\tiny $\delta_n$}}}(u_n(t))-\beta_{\text{{\tiny $\delta_n$}}}(u(t))\right\} \xi (t) dxdt \right|\\
\le&\,\int_0^T \|v(t)\|_{L^4(\Omega)} \|u_n(t)-u(t)\|_{L^4(\Omega)} \|\xi (t)\|_{L^2(\Omega)}dt\\
\le&\,T^{\frac{1}{2}} \left( \sup_{0 \le t \le T} \|v(t)\|_{L^4(\Omega)} \right) \left( \sup_{0 \le t \le T} \|u_n(t)-u(t)\|_{L^4(\Omega)} \right)
\left( \int_0^T \|\xi (t)\|_{L^2(\Omega)}^2dt \right)^{\frac{1}{2}}\\
\le&\,C_{24} T^{\frac{1}{2}} \left( \sup_{0 \le t \le T} \|u_n(t)-u(t)\|_{L^4(\Omega)} \right) \left( \int_0^T \|\xi (t)\|_{L^2(\Omega)}^2dt \right)^{\frac{1}{2}},
\end{align*}
which implies the following convergence by using \eqref{Equation-4-121}:
\begin{equation}\label{Equation-4-124}
\lim_{n \to \infty} \int_0^T \int_\Omega v(t) \left\{ \beta_{\text{{\tiny $\delta_n$}}}(u_n(t))-\beta_{\text{{\tiny $\delta_n$}}}(u(t))\right\} \xi (t) dxdt=0.
\end{equation}
Thirdly, since by Lemmas \ref{Lemma-4-4} and \ref{Lemma-4-5} we have not only
\begin{equation*}
\left| v \left\{ \beta_{\text{{\tiny $\delta_n$}}}(u)-u \right\} \xi \right| \le 2 |u| |v| |\xi|,\quad \text{a.e. in} \quad Q_T,
\end{equation*}
but also 
\begin{align*}
&\,\int_0^T \int_\Omega |u(t)| |v(t)| |\xi (t)| dxdt \le \int_0^T \|u(t)\|_{L^4(\Omega)} \|v(t)\|_{L^4(\Omega)} \|\xi (t)\|_{L^2(\Omega)}dt\\
\le&\,T^{\frac{1}{2}}\left( \sup_{0 \le t \le T} \|u(t)\|_{L^4(\Omega)}\right) \left( \sup_{0 \le t \le T} \|v(t)\|_{L^4(\Omega)}\right)
\left( \int_0^T \|\xi (t)\|_{L^2(\Omega)}^2dt \right)^{\frac{1}{2}}\\
\le&\,C_{24} C_{26} T^{\frac{1}{2}} \left( \int_0^T \|\xi (t)\|_{L^2(\Omega)}^2dt \right)^{\frac{1}{2}} < \infty,
\end{align*}
which guarantees $|u| |v| |\xi| \in L^1(Q_T)$, we apply the Lebesgue dominated convergence theorem and obtain 
\begin{equation}\label{Equation-4-125}
\lim_{n \to \infty} \int_0^T \int_\Omega v(t) \left\{ \beta_{\text{{\tiny $\delta_n$}}}(u(t))-u(t) \right\} \xi (t)dxdt=0
\end{equation}
by using the following convergence:
\begin{equation*}
\beta_{\text{{\tiny $\delta_n$}}}(u) \longrightarrow u \quad \text{a.e. in} \quad Q_T.
\end{equation*}
Using \eqref{Equation-4-123}--\eqref{Equation-4-125}, we obtain 
\begin{equation}\label{Equation-4-126}
\lim_{n \to \infty} \int_0^T (\beta_{\text{{\tiny $\delta_n$}}}(u_n(t))v_n(t),\xi (t))_{L^2(\Omega)}dt=\int_0^T (u(t) v(t),\xi (t))_{L^2(\Omega)}dt.
\end{equation}
\indent
Similarly, by replacing $v_n$ and $v$ with $w_n$ and $w$, respectively, in the argument of the derivation of \eqref{Equation-4-126}, we obtain
\begin{equation}\label{Equation-4-127}
\lim_{n \to \infty} \int_0^T (\beta_{\text{{\tiny $\delta_n$}}}(u_n(t))w_n(t),\xi (t))_{L^2(\Omega)}dt=\int_0^T (u(t) w(t),\xi (t))_{L^2(\Omega)}dt.
\end{equation}
\indent
Finally, we take the limit $n \to \infty$ in the both sides of \eqref{Equation-4-118}, and use \eqref{Equation-4-114}--\eqref{Equation-4-117},
\eqref{Equation-4-126} and \eqref{Equation-4-127}.
Then, we obtain the following equality for all $\xi \in L^2(0,T\,;L^2(\Omega))$:
\begin{align*}
&\int_0^T (u'(t),\xi (t))_{L^2(\Omega)}dt-D_1\int_0^T (\Delta_N u(t),\xi (t))_{L^2(\Omega)}dt\\
&\hspace*{0.5cm}=c_1 \int_0^T (u(t) v(t),\xi (t))_{L^2(\Omega)}dt-c_2 \int_0^T (u(t) w(t),\xi (t))_{L^2(\Omega)}dt,
\end{align*}
which implies that the property (s4) in Definition \ref{Definition-1-1} is satisfied.
\end{proof}
%%%%%
%%%%%
%%%%%
\begin{proof}[Proof of (s5) in Theorem \ref{Theorem-1-1}.]
From \eqref{Equation-4-45} in (b3) of Definition \ref{Definition-4-2}, for every $n \in \mathbb{N}$ we obtain that the following equality holds for all 
$\xi \in L^2(0,T\,;L^2(\Omega))$:
\begin{align}
\label{Equation-4-128}
&\int_0^T (v_n'(t),\xi (t))_{L^2(\Omega)}dt+\int_\Omega \left( \int_\Omega K(x,y)\alpha_{\text{{\tiny $\delta_n$}}}(v_n(y,t))dy \right) v_n(x,t) \xi (x,t)dxdt\\
\nonumber
&\hspace*{0.5cm}=c_3 \int_0^T (v_n(t)\beta_{\text{{\tiny $\delta_n$}}}(w_n(t)),\xi (t))_{L^2(\Omega)}dt
-c_4 \int_0^T (\beta_{\text{{\tiny $\delta_n$}}}(u_n(t))v_n(t),\xi (t))_{L^2(\Omega)}dt.
\end{align}
\indent
First of all, we obtain the following convergence by repeating the similar argumentation in the proof of (s4) in Definition \ref{Definition-1-1}:
\begin{equation}
\label{Equation-4-129}
\lim_{n \to \infty} \int_0^T (v_n(t)\beta_{\text{{\tiny $\delta_n$}}}(w_n(t)),\xi (t))_{L^2(\Omega)}dt=\int_0^T (v(t) w(t),\xi (t))_{L^2(\Omega)}dt.\\
\end{equation}
\indent
Secondly, we use the following equality for all $n \in \mathbb{N}$:
\begin{align}
\label{Equation-4-130}
&\,\int_0^T \int_\Omega \left( \int_\Omega K(x,y)\alpha_{\text{{\tiny $\delta_n$}}}(v_n(y,t))dy \right) v_n(x,t) \xi (x,t)dxdt\\
\nonumber
&\,\hspace*{1cm}-\int_0^T \int_\Omega \left( \int_\Omega K(x,y) v(y,t)dy \right) v(x,t) \xi (x,t)dxdt\\
\nonumber
=&\,\int_0^T \int_\Omega \left( \int_\Omega K(x,y)\alpha_{\text{{\tiny $\delta_n$}}}(v_n(y,t))dy \right) \{v_n(x,t)-v(x,t)\} \xi (x,t)dxdt\\
\nonumber
&\,\hspace*{1cm}+\int_0^T \int_\Omega \left( \int_\Omega K(x,y) \left\{ \alpha_{\text{{\tiny $\delta_n$}}}(v_n(y,t))
-\alpha_{\text{{\tiny $\delta_n$}}}(v(y,t)) \right\} dy \right) v(x,t) \xi (x,t)dxdt\\
\nonumber
&\,\hspace*{1cm}+\int_0^T \int_\Omega \left( \int_\Omega K(x,y) \left\{ \alpha_{\text{{\tiny $\delta_n$}}}(v(y,t))-v(y,t) \right\} dy \right) v(x,t) \xi (x,t)dxdt.
\end{align}
By Lemmas \ref{Lemma-4-4}, for the first term in the right-hand side of \eqref{Equation-4-130} we obtain 
\begin{align*}
&\,\left| \int_0^T \int_\Omega \left( \int_\Omega K(x,y)\alpha_{\text{{\tiny $\delta_n$}}}(v_n(y,t))dy \right) \{v_n(x,t)-v(x,t)\} \xi (x,t)dxdt \right|\\
\le&\,\int_0^T \int_\Omega \left( \int_\Omega |K(x,y)|^2dy \right)^{\frac{1}{2}} \|v_n(t)\|_{L^2(\Omega) } |v_n(x,t)-v(x,t)| |\xi (x,t)|dxdt\\
\le&\,K_3 \left( \max_{0 \le t \le T} \|v_n(t)\|_{L^2(\Omega)} \right) \left( \int_0^T \|v_n(t)-v(t)\|_{L^2(\Omega)}^2dt \right)^{\frac{1}{2}} 
\left( \int_0^T \|\xi (t)\|_{L^2(\Omega)}^2dt\right)^{\frac{1}{2}}\\
\le&\,K_3 (C_{24})^2 (T |\Omega|)^{\frac{1}{2}} \left(\max_{0 \le t \le T} \|v_n(t)-v(t)\|_{L^2(\Omega)} \right) 
\left( \int_0^T \|\xi (t)\|_{L^2(\Omega)}^2dt\right)^{\frac{1}{2}},
\end{align*}
which yields the following convergence:
\begin{equation}\label{Equation-4-131}
\lim_{n \to \infty} \int_0^T \int_\Omega \left( \int_\Omega K(x,y)\alpha_{\text{{\tiny $\delta_n$}}}(v_n(y,t))dy \right) \{v_n(x,t)-v(x,t)\} \xi (x,t)dxdt=0.
\end{equation}
For the second term in the right-hand side of \eqref{Equation-4-130} we obtain 
\begin{align*}
&\,\left| \int_0^T \int_\Omega \left( \int_\Omega K(x,y) \left\{ \alpha_{\text{{\tiny $\delta_n$}}}(v_n(y,t))-\alpha_{\text{{\tiny $\delta_n$}}}(v(y,t)) \right\}dy \right)
v(x,t) \xi (x,t)dxdt \right|\\
\le&\,\int_0^T \int_\Omega \left( \int_\Omega |K(x,y)|^2dy \right)^{\frac{1}{2}} \|v_n(t)-v(t)\|_{L^2(\Omega) } |v(x,t)| |\xi (x,t)|dxdt\\
\le&\,K_3 \left( \max_{0 \le t \le T} \|v_n(t)-v(t)\|_{L^2(\Omega)} \right) \left( \int_0^T \|v(t)\|_{L^2(\Omega)}^2dt \right)^{\frac{1}{2}} 
\left( \int_0^T \|\xi (t)\|_{L^2(\Omega)}^2dt\right)^{\frac{1}{2}},
\end{align*}
which yields the following convergence:
\begin{equation}\label{Equation-4-132}
\lim_{n \to \infty} \int_0^T \int_\Omega \left( \int_\Omega K(x,y) \left\{ \alpha_{\text{{\tiny $\delta_n$}}}(v_n(y,t))
-\alpha_{\text{{\tiny $\delta_n$}}}(v(y,t)) \right\}dy \right) v(x,t) \xi (x,t)dxdt=0.
\end{equation}
For the third term in the right-hand side of \eqref{Equation-4-130} we obtain the following inequality for a.e. $(x,t) \in Q_T$: 
\begin{gather*}
\left| \left( \int_\Omega K(x,y) \left\{ \alpha_{\text{{\tiny $\delta_n$}}}(v(y,t))-v(y,t) \right\} dy \right) v(x,t) \xi (x,t) \right|\\
\le 2\left| \left( \int_\Omega K(x,y) |v(y,t)|dy \right) v(x,t) \xi (x,t) \right| \le 2K_3 \|v(t)\|_{L^2(\Omega)} |v(x,t)| |\xi (x,t)|
\end{gather*}
and $2K_3 \|v\|_{L^2(\Omega)} |v| |\xi | \in L^1(Q_T)$.
Applying the Lebesgue dominated convergence theorem, we obtain 
\begin{equation}\label{Equation-4-133}
\lim_{n \to \infty} \int_0^T \int_\Omega\left( \int_\Omega K(x,y) \left\{ \alpha_{\text{{\tiny $\delta_n$}}}(v(y,t))-v(y,t) \right\} dy \right) v(x,t) \xi (x,t) dxdt=0
\end{equation}
by using the following convergence:
\begin{equation*}
\alpha_{\text{{\tiny $\delta_n$}}}(v) \longrightarrow v \quad \text{a.e. in} \quad Q_T.
\end{equation*}
Using \eqref{Equation-4-131}--\eqref{Equation-4-133}, we obtain
\begin{align}
\label{Equation-4-134}
&\,\lim_{n \to \infty} \int_0^T \int_\Omega \left( \int_\Omega K(x,y)\alpha_{\text{{\tiny $\delta_n$}}}(v_n(y,t))dy \right) v_n(x,t) \xi (x,t)dxdt\\
\nonumber
&\,\hspace*{1cm}=\int_0^T \int_\Omega \left( \int_\Omega K(x,y) v(y,t)dy \right) v(x,t) \xi (x,t)dxdt.
\end{align}
\indent
Finally, we take the limit $n \to \infty$ in the both sides of \eqref{Equation-4-128}, and use \eqref{Equation-4-114}--\eqref{Equation-4-117},
\eqref{Equation-4-126} and \eqref{Equation-4-134}.
Then, we obtain the following equality for all $\xi \in L^2(0,T\,;L^2(\Omega))$:
\begin{gather*}
\int_0^T (v'(t),\xi (t))_{L^2(\Omega)}dt+\int_\Omega \left( \int_\Omega K(x,y) v(y,t)dy \right) v(x,t) \xi (x,t)dxdt\\
\nonumber
=c_3 \int_0^T (v(t) w(t),\xi (t))_{L^2(\Omega)}dt-c_4 \int_0^T (u(t) v(t),\xi (t))_{L^2(\Omega)}dt,
\end{gather*}
which implies that the property (s5) in Definition \ref{Definition-1-1} is satisfied.
\end{proof}
%%%%%
%%%%%
%%%%%
\begin{proof}[Proof of (s6) in Theorem \ref{Theorem-1-1}.]
For every $n \in \mathbb{N}$ we obtain that the following equality holds for all $\xi \in L^2(0,T\,;L^2(\Omega))$:
\begin{gather}
\label{Equation-4-135}
\int_0^T (w_n'(t),\xi (t))_{L^2(\Omega)}dt-D_2\int_0^T (\Delta_N w_n(t),\xi (t))_{L^2(\Omega)}dt\\
\nonumber
+c_5\int_0^T (w_n(t),\xi (t))_{L^2(\Omega)}dt=c_6 \int_0^T (1,\xi (t))_{L^2(\Omega)}dt\\
\nonumber
-c_7 \int_0^T (v_n(t) \beta_{\text{{\tiny $\delta_n$}}}(w_n(t)),\xi (t))_{L^2(\Omega)}dt.
\end{gather}
We take the limit $n \to \infty$ in the both sides of \eqref{Equation-4-135}, and use \eqref{Equation-4-114}--\eqref{Equation-4-117} and \eqref{Equation-4-130}.
Then, we obtain the following equality for all $\xi \in L^2(0,T\,;L^2(\Omega))$:
\begin{gather*}
\int_0^T (w'(t)-D_2\Delta_N w(t)+c_5w(t),\xi (t))_{L^2(\Omega)}dt=\int_0^T (c_6 \cdot 1-c_7 v(t) w(t),\xi (t))_{L^2(\Omega)}dt,
\end{gather*}
which implies that the property (s6) in Theorem \ref{Theorem-1-1} is satisfied.
\end{proof}
%%%%%
%%%%%
%%%%%
\begin{remark}\label{Remark-4-1}
We assume that there exist constants $C_*>0$ and $C^*>0$ such that the following property is satisfied for all $t \in [0,T]$:
\begin{equation*}
C_* \le \int_\Omega K(x,y) v(y,t)dx \le C^*.
\end{equation*}
Under this assumption, we can consider the state-dependent spaces $\{L^2(v(t))\,;\,0 \le t \le T\}$, where $L^2(v(t))$ denotes a twisted space of $L^2(\Omega)$ 
endowed with the following inner product for each $t \in [0,T]$, which is analogous to \eqref{Equation-2-18}:
\begin{equation*}
(v_1,v_2)_{v(t)}:=\int_\Omega \left( \int_\Omega K(x,y) v(y,t)dy \right)^{-1} v_1 v_2 dx,\quad \forall v_1,v_2 \in L^2(\Omega).
\end{equation*}
Moreover, we denote by $\Phi_0$ the continuous convex function on $L^2(v(t))$ defined by
\begin{equation*}
\Phi (\bar{v}):=\frac{1}{2} \int_\Omega |\bar{v}(x)|^2dx,\quad \forall \bar{v} \in L^2(\Omega)=L^2(\bar{v}(t)).
\end{equation*}
Then, the perturbed evolution equation in (s5) of Definition \ref{Definition-1-1} is equivalent to the following perturbed twisted gradient flow:
\begin{gather*}
v'(t)+v^*(t)=v(t) \{c_3 w(t)-c_4 u(t)\} \quad \text{in} \quad L^2(v(t)),\quad \text{a.e.}~t \in (0,T),\\
v^*(t) \in \partial_{\text{{\tiny $v(t)$}}} \Phi (v(t)),\quad \text{a.e.}~t \in (0,T).
\end{gather*}
This perturbed twisted gradient flow is of particular interest since the state-dependent spaces $L^2(v(t))$ depend on the state $v(t)$ itself.
\end{remark}
\vspace{1cm}
\section*{Declarations}
\textbf{Conflict of interest} We declare that we have no conflict of interest.\\

%%% if required, the content of .bbl file can be included here once bbl is generated
%%% \input sn-article.bbl
\end{document}